\documentclass[a4paper]{amsart}
\usepackage[utf8]{inputenc}

\usepackage{amsmath,amsthm,amssymb,bbm,comment, mathrsfs, mathabx}
\usepackage{thmtools,mathtools}
\usepackage{graphicx,enumitem,stmaryrd}
\usepackage{multicol}
\usepackage{xparse,hyperref}
\usepackage{anyfontsize}
\SetSymbolFont{stmry}{bold}{U}{stmry}{m}{n}
\allowdisplaybreaks[4]
\usepackage[all]{xy}
\usepackage[table]{xcolor}
\usepackage[normalem]{ulem}
\usepackage{xcolor, bm, tensor}
\usepackage{tikz,tikz-cd}
\usetikzlibrary{matrix}
\usetikzlibrary{fit}
\usetikzlibrary{calc}

\makeatletter
\def\slashedarrowfill@#1#2#3#4#5{%
  $\m@th\thickmuskip0mu\medmuskip\thickmuskip\thinmuskip\thickmuskip
   \relax#5#1\mkern-7mu%
   \cleaders\hbox{$#5\mkern-2mu#2\mkern-2mu$}\hfill
   \mathclap{#3}\mathclap{#2}%
   \cleaders\hbox{$#5\mkern-2mu#2\mkern-2mu$}\hfill
   \mkern-7mu#4$%
}
\def\rightslashedarrowfill@{%
  \slashedarrowfill@\relbar\relbar\mapstochar\rightarrow}
\newcommand\xslashedrightarrow[2][]{%
  \ext@arrow 0055{\rightslashedarrowfill@}{#1}{#2}}
\makeatother

\makeatletter
\newcommand{\ostar}{\mathbin{\mathpalette\make@circled\star}}
\newcommand{\make@circled}[2]{%
  \ooalign{$\m@th#1\smallbigcirc{#1}$\cr\hidewidth$\m@th#1#2$\hidewidth\cr}%
}
\newcommand{\smallbigcirc}[1]{%
  \vcenter{\hbox{\scalebox{0.77778}{$\m@th#1\bigcirc$}}}%
}
\makeatother

\usepackage{adjustbox}

\newcommand{\yo}{\text{\usefont{U}{min}{m}{n}\symbol{'110}}}
\newcommand{\ta}{\text{\usefont{U}{min}{m}{n}\symbol{'037}}}
\newcommand{\na}{\text{\usefont{U}{min}{m}{n}\symbol{'052}}}

\DeclareFontFamily{U}{min}{}
\DeclareFontShape{U}{min}{m}{n}{<-> dmjhira}{}

\tikzcdset{scale cd/.style={every label/.append style={scale=#1},
    cells={nodes={scale=#1}}}}

\usepackage{kantlipsum} 

\calclayout

\usepackage[bb=dsserif]{mathalpha}

\usetikzlibrary{calc}
\usetikzlibrary{decorations.pathmorphing}
\tikzset{curve/.style={settings={#1},to path={(\tikztostart)
    .. controls ($(\tikztostart)!\pv{pos}!(\tikztotarget)!\pv{height}!270:(\tikztotarget)$)
    and ($(\tikztostart)!1-\pv{pos}!(\tikztotarget)!\pv{height}!270:(\tikztotarget)$)
    .. (\tikztotarget)\tikztonodes}},
    settings/.code={\tikzset{quiver/.cd,#1}
        \def\pv##1{\pgfkeysvalueof{/tikz/quiver/##1}}},
    quiver/.cd,pos/.initial=0.35,height/.initial=0}

\tikzset{tail reversed/.code={\pgfsetarrowsstart{tikzcd to}}}
\tikzset{2tail/.code={\pgfsetarrowsstart{Implies[reversed]}}}
\tikzset{2tail reversed/.code={\pgfsetarrowsstart{Implies}}}
\tikzset{no body/.style={/tikz/dash pattern=on 0 off 1mm}}

\definecolor{blue(pigment)}{rgb}{0.2, 0.2, 0.6}
\definecolor{americanrose}{rgb}{1.0, 0.01, 0.24}
\definecolor{nicegreen}{rgb}{0.0, 0.5, 0.0}
\definecolor{deepmagenta}{rgb}{0.8, 0.0, 0.8}
\definecolor{deepcarrotorange}{rgb}{0.91, 0.41, 0.17}
\definecolor{cadetgrey}{rgb}{0.57, 0.64, 0.69}

\newtheorem{theoremm}{Theorem}[section]
\newtheorem{theoremmm}{Theorem}
\declaretheorem[style=plain,name=Theorem,numberlike=theoremmm]{theoremat}

\declaretheorem[style=plain,name=Theorem,numberlike=theoremm]{theorem}
\declaretheorem[style=plain,name=Theorem,numbered=no]{theorem*}

\declaretheorem[style=plain,name=Lemma,numberlike=theoremm]{lemma}
\declaretheorem[style=plain,name=Proposition,numberlike=theoremm]{proposition}
\declaretheorem[style=plain,name=Corollary,numberlike=theoremm]{corollary}

\declaretheorem[style=definition,name=Definition,numberlike=theorem]{definition}
\declaretheorem[style=remark,name=Example,numberlike=theorem]{example}
\declaretheorem[style=remark,name=Remark,numberlike=theorem]{remark}
\declaretheorem[style=remark,name=Notation,numberlike=theorem]{notation}

\font\sc=rsfs10
\newcommand{\csym}[1]{\sc\mbox{#1}\hspace{1.0pt}}

\font\scc=rsfs7
\newcommand{\ccf}[1]{\scc\mbox{#1}\hspace{0.5pt}}

\font\sccc=rsfs5
\newcommand{\cccsym}[1]{\sccc\mbox{#1}\hspace{0.2pt}}

\newcommand{\on}[1]{\operatorname{#1}}
\newcommand{\setj}[1]{\left\{ #1 \right\}}
\newcommand{\hcomp}{\circ_{h}}

\newcommand{\xiso}{\xrightarrow{\sim}}

\DeclareMathSymbol{\blackdiamond}{\mathbin}{mathb}{"0C}

\DeclareMathAlphabet\EuRoman{U}{eur}{m}{n}
\SetMathAlphabet\EuRoman{bold}{U}{eur}{b}{n}
\newcommand{\euler}{\EuRoman}

\newcommand{\Hom}[1]{\left\langle #1 \right\rangle}
\newcommand{\Homint}[1]{\left\llbracket #1 \right\rrbracket}

\newcommand{\tam}{\csym{C}\!\on{-Tamb}(\csym{C},\csym{C})}

\newcommand{\cplus}[1]{\prescript{\ccf{C}\!}{+\!}{#1}}
\newcommand{\cotimes}{\overset{\cccsym{C}}{\otimes}}
\newcommand{\kotimes}{\otimes_{\Bbbk}}
\newcommand{\ktotimes}[1]{\overset{#1}{\otimes}}

\begin{document}
\title{Module categories, internal bimodules and Tambara modules}
\author{Mateusz Stroi{\' n}ski}

\begin{abstract}
 We use the theory of Tambara modules to extend and generalize the reconstruction theorem for module categories over a rigid monoidal category to the non-rigid case. We show a biequivalence between the $2$-category of cyclic module categories over a monoidal category $\ccf{C}$ and the bicategory of algebra and bimodule objects in the category of Tambara modules on $\ccf{C}$. Using it, we prove that a cyclic module category can be reconstructed as the category of certain free module objects in the category of Tambara modules on $\ccf{C}$, and give a sufficient condition for its reconstructability as module objects in $\ccf{C}$. To that end, we extend the definition of the Cayley functor to the non-closed case, and show that Tambara modules give a proarrow equipment for $\ccf{C}$-module categories, in which $\ccf{C}$-module functors are characterized as $1$-morphisms admitting a right adjoint. Finally, we show that the $2$-category of all $\ccf{C}$-module categories embeds into the $2$-category of categories enriched in Tambara modules on $\ccf{C}$, giving an ``action via enrichment'' result.

\vspace{2mm}
\end{abstract}

\maketitle

\tableofcontents

\section{Introduction}

Just like the main application of representation theory is describing linear algebraic symmetries controlled by an algebraic object, higher representation theory can be viewed as the study of categorical symmetries controlled by a monoidal category.
This representation theoretic perspective has been prominent ever since the concept of categorification was first formulated in \cite{Cr}, \cite{CF}.
Indeed, the development of higher representation theory also parallels that of classical representation theory: the study of categorical symmetries such as those exhibited by the Jones polynomial and described via Khovanov homology in \cite{Kh}, or those exhibited by the category $\mathcal{O}$ associated to a complex semisimple Lie algebra, and described via projective functors in \cite{BGG}, \cite{So}, has produced a number of monoidal categories defined to capture the observed categorical action.

This has deeply influenced the abstract study of monoidal categories suitable for categorification (thus, in particular, $\Bbbk$-linear over a ground field $\Bbbk$), and axiomatizations of their properties in notions such as tensor categories or fusion categories, both studied extensively in \cite{EGNO}.
This further led to the development of theories aiming at understanding abstractly the actions of a given monoidal category, or, again similarly to representation theory, understanding the monoidal category via its actions.
Notable examples of such theories include the study of $2$-Kac-Moody algebras of \cite{Rou}, the study of finitary $2$-representations initiated by Mazorchuk-Miemietz in \cite{MM1}, and the study of module categories over tensor categories, as described in \cite[Chapter~7]{EGNO}.

One of the main tools of the latter two is {\it internalization} of categorical actions - given a module category $\mathbf{M}$ over a monoidal category $\csym{C}$, one looks for an algebra object (which in this document is synonymous with a monoid object) $\mathrm{A}$ in $\csym{C}$ such that there is an equivalence, or embedding, of $\csym{C}$-module categories, from $\mathbf{M}$ to the category $\on{mod}_{\ccf{C}}\!\on{--}\!\!\mathrm{A}$ of right $\mathrm{A}$-module objects in $\csym{C}$.
The action of $\csym{C}$ on $\on{mod}_{\ccf{C}}\!\on{--}\!\!\mathrm{A}$ here is simply by tensoring over $\csym{C}$: the most elementary example is $\mathbf{Vec}_{\Bbbk}$ acting on $\on{mod}A$, for any $\Bbbk$-algebra $A$, where each vector space $V$ acts on $\on{mod}A$ as the functor sending a module $M$ to the module $V \otimes_{\Bbbk} M$.
Internalizing a module category gives it a much more explicit form which often greatly facilitates its study, and also closely ties the problems of understanding the structure of $\mathbf{M}$ with that of understanding the structure of $\mathrm{A}$.
Essentially all general results about internalizing module categories require three ingredients:
\begin{enumerate}
 \item\label{Ingredient1} some finiteness assumptions on $\csym{C}$ and $\mathbf{M}$;
 \item\label{Ingredient2} the presence of a generator $X \in \mathbf{M}$, i.e. an object such that the collection $\setj{\mathbf{M}\mathrm{F}X \; | \mathrm{F} \in \csym{C}}$ generates the category in some sense, e.g. under taking direct sums and direct summands;
 \item\label{Ingredient3} the assumption that $\csym{C}$ is rigid.
\end{enumerate}

Combining assumption \ref{Ingredient1} and  assumption \ref{Ingredient3} one concludes that, for any objects $Y,Z \in \mathbf{M}$, there is an object $\setj{Y,Z}$ representing the functor
\begin{equation}\label{Kopreschiff}
 \on{Hom}_{\mathbf{M}}(\mathbf{M}(-)Y,Z): \csym{C}^{\on{opp}} \rightarrow \mathbf{Vec}_{\Bbbk}; \text{ alternatively the functor }
 \on{Hom}_{\mathbf{M}}(Y,\mathbf{M}(-)Z): \csym{C} \rightarrow \mathbf{Vec}_{\Bbbk}.
\end{equation}
One then shows that the object $\setj{X,X}$ admits the structure of an algebra object and so we can define $\mathrm{A}$ as $\setj{X,X}$.

One of the earliest results of this kind is \cite[Theorem~1]{Os}:
\begin{theorem*}
 Let $\csym{C}$ be a semisimple rigid monoidal category with finitely many irreducible objects and irreducible unit object. Let $\mathbf{M}$ be a semisimple $\csym{C}$-module category. There is a semisimple algebra object $\mathrm{A} \in \csym{C}$ such that the $\csym{C}$-module categories $\mathbf{M}$ and $\on{mod}_{\ccf{C}}\!\on{--}\!\!\mathrm{A}$ are equivalent.
\end{theorem*}
Relaxing the semisimplicity assumption, we may instead assume that $\csym{C}$ is {\it finitary} i.e. that, as a category, it is equivalent to $A\!\on{--proj}$ for some finite-dimensional $\Bbbk$-algebra $A$. In this setting, we have the following result:
\begin{theorem*}[{\cite[Theorem~4.7]{MMMT}}]
 Let $\csym{C}$ be a finitary rigid monoidal category with an indecomposable unit object. Let $\mathbf{M}$ be a finitary $\csym{C}$-module category. There is a coalgebra object $\mathrm{C} \in \underline{\csym{C}}$ such that the $\csym{C}$-module categories $\mathbf{M}$ and $\on{comod}_{\ccf{C}}\mathrm{C}$ are equivalent.
\end{theorem*}

The two differences between these two results is that we obtain a coalgebra object rather than an algebra object, and that this coalgebra object lies in the {\it injective abelianization} $\underline{\csym{C}}$ of $\csym{C}$, rather than in $\csym{C}$ itself.
The injective abelianization is a diagrammatic construction introduced in \cite[Section~3.2]{MMMT}, which, as we show in this document, is monoidally equivalent to the finite completion of $\csym{C}$ endowed with the monoidal structure given by Day convolution.
The latter difference immediately indicates a strategy for removing the finiteness assumptions: in Section~\ref{s102} we show that the functor $\on{Hom}_{\mathbf{M}}(\mathbf{M}(-)X,X)$ is an algebra object in the presheaf category $[\csym{C}^{\on{opp}},\mathbf{Vec}_{\Bbbk}]$, again using Day convolution.
In the presence of finiteness assumptions ensuring the representability of this functor, the algebra object structure on $\setj{X,X}$ of \cite{Os} is obtained by pulling back the one we define on $\on{Hom}_{\mathbf{M}}(\mathbf{M}(-)X,X)$ along the Yoneda embedding. As we show in Theorem~\ref{Formulae}, if $\csym{C}$ is rigid, this allows us to realize $\mathbf{M}$ as the category of certain $\on{Hom}_{\mathbf{M}}(\mathbf{M}(-)X,X)$-modules in $[\csym{C}^{\on{opp}},\mathbf{Vec}_{\Bbbk}]$, namely those module objects which could be called {\it $\csym{C}$-projective}.

The fact that the object obtained in \cite{MMMT} is a coalgebra object comes from the fact that the injective abelianization is monoidally equivalent to the Isbell dual of the category of finite-dimensional presheaves on $\csym{C}$, i.e. $\underline{\csym{C}} \simeq [\csym{C},\mathbf{vec}_{\Bbbk}]^{\on{op}}$, and similarly to $\on{Hom}_{\mathbf{M}}(\mathbf{M}(-)X,X)$, the functor $\on{Hom}_{\mathbf{M}}(X,\mathbf{M}(-)X)$ used in \cite{MMMT} is an algebra object in the copresheaf category $[\csym{C},\mathbf{Vec}_{\Bbbk}]$. Thus, allowing the algebra and coalgebra objects to live in various cocompletions and completions of $\csym{C}$, we extend the results of \cite{Os}, \cite{MMMT} in the absence of finiteness/representability conditions, as long as we keep the rigidity assumption on $\csym{C}$.

The main aim of this document is to extend the internalization results beyond the rigid case. In order to do that, we need to embed $\csym{C}$ in yet another monoidal category, in which the monoid objects for the $\csym{C}$-module categories will live. This category is the category of (classical) Tambara modules over $\csym{C}$.

Tambara modules (under the name of {\it distributors of tensor categories}) were first introduced in \cite{Ta}.
A (classical) Tambara module is a profunctor $\csym{C} \xslashedrightarrow{} \csym{C}$ with a $\csym{C}$-action, satisfying certain coherence axioms.
In other words, a Tambara module is a functor $\mathtt{\Phi}: \csym{C}^{\on{opp}} \kotimes \csym{C} \rightarrow \mathbf{Vec}_{\Bbbk}$ together with maps $\ta_{\mathrm{H;F,G}}^{\mathtt{\Phi}}: \mathtt{\Phi}(\mathrm{F,G}) \rightarrow \mathtt{\Phi}(\mathrm{H}\cotimes \mathrm{F}, \mathrm{H}\cotimes \mathrm{G})$, for all $\mathrm{F,G,H} \in \csym{C}$, natural in $\mathrm{F,G}$ and extranatural in $\mathrm{H}$. Tambara modules can be composed similarly to profunctors, so we obtain a monoidal category $\csym{C}\!\on{-Tamb}(\csym{C},\csym{C})$.
In the introduction to \cite{Ta}, Tambara states:

``Such distributors arise in studying extensions of a tensor category.
Given a tensor functor $\mathcal{A} \rightarrow \mathcal{B}$, set
\[
L(X,Y) = \on{Hom}_{\mathcal{B}}(X,Y) \text{ for } X,Y \in \mathcal{A}\text{''. }
\]

Here, the tensor functor is assumed to be the identity on objects; for a general monoidal functor $\mathbb{F}: \csym{A} \rightarrow \csym{B}$, we write $\mathtt{\Phi}(\mathrm{K,L}) = \on{Hom}_{\ccf{B}}(\mathbb{F}K, \mathbb{F}L)$.

Observe that $\mathbb{F}$ endows $\csym{B}$ with the structure of an $\csym{A}$-module category via $\mathrm{A} \blackdiamond_{\ccf{B}} \mathrm{B} := \mathbb{F}(\mathrm{A}) \ktotimes{\cccsym{B}} \mathrm{B}$, for $\mathrm{A} \in \csym{A}$ and $\mathrm{B} \in \csym{B}$.
If $\mathbb{F}$ is essentially surjective, then the unit object $\mathbb{1}_{\ccf{B}}$ of $\csym{B}$ generates $\csym{B}$ under the $\csym{A}$-action.
We may thus write $\mathtt{\Phi}(\mathrm{K,L}) = \on{Hom}_{\ccf{B}}(\mathrm{K}\blackdiamond_{\ccf{B}} \mathbb{1}_{\ccf{B}}, \mathrm{L}\blackdiamond_{\ccf{B}} \mathbb{1}_{\ccf{B}})$.
Generalizing to an arbitrary $\csym{C}$-module category $\mathbf{M}$ with a generator $X$, we obtain the Tambara module $[X,X]$ given by $[X,X]({K,L}) = \on{Hom}_{\mathbf{M}}(\mathbf{M}(\mathrm{K})X,\mathbf{M}(\mathrm{L})X)$.
Observe that this is a ``two-sided'' version of the (co)presheaf considered in Equation~\eqref{Kopreschiff}.
This observation not only correctly indicates that $[X,X]$ is the monoid object we want, but in a sense serves as a guiding principle for this document.

Indeed, for $\csym{C}$-module categories $\mathbf{M}, \mathbf{N}$ with generators $X \in \mathbf{M}, Y \in \mathbf{N}$ and a $\csym{C}$-module functor $\euler{\Phi}: \mathbf{M} \rightarrow \mathbf{N}$, the Tambara module $\on{Hom}_{\mathbf{N}}(\mathbf{N}(-)Y,\mathbf{N}(-)\euler{\Phi}(X))$ is naturally a $[Y,Y]$-$[X,X]$-bimodule and, similarly, $\csym{C}$-module transformations give rise to bimodule morphisms.
However, not every bimodule gives a $\csym{C}$-module functor. From the category theoretic point of view, bimodules should correspond to ``$\csym{C}$-module profunctors''.
We argue that the correct notion of a $\csym{C}$-module profunctor is again that of a Tambara module, in the generalized sense introduced in \cite{CEGLMPR}.
A (generalized) Tambara module $\mathbf{M} \xslashedrightarrow{} \mathbf{N}$ is a profunctor $\mathbf{N}^{\on{op}} \kotimes \mathbf{M} \rightarrow \mathbf{Vec}_{\Bbbk}$ together with a $\csym{C}$-action similar to the classical case.

It should be observed here that the main source of current interest in Tambara modules lies in the study of the categorical aspects of functional programming, as for instance in \cite{CEGLMPR}.

Before describing the paper in greater detail, we list our main results, most of which we have already indicated.
\begin{theoremat}[{Theorem~\ref{Section4Main}}]\label{IntroThm1}
 The proarrow equipment $\mathbb{P}:\csym{C}\!\on{-Mod}\rightarrow \csym{C}\!\on{-Tamb}$ is a map equipment. A Tambara module between Cauchy complete $\csym{C}$-module categories has a right adjoint if and only if it is representable.
\end{theoremat}
Here, $\mathbb{P}$ is the pseudofunctor from the bicategory of $\csym{C}$-module categories, $\csym{C}$-module functors and $\csym{C}$-module transformations, to the bicategory of $\csym{C}$-module categories, (generalized) Tambara modules and their morphisms. It being a map equipment means that it is locally full and faithful, and that each $\csym{C}$-module functor admits a right adjoint in $\csym{C}\!\on{-Tamb}$. The latter part of the claim shows that a Tambara module admitting a right adjoint is necessarily isomorphic to a Tambara module given by a $\csym{C}$-module functor. This is completely analogous to the map equipment $\mathbf{Cat}\rightarrow \mathbf{Prof}$, making precise the claim that Tambara modules are to $\csym{C}$-module functors as profunctors are to functors.
\begin{theoremat}[{Theorem~\ref{MainThm}}]\label{IntroThm2}
 The pseudofunctor $\na: \csym{C}\!\on{-Tamb}_{\euler{0}} \rightarrow \on{Bimod}(\csym{C}\!\on{-Tamb}(\csym{C},\csym{C}))$ is a biequivalence.
\end{theoremat}
This is arguably {\it the} main theorem of the paper, showing not only an internalization statement, but a full biequivalence between the bicategory of $\csym{C}$-module categories with a cyclic generator, Tambara modules between these and their morphisms, and the bicategory of algebra objects in $\csym{C}\!\on{-Tamb}(\csym{C},\csym{C})$, bimodule objects between them and bimodule morphisms.

Combining Theorem~\ref{IntroThm1} with Theorem~\ref{IntroThm2}, we find the following:
\begin{theoremat} [{Theorem~\ref{MainConsequence}}]\label{IntroThm3}
 Let $\mathbf{M,N}$ be cyclic $\csym{C}$-module categories. Choose a cyclic generator $X$ for $\mathbf{M}$ and a cyclic generator $Y$ for $\mathbf{N}$. The module categories $\mathbf{M,N}$ are equivalent if and only if the algebra objects $[X,X], [Y,Y]$ in $\csym{C}\!\on{-Tamb}(\csym{C},\csym{C})$ are Morita equivalent.
\end{theoremat}

Using the proof of Theorem~\ref{MainThm}, one easily concludes that we may reconstruct a cyclic $\csym{C}$-module category $\mathbf{M}$ as the category of ``$\csym{C}$-projective'' $[X,X]$-module objects in $\csym{C}\!\on{-Tamb}(\csym{C},\csym{C})$; this is Lemma~\ref{MXCPlus}. We also find a canonical morphism $\omega_{X,X}$, ``from $\on{Hom}_{\mathbf{M}}(\mathbf{M}(-)X,X)$ to $\on{Hom}_{\mathbf{M}}(\mathbf{M}(-)X,\mathbf{M}(-)X)$'', which, in some sense, measures how well the one-sided construction in the presheaf category approximates the full, two-sided Tambara module. If $\omega_{X,X}$ is invertible, we may again reconstruct $\mathbf{M}$ as the category of free $\setj{X,X}$-modules in $\csym{C}$:
\begin{theoremat}[{Theorem~\ref{Formulae} and Corollary~\ref{ExplicitEquations}}]\label{IntroThm4}
 If $\omega_{X,X}$ is an isomorphism, then there is an equivalence of \mbox{$\csym{C}$-module} categories between $\mathbf{M}$ and the category of ``$\csym{C}$-projective'' $\on{Hom}_{\mathbf{M}}(\mathbf{M}(-)X,X)$-modules.
 If $\on{Hom}_{\mathbf{M}}(\mathbf{M}(-)X,X)$ is representable, then $\omega_{X,X}$ is invertible if and only if
 \[
 \begin{aligned}
  \csym{C}(\mathrm{F,G}\cotimes \setj{X,X}) &\rightarrow \Hom{\mathbf{M}\mathrm{F}X,\mathbf{M}\mathrm{G}X} \\
  \mathrm{f} &\mapsto \mathbf{M}\mathrm{G}\on{ev}_{X} \circ \mathbf{m}_{\mathrm{G},\setj{X,X}} \circ \mathbf{M}\mathrm{f}_{X}
 \end{aligned}
 \]
 is invertible, for all $\mathrm{F,G} \in \csym{C}$. In that case, there is an equivalence of $\csym{C}$-module categories between $\mathbf{M}$ and the Cauchy completion of the category of free $\setj{X,X}$-module objects in $\csym{C}$.
\end{theoremat}

Finally, if we do not want to assume the existence of a generator $X$, we may follow the approach considered in much greater generality in the category theoretic setting by \cite{GP}. In our case it specializes to the observation that, if $\csym{C}$ is rigid and for all objects $Y,Z \in \mathbf{M}$ the internal $\on{Hom}$ object $\setj{Y,Z}$ of $\csym{C}$ exists, then we obtain a $\csym{C}$-enriched category $\mathcal{M}$ with $\on{Ob}\mathbf{M} = \on{Ob}\mathcal{M}$ and $\on{Hom}_{\mathcal{M}}(Y,Z) = \setj{Y,Z}$. This gives a $2$-functor $\csym{C}\!\on{-IntMod} \rightarrow \csym{C}\!\mathbf{-Cat}$ from the $2$-ca\-te\-gory of internalizable $\csym{C}$-module categories to that of $\csym{C}$-enriched categories. This $2$-functor is a $2$-equivalence onto its essential image (which is given by $\csym{C}$-tensored $\csym{C}$-categories). We generalize this beyond the rigid or internalizable case:
\begin{theoremat}[{Theorem~\ref{ActionViaEnrichment}}]\label{IntroThm5}
 There is a $2$-functor $\mathbb{S}: \csym{C}\!\on{-Mod} \rightarrow \csym{C}\!\on{-Tamb}(\csym{C},\csym{C})^{\otimes\!\on{opp}}\!\mathbf{-Cat}$ which is a $2$-equivalence onto its essential image.
\end{theoremat}

The paper is organized as follows.
\begin{itemize}
 \item
Section~\ref{s2} consists of a summary of notational conventions as well as the necessary preliminaries on module categories and Tambara modules.
\item In \cite{PS}, it is shown that the category of (classical) Tambara modules is an Eilenberg-Moore category for a monad on $[\csym{C} \kotimes \csym{C}^{\on{opp}}, \mathbf{Vec}_{\Bbbk}]$. Section~\ref{s3} recalls the description of free objects in this Eilenberg-Moore category.
\item Section~\ref{s4} establishes Theorem~\ref{IntroThm1}, by first showing that $\csym{C}\!\on{-Mod}$ acts on $\csym{C}\!\on{-Tamb}$ via restrictions, similarly to the bicategory of $\Bbbk$-algebras acting on the bicategory of bimodules over $\Bbbk$-algebras by restrictions which ``twist'' the module structures.
It also shows that categories of Tambara modules are {\it tame}, i.e. satisfy cocompleteness conditions necessary for the categories of internal bimodules to be well-behaved.
\item Section~\ref{DefiningNa} defines the pseudofunctor $\na$ of Theorem~\ref{IntroThm2}, in particular showing its pseudofunctoriality, and showing that the definition of $\na$ does not depend on the choice of generators for our module categories - all choices give equivalent pseudofunctors.
\item Section~\ref{s6} shows that $\na$ is compatible with restrictions, i.e. that $\na$ is ``natural in $\csym{C}$''.
\item Section~\ref{s7} shows that $\na$ is essentially surjective.
\item Section~\ref{s8} shows that $\na$ is locally an equivalence.

It should be remarked that, with the exception of Section~\ref{s12}, we do not pass to strictifications to omit the coherence cells in our computations. Some of the applications we have in mind crucially use these cells, so our proofs rely on explicit calculations inevitably involving coherence cells. As a consequence of this approach, we need to tackle a minor coherence problem in Section~\ref{s81}, using a strategy similar to MacLane's proof of the coherence theorem for monoidal categories.
 \item Section~\ref{s9} states and shows Theorem~\ref{IntroThm3}.
 \item Section~\ref{s101} recalls the Cayley functor $[\csym{C},\mathbf{Vec}_{\Bbbk}] \rightarrow \csym{C}\!\on{-Tamb}(\csym{C},\csym{C})$ defined by \cite{PS} in the case of right-closed $\csym{C}$. We give a different interpretation of this functor, which allows us to easily extend it to the general case (relaxing the closedness assumption).
 \item In Section~\ref{s102}, we prove Corollary~\ref{OstrikMonoidStructure}, stating that the algebra and coalgebra structures of \cite{Os} and \cite{MMMT} can be defined already on the level of (co)presheaf categories. After that, we define the morphism $\omega_{X,X}$ of Theorem~\ref{IntroThm4} and, in \ref{RigidMonoids}, we show that it is a monoid morphism which is invertible if $\csym{C}$ is rigid.
 \item In Section~\ref{s103} we show that the central notion of $2$-representation theory, that is, the notion of a simple transitive $2$-repre\-sentation, corresponds precisely to the simplicity of the monoid associated to the $2$-representation. This is Corollary~\ref{NaSimpleSimple}. We then show that the projective and injective abelianizations are given by finite cocompletion and completion, and, in Corollary~\ref{ReproveMMMT} reprove the main results of \cite{MMMT} using the biequivalence $\na$.
 \item
Section~\ref{s104} uses the theory of monoidal posets to provide an example of an infinite family of module categories over a fixed monoidal category which cannot be distinguished using their Ostrik algebras or MMMTZ coalgebras, showing that these do not provide complete invariants of module categories, and that in the non-rigid case our results are strictly stronger than those of \cite{Os} and \cite{MMMT}.
\item
Section~\ref{s11} shows that our results can be used to reconstruct, or even internalize, cyclic module categories - its main result is Theorem~\ref{IntroThm4}.
\item
Finally, Section~\ref{s12} shows that (classical) Tambara modules can be used to give a ``action via enrichment'' result in the non-rigid case, namely Theorem~\ref{IntroThm5}.
\end{itemize}

\vspace{5mm}
\textbf{Acknowledgements.}
This research is partially supported by Göran Gustafssons Stiftelse. The author would like to thank his advisor Volodymyr Mazorchuk for many helpful comments and discussions.
Section~\ref{s12} and Section~\ref{s104} were heavily influenced by the ``Bicategories'' course given at Research School on Bicategories, Categorification and Quantum Theory at University of Leeds in July, 2022. The author would thus also like to thank the organizers of this research school, as well as Richard Garner and Christina Vasilakopoulou who taught the course. Finally, the author would like to thank the referee for very valuable remarks.

\section{Preliminaries}\label{s2}

Throughout, we fix a field $\Bbbk$. All the categories, functors, transformations, monoidal categories, bicategories etc. are assumed to be $\Bbbk$-linear. In other words, we implicitly work in the $\mathcal{V}$-enriched setting for $\mathcal{V} = \mathbf{Vec}_{\Bbbk}$. Our arguments apply for similarly for other familiar choices of $\mathcal{V}$ where one can prove the commutativity of diagrams on the level of elements, such as $\mathbf{Ab}, \mathbf{Set}$ or $R\!\on{-Mod}$ for $R$ a commutative domain.

\subsection{Notational conventions}

The notational conventions specified below will be followed throughout the document, unless otherwise stated. The definitions of some of the notions below are recalled later in this section.

\begin{itemize}
  \item Categories are denoted by calligraphic capital letters, $\mathcal{A,B,C}$ and the like.
 \item Functors are denoted by $\euler{F,G}$ and the like.
 \item Natural transformations are denoted by $\euler{t,d,g,k}$ and the like.
 \item Objects in an ordinary category are denoted by $X,Y,Z$, and the like.
 \item Morphisms in an ordinary category are denoted by $f,g,h$, and the like. The identity morphism of an object $X \in \mathcal{C}$ is denoted by $\on{id}_{X}$.
 \item Monoidal categories are denoted denoted by $\csym{C},\csym{D}$ and the like.
 \item Objects of a monoidal category $\csym{C}$ are denoted by roman capital letters, $\mathrm{F,G,H}$ and the like.
 \item Morphisms of a monoidal category $\csym{C}$ are denoted by roman lowercase letters, $\mathrm{f,g,h}$ and the like.
 \item We denote the tensor product functor of a monoidal category $\csym{C}$ by $-\cotimes -$. Thus, for any objects $\mathrm{F,G}$, we obtain objects $\mathrm{F}\cotimes \mathrm{G}$ and $\mathrm{G} \cotimes \mathrm{F}$. If there is no risk of ambiguity, we may omit the superscript or the tensor product symbol itself, thus denoting $\mathrm{F} \cotimes \mathrm{G}$ by $\mathrm{FG}$, and $\mathrm{G} \cotimes \mathrm{F}$ simply by $\mathrm{GF}$.
 If $\csym{C}$ is the bicategory $\on{Bimod}(\mathrm{A})$ of bimodule objects over a monoid object $\mathrm{A}$ in a monoidal category $\csym{D}$, we may alternatively write $\mathrm{F} \cotimes \mathrm{G}$ as $\mathrm{F} \otimes_{\mathrm{A}} \mathrm{G}$.
 \item The unit object of a monoidal category $\csym{C}$ is denoted by $\mathbb{1}_{\!\ccf{C}}$. If there is no risk of ambiguity, the superscript will be omitted.
 \item The left unitor, right unitor and associator transformations for a monoidal category $\csym{C}$ are denoted by $\mathsf{l}^{\ccf{C}},\mathsf{r}^{\!\ccf{C}},\mathsf{a}^{\!\ccf{C}}$. If there is no risk of ambiguity, the superscript will be omitted. For any $\mathrm{F,G,H} \in \on{Ob}\csym{C}$, we obtain the component isomorphisms $\mathsf{l}^{\ccf{C}}_{\mathrm{F}}: \mathbb{1}_{\!\ccf{C}} \cotimes \mathrm{F} \xiso \mathrm{F}, \; \mathsf{r}^{\!\ccf{C}}_{\mathrm{F}}: \mathrm{F}\mathbb{1} \xiso \mathrm{F}$ and $\mathsf{a}^{\!\ccf{C}}_{\mathrm{F,G,H}}: (\mathrm{FG})\mathrm{H} \xiso \mathrm{F}(\mathrm{GH})$.
 \item More generally, we use superscript decorations to denote the structures the decorated objects belong to, and subscript decorations to denote the components of the decorated objects.
 Identity objects, functors and morphisms are an exception to this.
 \item The identity morphisms of a $\csym{C}$-enriched category $\mathcal{A}$ are denoted by $\mathsf{e}^{\mathcal{A}}_{X}: \mathbb{1}_{\ccf{C}} \rightarrow \mathcal{A}(X,X)$, for any $X \in \mathcal{A}$. The composition morphisms of $\mathcal{A}$ are denoted by $\mathsf{c}^{\mathcal{A}}_{Y;X,Z}: \mathcal{A}(Y,Z) \cotimes \mathcal{A}(X,Y) \rightarrow \mathcal{A}(X,Z)$.
 \item $\csym{C}$-module categories are denoted by boldface capital letters, $\mathbf{K,M,N}$ and the like.
 \item Structure morphisms of a $\csym{C}$-module category $\mathbf{M}$ are denoted by $\mathbf{m}$; the unitality morphisms and the multiplicativity morphisms are distinguished by the number of subscripts. We thus have the invertible transformation $\mathbf{m}_{\mathbb{1}}: \mathbb{1}_{\mathbf{M}} \xiso \mathbf{M}\mathbb{1}_{\ccf{C}}$, and, for any $\mathrm{F,G} \in \csym{C}$, we obtain an invertible transformation $\mathbf{m}_{\mathrm{F,G}}: \mathbf{M}\mathrm{G} \mathbf{M}\mathrm{F} \xiso \mathbf{M}\mathrm{GF}$.
 \item $\csym{C}$-module functors are denoted by capital Greek letters, $\euler{\Phi, \Psi, \Sigma}$ and the like.
 \item $\csym{C}$-module transformations are denoted by $\euler{t,d,g,k}$ and the like. This notation coincides with our notation for ordinary natural transformations; we will specify the $\csym{C}$-module property explicitly.
 \item Tambara modules are denoted by uppercase typewriter Greek letters, $\mathtt{\Phi, \Psi, \Sigma}$ and the like. Composition of Tambara modules is denoted by $- \diamond -$.
 \item Tambara morphisms are denoted by typewriter lowercase letters, $\mathtt{t,d,g,k}$ and the like.
 \item We apply the same conventions for profunctors without Tambara structure as for Tambara modules.
 \item Bicategories are denoted by $\csym{A},\csym{B}$ and the like. This coincides with our notation for monoidal categories; in case of ambiguity we will explicitly specify which kind of structure is considered.
 \item Pseudofunctors are denoted by $\mathbb{F,G}$ and the like.
 \item Pseudonatural transformations are denoted by $\mathbb{t,d,g,k}$ and the like.
 \item We denote by $\mathbf{Cat}_{\Bbbk}$ the closed monoidal bicategory of ($\Bbbk$-linear) categories, functors and natural transformations. We denote the tensor product of ($\Bbbk$-linear) categories $\mathcal{A,B}$ by $\mathcal{A} \kotimes \mathcal{B}$.
 \item Given objects $X,Y$ of a category $\mathcal{A}$, we will interchangeably denote the space of morphisms in $\mathcal{A}$ from $X$ to $Y$ by $\on{Hom}_{\mathcal{A}}(X,Y), \mathcal{A}(X,Y)$ and $\Hom{X,Y}^{\mathcal{A}}$.
\end{itemize}

The usual notation for oppositization of categories may cause confusion. Viewing a monoidal category $\csym{C}$ as a category with additional structure, the opposite category $\csym{C}^{\on{op}}$ oppositizes the morphisms of $\csym{C}$. If we instead identify $\csym{C}$ with its delooping bicategory, $\csym{C}^{\on{op}}$ would usually correspond to oppositizing the $1$-morphisms of $\csym{C}$, which corresponds to oppositizing the tensor product, but not the morphisms of $\csym{C}$. In view of this, we make the following conventions:
\begin{itemize}
 \item The opposite of a category $\mathcal{C}$ is denoted by $\mathcal{C}^{\on{op}}$.
 \item The monoidal category obtained by oppositizing the tensor product of a monoidal category $\csym{C}$ is denoted by $\csym{C}^{\otimes\!\on{opp}}$; the monoidal category obtained by oppositizing the morphisms of $\csym{C}$ is denoted by $\csym{C}^{\on{opp}}$.
 \item The $1$-cell opposite of a bicategory $\csym{A}$ is denoted by $\csym{A}^{\on{op}}$; the $2$-cell opposite of $\csym{A}$ is denoted by $\csym{A}^{\on{co}}$.
\end{itemize}

For the rest of the document, we fix a monoidal category $\csym{C}$.

\subsection{Closed and rigid monoidal categories}
We say that $\csym{C}$ is {\it right-closed} if, for any $\mathrm{H} \in \csym{C}$, the functor $- \cotimes \mathrm{H}$ has a right adjoint $\Homint{ \mathrm{H}, - }_{r}$, yielding a natural isomorphism $\Hom{\mathrm{F} \cotimes \mathrm{H}, \mathrm{G}} \simeq \Hom{\mathrm{F}, \Homint{\mathrm{H},\mathrm{G}}_{r}}$. Similarly, we say that $\csym{C}$ is {\it left-closed} if the functor $\mathrm{H} \cotimes -$ has a right adjoint $\Homint{\mathrm{H},-}_{l}$, yielding a natural isomorphism $\Hom{\mathrm{H} \cotimes \mathrm{F}, \mathrm{G}} \simeq \Hom{\mathrm{F}, \Homint{\mathrm{H},\mathrm{G}}_{l}}$.

Note that \cite{PS} follow the opposite convention, so the notions of left- and right-closed in this document and in \cite{PS} are swapped.

Given $\mathrm{F} \in \csym{C}$, a {\it right dual} of $\mathrm{F}$ is an object $\mathrm{F}^{\vee}$ which admits morphisms $\eta_{\mathrm{F},\mathrm{F}^{\vee}}: \mathbb{1} \rightarrow \mathrm{F}^{\vee}\cotimes \mathrm{F}$ and $\varepsilon_{\mathrm{F,F}^{\vee}}: \mathrm{F} \cotimes \mathrm{F}^{\vee} \rightarrow \mathbb{1}$ satisfying zigzag equations.
Similarly, a {\it left dual} $\prescript{\vee}{}{\mathrm{F}}$ is an object such that there are morphisms $\eta_{\prescript{\vee}{}{\mathrm{F}},\mathrm{F}}: \mathbb{1} \rightarrow \mathrm{F} \cotimes \prescript{\vee}{}{\mathrm{F}}$ and $\varepsilon: \prescript{\vee}{}{\mathrm{F}} \cotimes \mathrm{F} \rightarrow \mathbb{1}$ satisfying zigzag equations.

Following \cite[Proposition~2.10.8]{EGNO}, if $\mathrm{F}$ has a left dual $\prescript{\vee}{}{\mathrm{F}}$, then $- \cotimes \prescript{\vee}{}{\mathrm{F}}$ is right adjoint to $- \cotimes \mathrm{F}$. Thus, if $\csym{C}$ has left duals, it is right-closed. Similarly, if $\csym{C}$ has right duals, it is left-closed. Further, if $\csym{C}$ is right-closed and $\mathrm{F}$ has a left dual $\prescript{\vee}{}{\mathrm{F}}$, then we necessarily have $\prescript{\vee}{}{\mathrm{F}} \simeq \Homint{\mathrm{F},\mathbb{1}}_{r}$. Similarly, if $\csym{C}$ is left-closed and $\mathrm{F}$ has a right dual $\mathrm{F}^{\vee}$, then we necessarily have $\mathrm{F}^{\vee} \simeq \Homint{\mathrm{F},\mathbb{1}}_{l}$. If $\csym{C}$ has left and right duals, we say that $\csym{C}$ is {\it rigid.}

\subsection{Module categories}

\begin{definition}
 A {\it $\csym{C}$-module category} is a category $\mathbf{M}$ together with a strong monoidal functor $\csym{C} \rightarrow \on{End}_{\mathbf{Cat}_{\Bbbk}}(\mathbf{M})$. As such, it sends an object $\mathrm{F}$ of $\csym{C}$ to an endofunctor $\mathbf{M}\mathrm{F}$ of $\mathbf{M}$, and a morphism $\mathrm{F} \rightarrow \mathrm{G}$ of $\csym{C}$ to a natural transformation $\mathbf{M}\alpha: \mathbf{M}\mathrm{F} \Rightarrow \mathbf{M}\mathrm{G}$. Further, it is equipped with natural isomorphisms $\mathbf{m}_{\mathrm{F,G}}: \mathbf{M}\mathrm{G}\mathbf{M}\mathrm{F} \xRightarrow{\simeq} \mathbf{M}\mathrm{GF}$, for all $\mathrm{F,G} \in \csym{C}$, and a natural isomorphism $\mathbf{m}_{\mathbb{1}}: \mathbb{1}_{\mathbf{M}} \xRightarrow{\simeq} \mathbf{M}\mathbb{1}_{\ccf{C}}$. The explicit definition specifying the coherence conditions the collection $\setj{\mathbf{m}_{\mathrm{F,G}} \; | \; \mathrm{F,G} \in \csym{C}} \sqcup \setj{\mathbf{m}_{\mathbb{1}}}$ of isomorphisms needs to satisfy can be found in \cite[Definition~7.1.1]{EGNO}.
\end{definition}

\begin{remark}
 Observe that we pose no further requirements on $\mathbf{M}$ beyond it being $\Bbbk$-linear; in particular, we do not assume it to be additive, idempotent split, abelian, hom-finite, or Krull-Schmidt. Whenever we require any of these, we will state that explicitly.
\end{remark}

\begin{definition}
 Given $\csym{C}$-module categories $\mathbf{M,N}$, a {\it $\csym{C}$-module functor} $\euler{\Phi}: \mathbf{M} \rightarrow \mathbf{N}$ consists of an underlying functor $\euler{\Phi}: \mathbf{M} \rightarrow \mathbf{N}$ together with a natural isomorphism
 \[\begin{tikzcd}[row sep = scriptsize]
	{\csym{C}} & {\mathbf{Cat}_{\Bbbk}(\mathbf{M,M})} \\
	{\mathbf{Cat}_{\Bbbk}(\mathbf{N,N})} & {\mathbf{Cat}_{\Bbbk}(\mathbf{M,N})}
	\arrow["{\mathbf{M}}", from=1-1, to=1-2]
	\arrow["{\mathbf{Cat}_{\Bbbk}(\mathbf{M},\euler{\Phi})}", from=1-2, to=2-2]
	\arrow["{\mathbf{N}}"', from=1-1, to=2-1]
	\arrow["{\mathbf{Cat}_{\Bbbk}(\euler{\Phi},\mathbf{N})}"', shift right=1, from=2-1, to=2-2]
	\arrow["\phi"', shorten <=6pt, shorten >=6pt, Rightarrow, from=2-1, to=1-2]
\end{tikzcd}\]
i.e. a collection of isomorphisms $\euler{\phi}_{\mathrm{F}}: \mathbf{N}\mathrm{F}\euler{\Phi} \xRightarrow{\simeq} \euler{\Phi}\mathbf{M}\mathrm{F}$, satisfying coherence conditions specified e.g. in \cite[Definition~7.2.1]{EGNO}.
\end{definition}

\begin{definition}
 Let
 $\begin{tikzcd}[sep = small]
	{\mathbf{M}} & {\mathbf{N}}
	\arrow["\euler{\Phi}", shift left=1, from=1-1, to=1-2]
	\arrow["\euler{\Psi}"', shift right=1, from=1-1, to=1-2]
\end{tikzcd}$ be $\csym{C}$-module functors.
A {$\csym{C}$-module transformation} $\euler{m}: \euler{\Phi} \Rightarrow \euler{\Psi}$ consists of an underlying natural transformation $\euler{m}: \euler{\Phi} \Rightarrow \euler{\Psi}$ such that
\[\begin{tikzcd}
	{\mathbf{N}\mathrm{F}\euler{\Phi}} & {\mathbf{N}\mathrm{F}\euler{\Psi}} \\
	{\euler{\Phi}\mathbf{M}\mathrm{F}} & {\euler{\Psi}\mathbf{M}\mathrm{F}}
	\arrow[from=1-1, to=1-2, "\mathbf{N}\mathrm{F}\bullet \euler{m}"]
	\arrow[from=1-1, to=2-1, "\phi_{\mathrm{F}}"]
	\arrow[from=1-2, to=2-2, "\psi_{\mathrm{F}}"]
	\arrow[from=2-1, to=2-2, "\euler{m}\bullet \mathbf{M}\mathrm{F}"]
\end{tikzcd}\]
commutes for all $\mathrm{F}$.
\end{definition}

\begin{definition}
 We denote by $\csym{C}\!\on{-Mod}$ the bicategory formed by $\csym{C}$-module categories, functors and transformations.
\end{definition}

\subsection{Profunctors and Tambara modules}

We give a brief summary of elementary definitions of profunctors and Tambara modules. For more detailed accounts, see \cite{Lo}, \cite{CEGLMPR}.

Recall that given a category $\mathcal{A}$ and a functor $\euler{F}: \mathcal{A}^{\on{op}} \kotimes \mathcal{A} \rightarrow \mathbf{Vec}_{\Bbbk}$, the {\it coend} $\int^{Z \in \mathcal{A}} \euler{F}(Z,Z)$ of $\euler{F}$ is the coequalizer of
\[\begin{tikzcd}[ampersand replacement=\&, column sep = huge]
	{\coprod_{X,Y \in \mathcal{A}}\euler{F}(Y,X) \kotimes \Hom{X,Y}^{\mathcal{A}}} \& {\coprod_{Z \in \mathcal{A}} \euler{F}(Z,Z)}
	\arrow["{v \otimes f \mapsto \euler{F}(Y,f)(v)}", shift left=1, from=1-1, to=1-2]
	\arrow["{v \otimes f \mapsto \euler{F}(f,X)(v)}"', shift right=1, from=1-1, to=1-2]
\end{tikzcd}\]
Thus, a linear map $\int^{Z \in \mathcal{A}} \euler{F}(Z,Z) \xrightarrow{\gamma} V$ is the same as a collection of maps $\setj{\gamma_{Z}: \euler{F}(Z,Z) \rightarrow V}$ such that for any $f: X \rightarrow Y$, we have $\gamma_{Y}\circ \euler{F}(Y,f) = \gamma_{X} \circ \euler{F}(f,X)$. We refer to such a collection as an {\it extranatural transformation} from $\euler{F}$ to $V$.

In particular, coends satisfy a Fubini rule: given a functor $\euler{T}$ from $(\mathcal{A} \kotimes \mathcal{B})^{\on{op}} \kotimes (\mathcal{A} \kotimes \mathcal{B}) \simeq (\mathcal{A}^{\on{op}} \kotimes \mathcal{A}) \kotimes (\mathcal{B}^{\on{op}} \kotimes \mathcal{B})$ to $\mathbf{Vec}_{\Bbbk}$, we have canonical isomorphisms
\[
 \int^{X \in \mathcal{A}} \int^{Y \in \mathcal{B}} \euler{T}(X,Y,X,Y) \simeq \int^{(X,Y) \in \mathcal{A} \kotimes \mathcal{B}} \euler{T}(X,Y,X,Y) \simeq \int^{Y \in \mathcal{B}} \int^{X \in \mathcal{A}} \euler{T}(X,Y,X,Y)
\]
which we generally omit in our computations. Dually to coends, one may define ends using equalizers and products, which yields the appropriate notion of an extranatural transformation from $V$ to $\euler{F}$.

\begin{definition}
 Given categories $\mathcal{A,B}$, a {\it profunctor} $\mathtt{\Phi}$ from $\mathcal{A}$ to $\mathcal{B}$, denoted $\mathtt{\Phi}: \mathcal{A} \xslashedrightarrow{} \mathcal{B}$, is given by a presheaf $\mathtt{\Phi}: \mathcal{B}^{\on{op}} \kotimes \mathcal{A} \rightarrow \mathbf{Vec}_{\Bbbk}$.
 A morphism of profunctors is a morphism of presheaves, thus a transformation natural in both arguments.

 Given another profunctor $\mathtt{\Psi}: \mathcal{B} \xslashedrightarrow{} \mathcal{C}$, the composite profunctor $\mathtt{\Psi} \diamond \mathtt{\Phi}: \mathcal{A} \rightarrow \mathcal{C}$ is defined by the coend
 \begin{equation}\label{ProfunctorComposition}
  (\mathtt{\Psi} \diamond \mathtt{\Phi})(C,A) = \int^{B \in \mathcal{B}} \mathtt{\Psi}(C,B) \kotimes \mathtt{\Phi}(B,A),
 \end{equation}
 on objects, and by maps induced on cocones for coends on morphisms. Similarly one can define horizontal composition of profunctor morphisms by inducing from cocones, yielding the bicategory $\mathbf{Prof}$.
 \end{definition}

 In particular, a morphism of profunctors $\mathtt{t}: \mathtt{\Psi} \diamond \mathtt{\Phi} \Rightarrow \mathtt{\Sigma}$ consists of a collection
 \begin{equation}\label{ExtranaturalCollection}
 \setj{\mathtt{\Psi}(C,B) \kotimes \mathtt{\Phi}(B,A) \xrightarrow{\mathtt{t}_{B;C,A}} \mathtt{\Sigma}(C,A) \; | \; A \in \mathcal{A}, B \in \mathcal{B}, C \in \mathcal{C}}
 \end{equation}
 natural in $A,C$ and extranatural in $B$. More generally, for a functor $\euler{F}: \mathcal{C}^{\on{op}} \kotimes \mathcal{B} \kotimes \mathcal{B}^{\on{op}} \kotimes \mathcal{A} \rightarrow \mathbf{Vec}_{\Bbbk}$ and a profunctor $\mathtt{\Sigma}: \mathcal{A} \xslashedrightarrow{} \mathcal{C}$, we refer to a profunctor morphism $\mathtt{t}: \int^{B \in \mathcal{B}} \euler{F}(-,B,B,-) \Rightarrow \mathtt{\Sigma}$, thus a collection $\setj{\mathtt{t}_{B;C,A}}$ similar to that in~\eqref{ExtranaturalCollection}, as an {\it extranatural collection}. As a convention, we list the extranatural indices for the collection on the left and the natural indices on the right, separating with a semicolon. Dualizing, we obtain the notion of an extranatural collection from $\mathtt{\Sigma}$ to $\euler{F}$, equivalently given by a profunctor morphism $\mathtt{\Sigma} \Rightarrow \int_{B \in \mathcal{B}} \euler{F}(-,B,B,-)$.

 The following lemma is necessary to show in order to determine the unit $1$-morphisms and left and right unitors in $\mathbf{Prof}$. A complete account is given in \cite[Proposition~7.8.2]{Bo}:
 \begin{lemma}[Yoneda Lemma]\label{ProYoneda}
  For any categories $\mathcal{A},\mathcal{C}$ and any profunctor $\mathtt{\Sigma}: \mathcal{A} \xslashedrightarrow{} \mathcal{C}$, we have isomorphisms
  \[
    \mathtt{y}^{\mathcal{C},\mathtt{\Sigma}}: \mathcal{C}(-,-) \diamond \mathtt{\Sigma} \xiso \mathtt{\Sigma},
  \]
 given by the extranatural collection
 \[
  \begin{aligned}
  \mathtt{y}^{\mathcal{C},\mathtt{\Sigma}}_{C;C',A}: \mathcal{C}(C',C) \kotimes \mathtt{\Sigma}(C,A) &\rightarrow \mathtt{\Sigma}(C',A) \\
  f \otimes v &\mapsto \mathtt{\Sigma}(f,A)(v),
  \end{aligned}
 \]
 and
  \[
   \mathtt{y}^{\mathtt{\Sigma},\mathcal{A}}: \mathtt{\Sigma} \diamond \mathcal{A}(-,-) \xiso \mathtt{\Sigma},
  \]
  given by the extranatural collection
 \[
 \begin{aligned}
  \mathtt{y}^{\mathtt{\Sigma},\mathcal{A}}_{A;A',C}: \mathtt{\Sigma}(C,A) \kotimes \mathcal{A}(A,A') &\rightarrow \mathtt{\Sigma}(C,A') \\
  v \otimes g &\mapsto \mathtt{\Sigma}(C,g)(v).
 \end{aligned}
 \]
 In particular, for any $v \in \mathtt{\Sigma}(C,A)$, the element $(\mathtt{y}^{\mathcal{C},\mathtt{\Sigma}}_{C',A})^{-1}(v) \in (\mathcal{C}(-,-)\diamond \mathtt{\Sigma})(C,A) = \int^{C'} \mathcal{C}(C,C') \kotimes \mathtt{\Sigma}(C',A)$ is represented by the equivalence class of $\on{id}_{C} \otimes v \in \coprod_{C'} \mathcal{C}(C,C') \kotimes \mathtt{\Sigma}(C',A)$, viewing the former coend as a quotient of the latter coproduct.
 \end{lemma}

\begin{definition}[{\cite[Definition 4.1]{CEGLMPR}}]
 Let $\mathbf{M,N} \in \csym{C}\!\on{-Mod}$. A {\it Tambara module} $\mathtt{\Phi}$ from $\mathbf{M}$ to $\mathbf{N}$ consists of an underlying profunctor $\mathbf{M} \xslashedrightarrow{} \mathbf{N}$ together with an extranatural collection $\setj{\ta_{\mathrm{H}; Y,X}: \mathtt{\Phi}(Y,X) \rightarrow \mathtt{\Phi}(\mathbf{N}\mathrm{H}Y, \mathbf{M}\mathrm{H}X)}$ from $\mathtt{\Phi}$ to the functor given by
 \[
  (X,\mathrm{F,G}, Y) \mapsto \mathtt{\Phi}(\mathbf{N}\mathrm{F}X, \mathbf{M}\mathrm{G}Y)
 \]
 satisfying the following axioms:
 \begin{enumerate}[label = (\roman*)]
  \item multiplicativity:
  \[\begin{tikzcd}
	{\mathtt{\Phi}(Y,X)} & {\mathtt{\Phi}(\mathbf{N}\mathrm{F}Y,\mathbf{M}\mathrm{F}X)} && {\mathtt{\Phi}(\mathbf{N}\mathrm{G}\mathbf{N}\mathrm{F}Y,\mathbf{M}\mathrm{G}\mathbf{M}\mathrm{F}X)} \\
	{\mathtt{\Phi}(\mathbf{N}\mathrm{GF}Y,\mathbf{M}\mathrm{GF}X)} &&& {\mathtt{\Phi}(\mathbf{N}\mathrm{G}\mathbf{N}\mathrm{F}Y,\mathbf{M}\mathrm{GF}X)}
	\arrow["{\ta_{\mathrm{F};Y,X}}", from=1-1, to=1-2]
	\arrow["{\ta_{\mathrm{G};\mathbf{N}\mathrm{F}Y, \mathbf{M}\mathrm{F}X}}", from=1-2, to=1-4]
	\arrow["{\ta_{\mathrm{GF};Y,X}}"', from=1-1, to=2-1]
	\arrow["{\mathtt{\Phi}((\mathbf{n}_{\mathrm{G,F}})_{Y},\mathbf{M}\mathrm{GF}X)}"', from=2-1, to=2-4]
	\arrow["{\mathtt{\Phi}(\mathbf{N}\mathrm{G}\mathbf{N}\mathrm{F}Y,(\mathbf{m}_{\mathrm{G,F}})_{X})}", from=1-4, to=2-4]
\end{tikzcd}\]
commutes for all $X,Y, \mathrm{F,G}$.
\item unitality:
\[\begin{tikzcd}[row sep = scriptsize]
	{\mathtt{\Phi}(Y,X)} && {\mathtt{\Phi}(\mathbf{N}\mathbb{1}Y, \mathbf{M}\mathbb{1}X)} \\
	& {\mathtt{\Phi}(Y,X)}
	\arrow["{\ta_{\mathrm{F};Y,X}}", from=1-1, to=1-3]
	\arrow["{\on{id}}"', from=1-1, to=2-2]
	\arrow["{\mathtt{\Phi}(\mathbf{n}_{Y}, \mathbf{m}_{X}^{-1})}", from=1-3, to=2-2]
\end{tikzcd}\]
commutes for all $Y,X$.
 \end{enumerate}
\end{definition}

\begin{remark}
 The objects we refer to as Tambara modules are called generalized Tambara modules in \cite{CEGLMPR}.
\end{remark}

\begin{example}\label{HomCK}
 Viewing $\csym{C}$ as a module category over itself, the $\on{Hom}$-profunctor $\csym{C}(-,-): \csym{C} \xslashedrightarrow{} \csym{C}$ together with Tambara structure given by $\ta^{\ccf{C}(-,-)}_{\mathrm{H;F,G}} = (\mathrm{H} \cotimes -)_{\mathrm{F,G}}$ gives a Tambara module.
 More generally, for an object $\mathrm{K} \in \csym{C}$, the Tambara module $\csym{C}(-,-\mathrm{K}): \csym{C} \xslashedrightarrow{} \csym{C}$ is given by the profunctor sending $(\mathrm{F,G})$ to $\csym{C}(\mathrm{F},\mathrm{GK})$, with Tambara structure given by composites
 \[
  \csym{C}(\mathrm{F},\mathrm{GK}) \xrightarrow{(\mathrm{H}\cotimes -)_{\mathrm{F},\mathrm{GK}}} \csym{C}(\mathrm{HF},\mathrm{H}(\mathrm{GK})) \xrightarrow{\ccf{C}(\mathrm{HF},\euler{a}_{\mathrm{H},\mathrm{G},\mathrm{K})}} \csym{C}(\mathrm{HF},(\mathrm{H}\mathrm{G})\mathrm{K}).
 \]
\end{example}

\begin{example}
 Let $\mathbf{M}$ be a $\csym{C}$-module category. The $\on{Hom}$-profunctor $\mathbf{M}(-,-): \mathbf{M} \xslashedrightarrow{} \mathbf{M}$ together with Tambara structure given by $\ta_{H;X,Y}^{\mathbf{M}(-,-)} = (\mathbf{M}\mathrm{H})_{X,Y}$ gives a Tambara module.
\end{example}

\begin{definition}
 Given Tambara modules $\mathtt{\Phi}, \mathtt{\Psi}: \mathbf{M} \xslashedrightarrow{} \mathbf{N}$, a morphism of Tambara modules $\mathtt{t}: \mathtt{\Phi} \Rightarrow \mathtt{\Psi}$ is a morphism of underlying profunctors satisfying the {\it Tambara axiom,} requiring the following diagram to commute for all $\mathrm{F},X,Y$:
  \[\begin{tikzcd}[row sep = scriptsize]
	{\mathtt{\Phi}(X,Y)} && {\mathtt{\Phi}(\mathbf{N}\mathrm{F}X, \mathbf{M}\mathrm{F}Y)} \\
	{\mathtt{\Psi}(X,Y)} && {\mathtt{\Psi}(\mathbf{N}\mathrm{F}X,\mathbf{M}\mathrm{F}Y)}
	\arrow["{\mathtt{t}_{X,Y}}"', from=1-1, to=2-1]
	\arrow["{\ta_{\mathrm{F};X,Y}^{\mathtt{\Phi}}}", from=1-1, to=1-3]
	\arrow["{\mathtt{t}_{\mathbf{M}\mathrm{F}Y,\mathbf{N}\mathrm{F}X}}", from=1-3, to=2-3]
	\arrow["{\ta_{\mathrm{F};X,Y}^{\mathtt{\Psi}}}", from=2-1, to=2-3]
\end{tikzcd}.\]
\end{definition}

\begin{definition}\label{CompositeTambara}
 Given Tambara modules $\mathtt{\Phi}: \mathbf{K} \xslashedrightarrow{} \mathbf{M}$ and $\mathtt{\Psi}: \mathbf{M} \xslashedrightarrow{} \mathbf{N}$, we define the composite $\mathtt{\Psi} \diamond \mathtt{\Phi}: \mathbf{K} \xslashedrightarrow{} \mathbf{N}$ as the underlying composite profunctor, together with Tambara structure $\setj{\ta^{\mathtt{\Psi} \diamond \mathtt{\Phi}}_{\mathrm{F},Y;X,Z}}$, where the component $\ta^{\mathtt{\Psi} \diamond \mathtt{\Phi}}_{\mathrm{F},Y;X,Z}$ is given by
\[\begin{tikzcd}[ampersand replacement=\&]
	{\mathtt{\Psi}(X,Y) \kotimes \mathtt{\Phi}(Y,Z)} \& {\mathtt{\Psi}(\mathbf{N}\mathrm{F}X,\mathbf{M}\mathrm{F}Y) \kotimes \mathtt{\Phi}(\mathbf{M}\mathrm{F}Y,\mathbf{K}\mathrm{F}Z)} \\
	{\coprod_{W \in \mathbf{M}} \mathtt{\Psi}(\mathbf{N}\mathrm{F}X,W) \kotimes \mathtt{\Phi}(W,\mathbf{K}\mathrm{F}Z)} \& {\int^{W \in \mathbf{M}} \mathtt{\Psi}(\mathbf{N}\mathrm{F}X,W) \kotimes \mathtt{\Phi}(W,\mathbf{K}\mathrm{F}Z)}
	\arrow["{\ta_{\mathrm{F};X,Y}^{\mathtt{\Psi}} \otimes \ta_{\mathrm{F};Y,Z}^{\mathtt{\Phi}}}", from=1-1, to=1-2]
	\arrow["\iota"', from=1-2, to=2-1]
	\arrow["{\pi^{\mathtt{\Psi}\diamond \mathtt{\Phi}}}", from=2-1, to=2-2]
\end{tikzcd},\]
where, $\iota$ is a component for the cocone of the coproduct, and $\pi^{\mathtt{\Psi} \diamond \mathtt{\Phi}}$ is the projection map defining the coend.

It is easy to verify that horizontal and vertical compositions of Tambara morphisms, defined as respective compositions of profunctor morphisms, satisfy the Tambara axiom. Further, for any $\csym{C}$-module category $\mathbf{M}$, the profunctor morphisms $\mathtt{y}^{\mathbf{M},-}, \mathtt{y}^{-,\mathbf{M}}$ defined in Lemma~\ref{ProYoneda} give Tambara morphisms satisfying coherence conditions for unitors, so we obtain a bicategory, denoted by $\csym{C}\!\on{-Tamb}$, consisting of $\csym{C}$-module categories, Tambara modules and Tambara morphisms.
\end{definition}

\section{Free Tambara modules}\label{s3}

\begin{theorem}[{\cite[Proposition 5.1]{PS}}]\label{LeftAdjTambaraPastro}
 The forgetful functor $\tam \rightarrow \mathbf{Prof}(\csym{C},\csym{C})$ admits a left adjoint $\euler{F}_{l}$, sending a profunctor $\mathtt{\Sigma}: \csym{C} \xslashedrightarrow{} \csym{C}$ to the profunctor
 \[
  (\mathrm{F,G}) \mapsto \int^{\mathrm{H,B,C}} \csym{C}(\mathrm{F, H \cotimes B}) \kotimes \mathtt{\Sigma}(\mathrm{B,C}) \kotimes \csym{C}(\mathrm{H \cotimes C, G})
 \]
with Tambara structure induced by the collection
\[
 \resizebox{.99\hsize}{!}{$
 \ta^{\ccf{C}}_{\mathrm{K;F,HB}} \kotimes \mathtt{\Sigma}(\mathrm{B,C}) \kotimes \ta^{\ccf{C}}_{\mathrm{K;HC,G}}: \csym{C}(\mathrm{F, H \cotimes B}) \kotimes \mathtt{\Sigma}(\mathrm{B,C}) \kotimes \csym{C}(\mathrm{H \cotimes C, G}) \rightarrow \csym{C}(\mathrm{K\cotimes F, K \cotimes H \cotimes B}) \kotimes \mathtt{\Sigma}(\mathrm{B,C}) \kotimes \csym{C}(\mathrm{K \cotimes H \cotimes C, K \cotimes G})
 $}
\]
similarly to the Tambara structure described in Definition~\ref{CompositeTambara}.
\end{theorem}

\begin{corollary}\label{FreeTambara}
 Given $\mathrm{L,K} \in \csym{C}$, the free Tambara module $\mathtt{\Theta}_{\mathrm{K,L}}$ from $\csym{C}$ to $\csym{C}$ associated to the representable profunctor
 \[
(\mathrm{F,G}) \mapsto \csym{C}(\mathrm{F,L}) \kotimes \csym{C}(\mathrm{K,G})
 \]
 has the profunctor
 \[
  (\mathrm{F,G}) \mapsto \int^{\mathrm{H}} \csym{C}\mathrm{(F, H \cotimes L)} \kotimes \csym{C}\mathrm{(H \cotimes K, G)}
 \]
 as its underlying profunctor, with its Tambara structure induced from the family
 \[
  \setj{\csym{C}(\mathrm{F,H \cotimes L}) \kotimes \csym{C}(\mathrm{H \cotimes K, G}) \xrightarrow{\mathrm{(D \cotimes -) \kotimes (D \cotimes -)}} \csym{C}(\mathrm{D \cotimes F,D \cotimes H \cotimes L}) \kotimes \csym{C}(\mathrm{D \cotimes H \cotimes K,D \cotimes G})}
 \]
Further, for any $\mathtt{R} \in \tam$, we have
\[
 \on{Hom}_{\ccf{C}\!\on{-Tamb}(\ccf{C},\ccf{C})}(\mathtt{\Theta}_{\mathrm{K,L}},\mathtt{R}) \simeq \mathtt{R}(\mathrm{L,K}).
\]
\end{corollary}

\begin{proof}
 The first assertion follows by two applications of the Yoneda lemma:
 \[
  \int^{\mathrm{H,B,C}} \csym{C}(\mathrm{F, H \cotimes B}) \kotimes \csym{C}(\mathrm{B,L}) \kotimes \csym{C}(\mathrm{K,C}) \kotimes \csym{C}(\mathrm{H\cotimes C, G}) \simeq \int^{\mathrm{H}} \csym{C}(\mathrm{F,H \cotimes L}) \kotimes \csym{C}(\mathrm{H \cotimes K, G})
 \]
 while the second follows by first using the adjunction of Theorem \ref{LeftAdjTambaraPastro}, and then Yoneda Lemma for  $\csym{C}^{\on{opp}} \kotimes \csym{C}$:
 \[
  \on{Hom}_{\ccf{C}\!\on{-Tamb}(\ccf{C},\ccf{C})}(\mathtt{\Theta}_{\mathrm{K,L}}, \mathtt{R}) \simeq \on{Hom}_{\mathbf{Prof}(\ccf{C},\ccf{C})}(\csym{C}(-,\mathrm{L}) \kotimes \csym{C}(\mathrm{K},-), \mathtt{R}) \simeq \mathtt{R}(\mathrm{L,K}).
 \]
\end{proof}

\begin{remark}\label{TambaraElements}
 Using Corollary \ref{FreeTambara}, we observe that the monoidal unit of $\tam$, given by the hom profunctor $\csym{C}(-,-)$ together with the Tambara structure $\ta_{\mathrm{H;F,G}} = (\mathrm{H} \cotimes -)_{\mathrm{F,G}}$, is isomorphic to the free Tambara module $\mathtt{\Theta}_{\mathbb{1,1}}$. In particular,
 \[
\on{Hom}_{\ccf{C}\!\on{-Tamb}(\ccf{C},\ccf{C})}(\csym{C}(-,-), \mathtt{R}) \simeq \mathtt{R}(\mathbb{1,1}).
 \]
 Further, the isomorphism
 \[
 \begin{aligned}
  &\int^{\mathrm{H,B,C}} \csym{C}(\mathrm{F, H \cotimes B}) \kotimes \csym{C}(\mathrm{B,K}) \kotimes \csym{C}(\mathrm{\mathbb{1},C}) \kotimes \csym{C}(\mathrm{H\cotimes C, G}) \simeq \int^{\mathrm{H}} \csym{C}(\mathrm{F,H \cotimes K}) \kotimes \csym{C}(\mathrm{H \cotimes \mathbb{1}, G}) \\
  &\simeq \int^{\mathrm{H}} \csym{C}(\mathrm{F,H \cotimes K}) \kotimes \csym{C}(\mathrm{H, G}) \simeq \csym{C}(\mathrm{F},\mathrm{G}\mathrm{K})
 \end{aligned}
 \]
 shows that the Tambara module $\csym{C}(-,-\mathrm{K})$ of Example~\ref{HomCK} is isomorphic to $\mathtt{\Theta}_{\mathrm{K},\mathbb{1}}$. As a consequence,
 \begin{equation}\label{UniPropCK}
  \on{Hom}_{\ccf{C}\!\on{-Tamb}(\ccf{C},\ccf{C})}(\csym{C}(-,-\mathrm{K}), \mathtt{R}) \simeq \on{Hom}_{\on{Prof}(\ccf{C},\ccf{C})}(\csym{C}(\mathrm{-,K}) \kotimes \csym{C}(\mathbb{1},-),\mathtt{R}) \simeq \mathtt{R}(\mathrm{K},\mathbb{1}).
 \end{equation}
\end{remark}

\section{Structure of the bicategory \texorpdfstring{$\csym{C}\!\on{-Tamb}$}{C-Tamb}}\label{s4}

\subsection{Restriction and corestriction of Tambara modules}

Recall that, given categories $\mathcal{C},\mathcal{C}', \mathcal{D}, \mathcal{D}'$, together with functors $\euler{F}: \mathcal{C} \rightarrow \mathcal{C}'$ and $\euler{G}: \mathcal{D} \rightarrow \mathcal{D}'$, for any profunctor $\mathtt{\Psi}: \mathcal{D}' \xslashedrightarrow{} \mathcal{C}'$, we may consider the {\it restriction of $\mathtt{\Psi}$ along $\euler{F} \kotimes \euler{G}$,} which is the profunctor given by the composite
\[
 \mathcal{D}^{\on{op}} \kotimes \mathcal{C} \xrightarrow{\euler{G}^{\on{op}} \kotimes \euler{F}} \mathcal{D}'^{\on{op}} \kotimes \mathcal{C}' \xrightarrow{\mathtt{\Psi}} \mathbf{Vec}_{\Bbbk}.
\]
For the purposes of this document, we will consider the restrictions along $\mathbb{1}_{\mathcal{D}^{\on{op}}} \kotimes \euler{F}$ and along $\euler{G}^{\on{op}} \kotimes \mathbb{1}_{\mathcal{C}}$ as separate, commuting operations, which we will refer to as restriction of $\mathtt{\Psi}$ along $\euler{F}$ and corestriction of $\mathtt{\Psi}$ along $\euler{G}$, respectively. Below, we give an explicit description of restrictions and corestrictions for Tambara modules and module functors.

Let $\mathbf{M},\mathbf{N}$ be $\csym{C}$-module categories, let $\mathtt{\Psi} \in \csym{C}\!\on{-Tamb}(\mathbf{M},\mathbf{N})$ and let
$\euler{\Phi} \in \csym{C}\!\on{-Mod}(\mathbf{K},\mathbf{M})$.

\begin{definition}\label{ResDef}
 The {\it restriction of $\mathtt{\Psi}$ along $\euler{\Phi}$} is the Tambara module $\mathtt{\Psi} \smalltriangleup \euler{\Phi} \in \csym{C}\!\on{-Tamb}(\mathbf{K,N})$,
 whose underlying profunctor is given by
 $
  (Z,X) \mapsto \mathtt{\Psi}(Z, \euler{\Phi} X)
 $
 and whose Tambara structure is given by
  \[\begin{tikzcd}[column sep = huge]
	{\mathtt{\Psi}(Z,\euler{\Phi} X)} & {\mathtt{\Psi}(\mathbf{N}\mathrm{F}Z, \mathbf{M}\mathrm{F}\euler{\Phi} X)} & {\mathtt{\Psi}(\mathbf{N}\mathrm{F}Z, \euler{\Phi}\mathbf{K}\mathrm{F}X)}
	\arrow["{\ta_{\mathrm{F};Z,\euler{\Phi} X}^{\mathtt{\Psi}}}", from=1-1, to=1-2]
	\arrow["{\mathtt{\Psi}(\mathbf{N}\mathrm{F}Z, \euler{\phi}_{\mathrm{F},X})}", from=1-2, to=1-3]
\end{tikzcd}.\]
\end{definition}

Let $\mathtt{t} \in \csym{C}\!\on{-Tamb}(\mathbf{M},\mathbf{N})(\mathtt{\Psi}, \mathtt{\Psi}')$.
\begin{lemma}\label{ResNat}
 The transformation $\mathtt{t} \circ (\mathbb{1}_{\mathbf{N}^{\on{op}}} \kotimes \euler{\Phi})$ gives a morphism $\mathtt{t} \smalltriangleup \euler{\Phi}: \mathtt{\Psi} \smalltriangleup \euler{\Phi} \Rightarrow \mathtt{\Psi}' \smalltriangleup \euler{\Phi}$ of Tambara modules.
\end{lemma}

\begin{proof}
 The diagram
 \[\begin{tikzcd}[column sep = huge, row sep = scriptsize]
	{\mathtt{\Psi} (Z , \euler{\Phi} X)} & {\mathtt{\Psi} (\mathbf{N}\mathrm{F}Z , \mathbf{M}\mathrm{F}\euler{\Phi} X)} & {\mathtt{\Psi} (\mathbf{N}\mathrm{F}Z , \euler{\Phi} \mathbf{K}\mathrm{F}X)} \\
	{\mathtt{\Psi}' (Z , \euler{\Phi} X)} & {\mathtt{\Psi}' (\mathbf{N}\mathrm{F}Z , \mathbf{M}\mathrm{F}\euler{\Phi} X)} & {\mathtt{\Psi}' (\mathbf{N}\mathrm{F}Z , \euler{\Phi} \mathbf{K}\mathrm{F}X)}
	\arrow["{\mathtt{\Psi}(\mathbf{N}\mathrm{F}Z, \phi_{\mathrm{F},X})}", from=1-2, to=1-3]
	\arrow["{\ta^{\mathtt{\Psi}}_{\mathrm{F}; Z, \euler{\Phi} X}}", from=1-1, to=1-2]
	\arrow["{\mathtt{\Psi}'(\mathbf{N}\mathrm{F}Z, \euler{\phi}_{\mathrm{F},X})}", from=2-2, to=2-3]
	\arrow["{\ta^{\mathtt{\Psi}'}_{\mathrm{F}; Z, \euler{\Phi} X}}", from=2-1, to=2-2]
	\arrow["{\mathtt{t}_{Z,\euler{\Phi} X}}"', from=1-1, to=2-1]
	\arrow["{\mathtt{t}_{\mathbf{N}\mathrm{F}Z,\mathbf{M}\mathrm{F}\euler{\Phi} X}}"', from=1-2, to=2-2]
	\arrow["{\mathtt{t}_{\mathbf{N}\mathrm{F}Z, \euler{\Phi} \mathbf{K}\mathrm{F}X}}", from=1-3, to=2-3]
\end{tikzcd}\]
commutes: the left square since $\euler{t}$ is a morphism of Tambara modules, the right square since $\euler{t}$ is natural in the right variable.
\end{proof}

\begin{lemma}\label{ResFunc}
Let $\euler{d} \in \csym{C}\!\on{-Mod}(\euler{\Phi}, \euler{\Phi}')$.
 The transformation $\mathtt{\Psi} \circ (\mathbb{1}_{\mathbf{N}^{\on{op}}} \kotimes \euler{d})$ gives a morphism $\mathtt{\Psi} \smalltriangleup \euler{d}: \mathtt{\Psi} \smalltriangleup \euler{\Phi} \Rightarrow \mathtt{\Psi}' \smalltriangleup \euler{\Phi}'$ of Tambara modules.
\end{lemma}

\begin{proof}
 The diagram
 \begin{equation}\label{FFGeneral}
 \begin{tikzcd}[column sep = huge, row sep = scriptsize]
	{\mathtt{\Psi} (Z , \euler{\Phi} X)} & {\mathtt{\Psi} (\mathbf{N}\mathrm{F}Z , \mathbf{M}\mathrm{F}\euler{\Phi} X)} & {\mathtt{\Psi} (\mathbf{N}\mathrm{F}Z , \euler{\Phi} \mathbf{K}\mathrm{F}X)} \\
	{\mathtt{\Psi} (Z , \euler{\Phi}' X)} & {\mathtt{\Psi} (\mathbf{N}\mathrm{F}Z , \mathbf{M}\mathrm{F}\euler{\Phi}' X)} & {\mathtt{\Psi} (\mathbf{N}\mathrm{F}Z , \euler{\Phi}' \mathbf{K}\mathrm{F}X)}
	\arrow["{\mathtt{\Psi}(Z,\euler{d}_{X})}"', from=1-1, to=2-1]
	\arrow["{\ta^{\mathtt{\Psi}}_{\mathrm{F}; Z, \euler{\Phi} X}}", from=1-1, to=1-2]
	\arrow["{\mathtt{\Psi}(\mathbf{N}\mathrm{F}Z, \phi_{\mathrm{F},X})}", from=1-2, to=1-3]
	\arrow["{\ta^{\mathtt{\Psi}}_{\mathrm{F}; Z, \euler{\Phi}' X}}", from=2-1, to=2-2]
	\arrow["{\mathtt{\Psi}(\mathbf{N}\mathrm{F}Z, \phi'_{\mathrm{F},X})}", from=2-2, to=2-3]
	\arrow["{\mathtt{\Psi}(\mathbf{N}\mathrm{F}Z,\euler{d}_{\mathbf{K}\mathrm{F}X})}", from=1-3, to=2-3]
	\arrow["{\mathtt{\Psi}(\mathbf{N}\mathrm{F}Z,\mathbf{M}\mathrm{F}\euler{d}_{X})}"', from=1-2, to=2-2]
\end{tikzcd}
 \end{equation}
commutes: the left square by the Tambara axiom for $\ta^{\mathtt{\Psi}}$, the right square since $\euler{d}$ is a $\csym{C}$-module transformation.
\end{proof}

\begin{proposition}\label{ResPseudoOneSide}
 Given diagrams
\[
\begin{tikzcd}[scale cd = 0.9]
	{\mathbf{J}} && {\mathbf{K}} && {\mathbf{M}}
	\arrow[""{name=0, anchor=center, inner sep=0}, "{\euler{\Omega}''}"', curve={height=14pt}, from=1-1, to=1-3]
	\arrow[""{name=1, anchor=center, inner sep=0}, "{\euler{\Omega}'}"{description}, from=1-1, to=1-3]
	\arrow[""{name=2, anchor=center, inner sep=0}, "\euler{\Phi}", curve={height=-14pt}, from=1-3, to=1-5]
	\arrow[""{name=3, anchor=center, inner sep=0}, "{\euler{\Phi}''}"', curve={height=14pt}, from=1-3, to=1-5]
	\arrow[""{name=4, anchor=center, inner sep=0}, "{\euler{\Phi}'}"{description}, from=1-3, to=1-5]
	\arrow[""{name=5, anchor=center, inner sep=0}, "\euler{\Omega}", curve={height=-14pt}, from=1-1, to=1-3]
	\arrow["{\euler{d}}", shorten <=2pt, shorten >=2pt, Rightarrow, from=2, to=4]
	\arrow["{\euler{g}}", shorten <=2pt, shorten >=2pt, Rightarrow, from=5, to=1]
	\arrow["{\euler{d}'}", shorten <=2pt, shorten >=2pt, Rightarrow, from=4, to=3]
	\arrow["{\euler{g}'}", shorten <=2pt, shorten >=2pt, Rightarrow, from=1, to=0]
\end{tikzcd} \text{ in } \csym{C}\!\on{-Mod} \text{ and }
\begin{tikzcd}[column sep = huge, scale cd = 0.9]
	{\mathbf{M}} & {\mathbf{N}} & {\mathbf{Q}}
	\arrow[""{name=0, anchor=center, inner sep=0}, "{\mathtt{\Psi}}", curve={height=-14pt}, from=1-1, to=1-2]
	\arrow[""{name=1, anchor=center, inner sep=0}, "{\mathtt{\Psi}''}"', curve={height=14pt}, from=1-1, to=1-2]
	\arrow[""{name=2, anchor=center, inner sep=0}, "{\mathtt{\Psi}'}"{description}, from=1-1, to=1-2]
	\arrow[""{name=3, anchor=center, inner sep=0}, "{\mathtt{\Sigma}}", curve={height=-14pt}, from=1-2, to=1-3]
	\arrow[""{name=4, anchor=center, inner sep=0}, "{\mathtt{\Sigma}''}"', curve={height=14pt}, from=1-2, to=1-3]
	\arrow[""{name=5, anchor=center, inner sep=0}, "{\mathtt{\Sigma}}"{description}, from=1-2, to=1-3]
	\arrow["{\mathtt{t}}", shorten <=2pt, shorten >=2pt, Rightarrow, from=0, to=2]
	\arrow["{\mathtt{t}'}", shorten <=2pt, shorten >=2pt, Rightarrow, from=2, to=1]
	\arrow["{\mathtt{k}}", shorten <=2pt, shorten >=2pt, Rightarrow, from=3, to=5]
	\arrow["{\mathtt{k}'}", shorten <=2pt, shorten >=2pt, Rightarrow, from=5, to=4]
\end{tikzcd} \text{ in } \csym{C}\!\on{-Tamb},
\]
the following equations hold:
\begin{multicols}{2}
\begin{enumerate}
 \item \label{prl1} $\on{id}_{\mathtt{\Psi}} \smalltriangleup \euler{\Phi} = \on{id}_{\mathtt{\Psi} \smalltriangleup \euler{\Phi}}$;
 \item \label{prl2} $(\mathtt{t}' \circ \mathtt{t}) \smalltriangleup \euler{\Phi} = (\mathtt{t}' \smalltriangleup \euler{\Phi}) \circ (\mathtt{t} \smalltriangleup \euler{\Phi})$;
  \item \label{prl3} $(\mathtt{t} \smalltriangleup \euler{\Phi}) \circ (\mathtt{\Psi} \smalltriangleup \euler{d}) = (\mathtt{\Psi} \smalltriangleup \euler{d}) \circ (\mathtt{t} \smalltriangleup \euler{\Phi}) =: \mathtt{t} \smalltriangleup \euler{d}$;
 \item \label{prl4} $\mathtt{\Psi} \smalltriangleup \on{id}_{\euler{\Phi}} = \on{id}_{\mathtt{\Psi} \smalltriangleup \euler{\Phi}}$;
 \item \label{prl5} $\mathtt{\Psi} \smalltriangleup (\euler{d}' \circ \euler{d}) = (\mathtt{\Psi} \smalltriangleup \euler{d}') \circ (\mathtt{\Psi} \smalltriangleup \euler{d})$;
 \item \label{prl6} $\mathtt{\Psi} \smalltriangleup \mathbb{1}_{\mathbf{M}} = \mathtt{\Psi}$;
 \item \label{prl7} $(\mathtt{\Psi} \smalltriangleup \euler{\Phi}) \smalltriangleup \euler{\Omega} = \mathtt{\Psi} \smalltriangleup (\euler{\Phi \circ \Omega})$;
 \item \label{prl8} $(\mathtt{\Sigma} \circ \mathtt{\Psi}) \smalltriangleup \euler{\Phi} = \mathtt{\Sigma} \circ (\mathtt{\Psi} \smalltriangleup \euler{\Phi})$;
 \item \label{prl9} $(\mathtt{k} \diamond \mathtt{t}) \smalltriangleup \euler{d} = \mathtt{k} \diamond (\mathtt{t} \smalltriangleup \euler{d})$.
 \item[\vspace{\fill}]
\end{enumerate}
\end{multicols}
\end{proposition}

\begin{proof}
 Two morphisms of Tambara modules are equal if and only if their underlying morphisms of profunctors are equal. Since the underlying profunctors and profunctor morphisms of restrictions for Tambara modules are restrictions for the underlying profunctors, the stated equations for morphisms of Tambara modules follow from analogous equations for profunctors. For the latter, the validity of the equations is a consequence of elementary properties of functor composition (viewing profunctors as functors to $\mathbf{Vec}_{\Bbbk}$). Equations~\eqref{prl1}, \eqref{prl2}, \eqref{prl3}, \eqref{prl4}, \eqref{prl5} and \eqref{prl9} follow.

 Two Tambara modules are equal if and only if their underlying profunctors and their Tambara structures coincide. For the remaining equations, the equality of underlying profunctors again follows from elementary properties of functor composition. Thus, it suffices to exhibit equality of Tambara structures in each case. And indeed, for Equation~\eqref{prl6}:
 \[
  \mathtt{\Psi}(\mathbf{N}\mathrm{F}Y, \on{id}_{\mathbf{M}\mathrm{F}X}) \circ \ta^{\mathtt{\Psi}}_{\mathrm{F}; Y, \mathbb{1}_{\mathbf{M}}X} = \on{id}_{\mathtt{\Psi}(\mathbf{N}\mathrm{F}Y, \mathbf{M}\mathrm{F}X)} \circ \ta^{\mathtt{\Psi}}_{\mathrm{F}; Y, \mathbb{1}_{\mathbf{M}}X} = \ta^{\mathtt{\Psi}}_{\mathrm{F}; Y, \mathbb{1}_{\mathbf{M}}X},
 \]
 for Equation~\eqref{prl7}:
    \[
    \begin{aligned}
     &\ta^{(\mathtt{\Psi} \smalltriangleup \euler{\Phi}) \smalltriangleup \euler{\Phi}'}_{\mathrm{F};Z,V} \qquad & \text{ by definition of } - \smalltriangleup \Phi' \\
     &=(\mathtt{\Psi} \smalltriangleup \euler{\Phi})(\mathbf{N}\mathrm{F}Z, \euler{\phi}'_{\mathrm{F},V}) \circ \ta^{\mathtt{\Psi} \smalltriangleup \euler{\Phi}}_{\mathrm{F};Z,\euler{\Phi}'V}&\text{ by definition of } - \smalltriangleup \Phi \\
     &=\mathtt{\Psi}(\mathbf{N}\mathrm{F}Z, \euler{\Phi} \euler{\phi}'_{\mathrm{F},V}) \circ \big(\mathtt{\Psi}(\mathbf{N}\mathrm{F}Z,\euler{\phi}_{\mathrm{F},\euler{\Phi} ' V}) \circ \ta^{\mathtt{\Psi}}_{\mathrm{F}; Z, (\euler{\Phi} \circ \euler{\Phi}') V}\big)&\text{ by functoriality of } \mathtt{\Psi} \\
     &=\mathtt{\Psi}(\mathbf{N}\mathrm{F}Z,\euler{\Phi} \euler{\phi}'_{\mathrm{F},V}\circ \euler{\phi}_{\mathrm{F},\euler{\Phi}' V}) \circ \ta^{\mathtt{\Psi}}_{\mathrm{F}; Z, (\euler{\Phi} \circ \euler{\Phi}') V} &\text{by definition of structure maps for } \euler{\Phi}\circ \euler{\Phi}'  \\
     &= \ta^{\mathtt{\Psi} \smalltriangleup (\euler{\Phi} \circ \euler{\Phi}')}_{\mathrm{F};Z,V}
    \end{aligned}
   \]
 and for Equation~\eqref{prl8}, we observe that the Tambara structure maps
 \[
\int^{Y} \mathtt{\Sigma}(Z,Y) \kotimes \mathtt{\Psi}(Y, \euler{\Phi} W) \rightarrow \int^{Y} \mathtt{\Sigma}(\mathbf{N}\mathrm{F}Z, Y) \kotimes \mathtt{\Psi}(Y, \euler{\Phi} \mathbf{K}\mathrm{F}W)
 \]
 of the respective Tambara modules are given by the extranatural families
\[
 \setj{\Big(\mathtt{\Sigma}(\mathbf{N}\mathrm{F}Z, \mathbf{M}\mathrm{F}Y) \kotimes \mathtt{\Psi}(\mathbf{M}\mathrm{F}Y, \euler{\phi}_{\mathrm{F},W})\Big) \circ \Big(\ta^{\mathtt{\Sigma}}_{\mathrm{F};Z,Y} \kotimes \ta^{\mathtt{\Psi}}_{\mathrm{F};Z,Y}\Big)} \text{ and }
 \setj{\ta^{\mathtt{\Sigma}}_{\mathrm{F}; Z, Y} \kotimes \Big(\mathtt{\Psi}(\mathbf{M}\mathrm{F}Y, \euler{\phi}_{\mathrm{F},W}) \circ \ta^{\mathtt{\Psi}}_{Y,\euler{\Phi} W}\Big)}
\]
 for the left-hand side and the right-hand side respectively. These families coincide due to naturality of the tensor product.
\end{proof}

Similarly to Definition~\ref{ResDef}, Definition~\ref{ResNat} and Definition~\ref{ResFunc}, there is a notion of {\it corestriction} of Tambara modules:
\begin{definition}\label{CoresPseudoOneSide}
 Given
 \[
\mathtt{\Sigma} \in \csym{C}\!\on{-Tamb}(\mathbf{N,Q}) \text{ and } \euler{\Phi} \in \csym{C}\!\on{-Mod}(\mathbf{K,Q}),
 \]
 define the {\it corestriction} $\mathtt{\Psi} \smalltriangledown \euler{\Phi}$ as the Tambara module with underlying profunctors $(V,Y) \mapsto \mathtt{\Psi}(\euler{\Phi}V,Y)$ and the evident Tambara structure, and also define Tambara morphisms $\mathtt{t} \smalltriangledown \euler{\Phi}: \mathtt{\Psi} \smalltriangledown \euler{\Phi} \Rightarrow \mathtt{\Psi'} \smalltriangledown \euler{\Phi}$ and $\mathtt{\Psi} \smalltriangledown \euler{b}: \mathtt{\Psi} \smalltriangledown \euler{\Phi}' \Rightarrow \mathtt{\Psi} \smalltriangledown \euler{\Phi}$, for any $\mathtt{t}: \mathtt{\Psi} \Rightarrow \mathtt{\Psi}'$ and $\euler{b}: \euler{\Phi} \Rightarrow \euler{\Phi}'$.
\end{definition}
Observe that $\mathtt{\Psi} \smalltriangledown -$ is contravariant.

The proof of the following proposition is completely analogous to the proof of Proposition~\ref{ResPseudoOneSide}:
\begin{proposition}\label{CoResPseudoOneSide}
 Given diagrams
\[
\begin{tikzcd}[scale cd = 0.9]
	{\mathbf{J}} && {\mathbf{K}} && {\mathbf{Q}}
	\arrow[""{name=0, anchor=center, inner sep=0}, "{\euler{\Omega}''}"', curve={height=14pt}, from=1-1, to=1-3]
	\arrow[""{name=1, anchor=center, inner sep=0}, "{\euler{\Omega}'}"{description}, from=1-1, to=1-3]
	\arrow[""{name=2, anchor=center, inner sep=0}, "\euler{\Phi}", curve={height=-14pt}, from=1-3, to=1-5]
	\arrow[""{name=3, anchor=center, inner sep=0}, "{\euler{\Phi}''}"', curve={height=14pt}, from=1-3, to=1-5]
	\arrow[""{name=4, anchor=center, inner sep=0}, "{\euler{\Phi}'}"{description}, from=1-3, to=1-5]
	\arrow[""{name=5, anchor=center, inner sep=0}, "\euler{\Omega}", curve={height=-14pt}, from=1-1, to=1-3]
	\arrow["{\euler{d}}", shorten <=2pt, shorten >=2pt, Rightarrow, from=2, to=4]
	\arrow["{\euler{g}}", shorten <=2pt, shorten >=2pt, Rightarrow, from=5, to=1]
	\arrow["{\euler{d}'}", shorten <=2pt, shorten >=2pt, Rightarrow, from=4, to=3]
	\arrow["{\euler{g}'}", shorten <=2pt, shorten >=2pt, Rightarrow, from=1, to=0]
\end{tikzcd} \text{ in } \csym{C}\!\on{-Mod} \text{ and }
\begin{tikzcd}[column sep = huge, scale cd = 0.9]
	{\mathbf{M}} & {\mathbf{N}} & {\mathbf{Q}}
	\arrow[""{name=0, anchor=center, inner sep=0}, "{\mathtt{\Psi}}", curve={height=-14pt}, from=1-1, to=1-2]
	\arrow[""{name=1, anchor=center, inner sep=0}, "{\mathtt{\Psi}''}"', curve={height=14pt}, from=1-1, to=1-2]
	\arrow[""{name=2, anchor=center, inner sep=0}, "{\mathtt{\Psi}'}"{description}, from=1-1, to=1-2]
	\arrow[""{name=3, anchor=center, inner sep=0}, "{\mathtt{\Sigma}}", curve={height=-14pt}, from=1-2, to=1-3]
	\arrow[""{name=4, anchor=center, inner sep=0}, "{\mathtt{\Sigma}''}"', curve={height=14pt}, from=1-2, to=1-3]
	\arrow[""{name=5, anchor=center, inner sep=0}, "{\mathtt{\Sigma}}"{description}, from=1-2, to=1-3]
	\arrow["{\mathtt{t}}", shorten <=2pt, shorten >=2pt, Rightarrow, from=0, to=2]
	\arrow["{\mathtt{t}'}", shorten <=2pt, shorten >=2pt, Rightarrow, from=2, to=1]
	\arrow["{\mathtt{k}}", shorten <=2pt, shorten >=2pt, Rightarrow, from=3, to=5]
	\arrow["{\mathtt{k}'}", shorten <=2pt, shorten >=2pt, Rightarrow, from=5, to=4]
\end{tikzcd} \text{ in } \csym{C}\!\on{-Tamb},
\]
the following equations hold:
\begin{multicols}{2}
\begin{enumerate}
 \item $\on{id}_{\mathtt{\Sigma}} \smalltriangledown \euler{\Phi} = \on{id}_{\mathtt{\Sigma} \smalltriangledown \euler{\Phi}}$;
 \item $(\mathtt{k}' \circ \mathtt{k}) \smalltriangledown \euler{\Phi} = (\mathtt{k}' \smalltriangledown \euler{\Phi}) \circ (\mathtt{k} \smalltriangledown \euler{\Phi})$;
  \item $(\mathtt{k} \smalltriangledown \euler{\Phi}) \circ (\mathtt{\Sigma} \smalltriangledown \euler{d}) = (\mathtt{\Sigma} \smalltriangledown \euler{d}) \circ (\mathtt{k} \smalltriangledown \euler{\Phi}) =: \mathtt{k} \smalltriangledown \euler{d}$;
 \item $\mathtt{\Sigma} \smalltriangledown \on{id}_{\euler{\Phi}} = \on{id}_{\mathtt{\Sigma} \smalltriangledown \euler{\Phi}}$;
 \item $\mathtt{\Sigma} \smalltriangledown (\euler{d}' \circ \euler{d}) = (\mathtt{\Sigma} \smalltriangledown \euler{d}) \circ (\mathtt{\Sigma} \smalltriangledown \euler{d}')$;
 \item $\mathtt{\Sigma} \smalltriangledown \mathbb{1}_{\mathbf{Q}} = \mathtt{\Sigma}$;
 \item $(\mathtt{\Sigma} \smalltriangledown \euler{\Phi}) \smalltriangledown \euler{\Omega} = \mathtt{\Psi} \smalltriangledown (\euler{\Phi \circ \Omega})$;
 \item $(\mathtt{\Sigma} \diamond \mathtt{\Psi}) \smalltriangledown \euler{\Phi} = (\mathtt{\Sigma} \smalltriangledown \euler{\Phi}) \diamond \mathtt{\Psi}$;
 \item $(\mathtt{k} \diamond \mathtt{t}) \smalltriangledown \euler{d} = (\mathtt{k} \smalltriangledown \euler{d}) \diamond \mathtt{t}$.
 \item[\vspace{\fill}]
\end{enumerate}
\end{multicols}
\end{proposition}

Let $\mathtt{\Psi} \in \csym{C}\!\on{-Tamb}(\mathbf{M,N})$, $\euler{\Phi,\Phi'} \in \csym{C}\!\on{-Mod}(\mathbf{K,M})$ and $\euler{\Lambda,\Lambda'} \in \csym{C}\!\on{-Mod}(\mathbf{L,N})$. Further, let $\euler{d}: \euler{\Phi} \Rightarrow \euler{\Phi'}$ and $\euler{b}: \euler{\Lambda}' \Rightarrow \euler{\Lambda}$.
\begin{lemma}\label{TensorComm}
 We have $(\mathtt{\Psi} \smalltriangleup \euler{\Phi}) \smalltriangledown \euler{\Lambda} = (\mathtt{\Psi} \smalltriangledown \euler{\Lambda}) \smalltriangleup \euler{\Phi}$, and, similarly, $(\mathtt{\Psi} \smalltriangleup \euler{d}) \smalltriangledown \euler{b} = (\mathtt{\Psi} \smalltriangledown \euler{b}) \smalltriangleup \euler{d}$.
\end{lemma}

\begin{proof}
 Again, the latter equation reduces to familiar properties of functors and natural transformations, while for the former we only need to verify that the Tambara structures coincide. And indeed:
     \[
    \begin{aligned}
     &\ta^{(\mathtt{\Psi} \smalltriangleup \euler{\Phi}) \smalltriangledown \euler{\Lambda}}_{\mathrm{F};V,X} \qquad & \text{ by definition } - \smalltriangleup \euler{\Phi} \text{ and } -\smalltriangledown \euler{\Lambda} \\
     &=\mathtt{\Psi}(\lambda_{\mathrm{F},V},\euler{\Phi} \mathbf{K}\mathrm{F}X) \circ \big(\mathtt{\Psi}(\mathbf{N}\mathrm{F}\euler{\Lambda} V, \phi_{\mathrm{F},X}) \circ \ta^{\mathtt{\Psi}}_{\mathrm{F}; \euler{\Lambda} V, \euler{\Phi} X}\big)&\text{ by functoriality of } \mathtt{\Psi} \\
     &=\mathtt{\Psi}(\euler{\Lambda} \mathbf{L}\mathrm{F}V, \phi_{\mathrm{F},X}) \circ \mathtt{\Psi}(\lambda_{\mathrm{F},V},\mathbf{M}\mathrm{F}\euler{\Phi} X) \circ  \ta^{\mathtt{\Psi}}_{\mathrm{F}; \euler{\Lambda} V, \euler{\Phi} X}&\text{ by definition of } -\smalltriangledown \euler{\Lambda} \text{ and } -\smalltriangleup \euler{\Phi} \\
     &= \ta^{(\mathtt{\Psi} \smalltriangledown \euler{\Lambda}) \smalltriangleup \euler{\Phi} }_{\mathrm{F};V,X}.
    \end{aligned}
   \]
\end{proof}

\begin{corollary}\label{CoresResStructure}
 The assignments
 \[
  \begin{aligned}
  (\mathbf{M},\mathbf{N}) &\rightarrow \csym{C}\!\on{-Tamb}(\mathbf{M},\mathbf{N}) \\
  (\euler{\Lambda, \Phi}) &\mapsto \Big(((-)\smalltriangleup \euler{\Phi}) \smalltriangledown \euler{\Lambda}: \csym{C}\!\on{-Tamb}(\mathbf{M},\mathbf{N}) \rightarrow \csym{C}\!\on{-Tamb}(\mathbf{K},\mathbf{Q})\Big) \\
  \euler{b} \otimes \euler{d} &\mapsto ((-)\smalltriangleup \euler{d}) \smalltriangledown \euler{b}
  \end{aligned}
 \]
 define a $2$-functor $\mathbb{H}: \csym{C}\!\on{-Mod}^{\on{op,co}} \kotimes \csym{C}\!\on{-Mod}^{\on{op}} \rightarrow \mathbf{Cat}_{\Bbbk}$. Thus, restriction and corestriction give two commuting, strict, right actions of $\csym{C}\!\on{-Mod}$ on $\csym{C}\!\on{-Tamb}$.
\end{corollary}

\begin{proof}
 Equations~\eqref{prl1} and \eqref{prl2} in Proposition~\ref{ResPseudoOneSide} give a functor $- \smalltriangleup \euler{\Phi} \in \mathbf{Cat}_{\Bbbk}(\csym{C}\!\on{-Tamb}(\mathbf{M},\mathbf{N}),\csym{C}\!\on{-Tamb}(\mathbf{K},\mathbf{N}))$, for any $\euler{\Phi}$. Equation~\eqref{prl3} yields a natural transformation $- \smalltriangleup \euler{d}: -\smalltriangleup \euler{\Phi} \Rightarrow - \smalltriangleup \euler{\Phi}'$, for any $\euler{d}$. Equations \eqref{prl4}, \eqref{prl5} show that the assignment $\euler{d} \mapsto - \smalltriangleup \euler{d}$ is functorial, so that for every $\mathbf{N}$ we obtain a functor
 \[
\mathbb{H}(-,\mathbf{N})_{\mathbf{K},\mathbf{M}}: \csym{C}\!\on{-Mod}(\mathbf{K},\mathbf{M}) \rightarrow \mathbf{Cat}_{\Bbbk}(\csym{C}\!\on{-Tamb}(\mathbf{M},\mathbf{N}), \csym{C}\!\on{-Tamb}(\mathbf{K},\mathbf{N})).
 \]
 Equations~\eqref{prl6} and \eqref{prl7} show that $\mathbb{H}(-,\mathbf{N})$ is a $2$-functor from $\csym{C}\!\on{-Mod}^{\on{op}}$ to $\mathbf{Cat}_{\Bbbk}$. Finally, Equations~\eqref{prl8} and \eqref{prl9} show the $2$-naturality of $\mathbb{H}(-,\mathbf{N})$ in $\mathbf{N}$.
 Similarly, using Proposition~\ref{CoResPseudoOneSide}, we find a family $\setj{\mathbb{H}(\mathbf{M},-) \; | \; \mathbf{M} \in \csym{C}\!\on{-Mod}}$ of $2$-functors from $\csym{C}\!\on{-Mod}^{\on{op,co}}$ which is $2$-natural in $\mathbf{M}$. Due to Lemma~\ref{TensorComm}, these two families together assemble to the above described $2$-functor $\mathbb{H}$.
\end{proof}

\begin{notation}
 For computational purposes it is often useful to denote the Tambara module $\mathtt{\Psi} \smalltriangleup \euler{\Phi}$ by $\mathtt{\Psi}(-,\euler{\Phi} -)$. Similarly, in such settings we will denote the Tambara module $\mathtt{\Psi} \smalltriangledown \euler{\Lambda}$ by $\mathtt{\Psi}(\euler{\Lambda} -,-)$. In case $\mathtt{\Psi} = \Hom{-,-}^{\mathbf{M}}$, we may omit the superscript in $\Hom{-,\euler{\Phi}-}^{\mathbf{M}}$, since it is clear that the domain and codomain of this Tambara module coincide with those of $\euler{\Phi}$.
\end{notation}

\subsection{\texorpdfstring{$\csym{C}\!\on{-Tamb}$}{C-Tamb} as a proarrow equipment}

The following is a central observation in the study of profunctors:
\begin{theorem}[{\cite[Chapter 7]{Bo}}]
 There is a locally full and faithful pseudofunctor $\mathbb{E}: \mathbf{Cat} \rightarrow \mathbf{Prof}$ which is the identity on objects, and sends a functor $\euler{F}: \mathcal{A} \rightarrow \mathcal{B}$ to the profunctor $\on{Hom}_{\mathcal{B}}(-,\euler{F}-)$; the definition extends to natural transformations in the evident way.
 Further, the profunctor $\on{Hom}_{\mathcal{B}}(-,\euler{F}-)$ admits a right adjoint, given by $\on{Hom}_{\mathcal{B}}(\euler{F}-,-)$.
\end{theorem}

Pseudofunctors admitting similar properties have been studied abstractly, initially in \cite{Wo}, under the name {\it proarrow equipments}. Our terminology below follows \cite{GS}:
\begin{definition}\label{ProArrowDef}
 A {\it proarrow equipment} is a locally full and faithful pseudofunctor $\mathbb{F}: \csym{A} \rightarrow \csym{B}$, which is the identity on objects. If for every $1$-morphism $\mathrm{A}$ of $\csym{A}$, the $1$-morphism $\mathbb{F}(\mathrm{A})$ admits a right adjoint in $\csym{B}$, we say that $\mathbb{F}$ is a {\it map equipment.} Given proarrow equipments $\mathbb{F}: \csym{A} \rightarrow \csym{B}$ and $\mathbb{F'}: \csym{A}\, ' \rightarrow \csym{B}'$, a (strict) morphism of proarrow equipments from $\mathbb{F}$ to $\mathbb{F}'$ is a pair of pseudofunctors $\mathbb{A}: \csym{A} \rightarrow \csym{A}\, '$ and $\mathbb{B}: \csym{B} \rightarrow \csym{B}'$ such that the following diagram commutes:
 \[
  \begin{tikzcd}
   {\csym{A}} \arrow[r, "\mathbb{F}"] \arrow[d, "\mathbb{A}"] & {\csym{B}} \arrow[d, "\mathbb{B}"] \\
   {\csym{A}\, '} \arrow[r, "\mathbb{F}'"] & {\csym{B}'}
  \end{tikzcd}.
 \]
\end{definition}

\begin{definition}\label{PsPDef}
 We define the pseudofunctor $\mathbb{P}: \csym{C}\!\on{-Mod} \rightarrow \csym{C}\!\on{-Tamb}$ as follows:
 \begin{itemize}
  \item It is the identity on objects;
  \item Given $\euler{\Phi} \in \csym{C}\!\on{-Mod}(\mathbf{M},\mathbf{N})$, we let $\mathbb{P}(\euler{\Phi}) := \Hom{-,\euler{\Phi}-}$.
  \item Given $\euler{d}: \euler{\Phi} \Rightarrow \euler{\Phi}'$, we let $\mathbb{P}(\euler{d}) = \Hom{-,\euler{d}-}$.
  \item The coherence maps are given by
  \begin{equation}\label{CoherenceProarrow}
   \mathbb{P}(\euler{\Phi}) \diamond \mathbb{P}(\euler{\Omega}) = \Hom{-,\euler{\Phi} -} \diamond (\Hom{-,-}_{\mathbf{M}} \smalltriangleup \euler{\Omega}) = (\Hom{-,\euler{\Phi} -} \diamond \Hom{-,-}_{\mathbf{M}}) \smalltriangleup \euler{\Omega} \xiso \Hom{-,\euler{\Phi}-}\smalltriangleup \euler{\Omega} = \mathbb{P}(\euler{\Phi} \circ \euler{\Omega}),
  \end{equation}
 \item Equation~\eqref{prl6} in Proposition~\ref{ResPseudoOneSide} shows that $\mathbb{P}(\mathbb{1}_{\mathbf{M}}) = \mathbb{1}_{\mathbb{P}(\mathbf{M})}$, so $\mathbb{P}$ is strictly unital.
 \end{itemize}
\end{definition}
Functoriality of $\mathbb{P}$ on $2$-morphisms follows from Equations~\eqref{prl4} and \eqref{prl5} in Proposition~\ref{ResPseudoOneSide}. All equalities in~\eqref{CoherenceProarrow} follow from Proposition~\ref{ResPseudoOneSide}, and the isomorphism $(\Hom{-,\euler{\Phi} -} \diamond \Hom{-,-}_{\mathbf{M}}) \xiso \Hom{-,\euler{\Phi}-}$ is the unitor associated to $\Hom{-,-}_{\mathbf{M}}$ as the identity Tambara module for $\mathbf{M}$. These are also the coherence maps for the pseudofunctor $\mathbb{E}$, so the coherence conditions follow from pseudofunctoriality of $\mathbb{E}$.

In particular, the diagram
\begin{equation}\label{ProArrowMorphism}
\begin{tikzcd}[row sep = scriptsize]
	{\csym{C}\!\on{-Mod}} & {\csym{C}\!\on{-Tamb}} \\
	{\mathbf{Cat}_{\Bbbk}} & {\mathbf{Prof}_{\Bbbk}}
	\arrow[from=1-2, to=2-2]
	\arrow[from=1-1, to=2-1]
	\arrow["{\mathbb{P}}", from=1-1, to=1-2]
	\arrow["{\mathbb{E}}", from=2-1, to=2-2]
\end{tikzcd}
\end{equation}
commutes strictly, where the vertical arrows are the forgetful pseudofunctors.
\begin{proposition}\label{PisFF}
 $\mathbb{P}$ is locally full and faithful.
\end{proposition}

\begin{proof}
 Since $\mathbb{E}$ is locally full and faithful, and both the forgetful pseudofunctors in Diagram~\eqref{ProArrowMorphism} are locally faithful (two modifications are equal if and only if their underlying transformations are equal, and similarly for morphisms of Tambara modules), it suffices to show that for a natural transformation $\euler{d}$ between two module functors, the profunctor morphism $\Hom{-,\euler{d}-}$ is a morphism of Tambara modules if and only if $\euler{d}$ is a $\csym{C}$-module transformation. Specializing Diagram~\eqref{FFGeneral} of Lemma~\ref{ResFunc} to the case $\mathtt{\Psi} = \Hom{-,-}^{\mathbf{N}}$, we find the diagram
  \begin{equation}\label{ModificationsTambara}
 \begin{tikzcd}[column sep = huge, row sep = scriptsize]
	{\Hom{Z , \euler{\Phi} X }} & {\Hom{\mathbf{N}\mathrm{F}Z , \mathbf{N}\mathrm{F}\euler{\Phi} X}} & {\Hom{\mathbf{N}\mathrm{F}Z , \euler{\Phi} \mathbf{M}\mathrm{F}X}} \\
	{\Hom{Z , \euler{\Phi}' X}} & {\Hom{\mathbf{N}\mathrm{F}Z , \mathbf{N}\mathrm{F}\euler{\Phi}' X}} & {\Hom{\mathbf{N}\mathrm{F}Z , \euler{\Phi}' \mathbf{M}\mathrm{F}X}}
	\arrow["{\Hom{Z,\euler{d}_{X}}}"', from=1-1, to=2-1]
	\arrow["{(\mathbf{N}\mathrm{F})_{Z,\euler{\Phi} X}}", from=1-1, to=1-2]
	\arrow["{\Hom{\mathbf{N}\mathrm{F}Z, \phi_{\mathrm{F},X}}}", from=1-2, to=1-3]
	\arrow["{(\mathbf{N}\mathrm{F})_{Z,\euler{\Phi}'X}}", from=2-1, to=2-2]
	\arrow["{\Hom{\mathbf{N}\mathrm{F}Z, \phi'_{\mathrm{F},X}}}", from=2-2, to=2-3]
	\arrow["{\Hom{\mathbf{N}\mathrm{F}Z,\euler{d}_{\mathbf{M}\mathrm{F}X}}}", from=1-3, to=2-3]
	\arrow["{\Hom{\mathbf{N}\mathrm{F}Z,\mathbf{N}\mathrm{F}\euler{d}_{X}}}"', from=1-2, to=2-2]
\end{tikzcd}
 \end{equation}
 The left square of this diagram commutes by functoriality of $\mathbf{N}\mathrm{F}$. As observed in Lemma~\ref{ResFunc}, if $\euler{d}$ is a $\csym{C}$-module transformation, then the right square also commutes, and thus so does the outer square, which shows that $\Hom{-,\euler{d}-}$ is a morphism of Tambara modules. Conversely, if $\Hom{-,\euler{d}-}$ is a morphism of Tambara modules, then the outer diagram commutes; letting $Z = \euler{\Phi} X$ and chasing $\on{id}_{\euler{\Phi} X}$ shows that $\euler{d}$ is a $\csym{C}$-module transformation.
\end{proof}

Considering corestrictions rather than restrictions, we obtain an analogue of Definition~\ref{PsPDef} and Proposition~\ref{PisFF}:
\begin{proposition}\label{ContravariantProArrow}
 There is a proarrow equipment $\mathbb{B}: \csym{C}\!\on{-Mod}^{\on{op,co}} \rightarrow \csym{C}\!\on{-Tamb}$, which is given by $\euler{\Lambda} \mapsto \Hom{\euler{\Lambda}-,-}$ on $1$-morphisms and by $\euler{b} \mapsto \Hom{\euler{b}-,-}$ on $2$-morphisms.
\end{proposition}

\begin{remark}
 The isomorphisms
 \[
  (\mathbb{B}\kotimes \mathbb{P})(\euler{\Lambda,\Phi})(\mathtt{\Psi}) = (\Hom{\euler{\Lambda}-,-} \diamond \mathtt{\Psi}) \diamond \Hom{-,\euler{\Phi}-} \xiso \mathtt{\Psi}(\euler{\Lambda}-,\euler{\Phi}-) = \mathbb{H}(\euler{\Lambda},\euler{\Phi})(\mathtt{\Psi})
 \]
 provide a pseudonatural equivalence
\[\begin{tikzcd}[sep = small]
	{\csym{C}\!\on{-Mod}^{\on{op,co}} \kotimes \csym{C}\!\on{-Mod}^{\on{op}}} \\
	\\
	{\csym{C}\!\on{-Tamb} \kotimes \csym{C}\!\on{-Tamb}^{\on{op}}} && {\mathbf{Cat}_{\Bbbk}}
	\arrow["{\mathbb{B} \kotimes \mathbb{P}}", from=1-1, to=3-1, swap]
	\arrow[""{name=0, anchor=center, inner sep=0}, "{\mathbb{H}}"', from=1-1, to=3-3, swap]
	\arrow["{\ccf{C}\!\on{-Tamb}(-,-)}"', from=3-1, to=3-3]
	\arrow["\simeq", shift right=1, shorten >=11pt, Rightarrow, from=3-1, to=0]
\end{tikzcd}\]
\end{remark}

\subsection{\texorpdfstring{$\csym{C}\!\on{-Tamb}$}{C-Tamb} as a map equipment}\label{Adjunctions}

Following \cite[Proposition~7.9.1]{Bo}, for any $\euler{\Phi} \in \mathbf{Prof}(\mathbf{M,N})$, the profunctor morphism $\Hom{-,\euler{\Phi}} \diamond \Hom{\euler{\Phi}-,-} \xslashedrightarrow{} \Hom{-,-}^{\mathbf{N}}$ given by the extranatural collection
\[
 \setj{\mathtt{\varepsilon}_{Y;Z,Z'}: \Hom{Z,\euler{\Phi} Y} \kotimes \Hom{\euler{\Phi} Y, Z'} \rightarrow \Hom{Z,Z'} \; | \; Z }, \text{ where } \mathtt{\varepsilon}_{Y;Z,Z'}(f \otimes g) = g \circ f
\]
is the counit of an adjunction $(\Hom{-,\euler{\Phi} -}, \Hom{\euler{\Phi} -,-}, \mathtt{\eta}, \mathtt{\varepsilon})$ in $\mathbf{Prof}_{\Bbbk}$.
Its unit $\mathtt{\eta}$ is given by the composite map
\[
 \mathtt{\eta}_{X,X'}: \Hom{X,X'} \xrightarrow{\euler{\Phi}_{X,X'}} \Hom{\euler{\Phi} X, \euler{\Phi} X'} \xiso \int^{Y} \Hom{\euler{\Phi} X, Y} \kotimes \Hom{Y, \euler{\Phi} X},
\]
where the isomorphism is that coming from Yoneda lemma.

\begin{proposition}
 The adjunction just described, $(\Hom{-,\euler{\Phi} -}, \Hom{\euler{\Phi} -,-}, \mathtt{\eta, \varepsilon})$ in $\mathbf{Prof}_{\Bbbk}$, gives an adjunction in $\csym{C}\!\on{-Tamb}$. In particular, the proarrow equipment $\mathbb{P}$ is a map equipment.
\end{proposition}

\begin{proof}
 It suffices to show that $\eta$ and $\varepsilon$ are Tambara morphisms.
 For the counit $\varepsilon$, we need to show that the diagram
 \[\begin{tikzcd}
	{\int^{Y} \Hom{Z,\euler{\Phi} Y} \kotimes \Hom{\euler{\Phi} Y, Z'}} && {\int^{Y} \Hom{\mathbf{N}\mathrm{F}Z,\euler{\Phi} Y} \kotimes \Hom{\euler{\Phi} Y, \mathbf{N}\mathrm{F}Z'}} \\
	{\Hom{Z,Z'}} && {\Hom{\mathbf{N}\mathrm{F}Z,\mathbf{N}\mathrm{F}Z'}}
	\arrow["{\varepsilon_{\mathbf{N}\mathrm{F}Z,\mathbf{N}\mathrm{F}Z'}}", from=1-3, to=2-3]
	\arrow["{\ta^{\Hom{-,\euler{\Phi}-} \diamond \Hom{\euler{\Phi}-,-}}_{\mathrm{F};Z,Z'}}", shift left=1, from=1-1, to=1-3]
	\arrow["{\ta^{\Hom{-,-}_{\mathbf{N}}}_{\mathrm{F},Z,Z'}}"', from=2-1, to=2-3]
	\arrow["{\varepsilon_{Z,Z'}}"', from=1-1, to=2-1]
\end{tikzcd}\]
commutes.
This is the case if the diagram
\[
 \resizebox{.99\hsize}{!}{$
\begin{tikzcd}[column sep = huge, ampersand replacement = \&]
	{\Hom{Z,\euler{\Phi} Y} \kotimes \Hom{\euler{\Phi} Y,Z'}} \& {\Hom{\mathbf{N}\mathrm{F}Z,\mathbf{N}\mathrm{F}\euler{\Phi} Y} \kotimes \Hom{\mathbf{N}\mathrm{F}\euler{\Phi} Y,\mathbf{N}\mathrm{F}Z'}} \& {\Hom{\mathbf{N}\mathrm{F}Z,\euler{\Phi} \mathbf{M}\mathrm{F} Y} \kotimes \Hom{\euler{\Phi} \mathbf{M}\mathrm{F} Y,\mathbf{N}\mathrm{F}Z'}} \\
	{\Hom{Z,Z'}} \& \& {\Hom{\mathbf{N}\mathrm{F}Z,\mathbf{N}\mathrm{F}Z'}}
	\arrow["{\mathbf{N}\mathrm{F}_{Z,\euler{\Phi} Y} \otimes \mathbf{N}\mathrm{F}_{\euler{\Phi} Y,Z'}}", shift left=1, from=1-1, to=1-2]
	\arrow["{\mathbf{N}\mathrm{F}_{Z,Z'}}", from=2-1, to=2-3]
	\arrow[from=1-1, to=2-1, "{\varepsilon_{Y;Z,Z'}}"]
	\arrow["{\Hom{\mathbf{N}\mathrm{F}Z, (\phi_{\mathrm{F}})_{Y}} \otimes \Hom{(\phi_{\mathrm{F}}^{-1})_{Y}, \mathbf{N}\mathrm{F}Z'}}", shift left=1, from=1-2, to=1-3]
	\arrow[from=1-3, to=2-3, "{\varepsilon_{\mathbf{M}\mathrm{F}Y;\mathbf{N}\mathrm{F}Z,\mathbf{N}\mathrm{F}Z'}}"]
\end{tikzcd}$}
\]
commutes. And indeed, for any $f: Z \rightarrow \euler{\Phi} Y$ and $g: \euler{\Phi} Y \rightarrow Z'$, chasing $f \otimes g$ around the above diagram we obtain
\[\begin{tikzcd}
	{f \otimes g} & {\mathbf{N}\mathrm{F}(f) \otimes \mathbf{N}\mathrm{F}(g)} & {(\phi_{\mathrm{F}})_{Y} \circ \mathbf{N}\mathrm{F}(f) \otimes \mathbf{N}\mathrm{F}(g) \circ (\phi_{\mathrm{F}}^{-1})_{Y}} \\
	{g \circ f} & {\mathbf{N}\mathrm{F}(g \circ f)} & { \mathbf{N}\mathrm{F}(g) \circ (\phi_{\mathrm{F}}^{-1})_{Y} \circ (\phi_{\mathrm{F}})_{Y} \circ \mathbf{N}\mathrm{F}(f)}
	\arrow[from=2-2, to=2-3, equal]
	\arrow[maps to, from=1-3, to=2-3]
	\arrow[maps to, from=1-2, to=1-3]
	\arrow[maps to, from=1-1, to=1-2]
	\arrow[maps to, from=1-1, to=2-1]
	\arrow[maps to, from=2-1, to=2-2]
\end{tikzcd}\]

Further, the commutativity of the following diagram shows that also $\eta$ is a morphism of Tambara modules:
\[\begin{tikzcd}
	{\Hom{Y,Y'}} && {\Hom{\mathbf{M}\mathrm{F}Y, \mathbf{M}\mathrm{F}Y'}} \\
	{\Hom{\euler{\Phi} Y,\euler{\Phi} Y'}} & {\Hom{\mathbf{N}\mathrm{F}\euler{\Phi} Y, \mathbf{N}\mathrm{F}\euler{\Phi} Y}} & {\Hom{\euler{\Phi} \mathbf{M}\mathrm{F}Y, \euler{\Phi} \mathbf{M}\mathrm{F}Y'}}
	\arrow["{\euler{\Phi}_{Y,Y'}}"', from=1-1, to=2-1]
	\arrow["{\mathbf{M}\mathrm{F}_{Y,Y'}}", from=1-1, to=1-3]
	\arrow["{\mathbf{N}\mathrm{F}_{\euler{\Phi} Y, \euler{\Phi} Y'}}"', from=2-1, to=2-2, shift right = 1]
	\arrow["{\Hom{(\phi_{\mathrm{F}}^{-1})_{Y}, (\phi_{\mathrm{F}})_{Y}}}"', from=2-2, to=2-3, shift right = 1]
	\arrow["{\euler{\Phi}_{\mathbf{M}\mathrm{F}Y,\mathbf{M}\mathrm{F}Y'}}", from=1-3, to=2-3]
\end{tikzcd}\]
\end{proof}

  \begin{lemma}\label{StructureTransport}
  Let $\mathtt{\Psi} \in \csym{C}\!\on{-Tamb}(\mathbf{M,N})$ and let $\mathtt{\Psi}' \in \mathbf{Prof}_{\Bbbk}(\mathbf{M,N})$ be such that there is an isomorphism of profunctors $\mathtt{s}: \mathtt{\Psi} \Rightarrow \mathtt{\Psi}'$. There is a unique Tambara module structure with which the profunctor $\mathtt{\Psi}'$ can be endowed so that $\mathtt{s}$ becomes an isomorphism of Tambara modules.
 \end{lemma}

 \begin{proof}
  It is easy to verify that the assignment $\ta^{\mathtt{\Psi}'}_{\mathrm{F};Y,X} := \euler{s}_{Y,X}^{-1} \circ \ta^{\mathtt{\Psi}}_{\mathrm{F};Y,X} \circ \mathtt{s}_{Y,X}$ gives a well-defined Tambara structure which also satisfies the uniqueness property.
 \end{proof}

  Whenever we endow a profunctor with a Tambara structure using the construction of Lemma \ref{StructureTransport}, we say that we {\it transport} the Tambara module structure of $\mathtt{\Psi}$ to $\mathtt{\Psi}'$, along the isomorphism $\mathtt{s}$.

 Recall that beyond module functors, we may also consider {\it lax module functors}, whose structure morphisms satisfy all the coherence conditions but are not necessarily invertible.
 \begin{lemma}\label{RepresentableGivesLax}
  Let $\mathtt{\Phi}: \mathbf{M} \xslashedrightarrow{} \mathbf{N}$ be a Tambara module whose underlying profunctor is isomorphic to $\Hom{-,\euler{\Phi} -}$, for some $\euler{\Phi} \in \mathbf{Cat}_{\Bbbk}(\mathbf{M,N})$.
  Then $\euler{\Phi}$ can be endowed with the structure of a lax $\csym{C}$-module functor, by setting
  \[
(\euler{\phi}_{\mathrm{F}})_{X} := \ta_{\mathrm{F},\euler{\Phi} X, X}^{\mathtt{\Phi}}(\on{id}_{\euler{\Phi} X}) \in \Hom{\mathbf{N}\mathrm{F}\euler{\Phi} X, \euler{\Phi} \mathbf{M}\mathrm{F}X}.
  \]
 \end{lemma}

  \begin{proof}
  Transporting the Tambara module structure from $\mathtt{\Phi}$ along a profunctor isomorphism $\mathtt{\Phi} \simeq \Hom{-,\euler{\Phi} -}$, we may assume that the underlying profunctor of $\mathtt{\Phi}$ equals $\Hom{-,\euler{\Phi} -}$.

  We need to verify that the collection $\setj{(\phi_{\mathrm{F}})_{X} \; | \; X \in \mathbf{M}, \mathrm{F} \in \csym{C}}$ is natural in $\mathrm{F}$ and $X$, and that it satisfies the associativity and unitality conditions.

  For any $\mathrm{F} \in \csym{C}$ and $X \in \mathbf{M}$, the collection
 \[
\ta^{\mathtt{\Phi}}_{\mathrm{F};-,X} = \setj{\ta^{\mathtt{\Phi}}_{\mathrm{F};Y,X}: \Hom{Y,\euler{\Phi} X} \rightarrow \Hom{\mathbf{N}\mathrm{F}Y,\euler{\Phi} \mathbf{M}\mathrm{F}X} \; | \; Y \in \mathbf{N} }
 \]
 gives a natural transformation $\Hom{-,\euler{\Phi} X} \rightarrow \Hom{\mathbf{N}\mathrm{F}-,\euler{\Phi} \mathbf{M}\mathrm{F}X}$. Thus, by Yoneda lemma, for any $f \in \Hom{Y,\euler{\Phi} X}$ we have
 \begin{equation}\label{YonedaFormula}
\ta^{\mathtt{\Phi}}_{\mathrm{F};Y,X}(f) = \ta_{\mathrm{F},\euler{\Phi} X, X}^{\mathtt{\Phi}}(\on{id}_{\Phi X}) \circ \mathbf{N}\mathrm{F}(f).
 \end{equation}
 Since $\ta^{\mathtt{\Phi}}_{\mathrm{F};Y,X}$ is natural in $X$, the diagram
 \[\begin{tikzcd}
	{\Hom{\euler{\Phi} X, \euler{\Phi} X}} && {\Hom{\mathbf{N}\mathrm{F}\euler{\Phi} X, \euler{\Phi} \mathbf{M}\mathrm{F}X}} \\
	{\Hom{\euler{\Phi} X, \euler{\Phi} X'}} && {\Hom{\mathbf{N}\mathrm{F}\euler{\Phi} X, \euler{\Phi} \mathbf{M}\mathrm{F}X'}}
	\arrow["{\ta_{\mathrm{F}; \euler{\Phi} X, X}}", from=1-1, to=1-3]
	\arrow["{\Hom{\euler{\Phi} X, \euler{\Phi} g}}"', from=1-1, to=2-1]
	\arrow["{\ta_{\mathrm{F}; \euler{\Phi} X, X'}}"', from=2-1, to=2-3]
	\arrow["{\Hom{\mathbf{N}\mathrm{F}\euler{\Phi} X, \euler{\Phi}\mathbf{M}\mathrm{F}g}}", from=1-3, to=2-3]
\end{tikzcd}\]
 commutes. Hence so does
 \[\begin{tikzcd}[column sep = huge]
	{\mathbf{N}\mathrm{F}\euler{\Phi} X} & {\euler{\Phi}\mathbf{M}\mathrm{F}X} \\
	{\mathbf{N}\mathrm{F}\euler{\Phi} X'} & {\euler{\Phi}\mathbf{M}\mathrm{F}X'}
	\arrow["{\ta_{\mathrm{F}; \Phi X, X}(\on{id}_{\Phi X})}", from=1-1, to=1-2]
	\arrow["{\mathbf{N}\mathrm{F}\Phi(g)}"', from=1-1, to=2-1]
	\arrow["{\Phi\mathbf{M}\mathrm{F}(g)}", from=1-2, to=2-2]
	\arrow["{\ta_{\mathrm{F};\Phi X', X'}(\on{id}_{\Phi X'})}"', from=2-1, to=2-2]
\end{tikzcd}\]
proving naturality in $X$.

By extranaturality of $\ta^{\mathtt{\Phi}}$ in $\mathrm{F}$, the diagram
\[\begin{tikzcd}[sep = small]
	& {\Hom{\euler{\Phi} X, \euler{\Phi} X}} \\
	{\Hom{\mathbf{N}\mathrm{F}\euler{\Phi} X, \euler{\Phi} \mathbf{M}\mathrm{F}X}} && {\Hom{\mathbf{N}\mathrm{F'}\euler{\Phi} X, \euler{\Phi} \mathbf{M}\mathrm{F'}X}} \\
	& {\Hom{\mathbf{N}\mathrm{F}\euler{\Phi} X, \euler{\Phi} \mathbf{M}\mathrm{F'}X}}
	\arrow["{\ta_{\mathrm{F}; \euler{\Phi} X,X}}"', from=1-2, to=2-1]
	\arrow["{\ta_{\mathrm{F'}; \euler{\Phi} X,X}}", from=1-2, to=2-3]
	\arrow["{\Hom{\mathbf{N}\mathrm{F}\euler{\Phi} X, \euler{\Phi} \mathbf{M}\alpha_{X}}}"'{pos=0.3}, from=2-1, to=3-2]
	\arrow["{\Hom{\mathbf{N}\alpha_{\euler{\Phi} X}, \euler{\Phi} \mathbf{M}\mathrm{F'}X}}", from=2-3, to=3-2]
\end{tikzcd}\]
commutes, for $\alpha \in \on{Hom}_{\ccf{C}}(\mathrm{F,F}')$. In particular, evaluating the two composite maps in the above diagram at $\on{id}_{\euler{\Phi} X}$ and using its commutativity, we find that
\[
 \ta_{\mathrm{F}';\euler{\Phi}X,X}(\on{id}_{\euler{\Phi}X}) \circ \mathbf{N}\alpha_{\euler{\Phi}X} = \euler{\Phi}\mathbf{M}\alpha_{X} \circ \ta_{\mathrm{F};\euler{\Phi}X,X}(\on{id}_{\euler{\Phi}X}),
\]
 which shows that our candidate collection is natural in $\mathrm{F}$.

The diagram
\[\begin{tikzcd}[column sep = huge]
	{\Hom{\euler{\Phi} X, \euler{\Phi} X}} && {\Hom{\mathbf{N}\mathrm{GF}\euler{\Phi} X, \euler{\Phi} \mathbf{M}\mathrm{GF}X}} \\
	{\Hom{\mathbf{N}\mathrm{F}\euler{\Phi} X, \euler{\Phi} \mathbf{M}\mathrm{F}X}} \\
	{\Hom{\mathbf{N}\mathrm{G}\mathbf{N}\mathrm{F}\euler{\Phi} X, \euler{\Phi} \mathbf{M}\mathrm{G}\mathbf{M}\mathrm{F}X}} && {\Hom{\mathbf{N}\mathrm{G}\mathbf{N}\mathrm{F}\euler{\Phi} X, \euler{\Phi} \mathbf{M}\mathrm{GF}X}}
	\arrow["{\ta_{\mathrm{F};\euler{\Phi} X,X}}", from=1-1, to=2-1]
	\arrow["{\ta_{\mathrm{G};\mathbf{N}\mathrm{F}\euler{\Phi} X, \mathbf{M}\mathrm{F} X}}", from=2-1, to=3-1]
	\arrow["{\ta_{\mathrm{GF};\euler{\Phi} X,X}}", from=1-1, to=1-3]
	\arrow["{\Hom{\mathbf{N}\mathrm{G}\mathbf{N}\mathrm{F}X, (\euler{\Phi} \mathbf{m}_{\mathrm{G,F}})_{X}}}", from=3-1, to=3-3]
	\arrow["{\Hom{(\mathbf{n}_{\mathrm{G,F}})_{\euler{\Phi} X}, \euler{\Phi} \mathbf{M}\mathrm{GF}X}}", from=1-3, to=3-3]
\end{tikzcd}\]
commutes by the associativity condition for the Tambara module $\mathtt{\Phi}$. We again evaluate the maps in the above diagram at $\on{id}_{\Phi X}$, via the following diagram chase:
\[\begin{tikzcd}[sep = small]
	{\on{id}_{\euler{\Phi} X}} && {\ta_{\mathrm{GF},\euler{\Phi} X, X}(\on{id}_{\euler{\Phi} X})} \\
	{\ta_{\mathrm{F};\euler{\Phi} X, X}(\on{id}_{\euler{\Phi} X})} \\
	{\ta_{\mathrm{G};\euler{\Phi}\mathbf{M}\mathrm{F} X, \mathbf{M}\mathrm{F} X}(\on{id}_{\euler{\Phi}\mathbf{M}\mathrm{F}X}) \circ \mathbf{N}\mathrm{G}\ta_{\mathrm{F};\euler{\Phi} X, X}(\on{id}_{\euler{\Phi} X})} \\
	{\euler{\Phi} (\mathbf{m}_{G,F})_{X} \circ \ta_{\mathrm{G};\euler{\Phi}\mathbf{M}\mathrm{F} X, \mathbf{M}\mathrm{F} X}(\on{id}_{\euler{\Phi}\mathbf{M}\mathrm{F}X}) \circ \mathbf{N}\mathrm{G}\ta_{\mathrm{F};\euler{\Phi} X, X}(\on{id}_{\euler{\Phi} X})} && {\ta_{\mathrm{GF},\euler{\Phi} X, X}(\on{id}_{\euler{\Phi} X}) \circ (\mathbf{n}_{\mathrm{G,F}})_{\euler{\Phi} X}}
	\arrow[maps to, from=1-1, to=2-1]
	\arrow[maps to, from=2-1, to=3-1]
	\arrow[from=4-1, to=4-3, equal]
	\arrow[maps to, from=3-1, to=4-1]
	\arrow[maps to, from=1-1, to=1-3]
	\arrow[maps to, from=1-3, to=4-3]
\end{tikzcd}\]
which is precisely the multiplicativity condition:
\[\begin{tikzcd}[column sep = huge, row sep = small]
	{\mathbf{N}\mathrm{G}\mathbf{N}\mathrm{F}\euler{\Phi}X} && {\mathbf{N}\mathrm{G}\euler{\Phi}\mathbf{M}\mathrm{F}X} && {\euler{\Phi}\mathbf{M}\mathrm{G}\mathbf{M}\mathrm{F}X} \\
	\\
	{\mathbf{N}\mathrm{GF}\euler{\Phi}X} &&&& {\euler{\Phi}\mathbf{M}\mathrm{GF}X}
	\arrow["{\mathbf{N}\mathrm{G}\ta_{\mathrm{F};\euler{\Phi}X,X}(\on{id}_{\euler{\Phi}X})}", from=1-1, to=1-3]
	\arrow["{\ta_{\mathrm{G};\euler{\Phi}\mathbf{M}\mathrm{F}X,\mathbf{M}\mathrm{F}X}(\on{id}_{\euler{\Phi}\mathbf{M}\mathrm{F}X})}", from=1-3, to=1-5]
	\arrow["{\euler{\Phi}(\mathbf{m}_{\mathrm{G,F}})_{X}}", from=1-5, to=3-5]
	\arrow["{(\mathbf{n}_{\mathrm{G,F}})_{\euler{\Phi}X}}"', from=1-1, to=3-1]
	\arrow["{\ta_{\mathrm{GF},\euler{\Phi}X,X}(\on{id}_{\euler{\Phi}X})}"', from=3-1, to=3-5]
\end{tikzcd}\]
Similarly, using the commutativity of
\[\begin{tikzcd}
	& {\Hom{\euler{\Phi} X, \euler{\Phi} X}} \\
	& {\Hom{\mathbf{N}\mathbb{1}\euler{\Phi} X, \euler{\Phi} \mathbf{M}\mathbb{1}X}} && {\Hom{\euler{\Phi} X, \euler{\Phi} \mathbf{M}\mathbb{1} X}}
	\arrow["{\ta_{\mathbb{1}; \euler{\Phi} X, X}}", from=1-2, to=2-2]
	\arrow["{\Hom{\euler{\Phi} X, \euler{\Phi} \mathbf{m}_{X}}}", from=1-2, to=2-4]
	\arrow["{\Hom{\mathbf{n}_{X},\euler{\Phi} \mathbf{M}\mathbb{1}X}}"', shift right=1, from=2-2, to=2-4]
\end{tikzcd}\]
  and evaluating the maps in the diagram at $\on{id}_{\euler{\Phi} X}$, we find that our candidate collection also satisfies the unitality condition for a lax $\csym{C}$-module functor.
 \end{proof}

 The following definitions were initially introduced in \cite[Section~3]{La} in the more general case of categories enriched in a symmetric monoidal category $\mathcal{V}$; as before, we specialize to the case $\mathcal{V} = \mathbf{Vec}_{\Bbbk}$. More detailed accounts are found in \cite[Section~4]{BD}, \cite[Section~7.9]{Bo}.
 \begin{definition}
  A $\Bbbk$-linear category $\mathcal{C}$ is said to be {\it Cauchy complete} if any $\Bbbk$-linear profunctor $\mathcal{B} \xslashedrightarrow{} \mathcal{C}$ admitting a right adjoint in $\mathbf{Prof}$ is {\it representable}, i.e. isomorphic to a profunctor of the form $\Hom{-,\Phi -}$ for some functor $\Phi: \mathcal{B} \rightarrow \mathcal{C}$.

  The {\it Cauchy completion} $\mathcal{C}^{\mathsf{c}}$ of a $\Bbbk$-linear category $\mathcal{C}$ is the subcategory of its presheaf category given by retracts of finite biproducts of representable presheaves.
 \end{definition}
 A $\Bbbk$-linear category is Cauchy complete if and only if it is idempotent split and additive.

As described in \ref{Adjunctions}, the embedding $\iota_{\mathcal{C}}^{\mathsf{c}}: \mathcal{C} \rightarrow \mathcal{C}^{\mathsf{c}}$ gives an adjunction $(\Hom{-,\iota_{\mathcal{C}}^{\mathsf{c}}-},\Hom{\iota_{\mathcal{C}}^{\mathsf{c}}-,-},\eta, \varepsilon)$ in $\mathbf{Prof}_{\Bbbk}$. This adjunction is an adjoint equivalence. An analogous statement holds in the bicategory $\csym{C}\!\on{-Tamb}$: using the universal property of Cauchy completion, given a $\csym{C}$-module category $\mathbf{M}$, there is a canonical way to extend the $\csym{C}$-module category structure to $\mathbf{M}^{\mathsf{c}}$; see e.g. \cite[Section~3.2]{Str}. The embedding $\iota_{\mathcal{C}}^{\mathsf{c}}$ then becomes a $\csym{C}$-module functor.
Since a morphism of Tambara modules is an isomorphism if and only if its underlying morphism of profunctors is an isomorphism, we find that the adjunction $(\Hom{-,\iota_{\mathcal{C}}^{\mathsf{c}}-},\Hom{\iota_{\mathcal{C}}^{\mathsf{c}}-,-},\eta, \varepsilon)$ in $\csym{C}\!\on{-Tamb}$ becomes an adjoint equivalence. We thus find an analogue of \cite[Theorem~7.9.4]{Bo}:

\begin{corollary}\label{CauchyTrick}
 Let $\mathbf{M},\mathbf{N}$ be a pair of $\csym{C}$-module categories. We have an equivalence $\mathbf{M} \simeq \mathbf{N}$ in $\csym{C}\!\on{-Tamb}$ if and only if there is an equivalence $\mathbf{M}^{\mathsf{c}} \simeq \mathbf{N}^{\mathsf{c}}$ in $\csym{C}\!\on{-Mod}$. In particular, if $\mathbf{M}$ and $\mathbf{N}$ are Cauchy complete, an equivalence $\mathbf{M} \simeq \mathbf{N}$ exists in $\csym{C}\!\on{-Tamb}$ if and only if it exists in $\csym{C}\!\on{-Mod}$.
\end{corollary}

 We now show that, similarly to profunctors, also Tambara modules with a Cauchy complete codomain and a right adjoint are representable:
\begin{proposition}
 Let $\mathbf{M},\mathbf{N}$ be $\csym{C}$-module categories, with $\mathbf{N}$ Cauchy complete. Let $\mathtt{\Phi}: \mathbf{M} \xslashedrightarrow{} \mathbf{N}$ be a Tambara module which admits a right adjoint in $\csym{C}\!\on{-Tamb}$. There is a (strong) $\csym{C}$-module functor $\euler{\Phi} \in \csym{C}\!\on{-Mod}(\mathbf{M,N})$ such that $\mathtt{\Phi} \simeq \Hom{-,\euler{\Phi} -}$.
\end{proposition}

\begin{proof}
 Let $(\mathtt{\Phi}, \mathtt{\Phi}^{\ast})$ be an adjoint pair in $\csym{C}\!\on{-Tamb}$. Passing it under the forgetful pseudofunctor, we obtain an underlying adjoint pair of profunctors. By \cite[Theorem~7.9.3]{Bo}, there is a functor $\euler{\Phi}$ such that the underlying profunctor of $\mathtt{\Phi}$ is isomorphic to $\Hom{-,\euler{\Phi} -}$. Transporting structures, we may assume that the underlying profunctors of $\mathtt{\Phi}$ and $\mathtt{\Phi}^{\ast}$ equal $\Hom{-,\euler{\Phi} -}$ and $\Hom{\euler{\Phi} -,-}$, respectively. Under this assumption, let $\eta$ and $\varepsilon$ be a unit and a counit giving an adjunction $(\mathtt{\Phi}, \mathtt{\Phi}^{\ast},\eta, \varepsilon)$ in $\csym{C}\!\on{-Tamb}$.

 In view of Lemma \ref{RepresentableGivesLax}, it suffices to show that the morphism $\ta_{\mathrm{F};\euler{\Phi} X, X}^{\mathtt{\Phi}}(\on{id}_{\euler{\Phi} X})$ is invertible for all $\mathrm{F}, X$.

From the isomorphism
\[
\mathbf{Prof}(\mathbf{M},\mathbf{M})(\Hom{-,-}_{\mathbf{M}}, \Hom{\euler{\Phi}-,\euler{\Phi}-}) \simeq \int_{\mathbf{M}} \Hom{\euler{\Phi}-,\euler{\Phi} -} \simeq \mathbf{Cat}_{\Bbbk}(\euler{\Phi},\euler{\Phi})
\]
we conclude that the underlying profunctor morphism of $\eta$ corresponds to an endotransformation $\widehat{\eta}$ of $\euler{\Phi}$. The components of this transformation are $\widehat{\eta}_{X} := \eta_{X,X}(\on{id}_{X})$.

On the other hand, the counit $\varepsilon$ is given by a family $\setj{\varepsilon_{Y,Y'}^{X}: \Hom{Y, \euler{\Phi} X} \kotimes \Hom{\euler{\Phi} X,Y'} \rightarrow \Hom{Y,Y'}}$, extranatural in $X$ and natural in $Y,Y'$. We conclude that for any $f \in \Hom{Y, \euler{\Phi} X}$ and $g \in \Hom{\euler{\Phi} X, Y'}$ we have
\[
 \varepsilon_{Y,Y'}^{X}(f \otimes g) = g \circ \varepsilon_{\euler{\Phi} X, \euler{\Phi} X}^{X}(\on{id}_{\euler{\Phi} X} \otimes \on{id}_{\euler{\Phi} X}) \circ f.
\]
Similarly to the proof of \cite[Proposition~7.9.2]{Bo}, the triangle equations coming from the adjunction yield
\[
 \begin{aligned}
  g = \varepsilon(\on{id}_{\euler{\Phi} X} \otimes g) \circ \widehat{\eta}_{X} ; \quad f = \widehat{\eta}_{X} \circ \varepsilon(f \otimes \on{id}_{\euler{\Phi} X}),
 \end{aligned}
\]
for $f,g$ as above. Setting $f = \on{id}_{\euler{\Phi} X} = g$, we find that
\begin{equation}\label{UnitCounitIso}
\varepsilon(\on{id}_{\euler{\Phi} X} \otimes \on{id}_{\euler{\Phi} X}) = \widehat{\eta}_{X}^{-1}.
\end{equation}

Since $\varepsilon$ is a morphism of Tambara modules, the diagram
\[\begin{tikzcd}
	{\Hom{Y,\euler{\Phi} X} \kotimes \Hom{\euler{\Phi} X, Y'}} & {\Hom{Y, Y'}} \\
	{\Hom{\mathbf{N}\mathrm{F}Y,\euler{\Phi} \mathbf{M}\mathrm{F} X} \kotimes \Hom{\euler{\Phi} \mathbf{M}\mathrm{F} X, \mathbf{N}\mathrm{F} Y'}} & {\Hom{\mathbf{N}\mathrm{F}Y,\mathbf{N}\mathrm{F}Y'}}
	\arrow["{\ta_{\mathrm{F};Y,X}^{\mathtt{\Phi}} \otimes \ta_{\mathrm{F};Y,X}^{\mathtt{\Phi}^{\ast}}}", from=1-1, to=2-1]
	\arrow["{\varepsilon_{Y,Y'}^{X}}", from=1-1, to=1-2]
	\arrow["{\mathbf{N}\mathrm{F}_{Y,Y'}}", from=1-2, to=2-2]
	\arrow["{\varepsilon_{\mathbf{N}\mathrm{F}Y,\mathbf{N}\mathrm{F}Y'}^{\mathbf{M}\mathrm{F}X}}"', from=2-1, to=2-2]
\end{tikzcd}\]
commutes. Setting $Y = \euler{\Phi} X = Y'$ and chasing $\on{id}_{\euler{\Phi} X} \otimes \on{id}_{\euler{\Phi} X}$, we obtain
\[
{\varepsilon_{\mathbf{N}\mathrm{F}\euler{\Phi}X, \mathbf{N}\mathrm{F}\euler{\Phi}X}\big(\ta_{\mathrm{F}; \euler{\Phi} X, X}^{\mathtt{\Phi}}(\on{id}_{\euler{\Phi} X}) \otimes \ta_{\mathrm{F}; X, \euler{\Phi} X}^{\mathtt{\Phi}^{\!\ast}}(\on{id}_{\euler{\Phi} X})\big)} = {\mathbf{N}\mathrm{F}\big(\varepsilon_{\euler{\Phi} X, \euler{\Phi} X}(\on{id}_{\euler{\Phi} X} \otimes \on{id}_{\euler{\Phi} X})\big)}
\]
which, using equations \eqref{YonedaFormula} and \eqref{UnitCounitIso}, can be rewritten as
\begin{equation}\label{EqFromCounit}
 {\ta_{\mathrm{F}; X, \euler{\Phi} X}^{\mathtt{\Phi}^{\!\ast}}(\on{id}_{\euler{\Phi} X}) \circ \widehat{\eta}_{\mathbf{M}\mathrm{F}X}^{-1} \circ \ta_{\mathrm{F}; \euler{\Phi} X, X}^{\mathtt{\Phi}}(\on{id}_{\euler{\Phi} X})} = {\mathbf{N}\mathrm{F}\widehat{\eta}_{X}^{-1}}.
\end{equation}

On the other hand, $\eta$ being a morphism of Tambara modules makes the diagram
\[\begin{tikzcd}
	{\Hom{X,X'}} && {\int^{Y}\Hom{\euler{\Phi} X, Y} \kotimes \Hom{Y, \euler{\Phi} X'}} \\
	{\Hom{\mathbf{M}\mathrm{F}X, \mathbf{M}\mathrm{F}X'}} && {\int^{Y}\Hom{\euler{\Phi} \mathbf{M}\mathrm{F}X, Y} \kotimes \Hom{Y, \euler{\Phi} \mathbf{M}\mathrm{F} X'}}
	\arrow["{\mathbf{M}\mathrm{F}_{X,X'}}"', from=1-1, to=2-1]
	\arrow["{\eta_{X,X'}}", from=1-1, to=1-3]
	\arrow["{\eta_{\mathbf{M}\mathrm{F}X, \mathbf{M}\mathrm{F}X'}}"', from=2-1, to=2-3]
	\arrow["{\ta_{\euler{\Phi} X, \euler{\Phi} X'}^{\mathtt{\Phi}^{\ast}\circ \mathtt{\Phi}}}", from=1-3, to=2-3]
\end{tikzcd}\]
commute. Applying Yoneda lemma to the right half of the above diagram, we may rewrite it as
\[\begin{tikzcd}[row sep = scriptsize]
	{\Hom{X,X'}} && {\Hom{\euler{\Phi} X, \euler{\Phi} X'}} \\
	&& {\Hom{\mathbf{N}\mathrm{F}\euler{\Phi} X,\mathbf{N}\mathrm{F}\euler{\Phi} X'}} \\
	{\Hom{\mathbf{M}\mathrm{F}X, \mathbf{M}\mathrm{F}X'}} && {\Hom{\euler{\Phi} \mathbf{M}\mathrm{F}X,\euler{\Phi} \mathbf{M}\mathrm{F}X'}}
	\arrow["{\mathbf{M}\mathrm{F}_{X,X'}}"', from=1-1, to=3-1]
	\arrow["{\mathbf{N}\mathrm{F}_{\euler{\Phi} X, \euler{\Phi} X'}}", from=1-3, to=2-3]
	\arrow["{\Hom{\ta^{\mathtt{\Phi}^{\!\ast}}_{\mathrm{F};X,\euler{\Phi} X}(\on{id}_{\euler{\Phi} X}), \ta^{\mathtt{\Phi}}_{\mathrm{F};\euler{\Phi} X', X'}(\on{id}_{\euler{\Phi} X'})}}", from=2-3, to=3-3]
	\arrow["{\eta_{X,X'}}", shift left=1, from=1-1, to=1-3]
	\arrow["{\eta_{\mathbf{M}\mathrm{F}X, \mathbf{M}\mathrm{F}X'}}"', shift right=1, from=3-1, to=3-3]
\end{tikzcd}\]
Setting $X = X'$ and chasing $\on{id}_{X}$ we find
\begin{equation}\label{EqFromUnit}
   \widehat{\eta}_{\mathbf{M}\mathrm{F}X} = \ta_{\mathrm{F};\euler{\Phi} X, X}^{\mathtt{\Phi}}(\on{id}_{\euler{\Phi} X}) \circ \mathbf{N}\mathrm{F}(\widehat{\eta}_{X}) \circ \ta_{\mathrm{F};X,\euler{\Phi} X}^{\mathtt{\Phi}^{\!\ast}}(\on{id}_{\euler{\Phi} X})
\end{equation}
Since $\widehat{\eta}$ is an isomorphism, the equations \eqref{EqFromCounit} and \eqref{EqFromUnit} imply that $\ta^{\mathtt{\Phi}}_{\mathrm{F};\euler{\Phi} X, X}(\on{id}_{\euler{\Phi} X})$ is simultaneously both a split epimorphism and a split monomorphism, and thus also an isomorphism.
\end{proof}

Combining the results of this section, we find the following:
\begin{theorem}\label{Section4Main}
 The proarrow equipment $\mathbb{P}$ is a map equipment. A Tambara module between Cauchy complete $\csym{C}$-module categories has a right adjoint if and only if it is representable.
\end{theorem}

\subsection{\texorpdfstring{$\csym{C}\!\on{-Tamb}$}{C-Tamb} as a tame bicategory}

Recall that for any monoidal category $\csym{A}$ we may consider monoids, modules and bimodules internal to $\csym{A}$, see e.g. \cite[Section~4.1]{GJ}. If $\csym{A}$ admits reflective coequalizers, we may define the balanced tensor product of a right module with a left module: if $A$ is a monoid in $\csym{A}$, $M$ is a right $A$-module and $N$ is a left $A$-module, we let $M \otimes_{A} N$ be the coequalizer of
\[\begin{tikzcd}
	{M \otimes A \otimes N} && {M \otimes N}
	\arrow["{M \otimes \mathtt{la}_{N}}", shift left=1, from=1-1, to=1-3]
	\arrow["{\mathtt{ra}_{M} \otimes N}"', shift right=1, from=1-1, to=1-3]
\end{tikzcd}\]
If the tensor product in $\csym{A}$ preserves reflexive coequalizers and $M$ and $N$ are instead taken to be a $B$-$A$-bimodule and an $A$-$C$-bimodule respectively (for further monoids $B,C \in \csym{A}$), then $M \otimes_{A} N$ can be endowed with the structure of a $B$-$C$-bimodule in the evident way - see \cite[Section~4.2]{GJ} for a detailed account. In particular, we obtain a bicategory $\on{Bimod}(\csym{A})$. In that case, following \cite[Definition~4.2.1]{GJ}, we say that $\csym{A}$ is {\it tame}. More generally, if $\csym{A}$ is a bicategory, we say that $\csym{A}$ is tame if $\csym{A}(\mathtt{i,j})$ admits reflexive coequalizers for all $\mathtt{i,j} \in \on{Ob}\csym{A}$, and if horizontal composition in $\csym{A}$ preserves coequalizers in each variable. In particular, if $\csym{A}$ is tame, then so is the monoidal category $\csym{A}(\mathtt{i,i})$, for every $\mathtt{i} \in \on{Ob}\csym{A}$.

\begin{proposition}\label{Tameness}
 $\csym{C}\!\on{-Tamb}$ is tame.
\end{proposition}

\begin{proof}
 Let $\mathbf{M},\mathbf{N} \in \csym{C}\!\on{-Tamb}$. Forgetting the $\csym{C}$-module structures, we may consider the category $\mathbf{Prof}_{\Bbbk}(\mathbf{M},\mathbf{N})$. Since $\mathbf{Prof}_{\Bbbk}(\mathbf{M},\mathbf{N}) = [\mathbf{N}^{\on{op}} \kotimes \mathbf{M},\mathbf{Vec}_{\Bbbk}]$, this category is cocomplete, and the colimits are constructed pointwise in $\mathbf{Vec}_{\Bbbk}$. Let $\mathtt{p,b} \in \csym{C}\!\on{-Tamb}(\mathbf{M},\mathbf{N})(\mathtt{\Psi},\mathtt{\Psi'})$. To construct the coequalizer of $\mathtt{p}$ and $\mathtt{b}$, it suffices to show that the evident assignment
 \[\begin{tikzcd}
	{\mathtt{\Psi}(Y,X)} & {\mathtt{\Psi}'(Y,X)} & {\on{coeq}(\mathtt{p},\mathtt{b})_{Y,X}} \\
	{\mathtt{\Psi}(\mathbf{N}\mathrm{F}Y,\mathbf{M}\mathrm{F}X)} & {\mathtt{\Psi}'(\mathbf{N}\mathrm{F}Y,\mathbf{M}\mathrm{F}X)} & {\on{coeq}(\mathtt{p},\mathtt{b})_{\mathbf{N}\mathrm{F}Y,\mathbf{N}\mathrm{F}X}}
	\arrow["{\ta^{\mathtt{\Psi}}_{\mathrm{F};Y,X}}", from=1-1, to=2-1]
	\arrow["{\ta^{\mathtt{\Psi}'}_{\mathrm{F};Y,X}}", from=1-2, to=2-2]
	\arrow["{\mathtt{p}_{Y,X}}", shift left=1, from=1-1, to=1-2]
	\arrow["{\mathtt{b}_{Y,X}}"', shift right=1, from=1-1, to=1-2]
	\arrow["{\mathtt{p}_{\mathbf{N}\mathrm{F}Y,\mathbf{M}\mathrm{F}X}}", shift left=1, from=2-1, to=2-2]
	\arrow["{\mathtt{b}_{\mathbf{N}\mathrm{F}Y,\mathbf{M}\mathrm{F}X}}"', shift right=1, from=2-1, to=2-2]
	\arrow[from=1-2, to=1-3]
	\arrow[from=2-2, to=2-3]
	\arrow[dashed, from=1-3, to=2-3]
\end{tikzcd}\]
for $\mathrm{F} \in \csym{C}, Y \in \mathbf{N}$ and $X \in \mathbf{M}$, gives a well-defined Tambara structure, which makes the profunctor morphisms from $\on{coeq}(\mathtt{p},\mathtt{b})$, induced by Tambara morphisms from $\mathtt{\Psi}'$, coequalizing $\euler{p}$ and $\euler{b}$, into Tambara morphisms. For the first claim, observe that since $\mathtt{\Psi}$ and $\mathtt{\Psi}'$ both satisfy the Tambara axiom, and since $\mathtt{p},\mathtt{b}$ are Tambara morphisms, the Tambara axiom diagrams for $\mathtt{\Psi},\mathtt{\Psi}'$ assemble to a commutative diagram of coequalizer diagrams. Taking the colimit of each, and the resulting morphisms between the colimits, shows that the Tambara axiom holds for the assignment above.

To see that also the latter claim holds, let $\mathtt{g}: \mathtt{\Psi}' \Rightarrow \mathtt{\Sigma}$ be a Tambara morphism coequalizing $\mathtt{p}$ and $\mathtt{b}$. We let $\overline{\mathtt{g}}: \on{coeq}(\mathtt{p},\mathtt{b}) \Rightarrow \mathtt{\Sigma}$ be the profunctor morphism coming from the universal property of $\on{coeq}(\mathtt{p},\mathtt{b})$ as the coequalizer in $\mathbf{Prof}_{\Bbbk}(\mathbf{M},\mathbf{N})$. It remains to observe that the right-back wall of the diagram
\[\begin{tikzcd}[ampersand replacement=\&]
	\&\& {\on{coeq}(\mathtt{p},\mathtt{b})_{Y,X}} \\
	{\mathtt{\Psi}(Y,X)} \& {\mathtt{\Psi}'(Y,X)} \&\& {\mathtt{\Sigma}(Y,X)} \\
	\&\& {\on{coeq}(\mathtt{p},\mathtt{b})_{\mathbf{N}\mathrm{F}Y,\mathbf{N}\mathrm{F}X}} \\
	{\mathtt{\Psi}(\mathbf{N}\mathrm{F}Y,\mathbf{M}\mathrm{F}X)} \& {\mathtt{\Psi}'(\mathbf{N}\mathrm{F}Y,\mathbf{M}\mathrm{F}X)} \&\& {\mathtt{\Sigma}(\mathbf{N}\mathrm{F}Y,\mathbf{M}\mathrm{F}X)}
	\arrow["{\ta^{\mathtt{\Psi}}_{\mathrm{F};Y,X}}", from=2-1, to=4-1]
	\arrow["{\ta^{\mathtt{\Psi}'}_{\mathrm{F};Y,X}}", from=2-2, to=4-2]
	\arrow["{\mathtt{p}_{Y,X}}", shift left=1, from=2-1, to=2-2]
	\arrow["{\mathtt{b}_{Y,X}}"', shift right=1, from=2-1, to=2-2]
	\arrow["{\mathtt{p}_{\mathbf{N}\mathrm{F}Y,\mathbf{M}\mathrm{F}X}}", shift left=1, from=4-1, to=4-2]
	\arrow["{\mathtt{b}_{\mathbf{N}\mathrm{F}Y,\mathbf{M}\mathrm{F}X}}"', shift right=1, from=4-1, to=4-2]
	\arrow[from=2-2, to=1-3]
	\arrow[from=4-2, to=3-3]
	\arrow[dashed, from=1-3, to=3-3]
	\arrow["{\ta^{\mathtt{\Sigma}}_{\mathrm{F};Y,X}}", from=2-4, to=4-4]
	\arrow["{\mathtt{g}_{Y,X}}"{description, pos=0.6}, from=2-2, to=2-4]
	\arrow["{\mathtt{g}_{\mathbf{N}\mathrm{F}Y,\mathbf{M}\mathrm{F}X}}"', from=4-2, to=4-4]
	\arrow["{\overline{\mathtt{g}}_{Y,X}}", from=1-3, to=2-4]
	\arrow["{\overline{\mathtt{g}}_{\mathbf{N}\mathrm{F}Y,\mathbf{M}\mathrm{F}X}}"{pos=0.6}, from=3-3, to=4-4]
\end{tikzcd}\]
commutes. Since the unlabelled maps are coequalizer maps, said commutativity follows from the fact that all paths in the diagram with $\mathtt{\Psi}'(Y,X)$ as domain and $\mathtt{\Sigma}(\mathbf{N}\mathrm{F}Y,\mathbf{M}\mathrm{F}X)$ as codomain give the same morphism, which in turn follows from the commutativity of the bottom, top, front, and left-back walls of the diagram.

To see that the horizontal composition in $\csym{C}\!\on{-Tamb}$ preserves the above constructed coequalizers, let $\mathtt{\Phi}: \mathbf{N} \xslashedrightarrow{} \mathbf{Q}$. We then have
\[
\begin{aligned}
 &\int^{Y \in \mathbf{N}} \mathtt{\Phi}(Z,Y) \kotimes \on{coeq}(\mathtt{p},\mathtt{b})_{Y,X} \simeq \int^{Y \in \mathbf{N}}\on{coeq}\big(\mathtt{\Phi}(Z,Y) \kotimes \mathtt{p}_{Y,X}, \mathtt{\Phi}(Z,Y) \kotimes \mathtt{b}_{Y,X}\big) \\
 &\simeq \on{coeq}\Bigg(
 \begin{tikzcd}[ampersand replacement=\&,column sep = huge]
	{\int^{Y} \mathtt{\Phi}(Z,Y) \kotimes \mathtt{\Psi}(Y,X)} \& {\int^{Y} \mathtt{\Phi}(Z,Y) \kotimes \mathtt{\Psi}'(Y,X)}
	\arrow["{\setj{\mathtt{\Phi}(Z,Y) \otimes \mathtt{p}_{Y,X}}}", shift left=3, from=1-1, to=1-2]
	\arrow["{\setj{\mathtt{\Phi}(Z,Y) \otimes \mathtt{b}_{Y,X}}}"', shift right=2, from=1-1, to=1-2]
\end{tikzcd}\Bigg),
\end{aligned}
\]
where the first isomorphism follows from the cocontinuity of the tensor product in $\mathbf{Vec}_{\Bbbk}$, and the latter from the commutativity of colimits. Functoriality of each side in $X,Z$ follows from functoriality of colimits, and thus the isomorphisms are natural in $X$ and $Z$. The Tambara structure on the left-hand side is induced by the collections $\setj{\ta^{\mathtt{\Phi}}_{\mathrm{F};Z,Y} \otimes \on{coeq}\big(\ta^{\mathtt{\Psi}}_{\mathrm{F};Y,X},\ta^{\mathtt{\Psi}'}_{\mathrm{F};Y,X}\big)}$, while the Tambara structure on the right-hand side is induced by the collections $\setj{\on{coeq}(\ta^{\mathtt{\Phi}}_{\mathrm{F};Y,Z} \otimes \ta^{\mathtt{\Psi}}_{\mathrm{F};Y,X},\ta^{\mathtt{\Phi}}_{\mathrm{F};Y,Z} \otimes \ta^{\mathtt{\Psi}'}_{\mathrm{F};Y,X})}$.  The left-hand side collection is mapped to the right-hand side collection under the above isomorphisms, showing that we have an isomorphism of Tambara modules.
\end{proof}

\section{Defining the pseudofunctor \texorpdfstring{$\na$}{na}}\label{DefiningNa}

Recall that by bicategorical Yoneda lemma, for any $\mathbf{M} \in \csym{C}\!\on{-Mod}$, we have an equivalence
\begin{equation}\label{BicatYo}
\mathbf{M}\xiso \csym{C}\!\on{-Mod}(\csym{C},\mathbf{M}),
\end{equation}
sending $X \in \mathbf{M}$ to $\euler{\Phi}_{X}$ satisfying $\euler{\Phi}_{X}(\mathrm{F}) = \mathbf{M}\mathrm{F}(X)$, for $\mathrm{F} \in \csym{C}$. For a detailed account, see \cite[Section~8.3]{JY}

Given $\mathtt{\Psi} \in \csym{C}\!\on{-Tamb}(\mathbf{M},\mathbf{N})$ and objects $X \in \mathbf{M}, Y \in \mathbf{N}$, we denote the Tambara module $\mathtt{\Psi} \smalltriangleup \euler{\Phi}_{X} \smalltriangledown \euler{\Phi}_{Y} \in \tam$ by $\mathtt{\Psi}[Y,X]$.
Its underlying profunctor is given by $(\mathrm{F},\mathrm{G}) \mapsto \mathtt{\Psi}(\mathbf{N}\mathrm{F}Y, \mathbf{M}\mathrm{G}X)$. If $\mathbf{M} = \mathbf{N}$ and $\mathtt{\Psi} = \Hom{-,-}_{\mathbf{M}}$, we simply write $[Y,X]$.
By definition, the Tambara structure of $\mathtt{\Psi}[Y,X]$ is given by
\begin{equation}\label{TambaraOnRestriction}
\begin{tikzcd}[column sep = huge]
	{\mathtt{\Psi}(\mathbf{N}\mathrm{F}Y, \mathbf{M}\mathrm{G}X)} & {\mathtt{\Psi}(\mathbf{N}\mathrm{H}\mathbf{N}\mathrm{F}Y, \mathbf{M}\mathrm{H}\mathbf{M}\mathrm{G}X)} & {\mathtt{\Psi}(\mathbf{N}\mathrm{HF}Y, \mathbf{M}\mathrm{HG}X)}
	\arrow["{\ta^{\mathtt{\Psi}}_{\mathrm{H;\mathbf{N}\mathrm{F}Y, \mathbf{M}\mathrm{G}X}}}", from=1-1, to=1-2]
	\arrow["{\mathtt{\Psi}(\mathbf{n}_{\mathrm{H},\mathrm{F}}^{-1},\mathbf{m}_{\mathrm{H},\mathrm{G}})}", from=1-2, to=1-3]
\end{tikzcd}
\end{equation}

\begin{proposition}\label{EndMonoid}
 Given a module category $\mathbf{M}$ and an object $X \in \mathbf{M}$, we endow $[X,X] \in \tam$ with the structure of a monoid object as follows:
 \begin{enumerate}
  \item
The unit morphism $\mathtt{e}: \csym{C}(-,-) \rightarrow [X,X]$ is that corresponding to $(\mathbf{M}\on{id}_{\mathbb{1}})_{X}$ under the isomorphism
\[
 \on{Hom}_{\ccf{C}\!\on{-Tamb}(\ccf{C},\ccf{C})}(\csym{C}(-,-), [X,X]) \simeq [X,X](\mathbb{1},\mathbb{1}) = \Hom{X,X}
\]
observed in Remark \ref{TambaraElements}. It is thus given by
 \[
 \begin{aligned}
  \mathtt{e}_{\mathrm{F,G}}: \csym{C}(\mathrm{F,G}) &\rightarrow [X,X](\mathrm{F,G}) = \Hom{\mathbf{M}\mathrm{F}X, \mathbf{M}\mathrm{G}X} \\
  \mathrm{b} &\mapsto (\mathbf{M}\mathrm{b})_{X}
 \end{aligned}.
 \]
 \item
 The multiplication morphism $\mathtt{m}$ is given by
 \[
 \begin{aligned}
  \int^{\mathrm{H}} \Hom{\mathbf{M}\mathrm{F}X, \mathbf{M}\mathrm{H}X} \kotimes \Hom{\mathbf{M}\mathrm{H}X, \mathbf{M}\mathrm{G}X} &\rightarrow \Hom{\mathbf{M}\mathrm{F}X,\mathbf{M}\mathrm{G}X} \\
  \mathrm{b} \otimes \mathrm{c} &\mapsto \mathrm{c \circ b}
 \end{aligned}
 \]
  \end{enumerate}
\end{proposition}

\begin{proof}
 It is clear that $\mathtt{m}$ is a well-defined morphism of profunctors. It is also a morphism of Tambara modules, since the diagram
 \[\begin{tikzcd}
	{\Hom{\mathbf{M}\mathrm{F}X, \mathbf{M}\mathrm{H}X} \kotimes \Hom{\mathbf{M}\mathrm{H}X, \mathbf{M}\mathrm{G}X}} && {\Hom{\mathbf{M}\mathrm{F}X, \mathbf{M}\mathrm{G}X}} \\
	{\Hom{\mathbf{M}\mathrm{K}\mathbf{M}\mathrm{F}X, \mathbf{M}\mathrm{K}\mathbf{M}\mathrm{H}X} \kotimes \Hom{\mathbf{M}\mathrm{K}\mathbf{M}\mathrm{H}X, \mathbf{M}\mathrm{K}\mathbf{M}\mathrm{G}X}} && {\Hom{\mathbf{M}\mathrm{K}\mathbf{M}\mathrm{F}X, \mathbf{M}\mathrm{K}\mathbf{M}\mathrm{G}X}} \\
	{\Hom{\mathbf{M}\mathrm{KF}X, \mathbf{M}\mathrm{KH}X} \kotimes \Hom{\mathbf{M}\mathrm{KH}X, \mathbf{M}\mathrm{KG}X}} && {\Hom{\mathbf{M}\mathrm{KF}X, \mathbf{M}\mathrm{KG}X}}
	\arrow["{\mathbf{M}\mathrm{K}_{\mathbf{M}\mathrm{F}X, \mathbf{M}\mathrm{H}X} \kotimes \mathbf{M}\mathrm{K}_{\mathbf{M}\mathrm{H}X, \mathbf{M}\mathrm{G}X}}"', from=1-1, to=2-1]
	\arrow["{\Hom{\mathbf{m}_{\mathrm{K,F}}^{-1}, \mathbf{m}_{\mathrm{K,H}}} \otimes \Hom{\mathbf{m}_{\mathrm{K,H}}^{-1}, \mathbf{m}_{\mathrm{K,G}}}}"', from=2-1, to=3-1]
	\arrow["{\mathtt{m}_{\mathrm{H;F,G}}}", from=1-1, to=1-3]
	\arrow["{\mathbf{M}\mathrm{K}_{\mathbf{M}\mathrm{F}X, \mathbf{M}\mathrm{G}X}}", from=1-3, to=2-3]
	\arrow["{\Hom{\mathbf{m}_{\mathrm{K,F}}^{-1}, \mathbf{m}_{\mathrm{K,G}}}}", from=2-3, to=3-3]
	\arrow["{\mathtt{m}_{\mathrm{KH;KF,KG}}}", from=3-1, to=3-3]
\end{tikzcd}\]
 commutes by functoriality of $\mathbf{M}\mathrm{K}$.
  Associativity of the multiplication morphism follows from the associativity of composition in $\mathbf{M}$.
 Left unitality follows by noting that the unit map $\int^{\mathrm{H}} \csym{C}(\mathrm{F,H}) \kotimes \Hom{\mathbf{M}\mathrm{H}X, \mathbf{M}\mathrm{G}X} \rightarrow \Hom{\mathbf{M}\mathrm{F}X, \mathbf{M}\mathrm{G}X}$
 is induced by the maps $\csym{C}(\mathrm{F,H}) \kotimes \Hom{\mathbf{M}\mathrm{H}X, \mathbf{M}\mathrm{G}X} \rightarrow \Hom{\mathbf{M}\mathrm{F}X, \mathbf{M}\mathrm{G}X}$
 given by $\alpha \otimes f \mapsto f \circ [X,X](\alpha, \mathrm{G}) = f \circ \mathbf{M}\alpha_{X}$, i.e. the same map as the morphism $\mathtt{e}_{\mathrm{F,H}} \otimes \Hom{\mathbf{M}\mathrm{H}X,\mathbf{M}\mathrm{G}X}$, followed by composition. Right unitality is analogous.
\end{proof}

\begin{proposition}\label{HomBimod}
 Given $\csym{C}$-module categories $\mathbf{M},\mathbf{N}$, a Tambara module $\mathtt{\Psi}: \mathbf{M} \xslashedrightarrow{} \mathbf{N}$ and objects $X \in \mathbf{M}$ and $Y \in \mathbf{N}$, the Tambara module $\mathtt{\Psi}[Y,X]$ can be endowed with the structure of a $[Y,Y]$-$[X,X]$-bimodule as follows:
  \begin{enumerate}
  \item The left action $\mathtt{la}: [Y,Y] \diamond \mathtt{\Psi}[Y,X] \rightarrow \mathtt{\Psi}[Y,X]$ is given by
  \[
  \begin{aligned}
   \mathtt{la}_{\mathrm{F,G}}: \int^{\mathrm{H}} \Hom{\mathbf{N}\mathrm{F}Y, \mathbf{N}\mathrm{H}Y} \kotimes \mathtt{\Psi}(\mathbf{N}\mathrm{H}Y, \mathbf{M}\mathrm{G}X) &\rightarrow \mathtt{\Psi}(\mathbf{N}\mathrm{F}Y, \mathbf{M}\mathrm{G}X) \\
   f \otimes v &\mapsto \mathtt{\Psi}(f, \mathbf{M}\mathrm{G}X)(v)
  \end{aligned}
  \]
    \item The right action $\mathtt{ra}: \mathtt{\Psi}[Y,X] \diamond [X,X] \rightarrow \mathtt{\Psi}[Y,X]$ is given by
  \[
  \begin{aligned}
   \mathtt{ra}_{\mathrm{F,G}}: \int^{\mathrm{H}} \mathtt{\Psi}(\mathbf{N}\mathrm{F}Y, \mathbf{M}\mathrm{H}X) \kotimes \Hom{\mathbf{M}\mathrm{H}X,\mathbf{M}\mathrm{G}X} &\rightarrow \mathtt{\Psi}(\mathbf{N}\mathrm{F}Y, \mathbf{M}\mathrm{G}X) \\
   v \otimes g &\mapsto \mathtt{\Psi}(\mathbf{N}\mathrm{F}Y, g)(v)
  \end{aligned}
  \]
 \end{enumerate}
\end{proposition}

\begin{proof}
 The verification of the left, respectively right module axioms for $\mathtt{la}$ and $\mathtt{ra}$ is analogous to the verifications of Proposition~\ref{EndMonoid}. The multiplicativity axioms follow from functoriality of $\mathtt{\Psi}$ in the respective variables.

 Since the unitors in $\tam$ are given by evaluating the respective actions of $\csym{C}$, the commutativity of the diagram
 \[\begin{tikzcd}[column sep = huge, row sep = scriptsize]
	{\int^{\mathrm{H}}\csym{C}(\mathrm{F,H}) \kotimes \mathtt{\Psi}(\mathbf{N}\mathrm{H}Y, \mathbf{M}\mathrm{G}X)} & {\int^{\mathrm{H}} \Hom{\mathbf{N}\mathrm{F}Y, \mathbf{N}\mathrm{H}Y} \kotimes \mathtt{\Psi}(\mathbf{N}\mathrm{H}Y, \mathbf{M}\mathrm{G}X)} \\
	& {\mathtt{\Psi}(\mathbf{N}\mathrm{F}Y,\mathbf{M}\mathrm{G}X)}
	\arrow["{(\mathtt{e}^{[Y,Y]}\diamond \mathtt{\Psi}[Y,X])_{\mathrm{F,G}}}", shift left=2, from=1-1, to=1-2]
	\arrow["{(\mathsf{l}^{-1}_{\mathtt{\Psi}[Y,X]})_{\mathrm{F,G}}}"', from=1-1, to=2-2]
	\arrow["{\mathtt{la}_{\mathrm{F,G}}}", from=1-2, to=2-2]
\end{tikzcd}\]
 follows from the commutativity of
 \[\begin{tikzcd}[row sep = scriptsize]
	{\mathrm{f} \otimes x} && {(\mathbf{N}\mathrm{f})_{Y} \otimes x} \\
	& {\mathtt{\Psi}[Y,X](\mathrm{f},\mathrm{G})(x)} & {\mathtt{\Psi}((\mathbf{N}\mathrm{f})_{Y},\mathbf{M}\mathrm{G}X)(x)},
	\arrow[maps to, from=1-1, to=1-3]
	\arrow[maps to, from=1-3, to=2-3]
	\arrow[maps to, from=1-1, to=2-2]
	\arrow[from=2-2, to=2-3, equal]
\end{tikzcd}.\]
 This shows unitality for the left module structure. Unitality for the right module structure is similar.
 Finally, we have $\mathtt{\Psi}(f,\mathbf{M}\mathrm{G}X)\mathtt{\Psi}(\mathbf{N}\mathrm{F}Y, g) = \mathtt{\Psi}(\mathbf{N}\mathrm{F}Y, g)\mathtt{\Psi}(f,\mathbf{M}\mathrm{G}X)$ for all $f$ and $g$, which shows that $\mathtt{la}$ and $\mathtt{ra}$ commute.
\end{proof}

Given a $\csym{C}$-module category $\mathbf{M}$ and an object $X \in \mathbf{M}$, let $\csym{C}\ostar X$ be the full subcategory of $\mathbf{M}$ whose collection of objects is given by $\setj{\mathbf{M}\mathrm{F}X \; | \; \mathrm{F} \in \csym{C}}$. In particular, given $\csym{C}$-module categories $\mathbf{M},\mathbf{N}$ and objects $X \in \mathbf{M}, Y \in \mathbf{N}$, the objects of $(\csym{C}^{\on{opp}}\kotimes \csym{C})\ostar (Y,X)$ are given by $\setj{(\mathbf{N}\mathrm{F}Y,\mathbf{M}\mathrm{G}X) \; | \; \mathrm{F},\mathrm{G} \in \csym{C}}$.

\begin{lemma}\label{BimoduleBinaturality}
 A morphism of Tambara modules $\mathtt{s} = (\mathtt{s}_{\mathrm{F},\mathrm{G}})_{\mathrm{F,G} \in \ccf{C}} \in \tam (\mathtt{\Psi}[Y,X],\mathtt{\Psi}'[Y,X])$ is a morphism of $[Y,Y]$-$[X,X]$-bimodules if and only if the assignment
 \begin{equation}\label{OtherNaturality}
\widetilde{\mathtt{s}}_{\mathbf{N}\mathrm{F}Y,\mathbf{M}\mathrm{G}X} := \mathtt{s}_{\mathrm{F},\mathrm{G}}
 \end{equation}
 defines a natural transformation from $\mathtt{\Psi}_{|(\ccf{C}^{\on{opp}}\kotimes \ccf{C})\ostar (Y,X)}$ to $\mathtt{\Psi}'_{|(\ccf{C}^{\on{opp}}\kotimes\ccf{C})\ostar (Y,X)}$.
\end{lemma}

\begin{proof}
 The latter condition holds if and only if for any $\mathrm{F,G,H} \in \csym{C}, f \in \Hom{\mathbf{N}\mathrm{F}Y, \mathbf{N}\mathrm{H}Y}$ and $g \in \Hom{\mathbf{M}\mathrm{H}X, \mathbf{M}\mathrm{G}X}$ we have
 \begin{equation}\label{fcb1}
 \mathtt{\Psi}'(f,\mathbf{M}\mathrm{G}X)\circ \mathtt{s}_{\mathrm{\mathrm{H},\mathrm{G}}} = \mathtt{s}_{\mathrm{F},\mathrm{G}}\circ \mathtt{\Psi}(f,\mathbf{M}\mathrm{G}X)
 \end{equation}
 and
 \begin{equation}\label{fcb2}
 \mathtt{\Psi}'(\mathbf{N}\mathrm{F}Y,g)\circ \mathtt{s}_{\mathrm{\mathrm{F},\mathrm{H}}} = \mathtt{s}_{\mathrm{F},\mathrm{G}}\circ \mathtt{\Psi}(\mathbf{N}\mathrm{F}Y,g).
 \end{equation}
 The equality $\setj{\mathtt{\Psi}'(f,\mathbf{M}\mathrm{G}X)\circ \mathtt{s}_{\mathrm{\mathrm{H},\mathrm{G}}} \; | \; \mathrm{H} \in \csym{C}} = \setj{\mathtt{s}_{\mathrm{F},\mathrm{G}}\circ \mathtt{\Psi}(f,\mathbf{M}\mathrm{G}X)\; | \; \mathrm{H} \in \csym{C}}$ of extranatural families is precisely the commutativity of the square
 \[\begin{tikzcd}[row sep = scriptsize]
	{\int^{\mathrm{H}} \Hom{\mathbf{N}\mathrm{F}Y,\mathbf{N}\mathrm{H}Y} \kotimes \mathtt{\Psi}(\mathbf{N}\mathrm{H}Y,\mathbf{M}\mathrm{G}X)} & {\mathtt{\Psi}(\mathbf{N}\mathrm{F}Y,\mathbf{M}\mathrm{G}X)} \\
	{\int^{\mathrm{H}} \Hom{\mathbf{N}\mathrm{F}Y,\mathbf{N}\mathrm{H}Y} \kotimes \mathtt{\Psi}'(\mathbf{N}\mathrm{H}Y,\mathbf{M}\mathrm{G}X)} & {\mathtt{\Psi}'(\mathbf{N}\mathrm{F}Y,\mathbf{M}\mathrm{G}X)}
	\arrow["{([Y,Y]\diamond\mathtt{s})_{\mathrm{F,G}}}", from=1-1, to=2-1]
	\arrow["{\mathtt{la}_{\mathrm{F,G}}^{\mathtt{\Psi}}}", from=1-1, to=1-2]
	\arrow["{\mathtt{la}_{\mathrm{F,G}}^{\mathtt{\Psi}'}}", from=2-1, to=2-2]
	\arrow["{\mathtt{s}_{\mathrm{F},\mathrm{G}}}", from=1-2, to=2-2]
\end{tikzcd}\]
which, by definition, holds for all $\mathrm{F,G}$ if and only if $\mathtt{s}$ is a morphism of left $[Y,Y]$-modules. We conclude that Equations~\eqref{fcb1} and~\eqref{fcb2} hold for all $\mathrm{F,G,H}$ and $f,g$ if and only if $\mathtt{s}$ is a morphism of $[X,X]$-$[Y,Y]$-bimodules.
\end{proof}

\begin{corollary}\label{LocalFunctors}
 For any $\mathtt{t}\in \csym{C}\!\on{-Tamb}(\mathbf{M},\mathbf{N})(\mathtt{\Psi},\mathtt{\Psi}')$, the morphism $\mathtt{t}[Y,X]\in \tam( \mathtt{\Psi}[Y,X] , \mathtt{\Psi}'[Y,X])$ is a morphism of $[Y,Y]$-$[X,X]$-bimodules.
 Functoriality of restrictions and corestrictions (Equations~\eqref{prl1} and \eqref{prl2} of Proposition~\ref{ResPseudoOneSide}) thus yields functors
\[
\widehat{\na}_{(\mathbf{N},Y),(\mathbf{M},X)}: \csym{C}\!\on{-Tamb}(\mathbf{M},\mathbf{N}) \rightarrow \on{Bimod}(\tam)
\]
for any choice of objects $Y \in \mathbf{N}$ and $X \in \mathbf{M}$, sending $\mathtt{\Psi}$ to $\mathtt{\Psi}[Y,X]$ and $\mathtt{t}$ to $\mathtt{t}[Y,X]$.

\end{corollary}

\begin{proof}
 The assignment described in Equation~\eqref{OtherNaturality} of Lemma~\ref{BimoduleBinaturality} produces the restriction of $\mathtt{t}$ to $(\csym{C}^{\on{opp}}\kotimes \csym{C})\ostar (Y,X)$, which clearly is natural since $\mathtt{t}$ is natural in $\mathbf{N}^{\on{op}} \kotimes \mathbf{M}$.
\end{proof}

\begin{definition}
 A $\csym{C}$-module category $\mathbf{M}$ is {\it very cyclic} if there is an essentially surjective $\csym{C}$-module functor $\csym{C} \xrightarrow{\euler{\Phi}_{X}} \mathbf{M}$. Equivalently, $\mathbf{M}$ is very cyclic if there is an object $X \in \mathbf{M}$ such that for every $Y \in \mathbf{M}$, there is $\mathrm{F} \in \csym{C}$ such that $Y \simeq \mathbf{M}\mathrm{F}(X)$.
\end{definition}

The reason we don't refer to a very cyclic $\mathbf{M}$ simply as cyclic is that in the sequel we will also consider the additive closure of the full image of $\euler{\Phi}_{X}$, and say that $\mathbf{M}$ is {\it cyclic} if that subcategory coincides with all of $\mathbf{M}$. This is the notion considered in \cite{MMMTZ1}.

\begin{definition}\label{CStarX}
 Given a $\csym{C}$-module category $\mathbf{M}$ and an object $X \in \mathbf{M}$, we let $\mathbf{M}\!\star\! X$ denote the full $\csym{C}$-module subcategory whose objects are of the form $\mathbf{M}\mathrm{F}_{n}\mathbf{M}\mathrm{F}_{n-1}\cdots \mathbf{M}\mathrm{F}_{1}(X)$, for $n \geq 0$ and $\mathrm{F}_{1},\ldots, \mathrm{F}_{n} \in \csym{C}$.
\end{definition}
The $\csym{C}$-module category $\mathbf{M}\!\star\! X$ is very cyclic, with a very cyclic generator $X$.

Clearly, $\mathbf{M}$ is very cyclic if and only if and only if there is $X \in \mathbf{M}$ such that the inclusion of $\mathbf{M}\star X$ into $\mathbf{M}$ is an equivalence of $\csym{C}$-module categories. This is equivalent to the inclusion of $\csym{C}\ostar X$ into $\mathbf{M}$ being an equivalence of categories.

\begin{definition}
 We define the bicategory $\csym{C}\!\on{-Tamb}_{\euler{0}}$ as follows:
 \begin{itemize}
  \item objects of $\csym{C}\!\on{-Tamb}_{\euler{0}}$ are pairs $(\mathbf{M},X)$, where $\mathbf{M}$ is a $\csym{C}$-module category and $X \in \mathbf{M}$;
  \item $\csym{C}\!\on{-Tamb}_{\euler{0}}((\mathbf{M},X),(\mathbf{N},Y)) := \csym{C}\!\on{-Tamb}(\mathbf{M},\mathbf{N})$, and similarly the composition and coherence cells are inherited from $\csym{C}\!\on{-Tamb}$.
 \end{itemize}
\end{definition}

Clearly, there is a biequivalence $\csym{C}\!\on{-Tamb}_{\euler{0}} \xiso \csym{C}\!\on{-Tamb}$ which is identity on $1$-morphisms and $2$-morphisms, and sends an object $(\mathbf{M},X)$ of $\csym{C}\!\on{-Tamb}_{\euler{0}}$ to $\mathbf{M}$. Any choice assigning an object $X$ to every $\csym{C}$-module category $\mathbf{M}$ gives a quasi-inverse to the biequivalence just described.

\begin{definition}
 We define the bicategory $\csym{C}\!\on{-Tamb}_{c}$ as the $1,2$-full subcategory of $\csym{C}\!\on{-Tamb}$ whose objects are very cyclic $\csym{C}$-module categories.
\end{definition}

\begin{definition}\label{ChooseGenerators}
 A choice $\mathtt{W}: \on{Ob}\csym{C}\!\on{-Tamb}_{c} \rightarrow \on{Ob}\csym{C}\!\on{-Tamb}_{\euler{0}}$ associating a very cyclic generator to each $\mathbf{M} \in \csym{C}\!\on{-Tamb}_{c}$, defines a pseudofunctor $\mathbb{G}_{\mathtt{W}}: \csym{C}\!\on{-Tamb}_{c} \rightarrow \csym{C}\!\on{-Tamb}_{\euler{0}}$ by $\on{Ob}\mathbb{G}_{\mathtt{W}} = \mathtt{W}$ and $(\mathbb{G}_{\mathtt{W}})_{\mathbf{M,N}} = \mathbb{1}_{\ccf{C}\!\on{-Tamb}(\mathbf{M,N})}$.
\end{definition}

\begin{theorem}\label{LaxPseudoFunctorialityNa}
 Let
 $\mathtt{\Psi}: \mathbf{K} \xslashedrightarrow{} \mathbf{M}$ and $\mathtt{\Sigma}: \mathbf{M} \xslashedrightarrow{} \mathbf{N}$ be $\csym{C}$-Tambara modules, and let $X \in \mathbf{K}, Y \in \mathbf{M}, Z \in \mathbf{N}$.

 There are canonical morphisms $\mathtt{c}_{\mathtt{\Sigma}[Z,Y],\mathtt{\Psi}[Y,X]}: \mathtt{\Sigma}[Z,Y] \otimes_{[Y,Y]} \mathtt{\Psi}[Y,X] \rightarrow (\mathtt{\Sigma}\diamond \mathtt{\Psi})[Z,X]$ of $[Z,Z]$-$[X,X]$-bimodules, which, together with the assignment $(\mathbf{M},X) \mapsto [X,X]$ and the functors $\widehat{\na}_{(\mathbf{N},Y),(\mathbf{M},X)}$ of Corollary~\ref{LocalFunctors} define a strictly unital lax functor $\widehat{\na}: \csym{C}\!\on{-Tamb}_{\euler{0}} \rightarrow \on{Bimod}(\tam)$.

 If $\mathbf{M}$ is very cyclic and $Y$ is a very cyclic generator for $\mathbf{M}$, the coherence morphism $\mathtt{c}_{\mathtt{\Sigma}[Z,Y],\mathtt{\Psi}[Y,X]}$ is an isomorphism.
\end{theorem}

\begin{proof}
 We have
 \[
 \resizebox{.99\hsize}{!}{$
 \begin{aligned}
  &\big(\mathtt{\Sigma}[Z,Y] \otimes_{[Y,Y]} \mathtt{\Psi}[Y,X]\big)(\mathrm{F,G}) = \on{coeq}\Big(
  \begin{tikzcd}[ampersand replacement = \&]
	{\mathtt{\Sigma}[Z,Y] \diamond [Y,Y] \diamond \mathtt{\Psi}[Y,X]} \&\& {\mathtt{\Sigma}[Z,Y] \diamond \mathtt{\Psi}[Y,X]}
	\arrow["{\mathtt{\Sigma}[Z,Y] \diamond \mathtt{la}^{\mathtt{\Psi}[Y,X]}}", shift left=1, from=1-1, to=1-3]
	\arrow["{\mathtt{ra}^{\mathtt{\Sigma}[Z,Y]} \diamond \mathtt{\Psi}[Y,X]}"', shift right=1, from=1-1, to=1-3]
\end{tikzcd}\Big)(\mathrm{F,G}) \\
 &= \on{coeq}\Big(
 \begin{tikzcd}[ampersand replacement=\&]
	{\int^{\mathrm{H,H'}}\mathtt{\Sigma}(\mathbf{N}\mathrm{F}Z,\mathbf{M}\mathrm{H}Y) \kotimes \Hom{\mathbf{M}\mathrm{H}Y, \mathbf{M}\mathrm{H}'Y} \kotimes \Psi(\mathbf{M}\mathrm{H}'Y, \mathbf{K}\mathrm{G}X)} \&\& {\int^{\mathrm{K}} \mathtt{\Sigma}(\mathbf{N}\mathrm{F}Z,\mathbf{M}\mathrm{K}Y) \kotimes \mathtt{\Psi}(\mathbf{M}\mathrm{K}Y,\mathbf{K}\mathrm{G}X)}
	\arrow["{\setj{\mathtt{\Sigma}(\mathbf{N}\mathrm{F}Z,\mathbf{M}\mathrm{H}Y)\otimes \mathtt{la}_{\mathrm{H';H,G}}^{\mathtt{\Psi}[Y,X]}}}", shift left=2.5, from=1-1, to=1-3]
	\arrow["{\setj{\mathtt{ra}_{\mathrm{H;F,H'}}^{\mathtt{\Psi}[Z,Y]} \otimes \mathtt{\Psi}(\mathbf{M}\mathrm{H}'Y,\mathbf{K}\mathrm{G}X)}}"', shift right=2.5, from=1-1, to=1-3]
\end{tikzcd}\Big)
 \end{aligned}
 $}
 \]
 Recall that by definition of a coend as a colimit, $\int^{\mathrm{K}} \mathtt{\Sigma}(\mathbf{N}\mathrm{F}Z,\mathbf{M}\mathrm{K}Y) \kotimes \mathtt{\Psi}(\mathbf{M}\mathrm{K}Y,\mathbf{K}\mathrm{G}X)$ is a coequalizer of a pair of maps into $\coprod_{\mathrm{K}} \mathtt{\Sigma}(\mathbf{N}\mathrm{F}Z,\mathbf{M}\mathrm{K}Y) \kotimes \mathtt{\Psi}(\mathbf{M}\mathrm{K}Y,\mathbf{K}\mathrm{G}X)$, and similarly, other coends in our computation are by definition coequalizers of maps into the corresponding coproduct.
 We may thus lift the morphisms $\mathtt{la},\mathtt{ra},\mathtt{e}$ to the corresponding coproducts. We denote the lifts by overlining. Using the definition of coends, together with properties of $\mathtt{la,ra,e}$, we find the following commutative diagram:
\[
 \resizebox{.99\hsize}{!}{$
\begin{tikzcd}[ampersand replacement=\&]
	\& {\coprod_{\mathrm{H,H'}}\mathtt{\Sigma}(\mathbf{N}\mathrm{F}Z, \mathbf{M}\mathrm{H}Y) \kotimes \Hom{\mathrm{H},\mathrm{H}'} \kotimes \mathtt{\Psi}(\mathbf{M}\mathrm{H}'Y, \mathbf{K}\mathrm{G}X)} \\
	\\
	{\coprod_{\mathrm{H,H'}}\mathtt{\Sigma}(\mathbf{N}\mathrm{F}Z, \mathbf{M}\mathrm{H}Y) \kotimes \Hom{\mathbf{M}\mathrm{H}Y, \mathbf{M}\mathrm{H}'Y} \kotimes \mathtt{\Psi}(\mathbf{M}\mathrm{H}'Y, \mathbf{K}\mathrm{G}X)} \& {\coprod_{\mathrm{K}}\mathtt{\Sigma}(\mathbf{N}\mathrm{F}Z,\mathbf{M}\mathrm{K}Y) \kotimes \mathtt{\Psi}(\mathbf{M}\mathrm{K}Y,\mathbf{K}\mathrm{G}X)} \& {\int^{U \in \ccf{C}\ostar Y} \mathtt{\Sigma}(\mathbf{N}\mathrm{F}Z, U) \kotimes \mathtt{\Psi}(U, \mathbf{K}\mathrm{G}X)} \\
	{\int^{\mathrm{H,H'}}\mathtt{\Sigma}(\mathbf{N}\mathrm{F}Z, \mathbf{M}\mathrm{H}Y) \kotimes \Hom{\mathbf{M}\mathrm{H}Y, \mathbf{M}\mathrm{H}'Y} \kotimes \mathtt{\Psi}(\mathbf{M}\mathrm{H}'Y, \mathbf{K}\mathrm{G}X)} \\
	{\int^{\mathrm{K}}\mathtt{\Sigma}(\mathbf{N}\mathrm{F}Z,\mathbf{M}\mathrm{K}Y) \kotimes \mathtt{\Psi}(\mathbf{M}\mathrm{K}Y,\mathbf{K}\mathrm{G}X)} \\
	{(\mathtt{\Sigma}[Z,Y]\otimes_{[Y,Y]}\mathtt{\Psi}[Y,X])(\mathrm{F,G})} \&\& {\int^{U \in \mathbf{M}} \mathtt{\Sigma}(\mathbf{N}\mathrm{F}Z, U) \kotimes \mathtt{\Psi}(U, \mathbf{K}\mathrm{G}X)}
	\arrow[two heads, from=3-1, to=4-1]
	\arrow["{\overline{\mathtt{ra}^{\mathtt{\Sigma}}\otimes \mathtt{\Psi}[Y,X]}}"', shift right=2, from=3-1, to=3-2]
	\arrow["{\overline{\mathtt{\Sigma}[Z,Y] \otimes \mathtt{ra}^{\mathtt{\Psi}}}}", shift left=2, from=3-1, to=3-2]
	\arrow["{\mathtt{\Sigma}[Z,Y] \otimes \mathtt{ra}^{\mathtt{\Psi}}}", shift left=2, from=4-1, to=5-1]
	\arrow["{\mathtt{la}^{\mathtt{\Sigma}}\otimes \mathtt{\Psi}[Y,X]}"', shift right=2, from=4-1, to=5-1]
	\arrow[curve={height=-18pt}, two heads, from=3-2, to=5-1]
	\arrow["{\overline{\mathsf{r}^{\mathtt{\Sigma}} \otimes \mathtt{\Psi}[Y,X]}}", shift left=2, from=1-2, to=3-2]
	\arrow["{\overline{\mathtt{\Sigma}[Z,Y] \otimes \mathsf{l}^{\mathtt{\Psi}[Y,X]}}}"', shift right=2, from=1-2, to=3-2]
	\arrow["{\overline{\mathtt{\Sigma}\otimes \mathtt{e}^{\Hom{-,-}_{\mathbf{M}}} \otimes \mathtt{\Psi}}}"'{pos=0.6}, curve={height=18pt}, from=1-2, to=3-1]
	\arrow[from=3-2, to=3-3]
	\arrow[from=5-1, to=6-1]
	\arrow["\simeq"{description, pos=0.6}, shift left=1, from=6-1, to=3-3]
	\arrow["\simeq"{description, pos=0.6}, shift left=1, from=3-3, to=6-1]
	\arrow[shift right=1, hook, from=3-3, to=6-3, "{\mathtt{i}_{\mathtt{\Sigma}[Z,Y],\mathtt{\Psi}[Y,X]}}"]
	\arrow[dashed, from=6-1, to=6-3, "{\mathtt{c}_{\mathtt{\Sigma}[Z,Y],\mathtt{\Psi}[Y,X]})_{\mathrm{F,G}}}"]
\end{tikzcd}$}
\]
 The morphism ${\mathtt{i}_{\mathtt{\Sigma}[Z,Y],\mathtt{\Psi}[Y,X]}}$ is obtained from the inclusion of $\setj{\mathtt{\Sigma}(\mathbf{N}\mathrm{F}Z,\mathbf{M}\mathrm{K}Y) \kotimes \mathtt{\Psi}(\mathbf{M}\mathrm{K}Y,\mathbf{K}\mathrm{G}X) \; | \; \mathrm{K} \in \csym{C}}$ into the collection $\setj{\mathtt{\Sigma}(\mathbf{N}\mathrm{F}Z,U) \kotimes \mathtt{\Psi}(U,\mathbf{K}\mathrm{G}X) \; | \; U \in \mathbf{M}}$.

 The objects
 $(\mathtt{\Sigma}[Z,Y]\otimes_{[Y,Y]}\mathtt{\Psi}[Y,X])(\mathrm{F,G})$ and $\int^{\ccf{C}\ostar Y} \mathtt{\Sigma}(\mathbf{N}\mathrm{F}Z, \mathbf{M}\mathrm{H}Y) \kotimes \mathtt{\Psi}(\mathbf{M}\mathrm{H}'Y, \mathbf{K}\mathrm{G}X)$ are by definition the coequalizers of their column and row in the diagram, respectively.

 The indicated, mutually inverse isomorphisms follow by applying the universal properties of their respective domains (in particular, they represent the same functor). Since all the morphisms in the diagram are (lifts of) morphisms of $[Z,Z]$-$[X,X]$-bimodules, so is the obtained morphism ${\mathtt{c}_{\mathtt{\Sigma}[Z,Y],\mathtt{\Psi}[Y,X]}}$. Similarly, we find naturality in $\mathtt{\Sigma},\mathtt{\Psi}$.

 Since ${\mathtt{c}_{\mathtt{\Sigma}[Z,Y],\mathtt{\Psi}[Y,X]}}$ is induced from ${\mathtt{i}_{\mathtt{\Sigma}[Z,Y],\mathtt{\Psi}[Y,X]}}$ via the unique isomorphism making the above diagram commute, it suffices to show that the morphisms ${\mathtt{i}_{\mathtt{\Sigma}[Z,Y],\mathtt{\Psi}[Y,X]}}$ satisfy the coherence axioms for a lax functor. The associativity
 \[
  \resizebox{.99\hsize}{!}{$
\begin{tikzcd}[ampersand replacement=\&]
	{\int^{U \in \ccf{C}\ostar Y} \big(\int^{V \in \ccf{C}\ostar Z} \mathtt{\Upsilon}(\mathbf{Q}\mathrm{F}W,V) \kotimes \mathtt{\Sigma}(V,U)\big) \kotimes \mathtt{\Psi}(U,\mathbf{K}\mathrm{G}X)} \& {\int^{U \in \ccf{C}\ostar Y} \big(\int^{V \in \mathbf{N}} \mathtt{\Upsilon}(\mathbf{Q}\mathrm{F}W,V) \kotimes \mathtt{\Sigma}(V,U)\big) \kotimes \mathtt{\Psi}(U,\mathbf{K}\mathrm{G}X)} \\
	{\int^{V \in \ccf{C}\ostar Z}  \mathtt{\Upsilon}(\mathbf{Q}\mathrm{F}W,V) \kotimes \big(\int^{U \in \ccf{C}\ostar Y}\mathtt{\Sigma}(V,U) \kotimes \mathtt{\Psi}(U,\mathbf{K}\mathrm{G}X)\big)} \& {\int^{V \in \mathbf{N}}  \mathtt{\Upsilon}(\mathbf{Q}\mathrm{F}W,V) \kotimes \big(\int^{U \in \ccf{C}\ostar Y}\mathtt{\Sigma}(V,U) \kotimes \mathtt{\Psi}(U,\mathbf{K}\mathrm{G}X)\big)} \\
	{\int^{V \in \ccf{C}\ostar Z}  \mathtt{\Upsilon}(\mathbf{Q}\mathrm{F}W,V) \kotimes \big(\int^{U \in \mathbf{M}}\mathtt{\Sigma}(V,U) \kotimes \mathtt{\Psi}(U,\mathbf{K}\mathrm{G}X)\big)} \& {\int^{V \in \mathbf{N}}  \mathtt{\Upsilon}(\mathbf{Q}\mathrm{F}W,V) \kotimes \big(\int^{U \in \mathbf{M}}\mathtt{\Sigma}(V,U) \kotimes \mathtt{\Psi}(U,\mathbf{K}\mathrm{G}X)\big)}
	\arrow["\simeq", from=1-1, to=2-1]
	\arrow["{\mathtt{\Upsilon}\diamond \mathtt{i}_{\mathtt{\Sigma},\mathtt{\Psi}}}"', from=2-1, to=3-1]
	\arrow["\simeq"', from=1-2, to=2-2]
	\arrow["{\mathtt{i}_{\mathtt{\Upsilon},\mathtt{\Sigma}} \diamond \mathtt{\Psi}}", shift left=1, from=1-1, to=1-2]
	\arrow["{\mathtt{\Upsilon}\diamond \mathtt{i}_{\mathtt{\Sigma},\mathtt{\Psi}}}", from=2-2, to=3-2]
	\arrow["{\mathtt{i}_{\mathtt{\Upsilon},\mathtt{\Sigma}\diamond\mathtt{\Psi}}}"', shift right=1, from=3-1, to=3-2]
\end{tikzcd}$}
\]
 follows from the Fubini rule for coends (which gives the isomorphisms in the diagram) together with the commutativity of
 \[
 \resizebox{.99\hsize}{!}{$
\begin{tikzcd}[ampersand replacement=\&]
	{\setj{\mathtt{\Upsilon}(\mathbf{Q}\mathrm{F}W,\mathbf{N}\mathrm{L}Z)\kotimes \mathtt{\Sigma}(\mathbf{N}\mathrm{L}Z,\mathbf{M}\mathrm{K}Y) \kotimes \mathtt{\Psi}(\mathbf{M}\mathrm{K}Y,\mathbf{K}\mathrm{G}X) \; | \; \mathrm{K,L} \in \csym{C}} } \& {\setj{\mathtt{\Upsilon}(\mathbf{Q}\mathrm{F}W,\mathbf{N}\mathrm{L}Z)\kotimes \mathtt{\Sigma}(\mathbf{N}\mathrm{L}Z,U) \kotimes \mathtt{\Psi}(U,\mathbf{K}\mathrm{G}X) \; | \; \mathrm{L} \in \csym{C},U \in \mathbf{M}} } \\
	{\setj{\mathtt{\Upsilon}(\mathbf{Q}\mathrm{F}W,V)\kotimes \mathtt{\Sigma}(V,\mathbf{M}\mathrm{K}Y) \kotimes \mathtt{\Psi}(\mathbf{M}\mathrm{K}Y,\mathbf{K}\mathrm{G}X) \; | \; \mathrm{K} \in \csym{C}, V \in \mathbf{N}} } \& {\setj{\mathtt{\Upsilon}(\mathbf{Q}\mathrm{F}W,V)\kotimes \mathtt{\Sigma}(V,U) \kotimes \mathtt{\Psi}(U,\mathbf{K}\mathrm{G}X) \; | \; V\in\mathbf{N},U \in \mathbf{M}} }
	\arrow["\subseteq", hook', from=1-1, to=2-1]
	\arrow["\subseteq"', hook, from=1-1, to=1-2]
	\arrow["\subseteq"', hook', from=1-2, to=2-2]
	\arrow["\subseteq", hook, from=2-1, to=2-2]
\end{tikzcd}
$}.
\]
Since $\widehat{\na}_{(\mathbf{M},X),(\mathbf{M},X)}(\Hom{-,-}_{\mathbf{M}}) = \Hom{-X,-X} = [X,X] = \mathbb{1}_{\on{Bimod}(\ccf{C}\!\on{-Tamb}(\ccf{C},\ccf{C}))([X,X],[X,X])}$, the lax functor $\widehat{\na}$ is strictly unital. To show right unitality of $\mathtt{c}$, we verify that the diagram
\[\begin{tikzcd}[ampersand replacement=\&]
	{\mathtt{\Psi}[Y,X] \otimes_{[X,X]}[X,X]} \&\& {(\mathtt{\Psi}\diamond \Hom{-,-}_{\mathbf{K}})[Y,X]} \\
	{\mathtt{\Psi}[Y,X]}
	\arrow["{\mathtt{c}_{\mathtt{\Psi}[Y,X],[X,X]}}", from=1-1, to=1-3]
	\arrow["{\mathsf{r}_{\mathtt{\Psi}[Y,X]}}"', from=1-1, to=2-1]
	\arrow["{\mathsf{r}_{\mathtt{\Psi}}[Y,X]}", from=1-3, to=2-1]
\end{tikzcd}\]
commutes. This is a consequence of the commutativity of
\[\begin{tikzcd}[ampersand replacement=\&]
	{\int^{U \in \ccf{C}\ostar X} \mathtt{\Psi}(\mathbf{M}\mathrm{K}Y,U) \kotimes \Hom{U, \mathbf{K}\mathrm{G}X}} \&\& {\int^{U \in \mathbf{K}} \mathtt{\Psi}(\mathbf{M}\mathrm{K}Y,U) \kotimes \Hom{U, \mathbf{K}\mathrm{G}X}} \\
	{\mathtt{\Psi}(\mathbf{M}\mathrm{K}Y,\mathbf{K}\mathrm{G}X)}
	\arrow["{\mathtt{i}_{\mathtt{\Psi}[Y,X],[X,X]}}", shift left=1, from=1-1, to=1-3]
	\arrow["{\underline{\mathtt{ra}}^{\mathtt{\Psi}[Y,X]}}"', shift right=1, from=1-1, to=2-1]
	\arrow["{\mathtt{\Psi}(\mathbf{M}\mathrm{K}Y,-)}", shift left=1, from=1-3, to=2-1]
\end{tikzcd},\]
where $\underline{\mathtt{ra}}^{\mathtt{\Psi}[Y,X]}$ is the morphism induced by $\mathtt{ra}[Y,X]$, using the fact that $\mathtt{ra}[Y,X]$ coequalizes $\mathtt{\Psi}[Y,X] \diamond \mathtt{m}^{[X,X]}$ and $\mathtt{ra}^{\mathtt{\Psi}[Y,X]} \diamond [X,X]$. This latter diagram commutes, since for any $x \in \mathtt{\Psi}(\mathbf{M}\mathrm{K}Y,\mathbf{K}\mathrm{L}X)$ and $g \in \Hom{\mathbf{K}\mathrm{L}X,\mathbf{K}\mathrm{G}X}$, we by definition have
\[
\mathtt{ra}^{\mathtt{\Psi}}_{\mathrm{L};\mathrm{K},\mathrm{G}}(x \otimes g) = \mathtt{\Psi}(\mathbf{M}\mathrm{K},g)(x).
\]

Left unitality of $\mathtt{c}$ is similar.

To see that $\mathtt{c}_{\mathtt{\Sigma}[Z,Y],\mathtt{\Psi}[Y,X]}$ becomes an isomorphism whenever $\mathbf{M}$ is very cyclic and $Y$ is a very cyclic generator, it suffices to show that in that case $\mathtt{i}_{\mathtt{\Sigma}[Z,Y],\mathtt{\Psi}[Y,X]}$ is an isomorphism.
This follows from the fact that $\mathtt{i}_{\mathtt{\Sigma}[Z,Y],\mathtt{\Psi}[Y,X]}$ is the comparison map induced by the change of weight: the coend $\int^{U \in \mathbf{M}} \mathtt{\Sigma}(\mathbf{N}\mathrm{F}Z, U) \kotimes \mathtt{\Psi}(U, \mathbf{K}\mathrm{G}X)$ can be written as the weighted colimit $\Hom{-,-}_{\mathbf{M}} \otimes_{(\mathbf{M}^{\on{op}}\kotimes \mathbf{M})} \Big(\mathtt{\Sigma}(\mathbf{N}\mathrm{F}Z,-)\kotimes \mathtt{\Psi}(-,\mathbf{K}\mathrm{G}X)\Big)$.
Denoting the inclusion of $\csym{C}\ostar Y$ into $\mathbf{M}$ by $\euler{I}_{Y}$ and using the fact that $\Hom{-,-}_{\mathbf{M}}\circ (\euler{I}_{Y}^{\on{op}}\kotimes \euler{I}_{Y}) = \Hom{-,-}_{\ccf{C}\ostar Y}$, we may write $\mathtt{i}_{\mathtt{\Sigma}[Z,Y],\mathtt{\Psi}[Y,X]}$ as the comparison map
\[
\begin{aligned}
 &\int^{U \in \mathbf{M}} \mathtt{\Sigma}(\mathbf{N}\mathrm{F}Z, U) \kotimes \mathtt{\Psi}(U, \mathbf{K}\mathrm{G}X) = \Hom{-,-}_{\mathbf{M}} \otimes_{(\mathbf{M}^{\on{op}}\kotimes \mathbf{M})} \Big(\mathtt{\Sigma}(\mathbf{N}\mathrm{F}Z,-)\kotimes \mathtt{\Psi}(-,\mathbf{K}\mathrm{G}X)\Big) \\
 &\rightarrow (\Hom{-,-}_{\mathbf{M}}\circ \euler{I}_{Y}) \otimes_{((\ccf{C}\ostar Y)^{\on{op}}\kotimes (\ccf{C}\ostar Y))} \Big(\mathtt{\Sigma}(\mathbf{N}\mathrm{F}Z,-)\kotimes \mathtt{\Psi}(-,\mathbf{K}\mathrm{G}X)\Big) \\
 &= \int^{U \in \ccf{C}\ostar Y} \mathtt{\Sigma}(\mathbf{N}\mathrm{F}Z, U) \kotimes \mathtt{\Psi}(U, \mathbf{K}\mathrm{G}X)
\end{aligned}
\]
induced by precomposition with the inclusion functor. If $Y$ is a very cyclic generator, then $\euler{I}_{Y}$ is an equivalence, and so the comparison map is invertible in that case.
\end{proof}

\begin{corollary}
 For every choice $\mathtt{W}$ of cyclic generators as in Definition~\ref{ChooseGenerators}, and the resulting pseudofunctor $\mathbb{G}_{\mathtt{W}}$, we obtain a pseudofunctor $\widehat{\na} \circ \mathbb{G}_{\mathtt{W}}: \csym{C}\!\on{-Tamb}_{c} \rightarrow \on{Bimod}(\tam)$. We denote this pseudofunctor by $\na_{\mathtt{W}}$. In the presence of a fixed choice $\mathtt{W}$ and no risk of ambiguity, we denote $\na_{\mathtt{W}}$ simply by $\na$.
\end{corollary}

Let $\mathtt{W},\mathtt{W}'$ be two choices assigning a very cyclic generator to every very cyclic $\csym{C}$-module category, as described in Definition~\ref{ChooseGenerators}.
\begin{lemma}
  The pseudofunctors $\mathbb{G}_{\mathtt{W}}$ and $\mathbb{G}_{\mathtt{W}'}$ are pseudonaturally isomorphic.
\end{lemma}

\begin{proof}
 The assignment $(\mathbb{g}_{\mathtt{W,W}'})_{\mathbf{M}} = \Hom{-,-}_{\mathbb{M}} = \mathbb{1}_{\mathbf{M}}$ defines a pseudonatural isomorphism $\mathbb{g}_{\mathtt{W,W}'}: \mathbb{G}_{\mathtt{W}} \rightarrow \mathbb{G}_{\mathtt{W}'}$. The same assignment defines its inverse.
\end{proof}

\begin{corollary}
 The pseudofunctors $\na_{\mathtt{W}}$ and $\na_{\mathtt{W}'}$ are pseudonaturally isomorphic.
\end{corollary}

\begin{proof}
 We have $\na_{\mathtt{W}} = \widehat{\na} \circ \mathbb{G}_{\mathtt{W}} \xrightarrow[\simeq]{\widehat{\na} \circ \mathbb{g}_{\mathtt{W,W}'}} \widehat{\na} \circ \mathbb{G}_{\mathtt{W}} = \na_{\mathtt{W}'}$
\end{proof}

Explicitly, given a $\csym{C}$-module category $\mathbf{M}$, the component $(\widehat{\na} \circ \mathbb{g}_{\mathtt{W,W}'})_{\mathbf{M}}$ is given by the bimodule $[\mathtt{W}'(\mathbf{M}),\mathtt{W}(\mathbf{M})]$, and given $\mathtt{\Sigma} \in \csym{C}\!\on{-Tamb}(\mathbf{M,N})$, the coherence $2$-morphism $(\widehat{\na} \circ \mathbb{g}_{\mathtt{W,W}'})_{\mathtt{\Sigma}}$ is given by
\[
\begin{aligned}
 &\mathtt{\Sigma}[\mathtt{W}'(\mathbf{N}),\mathtt{W}'(\mathbf{M})] \otimes_{[\mathtt{W}'(\mathbf{M}),\mathtt{W}'(\mathbf{M})]} [\mathtt{W}'(\mathbf{M}),\mathtt{W}(\mathbf{M})] \xrightarrow[\simeq]{\mathsf{c}_{\mathtt{\Sigma}[\mathtt{W}'(\mathbf{N}),\mathtt{W}'(\mathbf{M})], [\mathtt{W}'(\mathbf{M}),\mathtt{W}(\mathbf{M})]}} (\mathtt{\Sigma}\diamond \Hom{-,-}_{\mathbf{M}})[\mathtt{W}'(\mathbf{N}),\mathtt{W}(\mathbf{M})] \\
 & \xrightarrow{(\mathsf{l}^{-1}\circ \mathsf{r})_{\mathtt{\Sigma}}[\mathtt{W}'(\mathbf{N}),\mathtt{W}(\mathbf{M})]} (\Hom{-,-}_{\mathbf{N}}\diamond \mathtt{\Sigma})[\mathtt{W}'(\mathbf{N}),\mathtt{W}(\mathbf{M})] \xrightarrow{\mathsf{c}^{-1}_{[\mathtt{W}'(\mathbf{N}),\mathtt{W}(\mathbf{N})],\mathtt{\Sigma}[\mathtt{W}(\mathbf{N}),\mathtt{W}(\mathbf{M})]}}  [\mathtt{W}'(\mathbf{N}),\mathtt{W}(\mathbf{N})] \otimes_{[\mathtt{W}(\mathbf{N}),\mathtt{W}(\mathbf{N})]} \mathtt{\Sigma}[\mathtt{W}(\mathbf{N}),\mathtt{W}(\mathbf{M})].
\end{aligned}
\]

\section{Compatibility of \texorpdfstring{$\na$}{na} with restrictions}\label{s6}

Let $\mathbb{F}: \csym{C} \rightarrow \csym{D}$ be a (strong) monoidal functor. Given $\mathbf{M} \in \csym{D}\!\on{-Mod}$, denote by $\mathbb{F}^{\ast}\mathbf{M}$ the $\csym{C}$-module category obtained via the composition $\csym{C} \kotimes \mathbf{M} \rightarrow \csym{D} \kotimes \mathbf{M} \rightarrow \mathbf{M}$.
Given $\mathtt{\Psi} \in \csym{D}\!\on{-Tamb}(\mathbf{M},\mathbf{N})$, using the extranaturality of $\ta^{\mathtt{\psi}}$ in $\mathrm{F} \in \csym{D}$, one may verify that the collection $\setj{\ta^{\mathtt{\Psi}}_{\mathbb{F}(\mathrm{H});Y,X} \; | \; \mathrm{H} \in \csym{C}, Y \in \mathbf{N}, X \in \mathbf{M}}$ defines a Tambara module
$
\mathbb{F}^{\ast}\mathtt{\Psi}$ from $\mathbb{F}^{\ast}\mathbf{M}$ to $\mathbb{F}^{\ast}\mathbf{N}
$
whose underlying profunctor is identical to that of $\mathtt{\Psi}$.

Given a morphism $\mathtt{t} \in \csym{D}\!\on{-Tamb}(\mathbf{M},\mathbf{N})(\mathtt{\Psi},\mathtt{\Psi}')$, the collection $\setj{\mathtt{t}_{Y,X} \; | \; Y \in \mathbf{N}, X \in \mathbf{M}}$ defines a morphism
\[
\mathbb{F}^{\ast}\mathtt{t} \in \csym{C}\!\on{-Tamb}(\mathbb{F}^{\ast}\mathbf{M},\mathbb{F}^{\ast}\mathbf{N})(\mathbb{F}^{\ast}\mathtt{\Psi},\mathbb{F}^{\ast}\mathtt{\Psi}').
\]

We thus obtain a faithful functor $\mathbb{F}^{\ast}_{\mathbf{M},\mathbf{N}}$ from $\csym{D}\!\on{-Tamb}(\mathbf{M},\mathbf{N})$ to $\csym{C}\!\on{-Tamb}(\mathbb{F}^{\ast}\mathbf{M},\mathbb{F}^{\ast}\mathbf{N})$. This assignments extends to pseudofunctors
\[
\mathbb{F}^{\ast}: \csym{D}\!\on{-Tamb} \rightarrow \csym{C}\!\on{-Tamb}\qquad \text{ and } \qquad \mathbb{F}^{\ast}_{\euler{0}}: \csym{D}\!\on{-Tamb}_{\euler{0}} \rightarrow \csym{C}\!\on{-Tamb}_{\euler{0}}.
\]
whose coherence cells are identities. Given $\mathtt{\Sigma} \in \csym{D}\!\on{-Tamb}(\mathbf{M,N})$ and $\mathtt{\Psi} \in \csym{D}\!\on{-Tamb}(\mathbf{K,M})$, both $\mathbb{F}^{\ast}(\mathtt{\Sigma} \diamond \mathtt{\Psi})$ and $\mathbb{F}^{\ast}\mathtt{\Sigma} \diamond \mathbb{F}^{\ast}\mathtt{\Psi}$ give the Tambara module whose underlying profunctor is the composition of the underlying profunctors $\mathtt{\Sigma} \diamond \mathtt{\Psi}$; the (coinciding) Tambara structures and the isomorphism between them are depicted in the following diagram:
\[\begin{tikzcd}[ampersand replacement=\&]
	{(\mathbb{F}^{\ast}\mathtt{\Sigma} \diamond \mathbb{F}^{\ast}\mathtt{\Psi})(Z,X)} \& {(\mathbb{F}^{\ast}\mathtt{\Sigma} \diamond \mathbb{F}^{\ast}\mathtt{\Psi})((\mathbb{F}^{\ast}\mathbf{N})(\mathrm{G})Z,(\mathbb{F}^{\ast}\mathbf{K})(\mathrm{G})X)} \\
	{\int^{Y}\mathtt{\Sigma}(Z,Y) \diamond \mathtt{\Psi}(Y,X)} \& {\int^{Y}\mathtt{\Sigma}(\mathbf{N}(\mathbb{F}(\mathrm{H}))Z,Y) \diamond \mathtt{\Psi}(Y,\mathbf{K}(\mathbb{F}(\mathrm{H}))X)} \\
	{\mathbb{F}^{\ast}(\mathtt{\Sigma} \diamond \mathtt{\Psi})(Z,X)} \& {\mathbb{F}^{\ast}(\mathtt{\Sigma} \diamond \mathtt{\Psi})((\mathbb{F}^{\ast}\mathbf{N})(\mathrm{G})Z,(\mathbb{F}^{\ast}\mathbf{K})(\mathrm{G})X)}
	\arrow[from=2-1, to=3-1, equal]
	\arrow[from=2-1, to=1-1, equal]
	\arrow[from=2-1, to=2-2]
	\arrow[from=2-2, to=1-2,equal]
	\arrow[from=2-2, to=3-2, equal]
	\arrow[dashed, from=1-1, to=1-2]
	\arrow[dashed, from=3-1, to=3-2]
\end{tikzcd}\]
Similarly, the unitality coherence cells are identity morphisms: the identity Tambara module of $\mathbf{M}$ is given by the $\on{Hom}$-profunctor $\Hom{-,-}$ of $\mathbf{M}$, together with its $\csym{D}$-action inherited from the $\csym{D}$-module category structure on $\mathbf{M}$. The identity Tambara module of $\mathbb{F}^{\ast}\mathbf{M}$ is given by the same profunctor, with restricted action. But this is precisely $\mathbb{F}^{\ast}\Hom{-,-}_{\mathbf{M}}$.

Observe that $\mathbb{F}$ induces a $\csym{C}$-module functor, $\mathbb{F}: {}_{\ccf{C}}\csym{C} \rightarrow \mathbb{F}^{\ast}\csym{D}$.
The following statement is simpler, but analogous to the lax functoriality part of Theorem~\ref{LaxPseudoFunctorialityNa}:

\begin{lemma}
 The functor $\euler{Res}_{\mathbb{F}}: \csym{C}\!\on{-Tamb}(\mathbb{F}^{\ast}\csym{D}, \mathbb{F}^{\ast}\csym{D}) \rightarrow \csym{C}\!\on{-Tamb}(\csym{C}, \csym{C})$ obtained by restricting and corestricting along $\mathbb{F}$, is lax monoidal.
\end{lemma}

\begin{proof}
 Given $\mathtt{\Sigma},\mathtt{\Psi} \in \csym{C}\!\on{-Tamb}(\mathbb{F}^{\ast}\csym{D}, \mathbb{F}^{\ast}\csym{D})$, we have, for any $\mathrm{C,C}' \in \csym{C}$:
 \[
  \euler{Res}_{\mathbb{F}}(\mathtt{\Sigma} \diamond \mathtt{\Psi})(\mathrm{C},\mathrm{C}') = \int^{\mathrm{D} \in \ccf{D}} \mathtt{\Sigma}(\mathbb{F}\mathrm{C},\mathrm{D}) \kotimes \mathtt{\Psi}(\mathrm{D},\mathbb{F}\mathrm{C}')
 \]
 and
 \[
  (\euler{Res}_{\mathbb{F}}(\mathtt{\Sigma})\diamond \euler{Res}_{\mathbb{F}}(\mathtt{\Psi}))(\mathrm{C},\mathrm{C}') = \int^{\mathrm{C}''\in \ccf{C}} \mathtt{\Sigma}(\mathbb{F}\mathrm{C},\mathbb{F}\mathrm{C}'') \kotimes \mathtt{\Psi}(\mathbb{F}\mathrm{C}'',\mathbb{F}\mathrm{C}').
 \]
 We define the multiplicativity morphisms for $\euler{Res}_{\mathbb{F}}$ as those induced by
 \[\resizebox{.99\hsize}{!}{$
 \begin{tikzcd}[ampersand replacement=\&]
	{\coprod_{\mathrm{D,D}'} \mathtt{\Sigma}(\mathbb{F}\mathrm{C},\mathrm{D}) \kotimes \Hom{\mathrm{D,D}'} \kotimes \mathtt{\Psi}(\mathrm{D}',\mathbb{F}\mathrm{C}')} \& {\coprod_{\mathrm{D''}} \mathtt{\Sigma}(\mathbb{F}\mathrm{C},\mathrm{D}'') \kotimes \mathtt{\Psi}(\mathrm{D}'',\mathbb{F}\mathrm{C}')} \& {\int^{\mathrm{D}} \mathtt{\Sigma}(\mathbb{F}\mathrm{C},\mathrm{D}) \kotimes \mathtt{\Psi}(\mathrm{D},\mathbb{F}\mathrm{C}')} \\
	{\coprod_{\mathrm{C}'',\mathrm{C}'''} \mathtt{\Sigma}(\mathbb{F}\mathrm{C},\mathbb{F}\mathrm{C}'') \kotimes \Hom{\mathrm{C}'',\mathrm{C}'''} \kotimes \mathtt{\Psi}(\mathbb{F}\mathrm{C}''',\mathbb{F}\mathrm{C}')} \& {\coprod_{\mathrm{C}_{'}} \mathtt{\Sigma}(\mathbb{F}\mathrm{C},\mathbb{F}\mathrm{C}_{'}) \kotimes \mathtt{\Psi}(\mathbb{F}\mathrm{C}_{'},\mathbb{F}\mathrm{C}')} \& {\int^{\mathrm{C''}} \mathtt{\Sigma}(\mathbb{F}\mathrm{C},\mathbb{F}\mathrm{C}'')\kotimes \mathtt{\Psi}(\mathbb{F}\mathrm{C}'',\mathbb{F}\mathrm{C}')}
	\arrow[shift left=2, from=1-1, to=1-2]
	\arrow[shift right=2, from=1-1, to=1-2]
	\arrow[from=1-2, to=1-3]
	\arrow["{\coprod_{\mathrm{C}'',\mathrm{C}'''} \mathtt{\Sigma}(\mathbb{F}\mathrm{C},\mathbb{F}\mathrm{C}'') \otimes \mathbb{F}_{\mathrm{C}'',\mathrm{C}'''} \otimes \mathtt{\Psi}(\mathbb{F}\mathrm{C}''',\mathbb{F}\mathrm{C}')}", from=2-1, to=1-1]
	\arrow[hook, from=2-2, to=1-2]
	\arrow[shift left=2, from=2-1, to=2-2]
	\arrow[shift right=2, from=2-1, to=2-2]
	\arrow[from=2-2, to=2-3]
	\arrow[dashed, from=2-3, to=1-3]
\end{tikzcd}$}
\]
and the unitality structure morphisms as
\[
 \csym{C}(-,-) \xrightarrow{\mathbb{F}_{-,-}} \csym{D}(\mathbb{F}-,\mathbb{F}-) = \mathbb{F}^{\ast}\csym{D}(-,-).
\]
The multiplicative axiom follows analogously to the proof of Theorem~\ref{LaxPseudoFunctorialityNa}. The right unitality axiom is satisfied, since the diagram
\[\begin{tikzcd}[ampersand replacement=\&]
	{\int^{\mathrm{C}''} \mathtt{\Psi}(\mathbb{F}\mathrm{C},\mathbb{F}\mathrm{C}'') \kotimes \csym{C}(\mathrm{C}'',\mathrm{C}')} \& {\int^{\mathrm{C}''} \mathtt{\Psi}(\mathbb{F}\mathrm{C},\mathbb{F}\mathrm{C}'') \kotimes \csym{D}(\mathbb{F}\mathrm{C}'',\mathbb{F}\mathrm{C}')} \\
	{\mathtt{\Psi}(\mathbb{F}\mathrm{C},\mathbb{F}\mathrm{C}')} \& {\int^{\mathrm{D}} \mathtt{\Psi}(\mathbb{F}\mathrm{C},\mathrm{D}) \kotimes \csym{D}(\mathrm{D},\mathbb{F}\mathrm{C}')}
	\arrow["{\euler{r}^{\mathbb{F}^{\ast}\mathtt{\Psi}}}", from=1-1, to=2-1]
	\arrow["{\mathtt{\Psi} \otimes \mathbb{F}}", shift left=1, from=1-1, to=1-2]
	\arrow[hook, from=1-2, to=2-2]
	\arrow["{\euler{r}^{\mathtt{\Psi}}}"', from=2-2, to=2-1]
\end{tikzcd}\]
commutes: both the maps in it correspond to the extranatural collection
\[
 \begin{aligned}
  \mathtt{\Psi}(\mathbb{F}\mathrm{C},\mathbb{F}\mathrm{C}'') \kotimes \csym{C}(\mathrm{C}'',\mathrm{C}') \rightarrow \mathtt{\Psi}(\mathbb{F}\mathrm{C},\mathbb{F}\mathrm{C}') \\
  x \otimes \mathrm{f} \mapsto \mathtt{\Psi}(\mathbb{F}\mathrm{C},\mathbb{F}\mathrm{f})(x).
 \end{aligned}
\]
Left unitality is similar.
\end{proof}

Given a lax monoidal functor $\mathbb{G}: \csym{A} \rightarrow \csym{B}$, we obtain an induced lax functor
\[
\on{Bimod}(\mathbb{G}): \on{Bimod}(\csym{A}) \rightarrow \on{Bimod}(\csym{B}),
\]
which, for example, given a monoid object $A$ in $\csym{A}$, endows the object $\mathbb{G}(A)$ with the structure of a monoid via maps \[
\mathbb{1}_{\ccf{B}} \xrightarrow{\mathbf{g}_{\mathbb{1}}} \mathbb{G}(\mathbb{1}_{\ccf{A}}) \xrightarrow{\mathbb{G}(\mathtt{e}_{A})} \mathbb{G}(A)
\]
and
\[
 \mathbb{G}(A) \ktotimes{\cccsym{B}} \mathbb{G}(A) \xrightarrow{\mathbf{g}_{A,A}} \mathbb{G}(A \ktotimes{\cccsym{A}} A) \xrightarrow{\mathbb{G}(\mathtt{m}_{A})} \mathbb{G}(A).
\]
The multiplicative coherence maps of $\on{Bimod}(\mathbb{G})$ are given by
\[\begin{tikzcd}[ampersand replacement=\&]
	{\mathbb{G}(M) \ktotimes{\cccsym{B}} \mathbb{G}(A) \ktotimes{\cccsym{B}} \mathbb{G}(N)} \& {\mathbb{G}(M) \ktotimes{\cccsym{B}} \mathbb{G}(N)} \& {\mathbb{G}(M) \otimes_{\mathbb{G}(A)}\mathbb{G}(N)} \\
	{\mathbb{G}(M \ktotimes{\cccsym{A}} A \ktotimes{\cccsym{A}} N)} \& {\mathbb{G}(M \ktotimes{\cccsym{A}} N)} \& {\mathbb{G}(M \otimes_{A} N)}
	\arrow[shift left=1, from=1-1, to=1-2]
	\arrow["{\mathbf{g}}"', from=1-1, to=2-1]
	\arrow["{\mathbf{g}_{M,N}}", from=1-2, to=2-2]
	\arrow[shift right=1, from=1-1, to=1-2]
	\arrow[shift left=1, from=2-1, to=2-2]
	\arrow[shift right=1, from=2-1, to=2-2]
	\arrow[from=1-2, to=1-3]
	\arrow[dashed, from=1-3, to=2-3]
	\arrow[from=2-2, to=2-3]
\end{tikzcd}\]
from which it clearly follows that if $\mathbb{G}$ is strong monoidal, then $\on{Bimod}(\mathbb{G})$ is a pseudofunctor. The unitality coherence maps are the identities, since for any monoid $A \in \csym{A}$, the identity bimodule ${}_{A}A_{A}$ is mapped under $\mathbb{G}$ to the identity bimodule ${}_{\mathbb{G}(A)}\mathbb{G}(A)_{\mathbb{G}(A)}$ of $\mathbb{G}(A)$.

\begin{proposition}\label{HighNaturality}
 Given a strong monoidal functor $\mathbb{F}: \csym{C} \rightarrow \csym{D}$, the diagram
\begin{equation}\label{HighlyNatural}
 \begin{tikzcd}[ampersand replacement=\&]
	{\csym{D}\!\on{-Tamb}_{\euler{0}}} \& {\on{Bimod}(\csym{D}\!\on{-Tamb}(\csym{D},\csym{D}))} \\
	\& {\on{Bimod}(\csym{C}\!\on{-Tamb}(\csym{D},\csym{D}))} \\
	{\csym{C}\!\on{-Tamb}_{\euler{0}}} \& {\on{Bimod}(\csym{C}\!\on{-Tamb}(\csym{C},\csym{C}))}
	\arrow["{\mathbb{F}^{\ast}_{\euler{0}}}", from=1-1, to=3-1]
	\arrow["{\widehat{\na}_{\cccsym{D}}}", shift left=1, from=1-1, to=1-2]
	\arrow["{\widehat{\na}_{\cccsym{C}}}", shift left=1, from=3-1, to=3-2]
	\arrow["{\on{Bimod}(\mathbb{F}^{\ast}_{\cccsym{D},\cccsym{D}})}", from=1-2, to=2-2]
	\arrow["{\on{Bimod}(\euler{Res}_{\mathbb{F}})}", shift right=1, from=2-2, to=3-2]
\end{tikzcd}
\end{equation}
commutes strictly.
\end{proposition}

\begin{proof}
 First, for any Tambara module $\mathtt{\Psi} \in \csym{D}\!\on{-Tamb}((\mathbf{K},X),(\mathbf{M},Y))$, the underlying Tambara module of $(\widehat{\na}_{\ccf{C}} \circ \mathbb{F}^{\ast}_{\euler{0}})(\mathtt{\Psi})$ sends $(\mathrm{C},\mathrm{C}')$ to $\mathtt{\Psi}(\mathbf{M}\mathbb{F}(\mathrm{C})Y, \mathbf{K}\mathbb{F}(\mathrm{C}')X)$, and its Tambara structure is given by
 \[
 \begin{aligned}
 &\mathtt{\Psi}(\mathbf{M}\mathbb{F}(\mathrm{C})Y, \mathbf{K}\mathbb{F}(\mathrm{C}')X) \xrightarrow{\ta^{\mathtt{\Psi}}_{\mathbb{F}(\mathrm{C}''); \mathbf{M}\mathbb{F}(\mathrm{C})Y,\mathbf{K}\mathbb{F}(\mathrm{C}')X}} \mathtt{\Psi}(\mathbf{M}\mathbb{F}(\mathrm{C}'')\mathbf{M}\mathbb{F}(\mathrm{C})Y, \mathbf{K}\mathbb{F}(\mathrm{C}'')\mathbf{K}\mathbb{F}(\mathrm{C}')X)
 \xrightarrow{\mathtt{\Psi}\big(\mathbf{m}_{\mathbb{F}(\mathrm{C}''),\mathbb{F}(\mathrm{C})}^{-1},\mathbf{k}_{\mathbb{F}(\mathrm{C}''),\mathbb{F}(\mathrm{C}')}\big)} \\
 &\mathtt{\Psi}\big(\mathbf{M}\big(\mathbb{F}(\mathrm{C}'')\mathbb{F}(\mathrm{C})\big)Y, \mathbf{K}\big(\mathbb{F}(\mathrm{C}'')\mathbf{K}\mathbb{F}(\mathrm{C}')\big)X\big)
 \xrightarrow{\mathtt{\Psi}\big(\mathbf{M}\mathbf{f}_{\mathrm{C}'',\mathrm{C}}^{-1},\mathbf{K}\mathbf{f}_{\mathrm{C}'',\mathrm{C}'}\big)} \mathtt{\Psi}\big(\mathbf{M}\big(\mathbb{F}(\mathrm{C}''\mathrm{C})\big)Y, \mathbf{K}\big(\mathbb{F}(\mathrm{C}''\mathrm{C}')\big)X\big)
 \end{aligned}
 \]
 which coincides with the Tambara module obtained by chasing $\mathtt{\Psi}$ along the other path in Diagram~\eqref{HighlyNatural}. Similarly, chasing a morphism of Tambara modules gives equal morphisms of underlying Tambara modules, and since morphisms of bimodules do not require additional structure beyond the morphisms of underlying Tambara modules, we see that the diagram commutes on the level of $2$-morphisms.

 By definition, the multiplication map for the monoid $(\widehat{\na}_{\ccf{C}}\circ \mathbb{F}_{\euler{0}}^{\ast})((\mathbf{K},X))$ is given on components by the maps
 \[
  \int^{\mathrm{C}''} \Hom{\mathbf{K}\mathbb{F}(\mathrm{C})X,\mathbf{K}\mathbb{F}\mathrm{C}''X} \kotimes \Hom{\mathbf{K}\mathbb{F}(\mathrm{C}'')X, \mathbf{K}\mathbb{F}(\mathrm{C}')X} \rightarrow \Hom{\mathbf{K}\mathbb{F}(\mathrm{C})X,\mathbf{K}\mathbb{F}(\mathrm{C}')X}
 \]
 induced by composition. This map coincides with the composite
 \[
 \resizebox{.99\hsize}{!}{$
  \int^{\mathrm{C}''} \Hom{\mathbf{K}\mathbb{F}(\mathrm{C})X,\mathbf{K}\mathbb{F}(\mathrm{C}'')X} \kotimes \Hom{\mathbf{K}\mathbb{F}(\mathrm{C}'')X, \mathbf{K}\mathbb{F}(\mathrm{C}')X} \rightarrow \int^{\mathrm{D}} \Hom{\mathbf{K}\mathbb{F}(\mathrm{C})X,\mathbf{K}\mathrm{D}X} \kotimes \Hom{\mathbf{K}\mathrm{D}X, \mathbf{K}\mathbb{F}(\mathrm{C}')X} \rightarrow \Hom{\mathbf{K}\mathbb{F}(\mathrm{C})X,\mathbf{K}\mathbb{F}(\mathrm{C}')X},
  $}
 \]
 which gives the multiplication map of the monoid $(\on{Bimod}(\euler{Res}_{\mathbb{F}}) \circ \on{Bimod}(\mathbb{F}^{\ast}_{\cccsym{D},\cccsym{D}}) \circ \widehat{\na}_{\cccsym{D}})((\mathbf{K},X))$. Similarly, both paths in the diagram define the unit map $\csym{C}(\mathrm{C},\mathrm{C}') \rightarrow \Hom{\mathbf{K}\mathbb{F}(\mathrm{C})X,\mathbf{K}\mathbb{F}(\mathrm{C})'X}$ as the one sending $\mathrm{f}$ to $(\mathbf{K}\mathbb{F}(\mathrm{f}))_{X}$. Since the left and right action maps for the $[Y,Y]$-$[X,X]$-bimodule structure of $\mathtt{\Psi}[Y,X]$ are defined using the composition maps in $\mathbf{M}$ and $\mathbf{K}$ respectively, the coincidence of bimodule structures follows analogously to the coincidence of monoid structures.

 Let $\mathtt{\Sigma} \in \csym{D}\!\on{-Tamb}((\mathbf{M},Y),(\mathbf{N},Z))$. By the construction given in the proof of Theorem~\ref{LaxPseudoFunctorialityNa}, the multiplicative coherence morphism $\euler{c}_{\mathtt{\Sigma}[Z,Y],\mathtt{\Psi}[Y,X]}: \mathtt{\Sigma}[Z,Y] \otimes_{[Y,Y]} \mathtt{\Psi}[Y,X] \rightarrow (\mathtt{\Sigma} \diamond \mathtt{\Psi})[Z,X]$ of $\widehat{\na}_{\ccf{D}}$ is determined uniquely by the inclusions of
 $\setj{\mathtt{\Sigma}(\mathbf{N}\mathrm{D}Z,\mathbf{M}\mathrm{D}'Y) \kotimes \mathtt{\Psi}(\mathbf{M}\mathrm{D}'Y,\mathbf{K}\mathrm{D}''X) \; | \; \mathrm{D}' \in \csym{D}}$ into $\setj{\mathtt{\Sigma}(\mathbf{N}\mathrm{D}Z,U) \kotimes \mathtt{\Psi}(U,\mathbf{K}\mathrm{D}''X) \; | \; U \in \mathbf{M}}$, for $\mathrm{D,D}'$ and $\mathrm{D}''$ in $\csym{D}$. The multiplicative coherence morphism $\mathbb{F}^{\ast}\mathtt{\Sigma}[Z,Y] \otimes_{\mathbb{F}^{\ast}[Y,Y]} \mathbb{F}^{\ast}\mathtt{\Psi}[Y,X] \rightarrow (\mathbb{F}^{\ast}(\mathtt{\Sigma}\circ \mathtt{\Psi}))[Z,X]$ for either of the lax functors in Diagram~\eqref{HighlyNatural} is obtained by the same construction, from the inclusions of
 \[
\setj{\mathtt{\Sigma}(\mathbf{N}\mathbb{F}(\mathrm{C})Z,\mathbf{M}\mathbb{F}(\mathrm{C}')Y) \kotimes \mathtt{\Psi}(\mathbf{M}\mathbb{F}(\mathrm{C}')Y,\mathbf{K}\mathbb{F}(\mathrm{C}'')X) \; | \; \mathrm{C}' \in \csym{C}} \text{ into }
 \setj{\mathtt{\Sigma}(\mathbf{N}\mathbb{F}(\mathrm{C})Z,U) \kotimes \mathtt{\Psi}(U,\mathbf{K}\mathbb{F}(\mathrm{C})''X) \; | \; U \in \mathbf{M}}.
 \]

 Finally, observe that all the lax functors in Diagram~\eqref{HighlyNatural} are strictly unital, and hence both the paths in the diagram define strictly unital lax functors.
\end{proof}

\section{\texorpdfstring{$\na$}{na} is essentially surjective}\label{s7}

In the case $\mathbf{M} = \csym{C}$, the equivalence given by Equation~\eqref{BicatYo} becomes a monoidal equivalence $\csym{C}^{\otimes\!\on{opp}} \xiso \csym{C}\!\on{-Mod}(\csym{C},\csym{C})$. Composing this with the contravariant monoidal functor
\[
\mathbb{H}(\csym{C},-)_{\ccf{C},\ccf{C}}: \csym{C}\!\on{-Mod}(\csym{C},\csym{C})^{\on{op}} \rightarrow \mathbf{Cat}_{\Bbbk}(\tam, \tam)
\]
of Proposition~\ref{CoresResStructure}, we get a monoidal functor
\[
\csym{C} \rightarrow \mathbf{Cat}_{\Bbbk}(\tam, \tam),
\]
 which endows $\tam$ with the structure of a $\csym{C}$-module category. The action of $\mathrm{K} \in \csym{C}$ is by restriction - it maps a Tambara module $\mathtt{\Psi}$ to the Tambara module $\mathtt{\Psi}(-,-\mathrm{K})$. It should be noted that this action is not strict: the Tambara modules $\mathtt{\Psi}(-,(-\mathrm{L})\mathrm{K})$ and $\mathtt{\Psi}(-,-(\mathrm{LK}))$ are canonically isomorphic but not equal. If $\mathtt{T}$ is a monoid object in $\tam$, then $\mathtt{T}(-,-\mathrm{K})$ is a $\mathtt{T}$-module, with structure map $\mathtt{la}: \mathtt{T} \circ \mathtt{T}(-,-\mathrm{K}) \rightarrow \mathtt{T}(-,-\mathrm{K})$ given by
 \[
 \mathtt{la}_{\mathrm{F},\mathrm{G}} = \mathtt{m}_{\mathrm{F},\mathrm{GK}}: \int^{\mathrm{H}}\mathtt{T}(\mathrm{F},\mathrm{H}) \kotimes \mathtt{T}(\mathrm{H},\mathrm{GK}) \rightarrow \mathtt{T}(\mathrm{F},\mathrm{GK}).
 \]
 Transporting from the evident $\mathtt{T}$-module structure on $\mathtt{T} \circ \csym{C}(-,-\mathrm{K})$ under the isomorphism coming from Yoneda lemma gives the same $\mathtt{T}$-module structure on $\mathtt{T}(-,-\mathrm{K})$.

\begin{definition}\label{CTPlus}
 The $\csym{C}$-module category $\mathtt{T}^{\ccf{C}}_{+}$ is defined as the full subcategory of $\mathtt{T}\!\on{--mod}$ whose objects are the objects of $\tam \ast \mathtt{T}$, under the above described $\csym{C}$-module structure of $\tam$.
 As such, a general object of $\mathtt{T}^{\ccf{C}}_{+}$ is of the form $\mathtt{T}\big(-,((\ldots((-\mathrm{F}_{n})\mathrm{F}_{n-1})\ldots)\mathrm{F}_{1})\big)$, for $n \geq 0$ and $\mathrm{F}_{1},\ldots, \mathrm{F}_{n} \in \csym{C}$. Clearly, $\mathtt{T}_{+}^{\ccf{C}}$ is very cyclic, with a very cyclic generator given by $\mathtt{T}$.
\end{definition}

Using Proposition~\ref{EndMonoid}, the object $\mathtt{T} \in \mathtt{T}^{\ccf{C}}_{+}$ gives a monoid $[\mathtt{T},\mathtt{T}]$ in $\tam$.

\begin{proposition}
 There is an isomorphism $\mathtt{T} \simeq [\mathtt{T},\mathtt{T}]$ of monoid objects in $\tam$.
\end{proposition}
\begin{proof}
 We have
 \[
 \begin{aligned}
  &[\mathtt{T,T}](\mathrm{K,L}) = \on{Hom}_{\mathtt{T}\!\on{--mod}}(\mathtt{T}(-,-\mathrm{K}),\mathtt{T}(-,-\mathrm{L})) \xiso \on{Hom}_{\mathtt{T}\!\on{--mod}}(\mathtt{T} \circ \csym{C}(-,-\mathrm{K}), \mathtt{T}(-,-\mathrm{L}))\\
  &\xiso \on{Hom}_{\ccf{C}\!\on{-Tamb}(\ccf{C},\ccf{C})}(\csym{C}(-,-\mathrm{K}), \mathtt{T}(-,-\mathrm{L}))
  \xiso \mathtt{T}(-,-\mathrm{L})(\mathrm{K},\mathbb{1}) = \mathtt{T}(\mathrm{K},\mathbb{1}\mathrm{L}).
 \end{aligned}
 \]
 The first isomorphism is obtained by passing under the isomorphisms $\mathtt{T}(-,-\mathrm{K}) \simeq \mathtt{T}\circ \csym{C}(-,-\mathrm{K})$; the second is a general fact about monoids and modules in monoidal categories; the third follows from Equation~\eqref{UniPropCK} in Remark~\ref{TambaraElements}.

 All the isomorphisms in the construction are natural in $\mathrm{K,L}$, so we obtain an isomorphism $[\mathtt{T},\mathtt{T}] \simeq \mathtt{T}(-,\mathbb{1}-)$ of profunctors. Explicitly, the composite isomorphism, which we will from now on denote by $\euler{J}$, sends a morphism $\varphi$ to $\overline{\varphi} := \varphi_{\mathrm{K},\mathbb{1}}(\mathtt{e}_{\mathrm{K},\mathbb{1}\mathrm{K}}^{\mathtt{T}}(\euler{l}_{\mathrm{K}}))$. We show that this is an isomorphism of Tambara modules, where $\mathtt{T}(-,\mathbb{1}-)$ is endowed with Tambara structure transported from $\mathtt{T}$. We thus need to show that
 \[
  (\mathtt{T}(\mathrm{HK},\euler{l}_{\mathrm{HL}})\circ \ta^{\mathtt{T}}_{\mathrm{H;K,L}} \circ \mathtt{T}(\mathrm{K},\euler{l}_{\mathrm{L}}^{-1}) \circ \euler{J}_{\mathrm{K,L}})(\varphi) = (\euler{J}_{\mathrm{HK},\mathrm{HL}}\circ \ta^{[\mathtt{T},\mathtt{T}]}_{\mathrm{H};\mathrm{K,L}})(\varphi).
 \]
  And indeed, we have
   \[
    \begin{aligned}
     &(\mathtt{T}(\mathrm{HK},\euler{l}_{\mathrm{HL}}) \ta^{\mathtt{T}}_{\mathrm{H;K,L}}  \mathtt{T}(\mathrm{K},\euler{l}_{\mathrm{L}}^{-1})  \euler{J}_{\mathrm{K,L}})(\varphi) \qquad & \text{ by definition of }\euler{J}_{\mathrm{K,L}} \\
     &=\mathtt{T}(\mathrm{HK},\euler{l}_{\mathrm{HL}})\ta^{\mathtt{T}}_{\mathrm{H};\mathrm{K,L}}\mathtt{T}(\mathrm{K},\euler{l}_{L}^{-1})\varphi_{\mathrm{K},\mathbb{1}} \mathtt{e}_{\mathrm{K},\mathbb{1}\mathrm{K}}(\euler{l}_{\mathrm{K}}) &\text{Tambara structure of }\mathtt{T} \\
     &=\mathtt{T}(\mathrm{HK},\euler{l}_{\mathrm{HL}})\mathtt{T}(\mathrm{HK},\mathrm{H}\euler{l}_{L}^{-1})\ta^{\mathtt{T}}_{\mathrm{H};\mathrm{K},\mathbb{1}\mathrm{L}}\varphi_{\mathrm{K},\mathbb{1}} \mathtt{e}_{\mathrm{K},\mathbb{1}\mathrm{K}}(\euler{l}_{\mathrm{K}}) &\text{definition of }\ta^{\mathtt{T}(-,-\mathrm{L})}\\
     &=\mathtt{T}(\mathrm{HK},\euler{l}_{\mathrm{HL}})\mathtt{T}(\mathrm{HK},\mathrm{H}\euler{l}_{L}^{-1})\mathtt{T}(\mathrm{HK},\euler{a}_{\mathrm{H},\mathbb{1},\mathrm{L}})\ta^{\mathtt{T}(-,-\mathrm{L})}_{\mathrm{H};\mathrm{K},\mathbb{1}}\varphi_{\mathrm{K},\mathbb{1}} \mathtt{e}_{\mathrm{K},\mathbb{1}\mathrm{K}}(\euler{l}_{\mathrm{K}}) &\varphi \text{ is a Tambara morphism} \\
     &=\mathtt{T}(\mathrm{HK},\euler{l}_{\mathrm{HL}})\mathtt{T}(\mathrm{HK},\mathrm{H}\euler{l}_{L}^{-1})\mathtt{T}(\mathrm{HK},\euler{a}_{\mathrm{H},\mathbb{1},\mathrm{L}})\varphi_{\mathrm{HK},\mathrm{H}\mathbb{1}}\ta^{\mathtt{T}(-,-\mathrm{K})}_{\mathrm{H};\mathrm{K},\mathbb{1}} \mathtt{e}_{\mathrm{K},\mathbb{1}\mathrm{K}}(\euler{l}_{\mathrm{K}}) &\text{definition of }\ta^{\mathtt{T}(-,-\mathrm{K})} \\
     &=\mathtt{T}(\mathrm{HK},\euler{l}_{\mathrm{HL}})\mathtt{T}(\mathrm{HK},\mathrm{H}\euler{l}_{L}^{-1})\mathtt{T}(\mathrm{HK},\euler{a}_{\mathrm{H},\mathbb{1},\mathrm{L}})\varphi_{\mathrm{HK},\mathrm{H}\mathbb{1}}\mathtt{T}(\mathrm{HK},\euler{a}_{\mathrm{H},\mathbb{1},\mathrm{K}}^{-1})\ta^{\mathtt{T}}_{\mathrm{H};\mathrm{K},\mathbb{1}\mathrm{K}} \mathtt{e}_{\mathrm{K},\mathbb{1}\mathrm{K}}(\euler{l}_{\mathrm{K}}) &\mathtt{e} \text{ is a Tambara morphism} \\
     &=\mathtt{T}(\mathrm{HK},\euler{l}_{\mathrm{HL}})\mathtt{T}(\mathrm{HK},\mathrm{H}\euler{l}_{L}^{-1})\mathtt{T}(\mathrm{HK},\euler{a}_{\mathrm{H},\mathbb{1},\mathrm{L}})\varphi_{\mathrm{HK},\mathrm{H}\mathbb{1}}\mathtt{T}(\mathrm{HK},\euler{a}_{\mathrm{H},\mathbb{1},\mathrm{K}}^{-1})\mathtt{e}_{\mathrm{HK},\mathrm{H}(\mathbb{1}\mathrm{K})}(\mathrm{H}\euler{l}_{\mathrm{K}}) &\mathrm{H}\euler{l}_{L}^{-1}\circ \euler{a}_{\mathrm{H},\mathbb{1},\mathrm{L}} = \euler{r}_{\mathrm{H}}^{-1}\mathrm{L} \\
     &=\mathtt{T}(\mathrm{HK},\euler{l}_{\mathrm{HL}})\mathtt{T}(\mathrm{HK},\euler{r}_{\mathrm{H}}^{-1}\mathrm{L})\varphi_{\mathrm{HK},\mathrm{H}\mathbb{1}}\mathtt{T}(\mathrm{HK},\euler{a}_{\mathrm{H},\mathbb{1},\mathrm{K}}^{-1})\mathtt{e}_{\mathrm{HK},\mathrm{H}(\mathbb{1}\mathrm{K})}(\mathrm{H}\euler{l}_{\mathrm{K}}) &\varphi \text{ is a Tambara morphism} \\
     &=\mathtt{T}(\mathrm{HK},\euler{l}_{\mathrm{HL}})\varphi_{\mathrm{HK},\mathrm{H}}\mathtt{T}(\mathrm{HK},\euler{r}_{\mathrm{H}}^{-1}\mathrm{K})\mathtt{T}(\mathrm{HK},\euler{a}_{\mathrm{H},\mathbb{1},\mathrm{K}}^{-1})\mathtt{e}_{\mathrm{HK},\mathrm{H}(\mathbb{1}\mathrm{K})}(\mathrm{H}\euler{l}_{\mathrm{K}}) &\euler{r}_{\mathrm{H}}^{-1}\mathrm{K}\circ \euler{a}_{\mathrm{H},\mathbb{1},\mathrm{K}}^{-1} = \mathrm{H}\euler{l}_{\mathrm{K}}^{-1} \\
     &=\mathtt{T}(\mathrm{HK},\euler{l}_{\mathrm{HL}})\varphi_{\mathrm{HK},\mathrm{H}}\mathtt{T}(\mathrm{HK},\mathrm{H}\euler{l}_{\mathrm{K}}^{-1})\mathtt{e}_{\mathrm{HK},\mathrm{H}(\mathbb{1}\mathrm{K})}(\mathrm{H}\euler{l}_{\mathrm{K}}) &\mathtt{e} \text{ is a Tambara morphism} \\
     &=\mathtt{T}(\mathrm{HK},\euler{l}_{\mathrm{HL}})\varphi_{\mathrm{HK},\mathrm{H}}\mathtt{e}_{\mathrm{HK},\mathrm{HK}}(\on{id}_{\mathrm{HK}}) &\euler{l}_{\mathrm{HL}} = \euler{a}_{\mathbb{1},\mathrm{H},\mathrm{L}} \circ \euler{l}_{\mathrm{H}}\mathrm{L} \\
     &=\mathtt{T}(\mathrm{HK},\euler{a}_{\mathbb{1},\mathrm{H},\mathrm{L}})\mathtt{T}(\mathrm{HK}, \euler{l}_{\mathrm{H}}\mathrm{L})\varphi_{\mathrm{HK},\mathrm{H}}\mathtt{e}_{\mathrm{HK},\mathrm{HK}}(\on{id}_{\mathrm{HK}}) &\varphi \text{ is a Tambara morphism} \\
     &=\mathtt{T}(\mathrm{HK},\euler{a}_{\mathbb{1},\mathrm{H},\mathrm{L}})\varphi_{\mathrm{HK},\mathbb{1}\mathrm{H}}\mathtt{T}(\mathrm{HK}, \euler{l}_{\mathrm{H}}\mathrm{K})\mathtt{e}_{\mathrm{HK},\mathrm{HK}}(\on{id}_{\mathrm{HK}}) &\euler{l}_{\mathrm{H}}\mathrm{K}=\euler{a}_{\mathbb{1},\mathrm{H},\mathrm{K}}^{-1}\circ \euler{l}_{\mathrm{HK}} \\
     &=\mathtt{T}(\mathrm{HK},\euler{a}_{\mathbb{1},\mathrm{H},\mathrm{L}})\varphi_{\mathrm{HK},\mathbb{1}\mathrm{H}}\mathtt{T}(\mathrm{HK}, \euler{a}_{\mathbb{1},\mathrm{H},\mathrm{K}}^{-1})\mathtt{T}(\mathrm{HK},\euler{l}_{\mathrm{HK}})\mathtt{e}_{\mathrm{HK},\mathrm{HK}}(\on{id}_{\mathrm{HK}}) &\mathtt{e} \text{ is a Tambara morphism} \\
     &=\mathtt{T}(\mathrm{HK},\euler{a}_{\mathbb{1},\mathrm{H},\mathrm{L}})\varphi_{\mathrm{HK},\mathbb{1}\mathrm{H}}\mathtt{T}(\mathrm{HK}, \euler{a}_{\mathbb{1},\mathrm{H},\mathrm{K}}^{-1})\mathtt{e}_{\mathrm{HK},\mathrm{HK}}(\euler{l}_{\mathrm{HK}}) &\text{ definition of} \ta^{[\mathtt{T},\mathtt{T}]} \\
     &=(\euler{J}_{\mathrm{HK},\mathrm{HL}} \ta^{[\mathtt{T},\mathtt{T}]}_{\mathrm{H};\mathrm{K,L}})(\varphi).
    \end{aligned}
   \]
 By definition of $\mathtt{T}(-,\mathbb{1}-)$, the maps $\mathtt{T}(\mathrm{F},\mathbb{1}\mathrm{G}) \xrightarrow{\mathtt{T}(\mathrm{F}, \euler{l}_{\mathrm{G}}^{-1})} \mathtt{T}(\mathrm{F},\mathrm{G})$, for all $\mathrm{F,G} \in \csym{C}$, give an isomorphism $\mathtt{T}(-,\mathbb{1}-) \xiso \mathtt{T}$. We now show that the composite isomorphism $[\mathtt{T},\mathtt{T}] \xiso \mathtt{T}(-,\mathbb{1}-) \xiso \mathtt{T}$ is an isomorphism of monoid objects.

 To establish multiplicativity, we need to show that the diagram
 \[\begin{tikzcd}[row sep = scriptsize]
	{\int^{\mathrm{L}}[\mathtt{T},\mathtt{T}](\mathrm{K},\mathrm{L}) \kotimes [\mathtt{T},\mathtt{T}](\mathrm{L},\mathrm{M})} && {[\mathtt{T},\mathtt{T}](\mathrm{K},\mathrm{M})} \\
	{\int^{\mathrm{L}}\mathtt{T}(\mathrm{K},\mathbb{1}\mathrm{L}) \kotimes \mathtt{T}(\mathrm{L},\mathbb{1}\mathrm{M})} && {\mathtt{T}(\mathrm{K},\mathbb{1}\mathrm{M})} \\
	{\int^{\mathrm{L}}\mathtt{T}(\mathrm{K,L}) \kotimes \mathtt{T}(\mathrm{L,M})} && {\mathtt{T}(\mathrm{K,M})}
	\arrow["{-\circ -}", from=1-1, to=1-3]
	\arrow["{\int^{\mathrm{L}}\euler{J}_{\mathrm{K,L}} \otimes \euler{J}_{\mathrm{L,M}}}"', from=1-1, to=2-1]
	\arrow["{\int^{\mathrm{L}}\mathtt{T}(\mathrm{K},\euler{l}_{\mathrm{L}}^{-1}) \otimes \mathtt{T}(\mathrm{K},\euler{l}_{\mathrm{M}}^{-1})}"', from=2-1, to=3-1]
	\arrow["{\euler{J}_{\mathrm{K,M}}}", from=1-3, to=2-3]
	\arrow["{\mathtt{T}(\mathrm{K},\euler{l}_{\mathrm{M}}^{-1})}", from=2-3, to=3-3]
	\arrow["{\mathtt{m}_{\mathrm{K,M}}}"', from=3-1, to=3-3]
\end{tikzcd}\]
  commutes. And indeed, we have
 \[
 \begin{aligned}
  &(\mathtt{m}_{\mathrm{K,M}}\circ \mathtt{T}(\mathrm{K},\euler{l}_{\mathrm{L}}^{-1}) \otimes \mathtt{T}(\mathrm{K},\euler{l}_{\mathrm{M}}^{-1}) \circ \euler{J}_{\mathrm{K,L}} \otimes \euler{J}_{\mathrm{L,M}})(\varphi \otimes \psi) \qquad &\text{ by definition} \\
  &=\mathtt{m}_{\mathrm{K,M}}\big(\mathtt{T}(\mathrm{K},\euler{l}_{\mathrm{L}}^{-1})\varphi_{\mathrm{K},\mathbb{1}}\mathtt{e}_{\mathrm{K},\mathbb{1}\mathrm{K}}(\euler{l}_{\mathrm{K}}) \otimes \mathtt{T}(\mathrm{L},\euler{l}_{\mathrm{M}}^{-1}) \psi_{\mathrm{L},\mathbb{1}}\mathtt{e}_{\mathrm{L},\mathbb{1}\mathrm{L}}(\euler{l}_{\mathrm{L}})\big) &\text{definition of } \int^{\mathrm{L}} \euler{J}_{\mathrm{K,L}} \otimes \euler{J}_{\mathrm{L,M}} \\
  &=\mathtt{m}_{\mathrm{K,M}}\big(\varphi_{\mathrm{K},\mathbb{1}}\mathtt{e}_{\mathrm{K},\mathbb{1}\mathrm{K}}(\euler{l}_{\mathrm{K}}) \otimes \mathtt{T}(\euler{l}_{\mathrm{L}}^{-1},M)\mathtt{T}(\mathrm{L},\euler{l}_{\mathrm{M}}^{-1}) \psi_{\mathrm{L},\mathbb{1}}\mathtt{e}_{\mathrm{L},\mathbb{1}\mathrm{L}}(\euler{l}_{\mathrm{L}})\big) &\psi \text{ is a Tambara morphism} \\
  &=\mathtt{m}_{\mathrm{K,M}}\big(\varphi_{\mathrm{K},\mathbb{1}}\mathtt{e}_{\mathrm{K},\mathbb{1}\mathrm{K}}(\euler{l}_{\mathrm{K}}) \otimes \mathtt{T}(\mathrm{L},\euler{l}_{\mathrm{M}}^{-1}) \psi_{\mathbb{1}\mathrm{L},\mathbb{1}}\mathtt{T}(\euler{l}_{\mathrm{L}}^{-1},\mathbb{1}\mathrm{L})\mathtt{e}_{\mathrm{L},\mathbb{1}\mathrm{L}}(\euler{l}_{\mathrm{L}})\big) &\mathtt{e} \text{ is a Tambara morphism} \\
  &=\mathtt{m}_{\mathrm{K,M}}\big(\varphi_{\mathrm{K},\mathbb{1}}\mathtt{e}_{\mathrm{K},\mathbb{1}\mathrm{K}}(\euler{l}_{\mathrm{K}}) \otimes \mathtt{T}(\mathrm{L},\euler{l}_{\mathrm{M}}^{-1}) \psi_{\mathbb{1}\mathrm{L},\mathbb{1}}\mathtt{e}_{\mathbb{1}\mathrm{L},\mathbb{1}\mathrm{L}}(\on{id}_{\mathbb{1}\mathrm{L}})\big) &\mathtt{m} \text{ is a Tambara morphism} \\
  &=\mathtt{T}(\mathrm{K},\euler{l}_{\mathrm{M}}^{-1})\mathtt{m}_{\mathrm{K,\mathbb{1}M}}\big(\varphi_{\mathrm{K},\mathbb{1}}\mathtt{e}_{\mathrm{K},\mathbb{1}\mathrm{K}}(\euler{l}_{\mathrm{K}}) \otimes \psi_{\mathbb{1}\mathrm{L},\mathbb{1}}\mathtt{e}_{\mathbb{1}\mathrm{L},\mathbb{1}\mathrm{L}}(\on{id}_{\mathbb{1}\mathrm{L}})\big) &\mathtt{m} \text{ is a morphism of }\mathtt{T}\text{-modules} \\
  &=\mathtt{T}(\mathrm{K},\euler{l}_{\mathrm{M}}^{-1})\psi_{\mathbb{1}\mathrm{L},\mathbb{1}}\mathtt{m}_{\mathrm{K,\mathbb{1}\mathrm{L}}}\big(\varphi_{\mathrm{K},\mathbb{1}}\mathtt{e}_{\mathrm{K},\mathbb{1}\mathrm{K}}(\euler{l}_{\mathrm{K}}) \otimes \mathtt{e}_{\mathbb{1}\mathrm{L},\mathbb{1}\mathrm{L}}(\on{id}_{\mathbb{1}\mathrm{L}})\big) & \text{unitality of }\mathtt{m} \\
  &=\mathtt{T}(\mathrm{K},\euler{l}_{\mathrm{M}}^{-1})\psi_{\mathbb{1}\mathrm{L},\mathbb{1}}\varphi_{\mathrm{K},\mathbb{1}}\mathtt{e}_{\mathrm{K},\mathbb{1}\mathrm{K}}(\euler{l}_{\mathrm{K}}) & \text{by definition } \\
  &=\mathtt{T}(\mathrm{K},\euler{l}_{\mathrm{M}}^{-1})\euler{J}_{\mathrm{K,M}}(\psi \circ \varphi).
 \end{aligned}
 \]
 Recall that, by definition, the unitality map $\csym{C}(-,-) \rightarrow [\mathtt{T},\mathtt{T}]$ sends $\mathrm{k} \in \csym{C}(\mathrm{K,L})$ to $\mathtt{T} \smalltriangleup \mathrm{b}: \mathtt{T}(-,-\mathrm{K}) \rightarrow \mathtt{T}(-,-\mathrm{L})$. Using this, we verify unitality of $\mathtt{T}(-,\euler{l}^{-1}-)\euler{J}$:
 \[
  \begin{aligned}
   &(\mathtt{T}(-,\euler{l}^{-1})\circ \euler{J} \circ \mathtt{e}^{[\mathtt{T},\mathtt{T}]})_{\mathrm{K,L}}(\mathrm{k}) \qquad &\text{ by definition} \\
   &=\mathtt{T}(\mathrm{K},\euler{l}_{\mathrm{L}}^{-1})\mathtt{T}(\mathrm{K},\mathbb{1}\mathrm{k})\mathtt{e}_{\mathrm{K},\mathbb{1}\mathrm{K}}(\euler{l}_{\mathrm{K}}) &\text{ naturality of } \euler{l}_{-} \\
   &=\mathtt{T}(\mathrm{K},\mathrm{k})\mathtt{T}(\mathrm{K},\euler{l}_{\mathrm{K}}^{-1})\mathtt{e}_{\mathrm{K},\mathbb{1}\mathrm{K}}(\euler{l}_{\mathrm{K}}) &\mathtt{e} \text{ is a Tambara morphism} \\
   &=\mathtt{T}(\mathrm{K},\mathrm{k})\mathtt{e}_{\mathrm{K},\mathrm{K}}(\on{id}_{\mathrm{K}}) &\text{ Yoneda lemma} \\
   &=\mathtt{e}_{\mathrm{K},\mathrm{L}}(\mathrm{k})^.
  \end{aligned}
 \]
\end{proof}

\begin{corollary}\label{MainEssSurj}
 The pseudofunctor $\na: \csym{C}\!\on{-Tamb}_{c} \rightarrow \on{Bimod}(\tam)$ is essentially surjective.
\end{corollary}

\section{\texorpdfstring{$\na$}{na} is locally an equivalence}\label{s8}

\subsection{Fmo monoidal and module categories}\label{s81}

Given a monoidal category $(\csym{C}, \cotimes)$, we define a monoidal category $(\csym{C}^{\#},\#)$ as follows:
\begin{itemize}
 \item $\on{Ob}\csym{C}^{\#}$ is the free magma on $\on{Ob}\csym{C}$; it thus consists of finite fully parenthesized expressions $\omega(\mathrm{F}_{1},\ldots, \mathrm{F}_{n})$, for $n \geq 1$ and $\mathrm{F}_{1},\ldots,\mathrm{F}_{n} \in \csym{C}$.
 \item Let $\euler{Ev}_{\ccf{C}}: \on{Ob}\csym{C}^{\#} \rightarrow \on{Ob}\csym{C}$ be the magma morphism extending the identity map on $\on{Ob}\csym{C}$.
 We let
 \[
\on{Hom}_{\ccf{C}^{\#}}(\omega(\mathrm{F}_{1},\ldots, \mathrm{F}_{n}), \omega'(\mathrm{F}_{1}',\ldots,\mathrm{F}_{m}')) :=  \on{Hom}_{\ccf{C}}(\euler{Ev}_{\ccf{C}}\,\omega(\mathrm{F}_{1},\ldots, \mathrm{F}_{n}),\euler{Ev}_{\ccf{C}} \, \omega'(\mathrm{F}_{1}',\ldots,\mathrm{F}_{m}')).
 \]
 \item The tensor product functor is given by the multiplication operation on $\on{Ob}\csym{C}^{\#}$ and on morphisms given by
\[\begin{tikzcd}[ampersand replacement=\&,cramped,row sep=tiny]
	{{\on{Hom}_{\ccf{C}^{\#}}(\omega_{0}(\mathrm{F}_{1},\ldots, \mathrm{F}_{n}), \omega_{0}'(\mathrm{F}_{1}',\ldots,\mathrm{F}_{m}')) \kotimes \on{Hom}_{\ccf{C}^{\#}}(\omega_{1}(\mathrm{G}_{1},\ldots, \mathrm{G}_{l}), \omega_{1}'(\mathrm{G}_{1}',\ldots,\mathrm{G}_{k}'))}} \\
	{{\on{Hom}_{\ccf{C}}\big(\euler{Ev}_{\ccf{C}}\big(\omega_{0}(\mathrm{F}_{1},\ldots, \mathrm{F}_{n})\big), \euler{Ev}_{\ccf{C}}\big( \omega_{0}'(\mathrm{F}_{1}',\ldots,\mathrm{F}_{m}')\big)\big) \kotimes \on{Hom}_{\ccf{C}}\big(\euler{Ev}_{\ccf{C}}\big(\omega_{1}(\mathrm{G}_{1},\ldots, \mathrm{G}_{l})\big), \euler{Ev}_{\ccf{C}} \big(\omega_{1}'(\mathrm{G}_{1}',\ldots,\mathrm{G}_{k}')\big)\big)}} \\
	{{\on{Hom}_{\ccf{C}}\big(\euler{Ev}_{\ccf{C}}\big(\omega_{0}(\mathrm{F}_{1},\ldots, \mathrm{F}_{n})\big) \cotimes \euler{Ev}_{\ccf{C}}\big(\omega_{1}(\mathrm{G}_{1},\ldots, \mathrm{G}_{l})\big), \euler{Ev}_{\ccf{C}}\big( \omega_{0}'(\mathrm{F}_{1}',\ldots,\mathrm{F}_{m}')\big)\cotimes \euler{Ev}_{\ccf{C}} \big(\omega_{1}'(\mathrm{G}_{1}',\ldots,\mathrm{G}_{k}')\big)\big)}} \\
	{{\on{Hom}_{\ccf{C}}\big(\euler{Ev}_{\ccf{C}}\big(\omega_{0}\# \omega_{1}(\mathrm{F}_{1},\ldots, \mathrm{F}_{n},\mathrm{G}_{1},\ldots,\mathrm{G}_{l})\big),\euler{Ev}_{\ccf{C}}\big(\omega_{0}'\# \omega_{1}'(\mathrm{F}_{1}',\ldots, \mathrm{F}_{m}',\mathrm{G}_{1}',\ldots,\mathrm{G}_{k}')\big)\big)}} \\
	{{\on{Hom}_{\ccf{C}^{\#}}\big(\omega_{0}\# \omega_{1}(\mathrm{F}_{1},\ldots, \mathrm{F}_{n},\mathrm{G}_{1},\ldots,\mathrm{G}_{l}),\omega_{0}'\# \omega_{1}'(\mathrm{F}_{1}',\ldots, \mathrm{F}_{m}',\mathrm{G}_{1}',\ldots,\mathrm{G}_{k}')\big)}}
	\arrow["{=}", from=1-1, to=2-1]
	\arrow["{- \cotimes -}", from=2-1, to=3-1]
	\arrow["{=}", from=3-1, to=4-1]
	\arrow["{=}", from=4-1, to=5-1]
\end{tikzcd}\]
where the first arrow is by definition of the $\on{Hom}$-sets in $\csym{C}^{\#}$, the second uses the monoidal structure on $\csym{C}$, the third uses the fact that $\euler{Ev}_{\ccf{C}}$ is a magma morphism, and the fourth again is by definition of $\on{Hom}$-sets in $\csym{C}^{\#}$.
\end{itemize}
Clearly, $(\csym{C}^{\#},\#)$ is well-defined, and the functor $\euler{EV}_{\ccf{C}}: \csym{C}^{\#} \rightarrow \csym{C}$ given by $\euler{Ev}_{\ccf{C}}$ on objects and identity on morphisms is a monoidal equivalence.

We say that a $\csym{C}$-module category $\mathbf{M}$ is {\it fmo} (free module on objects) if the action $\on{Ob}\csym{C} \times \on{Ob}\mathbf{M} \rightarrow \on{Ob}\mathbf{M}$ is free. We denote by $\csym{C}\!\on{-Mod}^{\#}$ the $1,2$-full subbicategory of $\csym{C}\!\on{-Mod}$ whose objects are fmo $\csym{C}$-module categories. Similarly we define the $1,2$-full subbicategory $\csym{C}\!\on{-Tamb}^{\#}$ of $\csym{C}\!\on{-Tamb}$.

Given a $\csym{C}$-module category $\mathbf{M}$, we define $\mathbf{M}^{\#} \in \csym{C}\!\on{-Mod}^{\#}$ analogously to $\csym{C}^{\#}$. We thus let $\on{Ob}\mathbf{M}^{\#}$ be the free $(\on{Ob}\csym{C})$-module on $\on{Ob}\mathbf{M}$, let $\euler{Ev}_{\mathbf{M}}$ be the $\on{Ob}\csym{C}$-module morphism extending the identity map on $\on{Ob}\mathbf{M}$, and set
\[
 \on{Hom}_{\mathbf{M}^{\#}}\big(\omega(\mathrm{F}_{1},\ldots,\mathrm{F}_{n},X), \omega'(\mathrm{F}_{1}',\ldots,\mathrm{F}_{n}',X')\big) := \on{Hom}_{\mathbf{M}}\big(\euler{Ev}_{\mathbf{M}}\omega(\mathrm{F}_{1},\ldots,\mathrm{F}_{n},X),\euler{Ev}_{\mathbf{M}}\omega'(\mathrm{F}_{1}',\ldots,\mathrm{F}_{n}',X')\big).
\]
Composition and $\csym{C}$-module structure are completely analogous to the definition of $\csym{C}^{\#}$. The functor $\mathbf{M}^{\#} \rightarrow \mathbf{M}$ given by $\euler{Ev}_{\mathbf{M}}$ on objects and identity on morphisms is an equivalence of module categories. Thus, the respective inclusions of $\csym{C}\!\on{-Mod}^{\#}$ in $\csym{C}\!\on{-Mod}$ and of $\csym{C}\!\on{-Tamb}^{\#}$ in $\csym{C}\!\on{-Tamb}$ are biequivalences.

As a consequence of the special case of $\mathsf{F} := \euler{EV}_{\ccf{C}}$ of Proposition~\ref{HighNaturality}, in order to prove that $\na_{\cccsym{C}}$ is a biequivalence, it suffices to show that $\na_{\cccsym{C}^{\#}}$ is a biequivalence. And since the inclusion of $\csym{C}\!\on{-Tamb}^{\#}$ in $\csym{C}\!\on{-Tamb}$ is a biequivalence, it suffices to show that the restriction of $\na_{\cccsym{C}^{\#}}$ is a biequivalence. In other words, for the remainder of this section, we may assume that both the monoidal category $\csym{C}$ as well as its module categories are fmo.

\subsection{\texorpdfstring{$\na$}{na} is locally essentially surjective}

Let $\mathbf{M},\mathbf{N}$ be very cyclic $\csym{C}$-module categories with very cyclic generators $X \in \mathbf{M}, Y \in \mathbf{N}$. Recall that, following the notation introduced in Section~\ref{DefiningNa}, we write $\csym{C}\ostar X$ for the subcategory of $\mathbf{M}$ with objects $\setj{\mathbf{M}\mathrm{F}X \; | \; \mathrm{F} \in \csym{C}}$ and we write $\mathbf{M} \star X$ for the
full $\csym{C}$-module subcategory of $\mathbf{M}$ whose objects are of the form $\mathbf{M}\mathrm{F}_{n}\mathbf{M}\mathrm{F}_{n-1}\cdots \mathbf{M}\mathrm{F}_{1}(X)$, for $n \geq 1$ and $\mathrm{F}_{1},\ldots, \mathrm{F}_{n} \in \csym{C}$. In particular, $\csym{C}\ostar X$ is a subcategory of $\mathbf{M}\star X$, and the two categories are equivalent. We recall an elementary result about constructing such an equivalence:

\begin{lemma}\label{AxiomOfChoice}
 Given a category $\mathcal{C}$ together with a subcategory $\mathcal{S} \xhookrightarrow{I} \mathcal{C}$ and, for every object $c \in \mathcal{C}$, a choice of an object $s(c) \in \mathcal{S}$ and an isomorphism $\sigma_{c}: c \xiso s(c)$, we define an equivalence $\Sigma: \mathcal{C} \xiso \mathcal{S}$ by sending $c$ to $s(c)$, and sending $f \in \on{Hom}_{\mathcal{C}}(c,c')$ to $\sigma_{c'} \circ f \circ \sigma_{c}^{-1} \in \on{Hom}_{\mathcal{S}}(s(c),s(c'))$. This yields an adjoint equivalence $(\Sigma, I, \sigma, \sigma^{-1}_{|\mathcal{S}})$.
\end{lemma}

Let $\mathbf{m}_{\mathrm{F}_{n},\ldots,\mathrm{F}_{1},X}^{r}$ be the, unique due to the coherence theorem for monoidal categories, isomorphism from the object $\mathbf{M}\mathrm{F}_{n}\cdots \mathbf{M}\mathrm{F}_{1}X$ to $\mathbf{M}\big(\mathrm{F}_{n}(\cdots(\mathrm{F}_{3}(\mathrm{F}_{2}\mathrm{F}_{1}))\!\cdots\!)\big)X$ obtained by composing coherence cells of $\mathbf{M}$ and associators for $\csym{C}$. Our choice of the ``right-first'' parenthesizing is not essential to the arguments, we adapt it as a convention to facilitate the presentation. From now on, we denote the object obtained by right-first parenthesizing of $\mathrm{F}_{n}\ldots\mathrm{F}_{1}$ by $\mathrm{r}(\mathrm{F}_{n},\ldots,\mathrm{F}_{1})$. The main benefit of this choice is the following:

\begin{equation}\label{RightFirst}
 \mathrm{H} \cotimes \mathrm{r}(\mathrm{F}_{m},\ldots, \mathrm{F}_{1}) = \mathrm{r}(\mathrm{H},\mathrm{F}_{m},\ldots, \mathrm{F}_{1}).
\end{equation}

\begin{definition}\label{UpsilonMX}
Following Lemma~\ref{AxiomOfChoice}, the collection $\setj{\mathbf{m}_{\mathrm{F}_{n},\ldots,\mathrm{F}_{1},X}^{r} \; | \; \mathrm{F}_{n},\ldots, \mathrm{F}_{1} \in \csym{C}, n \geq 1}$ gives rise to an equivalence from $\mathbf{M}\star X$ to $\csym{C}\ostar X$, which we denote by $\euler{\Upsilon}_{\!(\mathbf{M},X)}$.
\end{definition}

 Let $\mathtt{\Psi} \in \on{Bimod}(\csym{C}\!\on{-Tamb}(\csym{C},\csym{C}))([X,X],[Y,Y])$
  be a $[Y,Y]$-$[X,X]$-bimodule with action maps
 $\mathtt{la}$ and $\mathtt{ra}$, given by extranatural collections $\setj{\mathtt{la}_{\mathrm{H;F,G}} \; | \; \mathrm{F,G,H} \in \csym{C}}$ and $\setj{\mathtt{ra}_{\mathrm{H;F,G}} \; | \; \mathrm{F,G,H} \in \csym{C}}$
 Passing under the tensor-hom adjunction in $\mathbf{Vec}_{\Bbbk}$, we obtain collections
 \begin{equation}\label{TensorHomExtra}
 \resizebox{.99\hsize}{!}{$
 \setj{\overline{\mathtt{ra}}_{\mathrm{F};\mathrm{H},\mathrm{G}}: \Hom{\mathbf{M}\mathrm{H}X, \mathbf{M}\mathrm{G}X} \rightarrow \on{Hom}_{\Bbbk}(\mathtt{\Psi}(\mathrm{F},\mathrm{H}),\mathtt{\Psi}(\mathrm{F},\mathrm{G}))}
 \text{ and }
 \setj{\overline{\mathtt{la}}_{\mathrm{G};\mathrm{F},\mathrm{H}}: \Hom{\mathbf{N}\mathrm{F}Y, \mathbf{N}\mathrm{H}Y}  \rightarrow \on{Hom}_{\Bbbk}(\mathtt{\Psi}(\mathrm{H},\mathrm{G}),\mathtt{\Psi}(\mathrm{F},\mathrm{G}))}.$}
 \end{equation}
The notation in~\eqref{TensorHomExtra} implicitly indicates that the dualized collections give extranatural collections, where $\overline{\mathtt{ra}}$ is extranatural in $\mathrm{F}$ and $\overline{\mathtt{la}}$ is extranatural in $\mathrm{G}$.
This is a consequence of the following lemma:

 \begin{lemma}
  Let $\mathcal{C}$ be a category, let $\euler{P,P'}: \mathcal{C} \rightarrow \mathbf{Vec}_{\Bbbk}$ and $\euler{R}: \mathcal{C}^{\on{op}} \rightarrow \mathbf{Vec}_{\Bbbk}$ be functors and let $V \in \mathbf{Vec}_{\Bbbk}$.
  \begin{enumerate}
   \item
Given a collection $\setj{\tau_{c}: \euler{P}(c) \kotimes \euler{R}(c) \rightarrow V}$ satisfying
  \[\begin{tikzcd}[ampersand replacement=\&, row sep=small]
	\& {\euler{P}(c) \kotimes \euler{R}(c')} \\
	{\euler{P}(c') \kotimes \euler{R}(c')} \&\& {\euler{P}(c) \kotimes \euler{R}(c)} \\
	\& V
	\arrow["{\euler{P}(f)\otimes \euler{R}(c')}"'{pos=0.6}, from=1-2, to=2-1]
	\arrow["{\euler{P}(c)\otimes \euler{R}(f)}"{pos=0.6}, from=1-2, to=2-3]
	\arrow["{\tau_{c'}}"', from=2-1, to=3-2]
	\arrow["{\tau_{c}}", from=2-3, to=3-2]
\end{tikzcd}\]
for all $c,c', f\in \mathcal{C}(c,c')$. The collection $\setj{\overline{\tau}_{c}: \euler{P}(c) \rightarrow \on{Hom}_{\Bbbk}(\euler{R}(c), V)}$ is natural in $c$.
\item Given a collection $\setj{\gamma_{c}: V\kotimes \euler{P}(c) \rightarrow \euler{P}'(c)}$ natural in $c$, the collection $\setj{\overline{\gamma}_{c}: V \rightarrow \on{Hom}_{\Bbbk}(\euler{P}(c), \euler{P}'(c))}$ satisfies
\[\begin{tikzcd}[ampersand replacement=\&, row sep = scriptsize]
	\& V \\
	{\on{Hom}_{\Bbbk}(\euler{P}(c),\euler{P}'(c))} \&\& {\on{Hom}_{\Bbbk}(\euler{P}(c'),\euler{P}'(c'))} \\
	\& {\on{Hom}_{\Bbbk}(\euler{P}(c),\euler{P}'(c'))}
	\arrow["{\overline{\gamma}_{c}}"', from=1-2, to=2-1]
	\arrow["{\overline{\gamma}_{c'}}", from=1-2, to=2-3]
	\arrow["{\on{Hom}_{\Bbbk}(\euler{P}(c),\euler{P}'(f))}"'{pos=0.4}, from=2-1, to=3-2]
	\arrow["{\on{Hom}_{\Bbbk}(\euler{P}(f),\euler{P}'(c'))}"{pos=0.4}, from=2-3, to=3-2]
\end{tikzcd}\]
for all $c,c', f \in \mathcal{C}(c,c')$.
  \end{enumerate}
 \end{lemma}
\begin{proof}
\begin{enumerate}
 \item We have
 \[
   (\on{Hom}_{\Bbbk}(\euler{R}(f),V) \circ \overline{\tau}_{c})(x)(y) = \tau_{c}(x \otimes \euler{R}(f)(y)) = \tau_{c'}(\euler{P}(f)(x)\otimes y) = (\overline{\tau}_{c'}\circ \euler{P}(f))(x)(y).
 \]
 \item We have
 \[
  (\on{Hom}_{\Bbbk}(\euler{P}(c), \euler{P}'(f))\circ \overline{\gamma}_{c})(v)(x) = \euler{P}'(f)(\gamma_{c}(v \otimes x)) = \gamma_{c'}(v \otimes \euler{P}(f)(x)) = (\on{Hom}_{\Bbbk}(\euler{P}(f),\euler{P}'(c'))\circ \overline{\gamma}_{c'})(v)(x).
 \]
\end{enumerate}
\end{proof}

\begin{proposition}
 The assignments
 $
   (\mathbf{N}\mathrm{F}Y, \mathbf{M}\mathrm{G}X) \mapsto \mathtt{\Psi}(\mathrm{F},\mathrm{G})
 $
 and that given by the common value of the commutative diagram
\[\resizebox{.99\hsize}{!}{$
\begin{tikzcd}[ampersand replacement=\&]
	{\Hom{\mathbf{N}\mathrm{F}Y,\mathbf{N}\mathrm{F}'Y} \kotimes \Hom{\mathbf{M}\mathrm{G}X,\mathbf{M}\mathrm{G}'X}} \& {\on{Hom}_{\Bbbk}(\mathtt{\Psi}(\mathrm{F}',\mathrm{G}'),\mathtt{\Psi}(\mathrm{F,G'})) \kotimes \on{Hom}_{\Bbbk}(\mathtt{\Psi}(\mathrm{F}',\mathrm{G}),\mathtt{\Psi}(\mathrm{F',G'}))} \\
	{\on{Hom}_{\Bbbk}(\mathtt{\Psi}(\mathrm{F}',\mathrm{G}),\mathtt{\Psi}(\mathrm{F,G})) \kotimes \on{Hom}_{\Bbbk}(\mathtt{\Psi}(\mathrm{F},\mathrm{G}),\mathtt{\Psi}(\mathrm{F,G'}))} \\
	{\on{Hom}_{\Bbbk}(\mathtt{\Psi}(\mathrm{F},\mathrm{G}),\mathtt{\Psi}(\mathrm{F,G'})) \kotimes \on{Hom}_{\Bbbk}(\mathtt{\Psi}(\mathrm{F}',\mathrm{G}),\mathtt{\Psi}(\mathrm{F,G}))} \& {\on{Hom}_{\Bbbk}(\mathtt{\Psi}(\mathrm{F',G}),\mathtt{\Psi}(\mathrm{F,G'}))}
	\arrow["{\overline{\mathtt{la}}_{\mathrm{F,F'}}^{\mathrm{G'}} \otimes \overline{\mathtt{ra}}_{\mathrm{G,G'}}^{\mathrm{F'}}}", from=1-1, to=1-2]
	\arrow["{\overline{\mathtt{la}}_{\mathrm{F,F'}}^{\mathrm{G}} \otimes \overline{\mathtt{ra}}_{\mathrm{G,G'}}^{\mathrm{F}}}", from=1-1, to=2-1]
	\arrow["{\mathsf{c}^{\mathbf{Vec}_{\Bbbk}}_{\mathtt{\Psi}(\mathrm{F',G}),\mathtt{\Psi}(\mathrm{F',G'}),\mathtt{\Psi}(\mathrm{F,G'})}}", from=1-2, to=3-2]
	\arrow["{\mathsf{b}^{\mathbf{Vec}_{\Bbbk}}_{\on{Hom}_{\Bbbk}(\mathtt{\Psi}(\mathrm{F}',\mathrm{G}),\mathtt{\Psi}(\mathrm{F,G})),\on{Hom}_{\Bbbk}(\mathtt{\Psi}(\mathrm{F},\mathrm{G}),\mathtt{\Psi}(\mathrm{F,G'}))}}", from=2-1, to=3-1]
	\arrow["{\mathsf{c}^{\mathbf{Vec}_{\Bbbk}}_{{\mathtt{\Psi}(\mathrm{F',G}),\mathtt{\Psi}(\mathrm{F,G}),\mathtt{\Psi}(\mathrm{F,G'})}}}", from=3-1, to=3-2,swap]
\end{tikzcd}$}
\]
define a functor $\widetilde{\mathtt{\Psi}}: (\csym{C}\ostar Y)^{\on{op}} \kotimes (\csym{C}\ostar X) \rightarrow \mathbf{Vec}_{\Bbbk}$. Here, $\mathsf{b}^{\mathbf{Vec}_{\Bbbk}}$ denotes the standard symmetric braiding on $\mathbf{Vec}_{\Bbbk}$.
\end{proposition}

\begin{proof}
 If, for a fixed $\mathrm{G} \in \csym{C}$, the assignments $\mathbf{N}\mathrm{F}Y \mapsto \mathtt{\Psi}(\mathrm{F,G})$ and $\Hom{\mathbf{N}\mathrm{F}Y, \mathbf{N}\mathrm{F}'Y} \xrightarrow{\overline{\mathtt{la}}_{\mathrm{F,F'}}^{\mathrm{G}}} \on{Hom}_{\Bbbk}(\mathtt{\Psi}(\mathrm{F',G}), \mathtt{\Psi}(\mathrm{F,G}))$ define a functor $\widetilde{\mathtt{\Psi}}(-,\mathrm{G}): (\csym{C} \ostar Y)^{\on{op}} \rightarrow \mathbf{Vec}_{\Bbbk}$, and, similarly, for any fixed $\mathrm{F} \in \csym{C}$, the assignments $\mathbf{M}\mathrm{G}X \mapsto \mathtt{\Psi}(\mathrm{F,G})$ and $\Hom{\mathbf{M}\mathrm{G}X, \mathbf{M}\mathrm{G}'X} \xrightarrow{\overline{\mathtt{ra}}_{\mathrm{G,G'}}^{\mathrm{F}}} \on{Hom}_{\Bbbk}(\mathtt{\Psi}(\mathrm{F,G}), \mathtt{\Psi}(\mathrm{F,G'}))$ give a functor $\widetilde{\mathtt{\Psi}}(\mathrm{F},-): \csym{C}\ostar X \rightarrow \mathbf{Vec}_{\Bbbk}$, then $\mathtt{\Psi}$ is well-defined by \cite[Proposition~II.3.1]{MacL}. We show the first of the above two claims, proving that we obtain a functor $\widetilde{\mathtt{\Psi}}(-,\mathrm{G}):(\csym{C} \ostar Y)^{\on{op}} \rightarrow \mathbf{Vec}_{\Bbbk}$; the second claim is analogous.

 The unit axiom for the left module structure of $\mathtt{\Psi}$ implies that, for any $\mathrm{F,G} \in \csym{C}$, the diagram
 \[\begin{tikzcd}[ampersand replacement=\&, row sep = scriptsize]
	{\int^{\mathrm{H}} \csym{C}(\mathrm{F,H}) \kotimes \mathtt{\Psi}(\mathrm{H,G})} \& {\int^{\mathrm{H}} \Hom{\mathbf{N}\mathrm{F}Y,\mathbf{N}\mathrm{H}Y} \kotimes \mathtt{\Psi}(\mathrm{H,G})} \\
	\& {\mathtt{\Psi}(\mathrm{F,G})}
	\arrow["{\mathtt{e}^{[Y,Y]} \diamond \mathtt{\Psi}}", from=1-1, to=1-2]
	\arrow["{\mathsf{l}^{\mathtt{\Psi}}}"', from=1-1, to=2-2]
	\arrow["{\mathtt{la}_{\mathrm{F,G}}}", from=1-2, to=2-2]
\end{tikzcd}\]
commutes, which gives the equality
$
\setj{\mathtt{la}_{\mathrm{H;F,G}} \circ ((\mathbf{N}-)_{Y} \otimes \mathtt{\Psi}(\mathrm{H,G})) \; | \; \mathrm{H} \in \csym{C}} = \setj{\mathsf{l}^{\mathtt{\Psi}}_{\mathrm{H;F,G}} \; | \; \mathrm{H} \in \csym{C}}
$
of extranatural collections. Recall that for any $\mathrm{f} \in \csym{C}(\mathrm{F,H})$, we have $(\mathsf{l}^{\mathtt{\Psi}})_{\mathrm{H;F,G}}(\mathrm{f} \otimes -) = \mathtt{\Psi}(\mathrm{f},\mathrm{G})(-)$.
Evaluating at $\mathrm{H} = \mathrm{F}$ and $\on{id}_{\mathrm{F}} \in \csym{C}(\mathrm{F,F})$, we obtain
\[
\begin{aligned}
 &\widetilde{\mathtt{\Psi}}(-,\mathrm{G})(\on{id}_{\mathbf{N}\mathrm{F}Y}) = \widetilde{\mathtt{\Psi}}(-,\mathrm{G})((\mathbf{N}\on{id}_{\mathrm{F}})_{Y}) = \overline{\mathtt{la}}_{\mathrm{G;F,F}}((\mathbf{N}\on{id}_{\mathrm{F}})_{Y}) = \mathtt{la}_{\mathrm{F;F,G}}((\mathbf{N}\on{id}_{\mathrm{F}})_{Y} \otimes -)\\
 &= \mathsf{l}^{\mathtt{\Psi}}_{\mathrm{F;F,G}}(\on{id}_{\mathrm{F}} \otimes -)
 = \mathtt{\Psi}(\on{id}_{\mathrm{F}},\mathrm{G}) = \on{id}_{\mathtt{\Psi}(\mathrm{F,G})} = \on{id}_{\widetilde{\mathtt{\Psi}}(-,\mathrm{G})(\mathbf{N}\mathrm{F}Y)}.
\end{aligned}
\]
This shows that $\widetilde{\mathtt{\Psi}}(-,\mathrm{G})$ preserves identity morphisms.

Similarly, on the level of extranatural collections, the multiplicative axiom for $\mathtt{la}$ shows that, for any $\mathrm{F,G,H,H'} \in \csym{C}$, the diagram
\[\begin{tikzcd}[ampersand replacement=\&, column sep = huge]
	{\Hom{\mathbf{N}\mathrm{F}Y,\mathbf{N}\mathrm{H}Y} \kotimes \Hom{\mathbf{N}\mathrm{H}Y,\mathbf{N}\mathrm{H'}Y} \kotimes \mathtt{\Psi}(\mathrm{H',G})} \&\& {\Hom{\mathbf{N}\mathrm{F}Y,\mathbf{N}\mathrm{H'}Y} \kotimes \mathtt{\Psi}(\mathrm{H',G})} \\
	{\Hom{\mathbf{N}\mathrm{F}Y,\mathbf{N}\mathrm{H}Y}\kotimes \mathtt{\Psi}(\mathrm{H,G})} \&\& {\mathtt{\Psi}(\mathrm{F,G})}
	\arrow["{\Hom{\mathbf{N}\mathrm{F}Y,\mathbf{N}\mathrm{H}Y}\kotimes \mathtt{la}_{\mathrm{H';H,G}}}", from=1-1, to=2-1]
	\arrow["{\mathsf{c}^{\mathbf{N}}_{\mathrm{\mathbf{N}\mathrm{F}Y,\mathbf{N}\mathrm{H}Y,\mathbf{N}\mathrm{H'}Y'}} \otimes \mathtt{\Psi}(\mathrm{H',G})}", shift left=1, from=1-1, to=1-3]
	\arrow["{\mathtt{la}_{\mathrm{H;F,G}}}", from=2-1, to=2-3]
	\arrow["{\mathtt{la}_{\mathrm{H';F,G}}}", from=1-3, to=2-3]
\end{tikzcd}\]
commutes. Thus, for any $f \in \Hom{\mathbf{N}\mathrm{F}Y,\mathbf{N}\mathrm{H}Y}$ and $g \in \Hom{\mathbf{N}\mathrm{H}Y,\mathbf{N}\mathrm{H'}Y}$, we have
\[
\begin{aligned}
 &\widetilde{\mathtt{\Psi}}(-,\mathrm{G})(g \circ f) = \overline{\mathtt{la}}_{\mathrm{G;F,H'}}(g \circ f) = \mathtt{la}_{\mathrm{H';F,G}}(g \circ f \otimes -) = \mathtt{la}_{\mathrm{H;F,G}}(f \otimes \mathtt{la}_{\mathrm{H';H,G}}(g \otimes -)) = \\
 &\overline{\mathtt{la}}_{\mathrm{G;F,H}}(f)(\mathtt{la}_{\mathrm{H';H,G}}(g \otimes -)) = \overline{\mathtt{la}}_{\mathrm{G;F,H}}(f)(\overline{\mathtt{la}}_{\mathrm{G;H,H'}}(g)(-)) = \widetilde{\mathtt{\Psi}}(-,\mathrm{G})(g)\circ \widetilde{\mathtt{\Psi}}(-,\mathrm{G})(f),
\end{aligned}
\]
which concludes the proof.
\end{proof}

\begin{definition}\label{PsiHat}
 We denote the composite functor
 \[
  (\mathbf{N}^{\on{op}} \kotimes \mathbf{M})\star (Y,X) \xrightarrow[\simeq]{\euler{\Upsilon}_{\!(\mathbf{N}^{\on{op}}\kotimes \mathbf{M},(Y,X))}} (\mathbf{N}^{\on{op}} \kotimes \mathbf{M}) \ostar (Y,X) \xrightarrow{\widetilde{\mathtt{\Psi}}} \mathbf{Vec}_{\Bbbk}
 \]
 by $\widehat{\mathtt{\Psi}}$.
\end{definition}

\begin{proposition}
 The collection
 \[
 \begin{aligned}
 \setj{\ta^{\mathtt{\Psi}}_{\mathrm{H};\mathrm{r}(\mathrm{F}_{m},\ldots,\mathrm{F}_{1}),\mathrm{r}(\mathrm{G}_{n},\ldots, \mathrm{G}_{1})} \in \Hom{\mathtt{\Psi}\big(\mathrm{r}(\mathrm{F}_{m},\ldots,\mathrm{F}_{1}),\mathrm{r}(\mathrm{G}_{n},\ldots,\mathrm{G}_{1})\big),\mathtt{\Psi}\big(\mathrm{H}\cotimes \mathrm{r}(\mathrm{F}_{m},\ldots,\mathrm{F}_{1}),\mathrm{H}\cotimes \mathrm{r}(\mathrm{G}_{n},\ldots,\mathrm{G}_{1})\big)}}
 \end{aligned}
 \]
 endows $\widehat{\mathtt{\Psi}}$ with the structure of a Tambara module $(\mathbf{M}\star X) \xslashedrightarrow{\widehat{\mathtt{\Psi}}} (\mathbf{N} \star Y)$.
\end{proposition}

\begin{proof}
 First, using Equation~\eqref{RightFirst} together with Definition~\ref{UpsilonMX} and Definition~\ref{PsiHat}, we find
 \[
 \begin{aligned}
 &\mathtt{\Psi}\big(\mathrm{H}\cotimes \mathrm{r}(\mathrm{F}_{m},\ldots,\mathrm{F}_{1}),\mathrm{H}\cotimes \mathrm{r}(\mathrm{G}_{n},\ldots,\mathrm{G}_{1})\big) =
 \mathtt{\Psi}\big(\mathrm{r}(\mathrm{H},\mathrm{F}_{m},\ldots,\mathrm{F}_{1}),\mathrm{r}(\mathrm{H},\mathrm{G}_{n},\ldots,\mathrm{G}_{1})\big) \\
 &=\euler{\widehat{\Psi}}(\mathbf{N}\mathrm{H}\mathbf{N}\mathrm{F}_{m}\cdots \mathbf{N}\mathrm{F}_{1}Y,\mathbf{M}\mathrm{H}\mathbf{M}\mathrm{G}_{n}\cdots \mathbf{M}\mathrm{G}_{1}X)\\
 &= \widehat{\mathtt{\Psi}}\big((\mathbf{N}\star Y)(\mathrm{H})(\mathbf{N}\mathrm{F}_{m}\cdots \mathbf{N}\mathrm{F}_{1}Y), (\mathbf{M}\star X)(\mathrm{H})(\mathbf{M}\mathrm{G}_{n}\cdots \mathbf{M}\mathrm{G}_{1}X)\big)
 \end{aligned}
 \]
showing that the maps in our candidate Tambara structure have correct domains and codomains.

Next, let $f \in \Hom{\mathbf{M}\mathrm{H}_{k}\ldots\mathbf{M}\mathrm{H}_{1}X,\mathbf{M}\mathrm{G}_{m}\ldots\mathbf{M}\mathrm{G}_{1}X}_{\mathbf{M}\star X}$. Naturality with respect to the morphisms of $\mathbf{M}\star X$ then follows from the commutativity of the following diagram:
\[\begin{tikzcd}
	{\mathtt{\Psi}(\mathrm{r}(\mathrm{F}_{n},\ldots,\mathrm{F}_{1}),\mathrm{r}(\mathrm{H}_{k},\ldots,\mathrm{H}_{1}))} &&& {\mathtt{\Psi}(\mathrm{r}(\mathrm{F}_{n},\ldots,\mathrm{F}_{1}),\mathrm{r}(\mathrm{G}_{m},\ldots,\mathrm{G}_{1}))} \\
	{\mathtt{\Psi}(\mathrm{r}(\mathrm{K},\mathrm{F}_{n},\ldots,\mathrm{F}_{1}),\mathrm{r}(\mathrm{K},\mathrm{H}_{k},\ldots,\mathrm{H}_{1}))} &&& {\mathtt{\Psi}(\mathrm{r}(\mathrm{K},\mathrm{F}_{n},\ldots,\mathrm{F}_{1}),\mathrm{r}(\mathrm{K},\mathrm{G}_{m},\ldots,\mathrm{G}_{1}))}
	\arrow["{\overline{\mathtt{ra}}_{\mathrm{r}(\mathrm{H}_{k},\ldots,\mathrm{H}_{1}),\mathrm{r}(\mathrm{G}_{m},\ldots,\mathrm{G}_{1})}^{\mathrm{r}(\mathrm{F}_{n},\ldots,\mathrm{F}_{1})}(\mathbf{m}_{\mathrm{G}_{m},\ldots,\mathrm{G}_{1},X}^{r}\circ f\circ (\mathbf{m}^{r})^{-1}_{\mathrm{H}_{k},\ldots,\mathrm{H}_{1},X})}", shift left=2, from=1-1, to=1-4]
	\arrow["{\ta^{\mathtt{\Psi}}_{\mathrm{K};\mathrm{r}(\mathrm{F}_{n},\ldots,\mathrm{F}_{1}),\mathrm{r}(\mathrm{H}_{k},\ldots,\mathrm{H}_{1})}}", shift left=1, from=1-1, to=2-1]
	\arrow["{\ta^{\mathtt{\Psi}}_{\mathrm{K};\mathrm{r}(\mathrm{F}_{n},\ldots,\mathrm{F}_{1}),\mathrm{r}(\mathrm{H}_{k},\ldots,\mathrm{H}_{1})}}"', shift right=1, from=1-4, to=2-4]
	\arrow["{\overline{\mathtt{ra}}_{\mathrm{r}(\mathrm{K},\mathrm{H}_{k},\ldots,\mathrm{H}_{1}),\mathrm{r}(\mathrm{K},\mathrm{G}_{m},\ldots,\mathrm{G}_{1})}^{\mathrm{r}(\mathrm{K},\mathrm{F}_{n},\ldots,\mathrm{F}_{1})}(\mathbf{m}^{r}_{\mathrm{K},\mathrm{G}_{m},\ldots,\mathrm{G}_{1},X}\circ \mathbf{M}\mathrm{K}f\circ (\mathbf{m}^{r})^{-1}_{\mathrm{K},\mathrm{H}_{k},\ldots,\mathrm{H}_{1},X})}"', shift right=2, from=2-1, to=2-4]
\end{tikzcd}\]
Observe that $\mathbf{M}\mathrm{K}(\mathbf{m}^{r}_{\mathrm{G}_{m},\ldots,\mathrm{G}_{1}}) = \mathbf{m}^{r}_{\mathrm{K},\mathrm{G}_{m},\ldots,\mathrm{G}_{1}}$. Thus, the above diagram commutes if so does
\[\begin{tikzcd}
	{\mathtt{\Psi}(\mathrm{r}(\mathrm{F}_{n},\ldots,\mathrm{F}_{1}),\mathrm{r}(\mathrm{H}_{k},\ldots,\mathrm{H}_{1}))} &&& {\mathtt{\Psi}(\mathrm{r}(\mathrm{F}_{n},\ldots,\mathrm{F}_{1}),\mathrm{r}(\mathrm{G}_{m},\ldots,\mathrm{G}_{1}))} \\
	{\mathtt{\Psi}(\mathrm{r}(\mathrm{K},\mathrm{F}_{n},\ldots,\mathrm{F}_{1}),\mathrm{r}(\mathrm{K},\mathrm{H}_{k},\ldots,\mathrm{H}_{1}))} &&& {\mathtt{\Psi}(\mathrm{r}(\mathrm{K},\mathrm{F}_{n},\ldots,\mathrm{F}_{1}),\mathrm{r}(\mathrm{K},\mathrm{G}_{m},\ldots,\mathrm{G}_{1}))}
	\arrow["{\overline{\mathtt{ra}}_{\mathrm{r}(\mathrm{H}_{k},\ldots,\mathrm{H}_{1}),\mathrm{r}(\mathrm{G}_{m},\ldots,\mathrm{G}_{1})}^{\mathrm{r}(\mathrm{F}_{n},\ldots,\mathrm{F}_{1})}(\mathbf{m}^{r}_{\mathrm{G}_{m},\ldots,\mathrm{G}_{1},X}\circ f\circ (\mathbf{m}^{r})^{-1}_{\mathrm{H}_{k},\ldots,\mathrm{H}_{1},X})}", shift left=2, from=1-1, to=1-4]
	\arrow["{\ta^{\mathtt{\Psi}}_{\mathrm{K};\mathrm{r}(\mathrm{F}_{n},\ldots,\mathrm{F}_{1}),\mathrm{r}(\mathrm{H}_{k},\ldots,\mathrm{H}_{1})}}", shift left=1, from=1-1, to=2-1]
	\arrow["{\ta^{\mathtt{\Psi}}_{\mathrm{K};\mathrm{r}(\mathrm{F}_{n},\ldots,\mathrm{F}_{1}),\mathrm{r}(\mathrm{H}_{k},\ldots,\mathrm{H}_{1})}}"', shift right=1, from=1-4, to=2-4]
	\arrow["{\overline{\mathtt{ra}}_{\mathrm{r}(\mathrm{K},\mathrm{H}_{k},\ldots,\mathrm{H}_{1}),\mathrm{r}(\mathrm{K},\mathrm{G}_{m},\ldots,\mathrm{G}_{1})}^{\mathrm{r}(\mathrm{K},\mathrm{F}_{n},\ldots,\mathrm{F}_{1})}( \mathbf{M}\mathrm{K}(\mathbf{m}^{r}_{\mathrm{G}_{m},\ldots,\mathrm{G}_{1},X}\circ f \circ (\mathbf{m}^{r})^{-1}_{\mathrm{H}_{k},\ldots,\mathrm{H}_{1},X}))}"', shift right=2, from=2-1, to=2-4]
\end{tikzcd}\]
The commutativity of the latter diagram is a consequence of $\mathtt{ra}$ being a Tambara morphism. Naturality with respect to the morphisms of $(\mathbf{N}\star Y)^{\on{op}}$ follows analogously.

The diagram describing the extranaturality of $\ta^{\mathtt{\Psi}}_{\mathrm{H};\mathrm{r}(\mathrm{F}_{m},\ldots,\mathrm{F}_{1}),\mathrm{r}(\mathrm{G}_{n},\ldots, \mathrm{G}_{1})}$ in $\mathrm{H}$ is the special case of that describing extranaturality of $\ta^{\mathtt{\Psi}}_{\mathrm{H};\mathrm{F},\mathrm{G}}$, for general $\mathrm{F}$ and $\mathrm{G}$. Similarly, the Tambara axiom is satisfied for the collection
 \[
 \begin{aligned}
 \setj{\ta^{\mathtt{\Psi}}_{\mathrm{H};\mathrm{r}(\mathrm{F}_{m},\ldots,\mathrm{F}_{1}),\mathrm{r}(\mathrm{G}_{n},\ldots, \mathrm{G}_{1})} \in \Hom{\mathtt{\Psi}\big(\mathrm{r}(\mathrm{F}_{m},\ldots,\mathrm{F}_{1}),\mathrm{r}(\mathrm{G}_{n},\ldots,\mathrm{G}_{1})\big),\mathtt{\Psi}\big(\mathrm{H}\cotimes \mathrm{r}(\mathrm{F}_{m},\ldots,\mathrm{F}_{1}),\mathrm{H}\cotimes \mathrm{r}(\mathrm{G}_{n},\ldots,\mathrm{G}_{1})\big)}}_{\Bbbk}
 \end{aligned}
 \]
and module categories $\mathbf{M}\star X$ and $\mathbf{N}\star Y$ since it is satisfied for the collection $\setj{\ta^{\mathtt{\Psi}}_{\mathrm{H};\mathrm{F},\mathrm{G}}\; | \; \mathrm{F,G,H} \in \csym{C}}$ and module categories $\mathbf{M},\mathbf{N}$.
\end{proof}

\begin{proposition}\label{MainLocEss}
 We have $\widehat{\mathtt{\Psi}}[Y,X] = \mathtt{\Psi}$. In particular, $\na$ is locally essentially surjective.
\end{proposition}

\begin{proof}
 First, we have $\widehat{\mathtt{\Psi}}[Y,X](\mathrm{-,-}) = \widehat{\mathtt{\Psi}}(\mathbf{N}\mathrm{-}Y, \mathbf{M}\mathrm{-}X) = \mathtt{\Psi}(\mathrm{-,-})$, by definition of restriction and corestriction and by definition of $\widehat{\mathtt{\Psi}}$.
 Following Equation~\eqref{TambaraOnRestriction}, the component $\widehat{\mathtt{\Psi}}[Y,X]_{\mathrm{F,G}}^{\mathrm{H}}$ is given by

 \[
  \begin{aligned}
   &\widehat{\mathtt{\Psi}}(\mathbf{n}_{\mathrm{H,F}}^{-1}, \mathbf{m}_{\mathrm{H,G}}) \circ \ta^{\widehat{\mathtt{\Psi}}}_{\mathrm{H};\mathbf{N}\mathrm{F}Y, \mathbf{M}\mathrm{G}X} \qquad &\text{ by definition of } \ta^{\widehat{\mathtt{\Psi}}} \\
   =&\widehat{\mathtt{\Psi}}(\mathbf{n}_{\mathrm{H,F}}^{-1}, \mathbf{m}_{\mathrm{H,G}}) \circ \ta^{\mathtt{\Psi}}_{\mathrm{H};\mathrm{F},\mathrm{G}} &\text{ by definition of }\widehat{\mathtt{\Psi}} \\
   =&\mathtt{\Psi}(\mathbf{n}_{\mathrm{H,F}}^{-1} \circ \mathbf{n}_{\mathrm{H,F}}, \mathbf{m}_{\mathrm{H,G}}\circ \mathbf{m}_{\mathrm{H,G}}^{-1}) \circ \ta^{\widehat{\mathtt{\Psi}}}_{\mathrm{H};\mathrm{F}, \mathrm{G}} &\text{ by functoriality of } \mathtt{\Psi}
   &=\ta^{\widehat{\mathtt{\Psi}}}_{\mathrm{H};\mathrm{F}Y, \mathrm{G}X},
  \end{aligned}
 \]
  showing that also the Tambara structures of $\widehat{\mathtt{\Psi}}[Y,X]$ and $\mathtt{\Psi}$ coincide.

 Recall that the components $\mathtt{ra}^{\widehat{\mathtt{\Psi}}[Y,X]}: \int^{\mathrm{H}} \Hom{\mathbf{N}\mathrm{F}Y, \mathbf{N}\mathrm{H}Y} \kotimes \widehat{\mathtt{\Psi}}(\mathbf{N}\mathrm{H}Y,\mathbf{M}\mathrm{G}X) \rightarrow \widehat{\mathtt{\Psi}}(\mathbf{N}\mathrm{F}Y,\mathbf{M}\mathrm{G}X)$ of the left action on $\widehat{\mathtt{\Psi}}[Y,X]$ correspond to collections
 \[
 \setj{\overline{\widehat{\mathtt{\Psi}}(\mathbf{N}\mathrm{F}Y,\mathbf{M}(-)_{X})_{\mathrm{H,G}}}: \widehat{\mathtt{\Psi}}(\mathbf{N}\mathrm{F}Y,\mathbf{M}\mathrm{G}X) \kotimes \Hom{\mathbf{M}\mathrm{H}X, \mathbf{M}\mathrm{G}X} \rightarrow \widehat{\mathtt{\Psi}}(\mathbf{N}\mathrm{F}Y,\mathbf{M}\mathrm{G}X) \; | \; \mathrm{F,G,H} \in \csym{C}},
 \]
 obtained by applying the tensor-hom adjunction on the morphisms
 \[
  \widehat{\mathtt{\Psi}}(\mathbf{N}\mathrm{F}Y,\mathbf{M}(-)_{X})_{\mathrm{F,H}}: \Hom{\mathbf{M}\mathrm{H}X,\mathbf{M}\mathrm{G}X} \rightarrow \Hom{\widehat{\mathtt{\Psi}}(\mathbf{N}\mathrm{F}Y, \mathbf{M}\mathrm{H}X),\mathtt{\Psi}(\mathbf{N}\mathrm{F}Y, \mathbf{M}\mathrm{H}X)},
 \]
defining the profunctor structure on $\widehat{\mathtt{\Psi}}$. By definition of $\widehat{\mathtt{\Psi}}$, we have $\widehat{\mathtt{\Psi}}(\mathbf{N}\mathrm{F}Y,\mathbf{M}(-)_{X})_{\mathrm{F,H}} = \overline{\mathtt{ra}}_{\mathrm{F;H,G}}$. Since we have dualized twice, we find $\overline{\widehat{\mathtt{\Psi}}(\mathbf{N}\mathrm{F}Y,\mathbf{M}(-)_{X})_{\mathrm{H,G}}} = \mathtt{ra}_{\mathrm{H;F,G}}$, so the right module structures on $\mathtt{\Psi}$ and $\widehat{\mathtt{\Psi}}[Y,X]$ coincide. Analogously, the left module structures coincide, which concludes the proof.
\end{proof}

\subsection{\texorpdfstring{$\na$}{na} is locally full and faithful}

Let $\mathtt{\Psi}, \mathtt{\Psi}' \in \csym{C}\!\on{-Tamb}(\mathbf{M},\mathbf{N})$. Let $\mathtt{t}: \mathtt{\Psi}[Y,X] \Rightarrow \mathtt{\Psi}'[Y,X]$ be a bimodule morphism. Recall that by Lemma~\ref{BimoduleBinaturality}, we obtain a natural transformation $\widetilde{\mathtt{s}}: \mathtt{\Psi}_{|(\ccf{C}^{\on{opp}}\kotimes \ccf{C})\ostar (Y,X)} \rightarrow \mathtt{\Psi}'_{|(\ccf{C}^{\on{opp}}\kotimes\ccf{C})\ostar (Y,X)}$. Whiskering with the equivalence $\mathtt{\Upsilon}_{\mathbf{N}^{\on{op}}\kotimes \mathbf{M}}$, we obtain the natural transformation $\widehat{\mathtt{s}}: \widehat{\mathtt{\Psi}[Y,X]} \Rightarrow \widehat{\mathtt{\Psi}'[Y,X]}$. Conjugating with the unit of the adjoint equivalence $(\iota, \mathtt{\Upsilon}_{\mathbf{N}^{\on{op}}\kotimes \mathbf{M}}, (\mathbf{n}^{-1},\mathbf{m}), \varepsilon)$ obtained by the construction described in Lemma~\ref{AxiomOfChoice}, we obtain the natural transformation $\underline{\mathtt{s}}: \mathtt{\Psi}_{|(\mathbf{N}^{\on{op}} \kotimes \mathbf{M}) \star (Y,X)} \Rightarrow \mathtt{\Psi}'_{|(\mathbf{N}^{\on{op}} \kotimes \mathbf{M}) \star (Y,X)}$ given by
\[\resizebox{.99\hsize}{!}{$
\begin{tikzcd}[ampersand replacement=\&]
	{\mathtt{\Psi}(\mathbf{N}\mathrm{F}_{m}\cdots \mathbf{N}\mathrm{F}_{1}Y, \mathbf{M}\mathrm{G}_{n}\cdots\mathbf{M}\mathrm{G}_{1}X)} \&\& {\mathtt{\Psi}(\mathbf{N}\mathrm{r}(\mathrm{F}_{m},\ldots,\mathrm{F}_{1})Y, \mathbf{M}\mathrm{r}(\mathrm{G}_{n},\ldots,\mathrm{G}_{1})X)} \& {\mathtt{\Psi}[Y,X](\mathrm{r}(\mathrm{F}_{m},\ldots,\mathrm{F}_{1}), \mathrm{r}(\mathrm{G}_{n},\ldots,\mathrm{G}_{1}))} \\
	{\mathtt{\Psi}'(\mathbf{N}\mathrm{F}_{m}\cdots \mathbf{N}\mathrm{F}_{1}Y, \mathbf{M}\mathrm{G}_{n}\cdots\mathbf{M}\mathrm{G}_{1}X)} \&\& {\mathtt{\Psi}(\mathbf{N}\mathrm{r}(\mathrm{F}_{m},\ldots,\mathrm{F}_{1})Y, \mathbf{M}\mathrm{r}(\mathrm{G}_{n},\ldots,\mathrm{G}_{1})X)} \& {\mathtt{\Psi}'[Y,X](\mathrm{r}(\mathrm{F}_{m},\ldots,\mathrm{F}_{1}), \mathrm{r}(\mathrm{G}_{n},\ldots,\mathrm{G}_{1}))}
	\arrow["{\mathtt{\Psi}(\mathbf{n}^{r\, -1}_{\mathrm{F}_{m},\ldots,\mathrm{F}_{1},Y},\mathbf{m}^{r}_{\mathrm{G}_{n},\ldots,\mathrm{G}_{1},X})}", shift left=2, from=1-1, to=1-3]
	\arrow[equal, from=1-3, to=1-4]
	\arrow["{\mathtt{s}_{\mathrm{r}(\mathrm{F}_{m},\ldots,\mathrm{F}_{1}), \mathrm{r}(\mathrm{G}_{n},\ldots,\mathrm{G}_{1})}}", from=1-4, to=2-4]
	\arrow["{\mathtt{\Psi}'(\mathbf{n}^{r}_{\mathrm{F}_{m},\ldots,\mathrm{F}_{1},Y},\mathbf{m}^{r\, -1}_{\mathrm{G}_{n},\ldots,\mathrm{G}_{1},X})}"', from=2-3, to=2-1, shift right = 2]
	\arrow[equal, from=2-4, to=2-3]
\end{tikzcd}$}
\]

\begin{lemma}
 $\underline{\mathtt{s}}$ gives a Tambara morphism from $\mathtt{\Psi}$ to $\mathtt{\Psi}'$.
\end{lemma}

\begin{proof}

It suffices to show that the following diagram commutes:

\[\resizebox{.99\hsize}{!}{$
\begin{tikzcd}[ampersand replacement=\&]
	{\mathtt{\Psi}(\mathbf{N}\mathrm{F}_{m}\cdots \mathbf{N}\mathrm{F}_{1}Y, \mathbf{M}\mathrm{G}_{n}\cdots\mathbf{M}\mathrm{G}_{1}X)} \&\&\& {\mathtt{\Psi}(\mathbf{N}\mathrm{H}\mathbf{N}\mathrm{F}_{m}\cdots \mathbf{N}\mathrm{F}_{1}Y, \mathbf{M}\mathrm{H}\mathbf{M}\mathrm{G}_{n}\cdots\mathbf{M}\mathrm{G}_{1}X)} \\
	{\mathtt{\Psi}(\mathbf{N}\mathrm{r}(\mathrm{F}_{m},\ldots,\mathrm{F}_{1})Y, \mathbf{M}\mathrm{r}(\mathrm{G}_{n},\ldots,\mathrm{G}_{1})X)} \&\& {\mathtt{\Psi}(\mathbf{N}\mathrm{H}\mathbf{N}\mathrm{r}(\mathrm{F}_{m},\ldots,\mathrm{F}_{1})Y, \mathbf{M}\mathrm{H}\mathbf{M}\mathrm{r}(\mathrm{G}_{n},\ldots,\mathrm{G}_{1})X)} \& {\mathtt{\Psi}(\mathbf{N}\mathrm{r}(\mathrm{H},\mathrm{F}_{m},\ldots,\mathrm{F}_{1})Y, \mathbf{M}\mathrm{r}(\mathrm{H},\mathrm{G}_{n},\ldots,\mathrm{G}_{1})X)} \\
	{\mathtt{\Psi}[Y,X](\mathrm{r}(\mathrm{F}_{m},\ldots,\mathrm{F}_{1}), \mathrm{r}(\mathrm{G}_{n},\ldots,\mathrm{G}_{1}))} \&\&\& {\mathtt{\Psi}[Y,X](\mathrm{r}(\mathrm{H},\mathrm{F}_{m},\ldots,\mathrm{F}_{1}), \mathrm{r}(\mathrm{H},\mathrm{G}_{n},\ldots,\mathrm{G}_{1}))} \\
	{\mathtt{\Psi}'[Y,X](\mathrm{r}(\mathrm{F}_{m},\ldots,\mathrm{F}_{1}), \mathrm{r}(\mathrm{G}_{n},\ldots,\mathrm{G}_{1}))} \&\&\& {\mathtt{\Psi}'[Y,X](\mathrm{r}(\mathrm{H},\mathrm{F}_{m},\ldots,\mathrm{F}_{1}), \mathrm{r}(\mathrm{H},\mathrm{G}_{n},\ldots,\mathrm{G}_{1}))} \\
	{\mathtt{\Psi}'(\mathbf{N}\mathrm{r}(\mathrm{F}_{m},\ldots,\mathrm{F}_{1})Y, \mathbf{M}\mathrm{r}(\mathrm{G}_{n},\ldots,\mathrm{G}_{1})X)} \&\& {\mathtt{\Psi}'(\mathbf{N}\mathrm{H}\mathbf{N}\mathrm{r}(\mathrm{F}_{m},\ldots,\mathrm{F}_{1})Y, \mathbf{M}\mathrm{H}\mathbf{M}\mathrm{r}(\mathrm{G}_{n},\ldots,\mathrm{G}_{1})X)} \& {\mathtt{\Psi}'(\mathbf{N}\mathrm{r}(\mathrm{H},\mathrm{F}_{m},\ldots,\mathrm{F}_{1})Y, \mathbf{M}\mathrm{r}(\mathrm{H},\mathrm{G}_{n},\ldots,\mathrm{G}_{1})X)} \\
	{\mathtt{\Psi}'(\mathbf{N}\mathrm{F}_{m}\cdots \mathbf{N}\mathrm{F}_{1}Y, \mathbf{M}\mathrm{G}_{n}\cdots\mathbf{M}\mathrm{G}_{1}X)} \&\&\& {\mathtt{\Psi}'(\mathbf{N}\mathrm{H}\mathbf{N}\mathrm{F}_{m}\cdots \mathbf{N}\mathrm{F}_{1}Y, \mathbf{M}\mathrm{H}\mathbf{M}\mathrm{G}_{n}\cdots\mathbf{M}\mathrm{G}_{1}X)}
	\arrow["{\mathtt{\Psi}(\mathbf{n}^{r\, -1}_{\mathrm{F}_{m},\ldots,\mathrm{F}_{1},Y},\mathbf{m}^{r}_{\mathrm{G}_{n},\ldots,\mathrm{G}_{1},X})}", shift left=2, from=1-1, to=2-1]
	\arrow["{=}", from=2-1, to=3-1]
	\arrow["{\mathtt{s}_{\mathrm{r}(\mathrm{F}_{m},\ldots,\mathrm{F}_{1}), \mathrm{r}(\mathrm{G}_{n},\ldots,\mathrm{G}_{1})}}", from=3-1, to=4-1]
	\arrow["{\mathtt{\Psi}'(\mathbf{n}^{r}_{\mathrm{F}_{m},\ldots,\mathrm{F}_{1},Y},\mathbf{m}^{r\, -1}_{\mathrm{G}_{n},\ldots,\mathrm{G}_{1},X})}", from=5-1, to=6-1]
	\arrow["{=}", from=4-1, to=5-1]
	\arrow["{\ta^{\mathtt{\Psi}}_{\mathrm{H;\mathbf{N}\mathrm{F}_{m}\cdots \mathbf{N}\mathrm{F}_{1}Y, \mathbf{M}\mathrm{G}_{n}\cdots\mathbf{M}\mathrm{G}_{1}X}}}", from=1-1, to=1-4]
	\arrow["{\mathtt{\Psi}(\mathbf{n}^{r\, -1}_{\mathrm{H},\mathrm{F}_{m},\ldots,\mathrm{F}_{1},Y},\mathbf{m}^{r}_{\mathrm{H},\mathrm{G}_{n},\ldots,\mathrm{G}_{1},X})}", from=1-4, to=2-4]
	\arrow["{=}", from=2-4, to=3-4]
	\arrow["{\ta^{\mathtt{\Psi}[Y,X]}_{\mathrm{H};\mathrm{r}(\mathrm{F}_{m},\ldots,\mathrm{F}_{1}),\mathrm{r}(\mathrm{G}_{n},\ldots,\mathrm{G}_{1})}}", from=3-1, to=3-4]
	\arrow[""{name=0, anchor=center, inner sep=0}, "{\ta^{\mathtt{\Psi}}_{\mathrm{H};\mathrm{r}(\mathrm{F}_{m},\ldots,\mathrm{F}_{1}),\mathrm{r}(\mathrm{G}_{n},\ldots,\mathrm{G}_{1}),X}}"', shift right=1, from=2-1, to=2-3]
	\arrow["{\mathtt{\Psi}(\mathbf{n}_{\mathrm{H},\mathrm{r}(\mathrm{F}_{m},\ldots,\mathrm{F}_{1})}^{-1},\mathbf{m}_{\mathrm{H},\mathrm{r}(\mathrm{G}_{n},\ldots,\mathrm{G}_{1}),X})}"', shift right=1, from=2-3, to=2-4]
	\arrow["{\mathtt{s}_{\mathrm{r}(\mathrm{H},\mathrm{F}_{m},\ldots,\mathrm{F}_{1}), \mathrm{r}(\mathrm{H},\mathrm{G}_{n},\ldots,\mathrm{G}_{1})}}", from=3-4, to=4-4]
	\arrow["{\ta^{\mathtt{\Psi}'[Y,X]}_{\mathrm{H};\mathrm{r}(\mathrm{F}_{m},\ldots,\mathrm{F}_{1}),\mathrm{r}(\mathrm{G}_{n},\ldots,\mathrm{G}_{1})}}", from=4-1, to=4-4]
	\arrow["{\ta^{\mathtt{\Psi}'}_{\mathrm{H};\mathrm{r}(\mathrm{F}_{m},\ldots,\mathrm{F}_{1}),\mathrm{r}(\mathrm{G}_{n},\ldots,\mathrm{G}_{1}),X}}", shift left=1, from=5-1, to=5-3]
	\arrow["{\mathtt{\Psi}'(\mathbf{n}_{\mathrm{H},\mathrm{r}(\mathrm{F}_{m},\ldots,\mathrm{F}_{1})}^{-1},\mathbf{m}_{\mathrm{H},\mathrm{r}(\mathrm{G}_{n},\ldots,\mathrm{G}_{1}),X})}", shift left=1, from=5-3, to=5-4]
	\arrow["{=}", shift right=1, from=4-4, to=5-4]
	\arrow["{\mathtt{\Psi}'(\mathbf{n}^{r}_{\mathrm{H},\mathrm{F}_{m},\ldots,\mathrm{F}_{1},Y},\mathbf{m}^{r\, -1}_{\mathrm{H},\mathrm{G}_{n},\ldots,\mathrm{G}_{1},X})}", from=5-4, to=6-4]
	\arrow["{\ta^{\mathtt{\Psi}'}_{\mathrm{H;\mathbf{N}\mathrm{F}_{m}\cdots \mathbf{N}\mathrm{F}_{1}Y, \mathbf{M}\mathrm{G}_{n}\cdots\mathbf{M}\mathrm{G}_{1}X}}}", from=6-1, to=6-4]
	\arrow["1"{description}, draw=none, from=1-1, to=2-4]
	\arrow["3"{description}, draw=none, from=3-1, to=4-4]
	\arrow["4"{description}, draw=none, from=4-1, to=5-4]
	\arrow["5"{description}, draw=none, from=5-1, to=6-4]
	\arrow["2"{description}, shift right=1, draw=none, from=0, to=3-4]
\end{tikzcd}$}
\]
And indeed, face $1$ commutes since
\begin{equation}\label{FirstStep}
  \mathbf{m}_{\mathrm{H},\mathrm{r}(\mathrm{G}_{n},\ldots, \mathrm{G}_{1}),X}^{-1} \circ \mathbf{m}_{\mathrm{H},\mathrm{G}_{n},\ldots,\mathrm{G}_{1},X}^{r} = \mathbf{M}\mathrm{H}\mathbf{m}_{\mathrm{G}_{n},\ldots,\mathrm{G}_{1},X},
\end{equation}
 and similarly for $\mathbf{n}$, so we have
 \[
  \begin{aligned}
   &\begin{aligned}
   &{\ta^{\mathtt{\Psi}}_{\mathrm{H};\mathrm{r}(\mathrm{F}_{m},\ldots,\mathrm{F}_{1}),\mathrm{r}(\mathrm{G}_{n},\ldots,\mathrm{G}_{1}),X}}\circ {\mathtt{\Psi}(\mathbf{n}^{r\, -1}_{\mathrm{F}_{m},\ldots,\mathrm{F}_{1},Y},\mathbf{m}^{r}_{\mathrm{G}_{n},\ldots,\mathrm{G}_{1},X})}      &\qquad  &\text{Tambara structure of } \mathtt{\Psi}\\
   &=\mathtt{\Psi}(\mathbf{N}\mathrm{H}\mathbf{n}^{r\, -1}_{\mathrm{F}_{m},\ldots,\mathrm{F}_{1},Y},\mathbf{M}\mathrm{H}\mathbf{m}^{r}_{\mathrm{G}_{n},\ldots,\mathrm{G}_{1},X}) \circ {\ta^{\mathtt{\Psi}}_{\mathrm{H;\mathbf{N}\mathrm{F}_{m}\cdots \mathbf{N}\mathrm{F}_{1}Y, \mathbf{M}\mathrm{G}_{n}\cdots\mathbf{M}\mathrm{G}_{1}X}}} &       &\text{ by } \eqref{FirstStep}
   \end{aligned} \\
   &=\mathtt{\Psi}(\mathbf{n}_{\mathrm{H},\mathrm{F}_{m},\ldots,\mathrm{F}_{1},Y}^{r\, -1} \circ \mathbf{n}_{\mathrm{H},\mathrm{r}(\mathrm{F}_{m},\ldots, \mathrm{F}_{1}),Y},\mathbf{m}_{\mathrm{H},\mathrm{r}(\mathrm{G}_{n},\ldots, \mathrm{G}_{1}),X}^{-1} \circ \mathbf{m}_{\mathrm{H},\mathrm{G}_{n},\ldots,\mathrm{G}_{1},X}^{r}) \circ {\ta^{\mathtt{\Psi}}_{\mathrm{H;\mathbf{N}\mathrm{F}_{m}\cdots \mathbf{N}\mathrm{F}_{1}Y, \mathbf{M}\mathrm{G}_{n}\cdots\mathbf{M}\mathrm{G}_{1}X}}} \\
      &=\mathtt{\Psi}(\mathbf{n}_{\mathrm{H},\mathrm{r}(\mathrm{F}_{m},\ldots, \mathrm{F}_{1}),Y},\mathbf{m}_{\mathrm{H},\mathrm{r}(\mathrm{G}_{n},\ldots, \mathrm{G}_{1}),X}^{-1}) \circ \Psi(\mathbf{n}_{\mathrm{H},\mathrm{F}_{m},\ldots,\mathrm{F}_{1},Y}^{r\, -1},\mathbf{m}_{\mathrm{H},\mathrm{G}_{n},\ldots,\mathrm{G}_{1},X}^{r}) \circ {\ta^{\mathtt{\Psi}}_{\mathrm{H;\mathbf{N}\mathrm{F}_{m}\cdots \mathbf{N}\mathrm{F}_{1}Y, \mathbf{M}\mathrm{G}_{n}\cdots\mathbf{M}\mathrm{G}_{1}X}}} \\
   &=\mathtt{\Psi}(\mathbf{n}_{\mathrm{H},\mathrm{r}(\mathrm{F}_{m},\ldots, \mathrm{F}_{1}),Y}^{-1},\mathbf{m}_{\mathrm{H},\mathrm{r}(\mathrm{G}_{n},\ldots, \mathrm{G}_{1}),X})^{-1} \circ \Psi(\mathbf{n}_{\mathrm{H},\mathrm{F}_{m},\ldots,\mathrm{F}_{1},Y}^{r\, -1},\mathbf{m}_{\mathrm{H},\mathrm{G}_{n},\ldots,\mathrm{G}_{1},X}^{r}) \circ {\ta^{\mathtt{\Psi}}_{\mathrm{H;\mathbf{N}\mathrm{F}_{m}\cdots \mathbf{N}\mathrm{F}_{1}Y, \mathbf{M}\mathrm{G}_{n}\cdots\mathbf{M}\mathrm{G}_{1}X}}},
  \end{aligned}
 \]
 showing the commutativity of $1$. Face $2$ commutes by definition of $\ta^{\mathtt{\Psi}[Y,X]}$; face $3$ commutes since $\mathtt{s}$ is a Tambara morphism; face $4$ commutes by definition of $\ta^{\mathtt{\Psi}'}$ and the commutativity of face $5$ is analogous to that of face $1$.
\end{proof}

\begin{lemma}\label{LocalBijection}
\begin{enumerate}
 \item
 Given $\mathtt{\Psi}, \mathtt{\Psi}' \in \csym{C}\!\on{-Tamb}(\mathbf{M},\mathbf{N})$ and a Tambara morphism $\mathtt{t}: \mathtt{\Psi} \Rightarrow \mathtt{\Psi}'$, we have $\underline{\mathtt{t}[Y,X]} = \mathtt{t}$.\vspace{5pt}
 \item
 Given a bimodule morphism $\mathtt{s} \in \Hom{\mathtt{\Psi}[Y,X], \mathtt{\Psi}'[Y,X]}$, we have an equality of bimodule morphisms $\underline{\mathtt{s}}[Y,X] = \mathtt{s}$.
\end{enumerate}

\end{lemma}

\begin{proof}
 In both cases it suffices to show equality of the underlying natural transformations. Further, as a consequence of Lemma~\ref{AxiomOfChoice}, it suffices to show that $\underline{\mathtt{t}[Y,X]} = \mathtt{t}$ coincide on $(\csym{C}^{\on{opp}}\kotimes\csym{C})\ostar (Y,X)$. Thus, we only need to establish $\underline{\mathtt{t}[Y,X]}_{\mathbf{N}\mathrm{F}Y, \mathbf{M}\mathrm{G}X} = \mathtt{t}_{\mathbf{N}\mathrm{F}Y,\mathbf{M}\mathrm{G}X}$. And indeed,
 \[
 \begin{aligned}
&\underline{\mathtt{t}[Y,X]}_{\mathbf{N}\mathrm{F}Y,\mathbf{M}\mathrm{G}X} = \mathtt{\Psi}'(\mathbf{n}_{\mathrm{F}}^{r},\mathbf{m}_{\mathrm{G}}^{r \, -1}) \circ \mathtt{t}[Y,X]_{\mathrm{F,G}} \circ \mathtt{\Psi}(\mathbf{n}_{\mathrm{F}}^{r \, -1}, \mathbf{m}_{\mathrm{G}}^{r}) = \on{id}_{\mathtt{\Psi}'(\mathbf{N}\mathrm{F}Y,\mathbf{M}\mathrm{G}X)} \circ \mathtt{t}[Y,X]_{\mathrm{F,G}} \circ \on{id}_{\mathtt{\Psi}(\mathbf{N}\mathrm{F}Y,\mathbf{M}\mathrm{G}X)} \\
&=\mathtt{t}[Y,X] = \mathtt{t}_{\mathbf{N}\mathrm{F}Y,\mathbf{M}\mathrm{G}X}.
 \end{aligned}
 \]
 For the second claim, we have
 \[
  \underline{\mathtt{s}}[Y,X]_{\mathrm{F,G}} = \underline{\mathtt{s}}_{\mathbf{N}\mathrm{F}Y,\mathbf{M}\mathrm{G}X} = \mathtt{\Psi}'(\mathbf{n}_{\mathrm{F}}^{r},\mathbf{m}_{\mathrm{G}}^{r \, -1}) \circ \mathtt{s}_{\mathrm{F},\mathrm{G}} \circ \mathtt{\Psi}(\mathbf{n}_{\mathrm{F}}^{r \, -1}, \mathbf{m}_{\mathrm{G}}^{r}) = \on{id}_{\mathtt{\Psi}'(\mathbf{N}\mathrm{F}Y,\mathbf{M}\mathrm{G}X)} \circ \mathtt{s}_{\mathrm{F,G}} \circ \on{id}_{\mathtt{\Psi}(\mathbf{N}\mathrm{F}Y,\mathbf{M}\mathrm{G}X)} = \mathtt{s}_{\mathrm{F,G}}.
 \]
\end{proof}

\begin{proposition}\label{MainLocFF}
 $\na$ is locally full and faithful.
\end{proposition}

\begin{proof}
 Lemma~\ref{LocalBijection} shows that the map $\mathtt{s} \mapsto \underline{\mathtt{s}}$ is inverse to the map $(\na_{\mathbf{M},\mathbf{N}})_{\mathtt{\Psi},\mathtt{\Psi}'}$ sending a Tambara morphism $\mathtt{t}$ to the bimodule morphism $\mathtt{t}[Y,X]$.
\end{proof}

\begin{theorem}\label{MainThm}
 $\na$ is a biequivalence.
\end{theorem}

\begin{proof}
 By Corollary~\ref{MainEssSurj}, $\na$ is essentially surjective; by Proposition~\ref{MainLocEss} and Proposition~\ref{MainLocFF}, $\na$ is locally an equivalence. The claim follows by \cite[Theorem~2.25]{SP}.
\end{proof}

\section{Classifying cyclic module categories}\label{s9}

We now turn to one of the main applications, if not the main application, of our results.

\begin{definition}[{\cite[Definition~2.19]{MMMTZ1}}]\label{CyclicCauchyDefinition}
 A Cauchy complete $\csym{C}$-module category is said to be {\it cyclic} if there is an object $X \in \mathbf{M}$ such that every $Y \in \mathbf{M}$ is a direct summand of a finite direct sum of objects of the form $\mathbf{M}\mathrm{F}X$. Equivalently, the Cauchy completion $(\mathbf{M}\star X)^{c}$ is equivalent to $\mathbf{M}$. We say that $X$ is a {\it cyclic generator for $\mathbf{M}$.}
\end{definition}

\begin{definition}
 Given monoid objects $\mathrm{A,B}$ in a tame monoidal category $\csym{C}$, we say that $\mathrm{A,B}$ are {\it Morita equivalent} if and only if there are bimodule objects $\prescript{}{\mathrm{A}}{\mathrm{M}}_{\mathrm{B}}$ and $\prescript{}{\mathrm{B}}{\mathrm{N}}_{\mathrm{A}}$ such that $\prescript{}{\mathrm{B}}{\mathrm{M \otimes_{B} N}}_{\mathrm{B}} \simeq \prescript{}{\mathrm{B}}{\mathrm{B}}_{\mathrm{B}}$ and $\prescript{}{\mathrm{A}}{\mathrm{N \otimes_{A} M}}_{\mathrm{A}} \simeq \prescript{}{\mathrm{A}}{\mathrm{A}}_{\mathrm{A}}$
\end{definition}

\begin{theorem}\label{MainConsequence}
 Let $\mathbf{M,N}$ be cyclic $\csym{C}$-module categories. Choose a cyclic generator $X$ for $\mathbf{M}$ and a cyclic generator $Y$ for $\mathbf{N}$. The module categories $\mathbf{M,N}$ are equivalent if and only if the monoids $[X,X], [Y,Y] \in \csym{C}\!\on{-Tamb}(\csym{C},\csym{C})$ are Morita equivalent.
\end{theorem}

\begin{proof}
 Assume $\mathbf{M,N}$ are equivalent. Then, since $\mathbf{M} \simeq (\mathbf{M}\star X)^{c}$ and $\mathbf{N} \simeq (\mathbf{N}\star Y)^{c}$, we find an equivalence $(\mathbf{M}\star X)^{c} \simeq (\mathbf{M}\star Y)^{c}$, which, by Corollary~\ref{CauchyTrick}, yields an equivalence $\mathbf{M}\star X \simeq \mathbf{N}\star Y$ in $\csym{C}\!\on{-Tamb}$. Thus we find an equivalence $[X,X] = \na(\mathbf{M}\star X) \simeq \na(\mathbf{N}\star Y) = [Y,Y]$ in $\on{Bimod}(\tam)$, i.e. a Morita equivalence between $[X,X]$ and $[Y,Y]$.

 Conversely, assume that $[X,X]$ and $[Y,Y]$ are Morita equivalent. Then, since $\na$ is a biequivalence, there is an equivalence $\mathbf{M}\star X \simeq \mathbf{N}\star X$ in $\csym{C}\!\on{-Tamb}$. Again applying Corollary~\ref{CauchyTrick}, we find an equivalence $(\mathbf{M}\star X)^{c} \simeq (\mathbf{N}\star Y)^{c}$ in $\csym{C}\!\on{-Mod}$. It follows that we also have $\mathbf{M} \simeq \mathbf{N}$ in $\csym{C}\!\on{-Mod}$.
\end{proof}

\section{Ostrik algebras and MMMTZ coalgebras}\label{s10}

\subsection{The Cayley functor and its variants}\label{s101}

In the case where $\csym{C}$ is rigid or right-closed, the Cayley functor defined in \cite[Section~4]{PS} will be of importance to our study of $\csym{C}$-module categories. We show that there are multiple ways to realize the Cayley functor, which allow us to define it also in the case when $\csym{C}$ is not rigid or closed, and to define a covariant variant thereof.

Recall that the category $[\csym{C}^{\on{opp}},\mathbf{Vec}_{\Bbbk}]$ of presheaves on the monoidal category $\csym{C}$ admits a canonical monoidal structure via so-called {\it Day convolution}, first defined in \cite{Day}. We denote the Day convolution tensor product by $\circledast$. By definition, given presheaves $\euler{P,Q} \in [\csym{C}^{\on{opp}},\mathbf{Vec}_{\Bbbk}]$, we have
\[
 (\euler{P} \circledast \euler{Q})(\mathrm{F}) = \int^{\mathrm{H,K}} \csym{C}(\mathrm{F}, \mathrm{H} \cotimes \mathrm{K}) \kotimes \euler{P}(\mathrm{H}) \kotimes \euler{Q}(\mathrm{K}).
\]

First, observe that the category $\tam$ is cocomplete: \cite[Section~5]{PS} shows that it can be realized as the Eilenberg-Moore category for a monad on $[\csym{C} \kotimes \csym{C}^{\on{opp}}, \mathbf{Vec}_{\Bbbk}]$, and we give an explicit construction of coequalizers in Proposition~\ref{Tameness}.

Recall that we have a monoidal equivalence $\csym{C}^{\otimes\!\on{opp}}\xiso \csym{C}\!\on{-Mod}(\csym{C},\csym{C})$, sending an object $\mathrm{F}$ to the module category endomorphism given by the functor $- \cotimes \mathrm{F}$. Postcomposing with the embedding of $\csym{C}\!\on{-Mod}(\csym{C},\csym{C})^{\otimes\!\on{opp},\on{opp}}$ in $\csym{C}\!\on{-Tamb}(\csym{C},\csym{C})$ obtained in Proposition~\ref{ContravariantProArrow}, we get a strong monoidal functor $\csym{C}^{\on{opp}} \rightarrow \csym{C}\!\on{-Tamb}(\csym{C},\csym{C})$, sending $\mathrm{H} \in \csym{C}$ to the Tambara module defined by
\[
 (\mathrm{F,G}) \mapsto \Hom{\mathrm{F} \cotimes \mathrm{H}, \mathrm{G}}.
\]

Using the universal property of Day convolution described in \cite[Theorem~5.1]{IK}, we may extend this to a strong monoidal functor
\begin{equation}
\begin{aligned}
 \mathbb{L}: [\csym{C},\mathbf{Vec}_{\Bbbk}] &\rightarrow \tam \\
 \euler{P} &\mapsto \Big((\mathrm{F,G}) \mapsto \int^{\mathrm{H}} \Hom{\mathrm{F} \cotimes \mathrm{H},\mathrm{G}} \otimes \euler{P}(\mathrm{H})\Big)
\end{aligned}
\end{equation}

The Tambara structure
\[
 \int^{\mathrm{H}} \Hom{\mathrm{F} \cotimes \mathrm{H},\mathrm{G}} \kotimes \euler{P}(\mathrm{H}) \rightarrow \int^{\mathrm{H}} \Hom{\mathrm{K}\mathrm{F} \cotimes \mathrm{H},\mathrm{K}\mathrm{G}} \kotimes \euler{P}(\mathrm{H})
\]
is given by the extranatural family $\setj{\ta^{\mathbb{L}(\euler{P})}_{\mathrm{K;\mathrm{F,G}}}: \Hom{\mathrm{F} \cotimes \mathrm{H},\mathrm{G}} \kotimes \euler{P}(\mathrm{H}) \rightarrow \Hom{\mathrm{K}\mathrm{F} \cotimes \mathrm{H},\mathrm{K}\mathrm{G}} \kotimes \euler{P}(\mathrm{H})}$ where
\begin{equation}\label{TambaraCayley}
\ta^{\mathbb{L}(\euler{P})}_{\mathrm{K;\mathrm{F,G}}} = (\mathrm{K} \cotimes -)_{\mathrm{F} \cotimes \mathrm{H},\mathrm{G}} \kotimes \euler{P}(\mathrm{H}).
\end{equation}

If $\csym{C}$ is left-closed, then the Cayley functor of \cite[Section~4]{PS} is defined as
\begin{equation}
\begin{aligned}
 \mathbb{K}: [\csym{C},\mathbf{Vec}_{\Bbbk}] &\rightarrow \tam \\
 \euler{P} &\mapsto \Big((\mathrm{F,G}) \mapsto \euler{P}(\Homint{\mathrm{F},\mathrm{G}}_{l})\Big),
\end{aligned}
\end{equation}
where the Tambara structure on $\mathbb{K}(\euler{P})$ is given by
\begin{equation}\label{WHomTamb}
 \euler{P}(\Homint{\mathrm{F},\mathrm{G}}_{l}) \xrightarrow{\euler{P}(\Homint{\mathrm{F},\on{coev}_{\mathrm{K},\mathrm{G}})}} \euler{P}(\Homint{\mathrm{F},\Homint{\mathrm{K},\mathrm{K}\cotimes \mathrm{G}}_{l}}_{l}) \xrightarrow[\sim]{\euler{P}(\phi)} \euler{P}(\Homint{\mathrm{K} \cotimes \mathrm{F}, \mathrm{K} \cotimes \mathrm{G}}).
\end{equation}
where $\phi$ is the isomorphism resulting from the coincidence of the respective universal properties of $\Homint{\mathrm{K} \cotimes \mathrm{F}, \mathrm{K} \cotimes \mathrm{G}}$ and $\Homint{\mathrm{F},\Homint{\mathrm{K},\mathrm{K}\cotimes \mathrm{G}}_{l}}_{l}$.

The Tambara structure in~\eqref{WHomTamb} is in fact also obtained by transporting the Tambara structure in~\label{TambaraCayley} along the following profunctor isomorphism:
\begin{equation}\label{ToCayley}
 \int^{\mathrm{H}} \Hom{\mathrm{F} \cotimes \mathrm{H},\mathrm{G}} \kotimes \euler{P}(\mathrm{H}) \xiso \int^{\mathrm{H}} \Hom{\mathrm{H},\Homint{\mathrm{F},\mathrm{G}}_{l}} \kotimes \euler{P}(\mathrm{H}) \xiso \euler{P}(\Homint{\mathrm{F},\mathrm{G}}_{l}),
\end{equation}

The isomorphism~\eqref{ToCayley} is natural in $\euler{P}$, and thus provides a natural isomorphism from $\mathbb{L}$ to the Cayley functor $\mathbb{K}$ defined in \cite[Section~4]{PS}. To see that $\mathbb{L}$ and $\mathbb{K}$ are isomorphic as monoidal functors, observe that both are cocontinuous and that they restrict to isomorphic monoidal functors on $\csym{C}^{\on{opp}}$.

An alternative way to construct the generalized Cayley functor $\mathbb{L}$ is as follows: we first embed $\csym{C}^{\on{opp}}$ in $[\csym{C}, \mathbf{Vec}_{\Bbbk}]$ via the Yoneda embedding $\yo_{\ccf{C}^{\on{opp}}}$, then postcompose with the embedding of
\[
[\csym{C}, \mathbf{Vec}_{\Bbbk}]^{\otimes\!\on{opp}} \hookrightarrow [\csym{C}, \mathbf{Vec}_{\Bbbk}]\!\on{-Mod}([\csym{C}, \mathbf{Vec}_{\Bbbk}],[\csym{C}, \mathbf{Vec}_{\Bbbk}])
\]
and further postcompose with the embedding of the latter into $[\csym{C}, \mathbf{Vec}_{\Bbbk}]\!\on{-Tamb}([\csym{C}, \mathbf{Vec}_{\Bbbk}],[\csym{C}, \mathbf{Vec}_{\Bbbk}])$, as given in Definition~\ref{PsPDef}.
Restricting to $\csym{C}^{\on{opp}}\!\on{-Tamb}(\csym{C}^{\on{opp}},\csym{C}^{\on{opp}})$
along $\yo_{\ccf{C}^{\on{opp}}}$, via the functor $\euler{Res}_{\yo_{\cccsym{C}^{\on{opp}}}} \circ (\yo_{\ccf{C}^{\on{opp}}})_{[\ccf{C},
\mathbf{Vec}_{\Bbbk}],[\ccf{C}, \mathbf{Vec}_{\Bbbk}]}^{*}$ defined in Section~\ref{s6}, we obtain a strong monoidal functor from $\csym{C}^{\otimes\!\on{opp},\on{opp}}$ to $\csym{C}^{\on{opp}}\!\on{-Tamb}(\csym{C}^{\on{opp}},\csym{C}^{\on{opp}})$.

The equality $((\csym{C})^{\on{opp}})^{\on{opp}} = \csym{C}$ gives a canonical equivalence $\csym{C}^{\on{opp}}\!\on{-Tamb}(\csym{C}^{\on{opp}},\csym{C}^{\on{opp}}) \simeq \csym{C}\!\on{-Tamb}(\csym{C},\csym{C})^{\otimes\!\on{opp}}$. The resulting strong monoidal functor from $\csym{C}^{\on{opp}}$ to $\csym{C}\!\on{-Tamb}(\csym{C},\csym{C})$ sends $\mathrm{H} \in \csym{C}$ to the Tambara module given by
\[
  (\mathrm{G,F}) \mapsto [\csym{C}, \mathbf{Vec}_{\Bbbk}](\Hom{\mathrm{F},-} , \Hom{\mathrm{G},-}\circledast \Hom{\mathrm{H},-}) \simeq \csym{C}^{\on{opp}}(\mathrm{F},\mathrm{G} \cotimes \mathrm{H}) = \csym{C}(\mathrm{G} \cotimes \mathrm{H},\mathrm{F}),
\]
 with the evident Tambara structure; the obtained functor
 coincides with the functor $\mathbb{L}$ defined above.

 Similarly, we may extend the composite monoidal functor
 \[
 \begin{aligned}
  \csym{C}^{\otimes\!\on{opp}} &\xiso \csym{C}\!\on{-Mod}(\csym{C},\csym{C}) \hookrightarrow \tam \\
  &
  \mathrm{H} \longmapsto \big((\mathrm{F,G}) \longmapsto \csym{C}(\mathrm{F}, \mathrm{G} \cotimes \mathrm{H})\big)
 \end{aligned}
 \]
 to a strong monoidal functor
 \begin{equation}\label{WIntroduced}
 \begin{aligned}
  \mathbb{W}: [\csym{C}^{\on{opp}}, \mathbf{Vec}_{\Bbbk}]^{\otimes\!\on{opp}} &\rightarrow \tam \\
  \euler{P} &\mapsto \Big( (\mathrm{F,G}) \mapsto \int^{\mathrm{H}} \csym{C}(\mathrm{F}, \mathrm{G} \cotimes \mathrm{H}) \kotimes \euler{P}(\mathrm{H}) \Big)
 \end{aligned}
 \end{equation}
 where the Tambara structure on $\mathbb{W}(\euler{P})$ is defined analogously to that given in Equation~\eqref{TambaraCayley}.

 Just like $\mathbb{L}$, $\mathbb{W}$ also admits an alternative, equivalent construction via the embedding of $[\csym{C}^{\on{opp}}, \mathbf{Vec}_{\Bbbk}]^{\otimes\!\on{opp}}$ in $[\csym{C}^{\on{opp}}, \mathbf{Vec}_{\Bbbk}]\!\on{-Tamb}([\csym{C}^{\on{opp}}, \mathbf{Vec}_{\Bbbk}],[\csym{C}^{\on{opp}}, \mathbf{Vec}_{\Bbbk}])$.

 If $\csym{C}$ has right duals, then, using the isomorphism $\csym{C}(\mathrm{G}^{\vee} \cotimes \mathrm{F}, \mathrm{H}) \simeq \csym{C}(\mathrm{F},\mathrm{G} \cotimes \mathrm{H})$, we find isomorphisms
 \begin{equation}\label{FutureReminder}
  \mathbb{W}(\euler{P}) = \int^{\mathrm{H}} \csym{C}(\mathrm{F}, \mathrm{G} \cotimes \mathrm{H}) \kotimes \euler{P}(\mathrm{H}) \simeq \int^{\mathrm{H}} \csym{C}(\mathrm{G}^{\vee} \cotimes \mathrm{F}, \mathrm{H}) \kotimes \euler{P}(\mathrm{H}) \simeq \euler{P}(\mathrm{G}^{\vee} \cotimes \mathrm{F}).
 \end{equation}
 If $\csym{C}$ is also closed, we have $\prescript{\vee}{}{\mathrm{G}} \cotimes \mathrm{F} \simeq \Homint{\prescript{\vee\vee}{}{\mathrm{G}},\mathrm{F}}_{l}$, and if $\csym{C}$ is pivotal, we find a further isomorphism to $\Homint{\mathrm{G},\mathrm{F}}_{l}$. In that case, transporting the Tambara structure to the profunctor sending $(\mathrm{F,G})$ to $\euler{P}(\Homint{\mathrm{G},\mathrm{F}}_{l})$ yields a Tambara structure analogous to that given in Equation~\eqref{WHomTamb}. Observe that this simplification of $\mathbb{W}$ in the pivotal case is due to $\Homint{\mathrm{G,-}}_{l}$ giving a left adjoint to $\mathrm{G} \cotimes -$, i.e. due to the closed structure being also the {\it coclosed} structure. In the non-pivotal case it is not clear to the author how one could define a Tambara structure on the profunctor sending $(\mathrm{F,G})$ to $\csym{C}(\Homint{\mathrm{G},\mathrm{F}}_{l},\mathrm{H})$.

 Finally, note that if $\csym{C}$ has right duals, we have
 $
  \csym{C}(\mathrm{F}, \mathrm{G} \cotimes \mathrm{H}) \simeq \csym{C}(\mathrm{F} \cotimes \mathrm{H}^{\vee}, \mathrm{G})
 $,
 and so $\mathbb{W} \simeq \mathbb{L} \circ (-)^{\vee}$.

The following result is crucial for us:
\begin{theorem}[{\cite[Proposition~4.2]{PS}}]
 If $\csym{C}$ is rigid, then the Cayley functor
 $
  \mathbb{K}: [\csym{C},\mathbf{Vec}_{\Bbbk}] \rightarrow \tam
 $
 is a monoidal equivalence. Further, we have $\mathbb{K}^{-1}(\mathtt{\Psi}) \simeq \mathtt{\Psi}(\mathbb{1},-)$.
\end{theorem}

\subsection{Ostrik algebras and their variants}\label{s102}

Recall that, for the purposes of studying module categories, our main theorem, Theorem~\ref{MainThm}, can be interpreted as associating monoid objects in a suitable monoidal category to module categories, in such a way that Morita equivalence of said monoid objects corresponds to equivalence of module categories. An important result of this kind is due to Ostrik, appearing first in \cite{Os}. The following is a brief outline:

\begin{itemize}
 \item Let $X,Y,Z \in \mathbf{M}$. Consider the functor $\Hom{-X,Y}: \csym{C} \rightarrow \mathbf{Vec}_{\Bbbk}$. In the setting of \cite{Os}, it is (ind-)representable. Write $\setj{X,Y}$ for the representing object. (This notation is motivated by the observation that, as observed in \cite[Proposition~2.15]{DSPS}, $\setj{X,Y}$ gives the cotensor for an enrichment.)
 It is known as the {\it internal hom} from $X$ to $Y$ with respect to $\mathbf{M}$.
 \item The isomorphism $\Hom{\mathbf{M}\mathrm{F}X, Y} \simeq \Hom{\mathrm{F}, \setj{X,Y}}$ can be shown to be natural in $Y$, thus realizing the functor $\setj{X,-}: \mathbf{M} \rightarrow \csym{C}$ as right adjoint to $\mathbf{M}(-)X$.
 \item Denoting the counit of the above adjunction by $\on{ev}_{X}$, and passing under the above hom-bijection the composite morphism
 \begin{equation}\label{OstrikAlgMult}
  \mathbf{M}\big(\!\setj{Y,Z} \cotimes \setj{X,Y}\!\big) X \xrightarrow{(\mathbf{m}_{\setj{Y,Z},\setj{X,Y}}^{-1})_{X}} \mathbf{M}\!\setj{Y,Z}\!\mathbf{M}\!\setj{X,Y}) X \xrightarrow{\mathbf{M}\!\setj{Y,Z}\on{ev}_{X,Y}} \mathbf{M}\!\setj{Y,Z}Y \xrightarrow{\on{ev}_{Y,Z}} Z
 \end{equation}
 yields a composition for the internal hom bifunctor. It can be shown that this composition is associative, which endows $\setj{X,X}$ with the structure of a monoid object. The unit morphism is given by passing the morphism $\on{id}_{X}$ under the bijection $\Hom{X,X} \simeq \Hom{\mathbb{1}, \setj{X,X}}$.
\end{itemize}

We now show that the composition maps for internal homs can be constructed on the level of copresheaf categories, using Day convolution.
\begin{definition}\label{AbstractMonoid}
 We define the natural transformation $\euler{c}_{X,Y,Z}: \Hom{-Y,Z} \circledast \Hom{-X,Y} \rightarrow \Hom{-X,Z}$ by letting its component $(\euler{c}_{X,Y,Z})_{\mathrm{F}}$ be given by the map
 \[
 (\Hom{-Y,Z} \circledast \Hom{-X,Y})(\mathrm{F}) = \int^{\mathrm{K,L}} \Hom{\mathbf{M}\mathrm{K}Y,Z} \kotimes \Hom{\mathbf{M}\mathrm{L}X,Y} \kotimes \csym{C}(\mathrm{F},\mathrm{K}\cotimes \mathrm{L}) \rightarrow \Hom{\mathbf{M}\mathrm{F}X,Z}
 \]
 induced by the extranatural collection of maps $\setj{(\euler{c}_{X,Y,Z})_{\mathrm{F}})_{\mathrm{K,L}} \; | \; \mathrm{K,L} \in \csym{C}}$ defined as
 \[\begin{tikzcd}[ampersand replacement=\&]
	{  \Hom{\mathbf{M}\mathrm{K}Y,Z} \kotimes \Hom{\mathbf{M}\mathrm{L}X,Y} \kotimes \csym{C}(\mathrm{F},\mathrm{K}\cotimes \mathrm{L})} \& {  \Hom{\mathbf{M}\mathrm{K}Y,Z} \kotimes \Hom{\mathbf{M}\mathrm{K}\mathbf{M}\mathrm{L}X,\mathbf{M}\mathrm{K}Y} \kotimes \csym{C}(\mathrm{F},\mathrm{K}\cotimes \mathrm{L})} \\
	{\Hom{\mathbf{M}\mathrm{K}Y,Z} \kotimes \Hom{\mathbf{M}\mathrm{KL}X,\mathbf{M}\mathrm{K}Y}\kotimes \Hom{\mathbf{M}\mathrm{F}X,\mathbf{M}\mathrm{KL}X}} \& {\Hom{\mathbf{M}\mathrm{K}Y,Z} \kotimes \Hom{\mathbf{M}\mathrm{K}\mathrm{L}X,\mathbf{M}\mathrm{K}Y} \kotimes \csym{C}(\mathrm{F},\mathrm{K}\cotimes \mathrm{L})} \\
	{\Hom{\mathbf{M}\mathrm{F}X,Z}}
	\arrow["{\Hom{\mathbf{M}\mathrm{K}Y,Z} \otimes \mathbf{M}\mathrm{K}_{\mathbf{M}\mathrm{L}X,Y}\otimes \ccf{C}(\mathrm{F},\mathrm{K}\otimes \mathrm{L})}", shift left=1, from=1-1, to=1-2]
	\arrow["{  \Hom{\mathbf{M}\mathrm{K}Y,Z} \otimes \Hom{\mathbf{m}_{\mathrm{K,L},X}^{-1},\mathbf{M}\mathrm{K}Y} \otimes \ccf{C}(\mathrm{F},\mathrm{K}\otimes \mathrm{L})}"', from=1-2, to=2-2]
	\arrow["{\Hom{\mathbf{M}\mathrm{K}Y,Z} \otimes \Hom{\mathbf{M}\mathrm{K}\mathrm{L}X,\mathbf{M}\mathrm{K}Y} \otimes (\mathbf{M}_{\mathrm{F,KL}})_{X}}", shift left=1, from=2-2, to=2-1]
	\arrow["{(-\circ - \circ-)}"', from=2-1, to=3-1]
\end{tikzcd}\]
 Thus, on elements we have $b \otimes c \otimes \mathrm{f} \mapsto b\circ \mathbf{M}\mathrm{K}(c) \circ \mathbf{m}_{\mathrm{K,L},X}^{-1} \circ (\mathbf{M}\mathrm{f})_{X}$. Extranaturality follows from naturality of $\mathbf{M}\mathrm{g}$, for $\mathrm{g}$ a morphism of $\csym{C}$.
 Given $\mathrm{h}: \mathrm{F} \rightarrow \mathrm{F}'$, we have
 \[
  \resizebox{.99\hsize}{!}{$
 {\begin{aligned}
&\Big(
((\euler{c}_{X,Y,Z})_{\mathrm{F'}})_{\mathrm{K,L}}\circ (\Hom{-Y,Z} \circledast \Hom{-X,Y})(\mathrm{h})_{\mathrm{K,L}}\Big)(b \otimes c \otimes \mathrm{f}) \\
&=(\euler{c}_{X,Y,Z})_{\mathrm{F}})_{\mathrm{K,L}}(b \otimes c \otimes (\mathrm{f} \circ \mathrm{h})) = b\circ \mathbf{M}\mathrm{K}(c) \circ \mathbf{m}_{\mathrm{K,L},X}^{-1} \circ (\mathbf{M}(\mathrm{f}\circ \mathrm{h}))_{X} = b\circ \mathbf{M}\mathrm{K}(c) \circ \mathbf{m}_{\mathrm{K,L},X}^{-1} \circ (\mathbf{M}(\mathrm{f}))_{X} \circ (\mathbf{M}(\mathrm{h}))_{X} \\
&= (\euler{c}_{X,Y,Z})_{\mathrm{F'}})_{\mathrm{K,L}}(b \otimes c \otimes \mathrm{f}) \circ (\mathbf{M}\mathrm{h})_{X} =  \Hom{-X,Z}(\mathrm{h})\circ (\euler{c}_{X,Y,Z})_{\mathrm{F'}})_{\mathrm{K,L}}(b \otimes c \otimes \mathrm{f}),
 \end{aligned}}$}
 \]
 showing naturality in $\mathrm{F}$.
\end{definition}

\begin{proposition}\label{AbstractAssoc}
 The diagram
 \begin{equation}\label{OstrikAssoc}
 \begin{tikzcd}[ampersand replacement=\&,column sep=small]
	{(\Hom{-Z,W} \circledast \Hom{-Y,Z}) \circledast \Hom{-X,Y}} \& {\Hom{-Z,W} \circledast (\Hom{-Y,Z} \circledast \Hom{-X,Y})} \& {\Hom{-Z,W}\circledast \Hom{-X,Z}} \\
	{\Hom{-Y,W} \circledast \Hom{-X,Y}} \&\& {\Hom{-X,W}}
	\arrow["\simeq", from=1-1, to=1-2]
	\arrow["{\euler{c}_{Y,Z,W}\circledast \Hom{-X,Y}}"', from=1-1, to=2-1]
	\arrow["{\Hom{-Z,W} \circledast \euler{c}_{X,Y,Z}}", shift left=1, from=1-2, to=1-3]
	\arrow["{\euler{c}_{X,Y,W}}", from=2-1, to=2-3]
	\arrow["{\euler{c}_{X,Z,W}}"', from=1-3, to=2-3]
\end{tikzcd}
 \end{equation}
commutes.
\end{proposition}

\begin{proof}
 First, observe that on components the associator given by Yoneda lemma takes the following form:
 \[
 \begin{aligned}
  &\big((\Hom{-Z,W} \circledast \Hom{-Y,Z}) \circledast \Hom{-X,Y}\big)(\mathrm{F}) = \int^{\mathrm{S,T}} (\Hom{-Z,W} \circledast \Hom{-Y,Z})(\mathrm{S}) \kotimes \Hom{\mathbf{M}\mathrm{T}X,Y} \kotimes \csym{C}(\mathrm{F}, \mathrm{ST}) \\
  &= \int^{\mathrm{S,T}} \Big( \int^{\mathrm{K,L}} \Hom{\mathbf{M}\mathrm{K}Z,W} \kotimes \Hom{\mathbf{M}\mathrm{L}Y,Z} \kotimes \csym{C}(\mathrm{S},\mathrm{KL})\Big) \kotimes \Hom{\mathbf{M}\mathrm{T}X,Y} \kotimes \csym{C}(\mathrm{F}, \mathrm{ST}) \\
  &\xiso \int^{\mathrm{T,K,L}} \Hom{\mathbf{M}\mathrm{K}Z,W} \kotimes \Hom{\mathbf{M}\mathrm{L}Y,Z} \kotimes \Hom{\mathbf{M}\mathrm{T}X,Y} \kotimes \csym{C}(\mathrm{F}, (\mathrm{KL})\mathrm{T}) \\
  &\xleftarrow{\sim} \int^{\mathrm{K,S}} \Hom{\mathbf{M}\mathrm{K}Z,W} \kotimes \Big( \int^{\mathrm{L,T}} \Hom{\mathbf{M}\mathrm{L}Y,Z} \kotimes \Hom{\mathbf{M}\mathrm{T}X,Y} \kotimes \csym{C}(\mathrm{S},\mathrm{LT}) \Big) \kotimes \csym{C}(\mathrm{F},\mathrm{KS}) \\
  &=\big(\Hom{-Z,W} \circledast (\Hom{-Y,Z} \circledast \Hom{-X,Y})\big)(\mathrm{F})
 \end{aligned}
 \]
 Precomposing Diagram~\eqref{OstrikAssoc} with the above isomorphism to the functor
 \[
\mathrm{F} \mapsto \int^{\mathrm{T,K,L}} \Hom{\mathbf{M}\mathrm{K}Z,W} \kotimes \Hom{\mathbf{M}\mathrm{L}Y,Z} \kotimes \Hom{\mathbf{M}\mathrm{T}X,Y} \kotimes \csym{C}(\mathrm{F}, (\mathrm{KL})\mathrm{T}),
 \]
the two maps of Diagram~\eqref{OstrikAssoc} both yield the morphism corresponding to the extranatural collection sending
\[
\begin{aligned}
 &\Hom{\mathbf{M}\mathrm{K}Z,W} \kotimes \Hom{\mathbf{M}\mathrm{L}Y,Z} \kotimes \Hom{\mathbf{M}\mathrm{T}X,Y} \kotimes \csym{C}(\mathrm{F}, (\mathrm{KL})\mathrm{T}) \rightarrow \Hom{\mathbf{M}\mathrm{F}X,W} \\
 &c \otimes b \otimes a \otimes \mathrm{f} \mapsto c \circ \mathbf{M}\mathrm{K}(b \circ \mathbf{M}\mathrm{L}a) \circ \mathbf{m}_{\mathrm{K,L}}^{-1}\mathbf{M}\mathrm{T} \circ \mathbf{m}_{\mathrm{KL,T}}^{-1} \circ (\mathbf{M}\mathrm{f})_{X}
\end{aligned}
\]
\end{proof}

\begin{definition}\label{AbstractUnit}
 We define the unit transformation $\euler{u}_{X}: \csym{C}(-,\mathbb{1}) \Rightarrow \Hom{-X,X}$, by setting $(\euler{u}_{X})_{\mathrm{F}}(\mathrm{f}) = (\mathbf{m}_{\mathbb{1}}^{-1}\circ \mathbf{M}\mathrm{f})_{X}$, for any $\mathrm{F}\in \csym{C}$ and $\mathrm{f} \in \csym{C}(\mathrm{F},\mathbb{1})$.
\end{definition}
Given $\mathrm{h} \in \csym{C}(\mathrm{F,F}')$ and $\mathrm{f} \in \csym{C}(\mathrm{F}',\mathbb{1})$, we have
 \[
  (\euler{u}_{\mathrm{F}} \circ \csym{C}(\mathrm{h},\mathbb{1}))(\mathrm{f}) = \mathbf{m}_{\mathbb{1}}^{-1} \circ \mathbf{M}(\mathrm{f} \circ \mathrm{h}) = \mathbf{m}_{\mathbb{1}}^{-1} \circ \mathbf{M}(\mathrm{f}) \circ \mathbf{M}(\mathrm{h}) = (\Hom{(\mathbf{M}\mathrm{h})_{X},X}\circ \euler{u}_{\mathrm{F}'})(\mathrm{f}),
 \]
 showing naturality of $\euler{u}_{X}$.

We also give a proof of left unitality; right unitality is analogous.
\begin{proposition}\label{AbstractUnitality}
 The diagram
 \[\begin{tikzcd}[ampersand replacement=\&,column sep=scriptsize,row sep=small]
	{\csym{C}(-,\mathbb{1})\circledast \Hom{-X,X}} \& {\Hom{-X,X} \circledast \Hom{-X,X}} \\
	\& {\Hom{-X,X}}
	\arrow["{\euler{u}\circledast \Hom{-X,X}}", shift left=1, from=1-1, to=1-2]
	\arrow["{\euler{c}_{X,X,X}}", from=1-2, to=2-2]
	\arrow["{\mathsf{l}_{\Hom{-X,X}}}"', from=1-1, to=2-2]
\end{tikzcd}\]
commutes.
\end{proposition}

\begin{proof}
 Recall that the left unitor $\mathsf{l}_{\Hom{-X,X}}$ is given by the composition
 \[\begin{tikzcd}[ampersand replacement=\&,column sep=scriptsize,row sep=small]
	{\int^{\mathrm{K,L}}\csym{C}(\mathrm{K},\mathbb{1}) \kotimes \Hom{\mathbf{M}\mathrm{L}X,X} \kotimes \csym{C}(\mathrm{F,KL})} \& {\int^{\mathrm{L}} \Hom{\mathbf{M}\mathrm{L}X,X} \kotimes \csym{C}(\mathrm{F},\mathbb{1}\mathrm{L})} \\
	{\Hom{\mathbf{M}\mathrm{F}X,X}} \& {\int^{\mathrm{L}} \Hom{\mathbf{M}\mathrm{L}X,X} \kotimes \csym{C}(\mathrm{F},\mathrm{L})}
	\arrow["\simeq", from=1-1, to=1-2]
	\arrow["\simeq"', from=1-2, to=2-2]
	\arrow["\simeq"', from=2-2, to=2-1]
\end{tikzcd}\]
Thus, the corresponding extranatural collection $\setj{\csym{C}(\mathrm{K},\mathbb{1}) \kotimes \Hom{\mathrm{L}X,X} \kotimes \csym{C}(\mathrm{F,KL}) \rightarrow \Hom{\mathrm{F}X,X} \; | \; \mathrm{K,L}}$ consists of maps
\[
\begin{aligned}
 \csym{C}(\mathrm{K},\mathbb{1}) \kotimes \Hom{\mathbf{M}\mathrm{L}X,X} \kotimes \csym{C}(\mathrm{F,KL}) \rightarrow \Hom{\mathrm{F}X,X} \\
 \mathrm{k} \otimes b \otimes \mathrm{f} \mapsto \mathrm{b} \circ \mathbf{M}\mathsf{l}_{\mathrm{L}}^{\ccf{C}} \circ \mathbf{M}(\mathrm{kL}) \circ \mathbf{M}\mathrm{f}.
\end{aligned}
\]
On the other hand, the corresponding map in the extranatural collection giving
$\euler{c}_{X,X,X} \circ \euler{u}\circledast \Hom{-X,X}$ sends
$\mathrm{k} \otimes b \otimes \mathrm{f}$ to
$(\mathbf{m}_{\mathbb{1}}^{-1})_{X} \circ \mathbf{M}\mathrm{k} \circ \mathbf{M}\mathrm{K}b \circ \mathbf{m}_{\mathrm{K,L}}^{-1} \circ \mathbf{M}\mathrm{f}$. The sought equality
\[
 b \circ \mathbf{M}\euler{l}_{\mathrm{L}}^{\ccf{C}} \circ \mathbf{M}(\mathrm{kL}) \circ \mathbf{M}\mathrm{f} = (\mathbf{m}_{\mathbb{1}}^{-1})_{X} \circ \mathbf{M}\mathrm{k} \circ \mathbf{M}\mathrm{K}b \circ \mathbf{m}_{\mathrm{K,L}}^{-1} \circ \mathbf{M}\mathrm{f}
\]
follows from the commutativity of
\[\begin{tikzcd}[ampersand replacement=\&]
	{\mathbf{M}\mathrm{KL}X} \& {\mathbf{M}\mathrm{K}\mathbf{M}\mathrm{L}X} \& {\mathbf{M}\mathrm{K}X} \\
	{\mathbf{M}\mathbb{1}\mathrm{L}X} \& {\mathbf{M}\mathbb{1}\mathbf{M}\mathrm{L}X} \& {\mathbf{M}\mathbb{1}X} \\
	{\mathbf{M}\mathrm{L}X} \& X
	\arrow["{(\mathbf{m}_{\mathbb{1}}^{-1})_{\mathbf{M}\mathrm{L}X}}"{description}, from=2-2, to=3-1]
	\arrow["{\mathbf{m}_{\mathbb{1}}^{-1}}"', from=2-3, to=3-2]
	\arrow["b"{pos=0.6}, from=3-1, to=3-2]
	\arrow["{\mathbf{M}(\mathsf{l}_{\mathrm{L}}^{\cccsym{C}})}"', from=2-1, to=3-1]
	\arrow["{\mathbf{m}_{\mathbb{1},\mathrm{L}}^{-1}}", from=2-1, to=2-2]
	\arrow["{\mathbf{M}\mathrm{k}L}"', from=1-1, to=2-1]
	\arrow["{\mathbf{m}_{\mathrm{K,L}}^{-1}}", from=1-1, to=1-2]
	\arrow["{(\mathbf{M}\mathrm{k})_{\mathbf{M}\mathrm{L}X}}"', from=1-2, to=2-2]
	\arrow["{\mathbf{M}\mathrm{K}b}", from=1-2, to=1-3]
	\arrow["{(\mathbf{M}\mathrm{k})_{X}}"', from=1-3, to=2-3]
	\arrow["{\mathbf{M}\mathbb{1}b}", from=2-2, to=2-3]
\end{tikzcd},\]
whose faces commute by coherence axioms for $\mathbf{M}$.
\end{proof}

Proposition~\ref{AbstractAssoc} and Proposition~\ref{AbstractUnitality} together show the following:
\begin{corollary}
 $(\Hom{-X,X}, \euler{c}_{X,X,X}, \euler{u}_{X})$ is a monoid in $[\csym{C}^{\on{opp}},\mathbf{Vec}_{\Bbbk}]$.
 Similarly, the functor $\Hom{X,-X}$ can be endowed with the structure of a monoid in $[\csym{C},\mathbf{Vec}_{\Bbbk}]$.
\end{corollary}

\begin{proposition}\label{RepresentableMonoids}
 If the presheaves $\Hom{-X,Y}$ and $\Hom{-Y,Z}$ are representable via natural isomorphisms
 \[
\tau_{X,Y}: \csym{C}(-,\setj{X,Y}) \rightarrow \Hom{-X,Y} \text{ and } \tau_{Y,Z}: \csym{C}(-,\setj{Y,Z}) \rightarrow \Hom{-Y,Z},
 \]
 and $\euler{z}_{X,Y,Z}: \setj{Y,Z} \cotimes \setj{X,Y} \rightarrow \setj{X,Z}$ is the composition morphism used to define the monoid structure on $\setj{X,X}$ in \cite{Os}, then the composite
 \begin{equation}\label{RecoverMonoid}
  \resizebox{.99\hsize}{!}{$
  \csym{C}(-,\setj{Y,Z}\cotimes \setj{X,Y})\xiso \csym{C}(-,\setj{Y,Z}) \circledast \csym{C}(-,\setj{X,Y}) \xrightarrow{\tau_{Y,Z} \circledast \tau_{X,Y}} \Hom{-Y,Z} \circledast \Hom{-X,Y} \xrightarrow{\euler{c}_{X,Y,Z}} \Hom{-X,Z} \xrightarrow{\tau_{X,Z}^{-1}} \csym{C}(-,\setj{X,Z})$}
 \end{equation}
 coincides with $\csym{C}(-,\euler{z}_{X,Y,Z})$.
\end{proposition}

\begin{proof}
 By Yoneda lemma, it suffices to show that the components of the respective transformations at $\setj{Y,Z} \cotimes \setj{X,Y}$ send $\on{id}_{\setj{Y,Z} \cotimes \setj{X,Y}}$ to the same morphism.

 The image of $\on{id}_{\setj{Y,Z} \cotimes \setj{X,Y}}$ under the first isomorphism in \ref{RecoverMonoid} is given by the image of $\on{id}_{\setj{Y,Z}} \otimes \on{id}_{\setj{X,Y}} \otimes \on{id}_{\setj{Y,Z}\otimes \setj{X,Y}}$ under the cocone map for the component indexed by $(\setj{Y,Z},\setj{X,Y})$:
 \[\begin{tikzcd}[ampersand replacement=\&]
	{ \csym{C}(\setj{Y,Z},\setj{Y,Z}) \kotimes \csym{C}(\setj{X,Y},\setj{X,Y}) \kotimes \csym{C}(\setj{Y,Z}\cotimes \setj{X,Y},\setj{Y,Z}\cotimes \setj{X,Y})} \\
	{\int^{\mathrm{K,L}} \csym{C}(\mathrm{K},\setj{Y,Z}) \kotimes \csym{C}(\mathrm{L},\setj{X,Y}) \kotimes \csym{C}(\setj{Y,Z}\cotimes \setj{X,Y},\mathrm{K\cotimes L})} \\
	{  (\csym{C}(-,\setj{Y,Z}) \circledast \csym{C}(-,\setj{X,Y}))(\setj{Y,Z}\cotimes \setj{X,Y})}
	\arrow[from=1-1, to=2-1]
	\arrow["{=}", from=2-1, to=3-1]
\end{tikzcd}\]
The image of $\on{id}_{\setj{Y,Z}} \otimes \on{id}_{\setj{X,Y}} \otimes \on{id}_{\setj{Y,Z}\otimes \setj{X,Y}}$ under $\tau_{Y,Z} \otimes \tau_{X,Y}$ is given by $\on{ev}_{Y,Z} \otimes \on{ev}_{X,Y} \otimes \on{id}_{\setj{Y,Z}\otimes \setj{X,Y}}$. By definition,
\[
\euler{c}_{X,Y,Z}(\on{ev}_{Y,Z} \otimes \on{ev}_{X,Y} \otimes \on{id}_{\setj{Y,Z}\otimes \setj{X,Y}}) = \on{ev}_{Y,Z} \circ \mathbf{M}\!\setj{Y,Z}\on{ev}_{X,Y} \circ \mathbf{m}_{\setj{Y,Z},\setj{X,Y},X}^{-1} \circ \mathbf{M}(\on{id}_{\setj{Y,Z}\otimes \setj{X,Y}}).
\]
Comparing to Equation~\eqref{OstrikAlgMult}, we see that this is precisely $\euler{z}_{X,Y,Z}$.
\end{proof}

One may also show a similar statement for the unitality maps, which yields the following consequence:
\begin{corollary}\label{OstrikMonoidStructure}
  If $\Hom{-X,X}$ is representable via natural isomorphism $\tau_{X,X}: \csym{C}(-,\setj{X,X}) \xiso \Hom{-X,X}$, then the monoid structure on $\Hom{-,\setj{X,X}}$ obtained by transporting from $\Hom{-X,X}$ along $\tau$ coincides with the monoid structure defined in \cite{Os}.
\end{corollary}

\begin{proposition}
 The composite maps
\[\begin{tikzcd}[row sep=scriptsize, column sep = huge]
	{\int^{\mathrm{H}} \csym{C}(\mathrm{F}, \mathrm{GH}) \kotimes \Hom{\mathbf{M}\mathrm{H}X,Y}} & & {\int^{\mathrm{H}} \csym{C}(\mathrm{F}, \mathrm{GH}) \kotimes \Hom{\mathbf{M}\mathrm{G}\mathbf{M}\mathrm{H}X,\mathbf{M}\mathrm{G}Y}} \\
	{\Hom{\mathbf{M}\mathrm{F}X,\mathbf{M}\mathrm{G}Y}} & & {\int^{\mathrm{H}} \csym{C}(\mathrm{F}, \mathrm{GH}) \kotimes \Hom{\mathbf{M}\mathrm{GH}X,\mathbf{M}\mathrm{G}Y}}
	\arrow["{\ccf{C}(\mathrm{F,GH}) \otimes \mathbf{M}\mathrm{G}_{\mathbf{M}\mathrm{H}X,Y}}",  from=1-1, to=1-3]
	\arrow["{\ccf{C}(\mathrm{F,GH}) \otimes \Hom{(\mathbf{m}^{-1}_{\mathrm{G,H}})_{X},\mathbf{M}\mathrm{G}Y}}", from=1-3, to=2-3]
	\arrow["{\mathrm{f}\otimes b \mapsto b \circ \mathbf{M}\mathrm{f}_{X}}"', from=2-3, to=2-1]
\end{tikzcd}\]
given by the extranatural collections
\[
 \begin{aligned}
 \csym{C}(\mathrm{F,GH}) \kotimes \Hom{\mathbf{M}\mathrm{H}X, Y} &\rightarrow \Hom{\mathbf{M}\mathrm{F}X,\mathbf{M}\mathrm{G}Y} \\
 \mathrm{f} \otimes b &\mapsto \mathbf{M}\mathrm{G}b \circ \mathbf{m}^{-1}_{\mathrm{G,H}} \mathbf{M}\mathrm{f}_{X}
 \end{aligned}
\]
give a morphism $\mathtt{\omega}_{X,Y}: \mathbb{W}(\Hom{-X,Y}) \rightarrow [X,Y]$ in $\tam$.
Further, this morphism is an isomorphism if $\csym{C}$ is rigid.
\end{proposition}

\begin{proof}
 Clearly, we obtain a morphism of profunctors. To see that it is a morphism of Tambara modules, we chase $\mathrm{f} \otimes b \in \csym{C}(\mathrm{F,GH}) \kotimes \Hom{\mathbf{M}\mathrm{H}X,Y}$ in the diagram
 \[\begin{tikzcd}[row sep=scriptsize, column sep = huge]
	{\int^{\mathrm{H}} \csym{C}(\mathrm{F}, \mathrm{GH}) \kotimes \Hom{\mathbf{M}\mathrm{H}X,Y}} && {\coprod_{\mathrm{H}} \csym{C}(\mathrm{KF}, \mathrm{K(GH)}) \kotimes \Hom{\mathbf{M}\mathrm{K}\mathbf{M}\mathrm{H}X,\mathbf{M}\mathrm{K}Y}} \\
	{\Hom{\mathbf{M}\mathrm{F}X,\mathbf{M}\mathrm{G}Y}} && {\int^{\mathrm{H}} \csym{C}(\mathrm{KF}, \mathrm{(KG)H}) \kotimes \Hom{\mathbf{M}\mathrm{H}X,Y}} \\
	{\Hom{\mathbf{M}\mathrm{K}\mathbf{M}\mathrm{F}X,\mathbf{M}\mathrm{K}\mathbf{M}\mathrm{G}Y}} && {\Hom{\mathbf{M}\mathrm{KF}X,\mathbf{M}\mathrm{KG}Y}}
	\arrow["(\omega_{X,Y})_{\mathrm{F,G}}", from=1-1, to=2-1]
	\arrow["{\mathbf{M}\mathrm{K}_{\mathbf{M}\mathrm{F}X,\mathbf{M}\mathrm{G}Y}}", from=2-1, to=3-1]
	\arrow[from=1-1, to=1-3]
	\arrow[from=1-3, to=2-3]
	\arrow["\ta_{\mathrm{K;F,G}}^{\mathbb{W}(\Hom{-X,Y})}"', from=1-1, to=2-3]
	\arrow["(\omega_{X,Y})_{\mathrm{KF,KG}}"', swap, from=2-3, to=3-3]
	\arrow["{\Hom{(\mathbf{m}^{-1}_{\mathrm{K,F}}),(\mathbf{m}_{\mathrm{K,G}})_{Y}}}", swap, from=3-1, to=3-3]
	\arrow["\ta_{\mathrm{K;F,G}}^{[X,Y]}", from=2-1, to=3-3]
\end{tikzcd}\]
to show that it commutes - indeed, the commutativity of its exterior is precisely the condition for a Tambara morphism. Its upper face serves only to indicate the definition of the Tambara structure $\ta_{\mathrm{K;F,G}}^{\mathbb{W}(\Hom{-X,Y})}$, and commutes by definition.

\[\begin{tikzcd}[row sep=scriptsize]
	{\mathrm{f} \otimes b} && {\mathrm{K}\mathrm{f} \otimes \mathbf{M}\mathrm{K}b} \\
	{\mathbf{M}\mathrm{G}b \circ \mathbf{m}^{-1}_{\mathrm{G,H}} \circ \mathbf{M}\mathrm{f}_{X}} && {\mathrm{K}\mathrm{f} \otimes \mathbf{M}\mathrm{K}b \circ \mathbf{m}^{-1}_{\mathrm{K,H}}} \\
	{\mathbf{M}\mathrm{K}(\mathbf{M}\mathrm{G}b \circ \mathbf{m}^{-1}_{\mathrm{G,H}} \circ \mathbf{M}\mathrm{f}_{X})} && {\mathbf{m}_{\mathrm{K,G}}\circ \mathbf{M}\mathrm{KG}b \circ \mathbf{m}^{-1}_{\mathrm{K,GH}} \circ \mathbf{M}\mathrm{Kf}_{X} \circ \mathbf{m}_{\mathrm{K,G}}}
	\arrow[maps to, from=1-1, to=1-3]
	\arrow[maps to, from=1-3, to=2-3]
	\arrow[maps to, from=1-1, to=2-1]
	\arrow[maps to, from=2-1, to=3-1]
	\arrow[maps to, from=3-1, to=3-3]
	\arrow[maps to, from=2-3, to=3-3]
\end{tikzcd}\]

 To see that this morphism becomes an isomorphism in the rigid case, we recall from Equation~\ref{FutureReminder} in Subsection~\ref{s101} that we then have an isomorphism of Tambara modules
 \[
\mathbb{W}(\Hom{-X,Y}) \simeq \big( (\mathrm{F,G})\mapsto \Hom{\mathbf{M}(\prescript{\vee}{}{\mathrm{G}}\cotimes \mathrm{F}) X, Y}\big) \simeq [X,Y].
 \]
\end{proof}

\begin{proposition}\label{NotationTambara}
 For any $X,Y,Z \in \mathbf{M}$, the diagram
 \[\begin{tikzcd}
	{\mathbb{W}(\Hom{-X,Y}) \diamond \mathbb{W}(\Hom{-Y,Z})} & {\mathbb{W}(\Hom{-Y,Z}\circledast \Hom{-X,Y})} & {\mathbb{W}(\Hom{-X,Z})} \\
	{[X,Y]\diamond [Y,Z]} && {[X,Z]}
	\arrow["\simeq", shift left=2, from=1-1, to=1-2]
	\arrow["{\mathbb{W}(\euler{c}_{X,Y,Z})}", shift left=2, from=1-2, to=1-3]
	\arrow["{\omega_{X,Y}\circ\omega_{Y,Z}}"', shift right=2, from=1-1, to=2-1]
	\arrow["{\mathtt{a}_{X,Y,Z}}"', from=2-1, to=2-3]
	\arrow["{\omega_{X,Z}}", from=1-3, to=2-3]
\end{tikzcd},\]
where $\mathtt{a}_{X,Y,Z}$ is the Tambara morphism induced by composition, commutes.
\end{proposition}

\begin{proof}
 We show that the diagram
 \[
  \resizebox{.99\hsize}{!}{$
 \begin{tikzcd}[ampersand replacement=\&]
	{\int^{\mathrm{L}} \big(\int^{\mathrm{H'}}\csym{C}(\mathrm{F,LH'}) \kotimes \Hom{\mathbf{M}\mathrm{H'}X,Y}\big) \kotimes \big(\int^{\mathrm{H}}\csym{C}(\mathrm{L,GH}) \kotimes \Hom{\mathbf{M}\mathrm{H}Y,Z}\big)} \& {\int^{\mathrm{H,H',L}} \Hom{\mathbf{M}\mathrm{H}Y,Z} \kotimes \Hom{\mathbf{M}\mathrm{H'}X,Y} \kotimes \csym{C}(\mathrm{L,HH'}) \kotimes \csym{C}(\mathrm{F,GL})} \& {\int^{\mathrm{K}} \Hom{\mathbf{M}\mathrm{K}X,Z} \kotimes \csym{C}(\mathrm{F,GK})} \\
	\& {\int^{\mathrm{H,H'}} \Hom{\mathbf{M}\mathrm{H}Y,Z} \kotimes \Hom{\mathbf{M}\mathrm{H'}X,Y} \kotimes \csym{C}(\mathrm{F,GHH'})} \\
	{\int^{\mathrm{L}} \Hom{\mathbf{M}\mathrm{F}X,\mathbf{M}\mathrm{L}Y} \kotimes \Hom{\mathbf{M}\mathrm{L}Y,\mathbf{M}\mathrm{G}Z}} \&\& {\Hom{\mathbf{M}\mathrm{F}X,\mathbf{M}\mathrm{G}Z}}
	\arrow[from=1-1, to=3-1]
	\arrow[from=1-1, to=1-2]
	\arrow["\simeq"{description}, from=1-1, to=2-2]
	\arrow["\simeq"', from=2-2, to=1-2]
	\arrow[from=2-2, to=1-3]
	\arrow[from=3-1, to=3-3]
	\arrow[from=1-3, to=3-3]
	\arrow[from=1-2, to=1-3]
\end{tikzcd}$}\]
commutes. Its upper two faces commute by definition of $\mathbb{W}$. Chasing an element in the diagram, we obtain
\[
  \resizebox{.99\hsize}{!}{$
\begin{tikzcd}[ampersand replacement=\&,row sep=small]
	{\mathrm{f} \otimes x \otimes \mathrm{l} \otimes y} \& {\mathrm{f} \otimes x \otimes \on{id}_{\mathrm{HH'}} \otimes \big(\mathrm{l}\mathrm{H'} \circ\mathrm{f}\big)} \& {\big( y \circ \mathbf{M}\mathrm{H}x \circ \mathbf{m}_{\mathrm{H,H'}}^{-1}\big) \otimes \big(\mathrm{l}\mathrm{H'} \circ\mathrm{f}\big)} \\
	\& {y \otimes x \otimes \big(\mathrm{l}\mathrm{H'} \circ\mathrm{f}\big)} \\
	{\big(\mathbf{M}\mathrm{L}x \circ \mathbf{m}_{\mathrm{L,H'}}^{-1} \circ \mathbf{M}\mathrm{f}_{X}\big) \otimes \big(\mathbf{M}\mathrm{G}y \circ \mathbf{m}_{\mathrm{G,H}}^{-1} \circ \mathbf{M}\mathrm{l}_{Y}\big)} \& {\big(\mathbf{M}\mathrm{G}y \circ \mathbf{m}_{\mathrm{G,H}}^{-1} \circ \mathbf{M}\mathrm{l}_{Y}\big)\circ \big(\mathbf{M}\mathrm{L}x \circ \mathbf{m}_{\mathrm{L,H'}}^{-1} \circ \mathbf{M}\mathrm{f}_{X}\big)} \& {\mathbf{M}\mathrm{G}\big( y \circ \mathbf{M}\mathrm{H}x \circ \mathbf{m}_{\mathrm{H,H'}}^{-1}\big) \circ \mathbf{m}_{\mathrm{G,HH'}}^{-1} \circ \mathbf{M}\big(\mathrm{l}\mathrm{H'} \circ\mathrm{f}\big)_{X}}
	\arrow[maps to, from=1-3, to=3-3]
	\arrow[equal, from=3-2, to=3-3]
	\arrow[maps to, from=2-2, to=1-3]
	\arrow[maps to, from=2-2, to=1-2]
	\arrow[maps to, from=1-2, to=1-3]
	\arrow[maps to, from=3-1, to=3-2]
	\arrow[maps to, from=1-1, to=2-2]
	\arrow[maps to, from=1-1, to=3-1]
	\arrow[maps to, from=1-1, to=1-2]
\end{tikzcd}$}\]
where the equality indicated on the bottom of the diagram chase holds due to the commutativity of
\[
\begin{tikzcd}[ampersand replacement=\&,column sep = huge]
	{\mathbf{M}\mathrm{F}X} \& {\mathbf{M}\mathrm{LH'}X} \& {\mathbf{M}\mathrm{GHH'}X} \& {\mathbf{M}\mathrm{G}\mathbf{M}\mathrm{HH'}X} \\
	\& {\mathbf{M}\mathrm{L}\mathbf{M}\mathrm{H}'X} \& {\mathbf{M}\mathrm{GH}\mathbf{M}\mathrm{H'}X} \& {\mathbf{M}\mathrm{G}\mathbf{M}\mathrm{H}\mathbf{M}\mathrm{H'}X} \\
	\& {\mathbf{M}\mathrm{L}Y} \& {\mathbf{M}\mathrm{GH}Y} \& {\mathbf{M}\mathrm{G}\mathbf{M}\mathrm{H}Y} \\
	\&\&\& {\mathbf{M}\mathrm{G}Z}
	\arrow["{\mathbf{M}\mathrm{G}\mathbf{m}^{-1}}", from=1-4, to=2-4]
	\arrow["{\mathbf{M}\mathrm{G}\mathbf{M}\mathrm{H}x}", from=2-4, to=3-4]
	\arrow["{\mathbf{M}\mathrm{G}y}", from=3-4, to=4-4]
	\arrow["{\mathbf{m}^{-1}}", from=3-3, to=3-4]
	\arrow["{\mathbf{m}^{-1}\mathbf{M}\mathrm{H'}}", from=2-3, to=2-4]
	\arrow["{\mathbf{M}\mathrm{GH}x}", from=2-3, to=3-3]
	\arrow["{\mathbf{m}^{-1}}", from=1-3, to=2-3]
	\arrow["{\mathbf{m}^{-1}}", from=1-3, to=1-4]
	\arrow["{\mathbf{M}(\mathrm{lH'})_{X}}", from=1-2, to=1-3]
	\arrow["{\mathbf{M}\mathrm{l}_{Y}}", from=3-2, to=3-3]
	\arrow["{\mathbf{m}^{-1}}"{description}, from=1-2, to=2-2]
	\arrow["{\mathbf{M}\mathrm{l}_{\mathbf{M}\mathrm{H}'X}}", from=2-2, to=2-3]
	\arrow["{\mathbf{M}\mathrm{L}x}", from=2-2, to=3-2]
	\arrow["{\mathbf{M}\mathrm{f}_{X}}", from=1-1, to=1-2]
\end{tikzcd},\]
which in turn commutes due to coherence of $\mathbf{M}$.
\end{proof}

\begin{corollary}\label{RigidMonoids}
 There is a canonical monoid morphism
 \[
 \omega_{X,X}: \mathbb{W}(\Hom{-X,X}) \rightarrow [X,X],
 \]
 which is an isomorphism if $\csym{C}$ is rigid.

 Similarly, for any $X,Y \in \mathbf{M}$ we obtain a morphism of $\mathbb{W}(\Hom{-X,X})$-$\mathbb{W}(\Hom{-X,X})$-bimodules
 \[
 \omega_{X,Y}: \mathbb{W}(\Hom{-X,Y}) \rightarrow [X,Y],
 \]
 which is an isomorphism if $\csym{C}$ is rigid.
\end{corollary}

Thus, our results are a generalization of those of \cite{Os}.

\subsection{MMMTZ coalgebras and their variants}\label{s103}

First, we give a brief summary of the setting of $2$-representation theory as studied in \cite{MMMT}, \cite{MMMTZ1}.

\begin{definition}
 A ($\Bbbk$-linear) category $\mathcal{C}$ is said to be {\it finitary} if it is Hom-finite, Cauchy complete and admits an additive generator, i.e. an object $X \in \mathcal{C}$ such that every $Y$ in $\mathcal{C}$ is a direct summand of $X^{\oplus m}$ for some $m \geq 0$.

 A category $\mathcal{A}$ is said to be {\it finite abelian} if it is equivalent to the finite cocompletion of a finitary category.
\end{definition}
As a consequence, the category of projective objects in a finite abelian category is finitary.
 If $A$ is a finite-dimensional algebra and $\mathcal{A}$ the category with one object whose endomorphism algebra is $A$, then the Cauchy completion $\mathcal{A}^{c}$ is finitary and equivalent to $\on{proj-}\!A$. Up to equivalence any finitary category $\mathcal{C}$ is of this form: we have $\mathcal{C} \simeq \mathcal{A}^{c}$ if we let $A := \on{End}_{\mathcal{C}}(X)$.
 As a consequence, every finite abelian category is equivalent to the category $\on{mod--}\!A$ of finitely generated modules over a finite-dimensional algebra $A$.

 A $\Bbbk$-linear bicategory $\csym{B}$ is {\it finitary} if it has finitely many objects, and for all $\mathtt{i,j} \in \csym{B}$, the category $\csym{B}(\mathtt{i,j})$ is finitary. A {\it finitary birepresentation} ($2$-representation in the strict setting) is a pseudofunctor $\mathbf{M}:\csym{B} \rightarrow \mathbf{Cat}_{\Bbbk}$ such that, for all $\mathtt{i}$, the category $\mathbf{M}(\mathtt{i})$ is finitary. $\csym{B}$ is said to be {\it fiab} (fiat in the strict case) if all of its $1$-morphisms have both left and right adjoints.

 We may associate to $\csym{B}$ the monoidal category $\csym{B}_{\circ} = \bigoplus_{\mathtt{i,j}} \csym{B}(\mathtt{i,j})$. If $\csym{B}$ is fiab, then $\csym{B}_{\circ}$ is rigid. A birepresentation $\mathbf{M}$ gives a $\csym{B}_{\circ}$-module category $\mathbf{M}_{\circ} = \bigoplus_{\mathtt{i}} \mathbf{M}(\mathtt{i})$. Given birepresentations $\mathbf{M,N}$, we have $\mathbf{M} \simeq \mathbf{N}$ if and only if $\mathbf{M}_{\circ} \simeq \mathbf{N}_{\circ}$. Thus, classification problems for finitary birepresentations are equivalent to classification problems for finitary module categories over finitary monoidal categories.

 An {\it ideal $\mathbf{I}$ in a birepresentation $\mathbf{M}$} is a collection of (two-sided) ideals of categories $\mathbf{I}_{\mathtt{i}} \vartriangleleft \mathbf{M}(\mathtt{i})$ such that for any $\mathtt{i,j}$ and any $\mathrm{F} \in \csym{B}(\mathtt{i,j})$, we have $\mathbf{M}\mathrm{F}(\mathbf{I}_{\mathtt{i}}) \subseteq \mathbf{I}_{\mathtt{j}}$. A birepresentation $\mathbf{M}$ is said to be {\it simple transitive} if its only ideals are the zero ideal and $\mathbf{M}$ itself.

 Under the above correspondence of birepresentations with module categories, we obtain the following notion:
 \begin{definition}
  An {\it ideal of a $\csym{C}$-module category $\mathbf{M}$} is an ideal of categories $\mathbf{I} \vartriangleleft \mathbf{M}$ such that, for any $\mathrm{F} \in \csym{C}$, we have $\mathbf{M}\mathrm{F}(\mathbf{I}) \subseteq \mathbf{I}$.

  A $\csym{C}$-module category $\mathbf{M}$ is {\it simple transitive} if its only ideals are the zero ideal and $\mathbf{M}$ itself.
 \end{definition}

 \begin{remark}
  We may also recast the notion of a {\it transitive} birepresentation in the setting of $\csym{C}$-module categories: we say that a $\csym{C}$-module category $\mathbf{M}$ is {\it transitive} if any non-zero object of $\mathbf{M}$ is a cyclic generator for $\mathbf{M}$.

  Following \cite[Lemma~4]{MM5}, any simple transitive $\csym{C}$-module category is transitive.
 \end{remark}

 Given a category $\mathcal{C}$ and an ideal $\mathcal{I}$ in $\mathcal{C}$, using matrix notation for morphisms between direct sums, we have
 \[
  \mathcal{I}\left(\bigoplus_{i=1}^{k} X_{i}, \bigoplus_{j=1}^{l} Y_{j}\right) =
  \begin{pmatrix}
   \mathcal{I}(X_{1},Y_{1}) & \mathcal{I}(X_{1},Y_{2}) & \cdots & \mathcal{I}(X_{1},Y_{l}) \\
   \mathcal{I}(X_{2},Y_{1}) & \mathcal{I}(X_{2},Y_{2}) & \cdots & \mathcal{I}(X_{2},Y_{l}) \\
   \vdots & \vdots & \ddots & \vdots \\
   \mathcal{I}(X_{k},Y_{1}) & \mathcal{I}(X_{k},Y_{2}) & \cdots & \mathcal{I}(X_{k},Y_{l})
  \end{pmatrix}
 \]
 For a proof, see e.g. \cite[Lemma~A.3.4]{ASS}. This gives a bijective correspondence between ideals in $\mathcal{C}$ and ideals in $\mathcal{C}^{c}$.
 \begin{corollary}
  There is a bijective correspondence between ideals in $\mathbf{M}$ and ideals in $\mathbf{M}^{c}$.
 \end{corollary}

 \begin{proof}
  Since (by definition of Cauchy completeness) any $\Bbbk$-linear functor preserves direct sums and idempotent splittings, in particular so does $\mathbf{M}\mathrm{F}$ for any $\mathrm{F} \in \csym{C}$. Thus, we have
 \[
  \mathbf{M}\mathrm{F}\left(\mathcal{I}\left(\bigoplus_{i=1}^{k} X_{i}, \bigoplus_{j=1}^{l} Y_{j}\right)\right) =
  \begin{pmatrix}
   \mathbf{M}\mathrm{F}(\mathcal{I}(X_{1},Y_{1})) & \mathbf{M}\mathrm{F}(\mathcal{I}(X_{1},Y_{2})) & \cdots & \mathbf{M}\mathrm{F}(\mathcal{I}(X_{1},Y_{l})) \\
   \mathbf{M}\mathrm{F}(\mathcal{I}(X_{2},Y_{1})) & \mathbf{M}\mathrm{F}(\mathcal{I}(X_{2},Y_{2})) & \cdots & \mathbf{M}\mathrm{F}(\mathcal{I}(X_{2},Y_{l})) \\
   \vdots & \vdots & \ddots & \vdots \\
   \mathbf{M}\mathrm{F}(\mathcal{I}(X_{k},Y_{1})) & \mathbf{M}\mathrm{F}(\mathcal{I}(X_{k},Y_{2})) & \cdots & \mathbf{M}\mathrm{F}(\mathcal{I}(X_{k},Y_{l}))
  \end{pmatrix}
 \]
 Thus, an ideal $\mathbf{I}$ in $\mathbf{M}^{c}$ is uniquely determined by $(\mathbf{I}(X,Y))_{X,Y \in \mathbf{M}}$, where we identify $\mathbf{M}$ with its image under the embedding $\mathbf{M} \rightarrow \mathbf{M}^{c}$. On the other hand, given an ideal $\mathbf{J}$ in $\mathbf{M}^{c}$, the collection $(\mathbf{J}(X,Y))_{X,Y \in \mathbf{M}}$ defines an ideal in $\mathbf{M}$. The mutually inverse bijections are thus given by extending to matrices and restricting to objects of $\mathbf{M}$, respectively.
 \end{proof}

 \begin{proposition}
  There is a bijective correspondence between ideals in $\mathbf{M}\star X$ and subbimodules of $[X,X]$.
 \end{proposition}

 \begin{proof}
  A subbimodule $\mathtt{\Sigma}$ of $[X,X]$ consists of a collection of subspaces $\mathtt{\Sigma}(\mathrm{F},\mathrm{G}) \xhookrightarrow{\euler{\iota}_{\mathrm{F,G}}} \Hom{\mathbf{M}\mathrm{F}X,\mathbf{M}\mathrm{G}X}$ such that:
  \begin{itemize}
   \item $\mathtt{\Sigma}$ yields a subprofunctor, i.e. for any morphisms $\mathrm{F \xrightarrow{f} F'}$ and $\mathrm{G \xrightarrow{g} G'}$ we have
   \[
  \Hom{\mathbf{M}\mathrm{f}_{X},\mathbf{M}\mathrm{G}X}\mathtt{\Sigma}(\mathrm{F}',\mathrm{G}) \subseteq \mathtt{\Sigma}(\mathrm{F},\mathrm{G})\text{ and }\Hom{\mathbf{M}\mathrm{F}X, \mathbf{M}\mathrm{g}_{X}}\mathtt{\Sigma}(\mathrm{F},\mathrm{G}) \subseteq \mathtt{\Sigma}(\mathrm{F},\mathrm{G}')
   \]
   \item it is a Tambara submodule, thus: $\ta^{[X,X]}_{\mathrm{K;F,G}}(\mathtt{\Sigma}(\mathrm{F,G})) \subseteq \mathtt{\Sigma}(\mathrm{KF,KG})$;
   \item it is stable under the left and right actions of $[X,X]$ on itself, i.e. the images of the maps
   \[
   \begin{aligned}
   &\begin{tikzcd}[ampersand replacement = \&]
	{\int^{\mathrm{H}} \Hom{\mathbf{M}\mathrm{F}X,\mathbf{M}\mathrm{H}X} \kotimes \mathtt{\Sigma}(\mathrm{H},\mathrm{G})} \& {\int^{\mathrm{H}} \Hom{\mathbf{M}\mathrm{F}X,\mathbf{M}\mathrm{H}X} \kotimes \Hom{\mathbf{M}\mathrm{H}X,\mathbf{M}\mathrm{G}X}} \& {\Hom{\mathbf{M}\mathrm{F}X,\mathbf{M}\mathrm{G}X}}
	\arrow["{[X,X] \kotimes \iota}", shift left=1, from=1-1, to=1-2]
	\arrow["{\mathtt{m}_{\mathrm{F,G}}}", shift left=1, from=1-2, to=1-3]
\end{tikzcd} \\
 &\begin{tikzcd}[ampersand replacement=\&]
	{\int^{\mathrm{H}}\mathtt{\Sigma}(\mathrm{F,H}) \kotimes \Hom{\mathbf{M}\mathrm{H}X,\mathbf{M}\mathrm{G}X}} \& {\int^{\mathrm{H}} \Hom{\mathbf{M}\mathrm{F}X,\mathbf{M}\mathrm{H}X} \kotimes \Hom{\mathbf{M}\mathrm{H}X,\mathbf{M}\mathrm{G}X}} \& {\Hom{\mathbf{M}\mathrm{F}X,\mathbf{M}\mathrm{G}X}}
	\arrow["{\iota \kotimes [X,X]}", shift left=1, from=1-1, to=1-2]
	\arrow["{\mathtt{m}_{\mathrm{F,G}}}", shift left=1, from=1-2, to=1-3]
\end{tikzcd}
\end{aligned}
\]
 are contained in $\mathtt{\Sigma}(\mathrm{F,G})$. Equivalently, for any $a \in \Hom{\mathbf{M}\mathrm{F}X,\mathbf{M}\mathrm{H}X}, b \in \mathtt{\Sigma}(\mathrm{H,K}), c \in \Hom{\mathbf{M}\mathrm{K}X,\mathbf{M}\mathrm{G}X}$, we have $\iota(b) \circ a \in \mathtt{\Sigma}(\mathrm{F,K})$ and $c \circ \iota(b) \in \mathtt{\Sigma}(\mathrm{H},\mathrm{G})$.
  \end{itemize}
  The last axiom subsumes the first and shows that the collection $(\mathtt{\Sigma}(\mathrm{F,G}))_{\mathrm{F,G}}$ defines an ideal in the category $\csym{C}\ostar X$. Since any object in $\mathbf{M}\star X$ is isomorphic to an object of $\csym{C}\ostar X$, an ideal $\widehat{\mathtt{\Sigma}}$ in $\csym{C}\ostar X$ uniquely determines an ideal in $\mathbf{M}\star X$, which coincides with $\mathtt{\Sigma}$ on $\csym{C}\ostar X$. It remains to show that for $a$ and $\mathrm{K}$ as above, the morphism $\mathbf{M}\mathrm{K}a$ lies in $\widehat{\mathtt{\Sigma}}(\mathbf{M}\mathrm{K}\mathbf{M}\mathrm{F},\mathbf{M}\mathrm{K}\mathbf{M}\mathrm{G})$. By the second axiom, we have
  \[
\ta^{[X,X]}_{\mathrm{K;F,G}}(a) = \mathbf{m}_{\mathrm{K,G}} \circ \mathbf{M}\mathrm{K}a \circ \mathbf{m}_{\mathrm{K,F}}^{-1} \in \mathtt{\Sigma}(\mathrm{KF,KG}) = \widehat{\mathtt{\Sigma}}(\mathbf{M}\mathrm{KF}X, \mathbf{M}\mathrm{KG}X).
  \]
  Since $\widehat{\mathtt{\Sigma}}$ is an ideal in the category $\mathbf{M}\star X$, and $\mathbf{m}_{\mathrm{K,F}}^{-1},\mathbf{m}_{\mathrm{K,G}}$ are isomorphisms, we have $\mathbf{M}\mathrm{K}a \in \widehat{\mathtt{\Sigma}}(\mathbf{M}\mathrm{K}\mathbf{M}\mathrm{F},\mathbf{M}\mathrm{K}\mathbf{M}\mathrm{G})$.

  Conversely, given an ideal $\mathbf{S}$ in the $\csym{C}$-module category $\mathbf{M}\star X$, we define a subbimodule $\underline{\mathbf{S}}$ as the collection $(\mathbf{S}(\mathrm{F,G}))_{\mathrm{F,G}}$.
  Since $\mathbf{S}$ is an ideal in the underlying category, we find that the third, and thus also the first, axiom for a subbimodule is satisfied. To see that also the second axiom is satisfied, we again use the invertibility of $\mathbf{m}_{\mathrm{K,F}}^{-1},\mathbf{m}_{\mathrm{K,G}}$ to conclude that, since $\mathbf{M}\mathrm{K}a \in \mathbf{S}(\mathbf{M}\mathrm{K}\mathbf{M}\mathrm{F},\mathbf{M}\mathrm{K}\mathbf{M}\mathrm{G})$, we also have
  \[
\ta^{[X,X]}_{\mathrm{K;F,G}}(a) = \mathbf{m}_{\mathrm{K,G}} \circ \mathbf{M}\mathrm{K}a \circ \mathbf{m}_{\mathrm{K,F}}^{-1} \in \mathbf{S}(\mathrm{KF,KG}) =: \widehat{\mathtt{\Sigma}}(\mathbf{M}\mathrm{KF}X, \mathbf{M}\mathrm{KG}X),
  \]
 showing that $\underline{\mathbf{S}}$ indeed defines a subbimodule of $[X,X]$. Clearly, $\underline{\widehat{\mathtt{\Sigma}}} = \mathtt{\Sigma}$ and $\widehat{\underline{\mathbf{S}}} = \mathbf{S}$.
 \end{proof}

 \begin{definition}
  We say that a monoid object $A$ in a monoidal category $\csym{C}$ with a zero object $0$ is a {\it simple monoid} if and only if its only subobjects in $A\!\on{-bimod-}\!A$ are $A$ itself and $0$.
 \end{definition}

 In particular, we conclude that the natural notion of simplicity of a $\csym{C}$-module category arising from the perspective of bimodules in $\tam$ coincides with the notion introduced in \cite{MM5} and extensively studied in the field of $2$-representation theory:
 \begin{corollary}
  A $\csym{C}$-module category $\mathbf{M}$ is simple transitive if and only if any non-zero object $X$ of $\mathbf{M}$ is a cyclic generator for $\mathbf{M}$ and the obtained monoid $[X,X]$ in $\tam$ is simple.
 \end{corollary}

 \begin{corollary}\label{NaSimpleSimple}
  The pseudofunctor $\na$ induces a bijection
  \[
   \setj{\text{Simple transitive }\csym{C}\text{-module categories}}/\!\sim \; \leftrightarrow \setj{\text{Simple monoids in } \tam}/\sim_{\on{Morita}}
  \]
 \end{corollary}

 The fundamental difference between Ostrik's approach and the MMMTZ approach is not the oppositization, resulting from \cite{Os} considering presheaves of the form $\Hom{-X,X}$ and \cite{MMMT} considering copresheaves of the form $\Hom{X,-X}$, resulting in the former giving algebras and the latter coalgebras. It is rather the fact that, due to the setting of finitary categories which severely restricts representability of functors, the coalgebras obtained in \cite{MMMT} and \cite{MMMTZ1} do lie in a cocompletion, which in the setting of $2$-representation theory is known as {\it abelianization.}

 We now show that abelianizations are in fact finite (co)completions given by finite-dimensional (co)presheaves on the studied finitary monoidal category $\csym{C}$.

 The reductions of calculations of coends presented in Lemma~\ref{FinCoend} and Corollary~\ref{GeneratorReduction} can be deduced from \cite[Proposition~5.1.7]{KL}. Our results are less general, and thus admit simpler proofs, which we present below.

 \begin{lemma}\label{FinCoend}
  Let $\mathcal{C}$ be a $\Bbbk$-linear category $\mathcal{C}$ and let $\euler{F}: \mathcal{C}^{\on{op}} \kotimes \mathcal{C} \rightarrow \mathbf{Vec}_{\Bbbk}$ be a $\Bbbk$-linear functor.
  Let $V \in \mathbf{Vec}_{\Bbbk}$ and let $\setj{\sigma_{X}: \euler{F}(X,X) \rightarrow V \; | \; X \in \on{Ob}\mathcal{C}}$ be an extranatural collection.
  For any $X,Y \in \on{Ob}\mathcal{C}$, if the biproduct $X \oplus Y$ exists then, under the decomposition $\euler{F}(X \oplus Y, X \oplus Y) \simeq \euler{F}(X,X) \oplus \euler{F}(X,Y) \oplus \euler{F}(Y,X) \oplus \euler{F}(Y,Y)$, the map $\sigma_{X \oplus Y}$ is given by the matrix $\begin{pmatrix} \sigma_{X} & 0 & 0 & \sigma_{Y} \end{pmatrix}$.
 \end{lemma}

 \begin{proof}
  Since $X \oplus Y$ is the biproduct of $X$ and $Y$, we obtain split monos $\iota_{X},\iota_{Y}$ and split epis $\pi_{X}, \pi_{Y}$. Let $e_{X} = \iota_{X} \circ \pi_{X}$ and $e_{Y} = \iota_{Y} \circ \pi_{Y}$ be the resulting idempotent endomorphisms of $X \oplus Y$. Recall that $e_{X}\circ e_{Y} = 0 = e_{Y} \circ e_{X}$.

  From the extranaturality equations
\begin{multicols}{2}
\begin{enumerate}
 \item $\sigma_{X \oplus Y} \circ \euler{F}(X \oplus Y,\iota_{Y}) = \sigma_{Y} \circ \euler{F}(\iota_{Y}, Y)$
 \item $\sigma_{X \oplus Y} \circ \euler{F}(X \oplus Y,\iota_{X}) = \sigma_{X} \circ \euler{F}(\iota_{X}, Y)$
 \item $\sigma_{X \oplus Y} \circ \euler{F}(\pi_{Y}, X \oplus Y) = \sigma_{Y} \circ \euler{F}(Y, \pi_{Y})$
 \item $\sigma_{X \oplus Y} \circ \euler{F}(\pi_{X}, X \oplus Y) = \sigma_{X} \circ \euler{F}(X, \pi_{X})$
\end{enumerate}
\end{multicols}
 one easily derives the equations $\sigma_{X \oplus Y} \circ \euler{F}(X \oplus Y, e_{Y}) = \sigma_{X \oplus Y} \circ \euler{F}(e_{Y}, X \oplus Y)$ and $\sigma_{Y} = \sigma_{X \oplus Y} \circ \euler{F}(\pi_{Y}, \iota_{Y})$, and analogous equations for $X$. This yields
 \[
 \begin{aligned}
  &\sigma_{X \oplus Y} = \sigma_{X} \circ \euler{F}(\iota_{X},\pi_{X}) + \sigma_{Y} \circ \euler{F}(\iota_{Y},\pi_{Y}).
 \end{aligned}
 \]
 Observing that $\euler{F}(\iota_{X},\pi_{X}), \euler{F}(\iota_{Y},\pi_{Y})$ are the respective split epimorphisms in the decomposition of the lemma, the result follows.
 \end{proof}

 \begin{corollary}\label{GeneratorReduction}
  If $\mathcal{C}$ admits an additive generator $Z$, then any extranatural collection $\setj{\sigma_{X}: \euler{F}(X,X) \rightarrow V \; | \; X \in \mathcal{C}}$ is completely determined by $\sigma_{Z}$. As a consequence, we have
  \[
   \int^{X} \euler{F}(X,X) \simeq \on{coeq}\left(
   \begin{tikzcd}[ampersand replacement=\&]
	{\euler{F}(Z,Z) \kotimes \Hom{Z,Z}} \& {\euler{F}(Z,Z)}
	\arrow["{\euler{F}(Z,-)_{Z,Z}}", shift left=2, from=1-1, to=1-2]
	\arrow["{\euler{F}(-,Z)_{Z,Z}}"', shift right=2, from=1-1, to=1-2]
\end{tikzcd}
   \right)
  \]
 In particular, if $\mathcal{C}$ is finitary and $\euler{F}(Z,Z)$ is finite-dimensional, then the above coend is finite-dimensional.
 \end{corollary}

 \begin{proof}
  By definition, a map from the above coequalizer to a vector space $V$ is given by a map $\sigma: \euler{F}(Z,Z) \rightarrow V$ such that, for any morphism $b \in \mathcal{C}(Z,Z)$, we have $\sigma \circ \euler{F}(Z,b) = \sigma \circ \euler{F}(b,Z)$. By Lemma~\ref{FinCoend}, such a morphism uniquely determines an extranatural transformation from $\euler{F}$ to $V$, thus a map from $\int^{X}\euler{F}(X,X)$ to $V$. This extends to a natural isomorphism
  \[
   \mathbf{Vec}_{\Bbbk}\Bigg(\int^{X} \euler{F}(X,X),-\Bigg) \simeq \mathbf{Vec}_{\Bbbk}\left(\on{coeq}\left(
   \begin{tikzcd}[ampersand replacement=\&]
	{\euler{F}(Z,Z) \kotimes \Hom{Z,Z}} \& {\euler{F}(Z,Z)}
	\arrow["{\euler{F}(Z,-)_{Z,Z}}", shift left=2, from=1-1, to=1-2]
	\arrow["{\euler{F}(-,Z)_{Z,Z}}"', shift right=2, from=1-1, to=1-2]
\end{tikzcd}
   \right),-\right).
  \]
  The claimed isomorphism follows from Yoneda lemma.
 \end{proof}

 Since the Day convolution of (co)presheaves can be defined in terms of coends of functors as above, we have the following:
 \begin{corollary}
  If $\csym{Ć}$ is a finitary monoidal category, then $[\csym{C}^{\on{opp}},\mathbf{vec}_{\Bbbk}]$ is a monoidal subcategory of $[\csym{C}^{\on{opp}},\mathbf{Vec}_{\Bbbk}]$ endowed with the Day convolution monoidal structure.
 \end{corollary}

 We give a slightly simplified variant of projective abelianization of monoidal categories defined in \cite{MMMT}. Our variant produces a category monoidally equivalent to the abelianization of \cite{MMMT} - the difference between the construction therein and that specified below is that if $\csym{C}$ is {\it strict monoidal}, then our version does not guarantee that the abelianization remains strict, which is the case for the abelianization of \cite{MMMT}. Briefly put, we define the monoidal structure using colimits (biproducts), while the monoidal structure of \cite{MMMT} remembers the diagram for the colimit.

 Since $\csym{C}$ is additive monoidal $\Bbbk$-linear, the category $[\mathbb{N}_{0},\csym{C}]$ of sequences
 \[
  \cdots \rightarrow \mathrm{P}_{2} \rightarrow \mathrm{P}_{1} \rightarrow \mathrm{P}_{0}
 \]
 in $\csym{C}$ is monoidal under the  tensor product $\overline{\boxtimes}$ defined analogously to the tensor product of chain complexes.
 The inclusion of posets $\setj{0,1} \hookrightarrow \mathbb{N}_{0}$ gives a {\it truncation functor} $\mathrm{T}: [\mathbb{N}, \csym{C}] \rightarrow [\setj{0,1}, \csym{C}] = \csym{C}^{\rightarrow}$ to the arrow category on $\csym{C}$.
 Further, we have the inclusion functor $\mathrm{I}: \csym{C}^{\rightarrow} \hookrightarrow [\mathbb{N}_{0},\csym{C}]$ of sequences concentrated on $\setj{0,1}$ (continuing the sequence by zero for indices greater than $1$). In fact, $\mathrm{T}$ is a retraction, i.e. we have $\mathrm{T} \circ \mathrm{I} = \mathbb{1}_{\ccf{C}^{\rightarrow}}$.

 Defining $\boxtimes = \mathrm{T} \circ \overline{\boxtimes} \circ \mathrm{I} \kotimes \mathrm{I}$, we obtain a monoidal structure on $\csym{C}^{\rightarrow}$. Explicitly, on objects we have
 \[
  \big(\mathrm{P}_{1} \xrightarrow{\mathrm{p}} \mathrm{P}_{0}\big) \boxtimes \big(\mathrm{Q}_{1} \xrightarrow{\mathrm{q}} \mathrm{Q}_{0}\big) = \mathrm{P}_{1} \cotimes \mathrm{Q}_{0} \oplus \mathrm{P}_{0} \cotimes \mathrm{Q}_{1} \xrightarrow{(\begin{smallmatrix} \mathrm{p} \otimes \mathrm{Q}_{0} & \mathrm{P}_{0}\otimes \mathrm{q} \end{smallmatrix})} \mathrm{P}_{0} \cotimes \mathrm{Q}_{0},
 \]
 and similarly on morphisms. The ideal $\mathcal{H}$ in $\csym{C}^{\rightarrow}$ given by nullhomotopic maps, i.e. morphisms given by pairs $(\tau_{1},\tau_{0}): (\mathrm{P}_{1},\mathrm{P}_{0}) \rightarrow (\mathrm{P}_{1}',\mathrm{P}_{0}')$ such that there is a morphism $\mathrm{h}: \mathrm{P}_{0} \rightarrow \mathrm{P}_{1}'$ yielding $\mathrm{p}' \circ \mathrm{h} = \tau_{0}$, is a tensor ideal. Indeed, given such a homotopy, the lower triangle in the diagram
 \[\begin{tikzcd}[ampersand replacement=\&]
	{\mathrm{P}_{1} \cotimes \mathrm{Q}_{0} \oplus \mathrm{P}_{0} \cotimes \mathrm{Q}_{1}} \&\& {\mathrm{P}_{0} \cotimes \mathrm{Q}_{0}} \\
	{\mathrm{P}_{1}' \cotimes \mathrm{Q}_{0} \oplus \mathrm{P}_{0}' \cotimes \mathrm{Q}_{1}} \&\& {\mathrm{P}_{0}' \cotimes \mathrm{Q}_{0}}
	\arrow["{\tau_{0} \otimes \mathrm{Q}_{0}}", from=1-3, to=2-3]
	\arrow["{\left(\begin{smallmatrix}\tau_{1}\otimes \mathrm{Q}_{0} & 0 \\0 & \tau_{0}\otimes \mathrm{Q}_{1}\end{smallmatrix}\right)}"', from=1-1, to=2-1]
	\arrow["{\left(\begin{smallmatrix} \mathrm{p}\otimes \mathrm{Q}_{0} & \mathrm{P}_{0}\otimes \mathrm{q} \end{smallmatrix}\right)}", from=1-1, to=1-3]
	\arrow["{\left(\begin{smallmatrix} \mathrm{p}\otimes \mathrm{Q}_{0} & \mathrm{P}_{0}\otimes \mathrm{q} \end{smallmatrix}\right)}", from=2-1, to=2-3]
	\arrow["{\left(\begin{smallmatrix} \mathrm{h} \otimes \mathrm{Q}_{0} \\ 0 \end{smallmatrix}\right)}"{description}, from=1-3, to=2-1]
\end{tikzcd}\]
commutes, showing that $\tau \otimes \mathrm{Q}$ is nullhomotopic. Similarly one can show that $\mathrm{Q} \otimes \tau$ is nullhomotopic.

The {\it projective abelianization} $\overline{\csym{C}}$ of $\csym{C}$ is defined as the monoidal category $\csym{C}^{\rightarrow}/\mathcal{H}$.

Given a finite-dimensional algebra $A$, we have the equivalences
\[\begin{tikzcd}[ampersand replacement=\&]
	{[A \!\on{--proj}^{\on{op}},\mathbf{vec}_{\Bbbk}]} \&\& {A\!\on{--mod}} \&\& {(A\!\on{--proj})^{\rightarrow}/\mathcal{H}}
	\arrow["{\on{Coker}}", from=1-5, to=1-3, shift left=1]
	\arrow["{\euler{P}}", from=1-3, to=1-5, shift left=1]
	\arrow["{\on{ev}_{A}}", shift left=1, from=1-1, to=1-3]
	\arrow["{M\mapsto\on{Hom}(-,M)}", shift left=1, from=1-3, to=1-1]
\end{tikzcd}\]
where, for $\euler{F}: A\!\on{--proj}^{\on{op}} \rightarrow \mathbf{vec}_{\Bbbk}$, the $A$-module structure on $\euler{F}(A)$ is given by
\[
A \simeq \on{End}_{A\!\on{--proj}}(A)^{\on{op}} \xrightarrow{\euler{F}_{A,A}} \Hom{\euler{F}(A),\euler{F}(A)},
\]
and $\euler{P}$ sends a module $M$ to the homotopy class of diagrams giving projective presentations of $M$.

This extends to the case of a finitary monoidal category $\csym{C}$, where we obtain an equivalence $\csym{C}/\mathcal{H} \xrightarrow{\on{Coker}} [\csym{C}^{\on{opp}},\mathbf{vec}_{\Bbbk}]$. We now show that the diagram
\begin{equation}\label{MonoidalCokernel}
\begin{tikzcd}[row sep=scriptsize]
	{\csym{C}/\mathcal{H} \kotimes \csym{C}/\mathcal{H}} && {[\csym{C}^{\on{opp}},\mathbf{vec}_{\Bbbk}] \kotimes [\csym{C}^{\on{opp}},\mathbf{vec}_{\Bbbk}]} \\
	{\csym{C}/\mathcal{H}} && {[\csym{C}^{\on{opp}},\mathbf{vec}_{\Bbbk}]}
	\arrow["\boxtimes", from=1-1, to=2-1]
	\arrow["{\on{Coker} \kotimes \on{Coker}}", from=1-1, to=1-3]
	\arrow["{\circledast}", from=1-3, to=2-3]
	\arrow["{\on{Coker}}", from=2-1, to=2-3]
\end{tikzcd}
\end{equation}
commutes up to natural isomorphism. Observe that this does not immediately show that $\on{Coker}$ is a monoidal equivalence, since we do not claim our natural isomorphism to be coherent. However, it does show that $\boxtimes$ is right exact in both variables. Thus, by universal property of Day convolution (\cite[Theorem~5.1]{IK}), the functor extending the monoidal embedding
\[
\begin{aligned}
 \csym{C} \rightarrow \csym{C}^{\rightarrow}/\mathcal{H} \\
 \mathrm{F} \mapsto (0 \rightarrow \mathrm{F})
\end{aligned}
\]
to a functor from $[\csym{C}^{\on{opp}},\mathbf{vec}_{\Bbbk}]$ is strong monoidal. In the case $\csym{C} = A\!\on{--proj}$, it is easy to verify that the latter functor is precisely $\mathrm{P}$ - thus, in the general case it is a quasi-inverse to $\on{Coker}$, which shows that $\on{Coker}$ is indeed a monoidal equivalence.

\begin{lemma}
 Diagram~\eqref{MonoidalCokernel} commutes up to natural isomorphism.
\end{lemma}

\begin{proof}
 For $\mathrm{P = P_{1} \xrightarrow{p} P_{0}}$ and $\mathrm{Q = Q_{1} \xrightarrow{q} Q_{0}}$ in $\csym{C}^{\rightarrow}/\mathcal{H}$, we have
 \[
 \begin{aligned}
  &(\on{Coker}(\Hom{-,\mathrm{p}}) \circledast \on{Coker}(\Hom{-,\mathrm{q}}))(-) = \int^{\mathrm{H,K}}  \on{Coker}(\Hom{\mathrm{H},\mathrm{p}}) \kotimes \on{Coker}(\Hom{\mathrm{K,q}}) \kotimes \Hom{-,\mathrm{H} \cotimes \mathrm{K}} \\
  &\simeq \varinjlim_{\mathrm{p,q}} \left( \int^{\mathrm{H,K}} \Hom{\mathrm{H,p}} \kotimes \Hom{\mathrm{K,q}} \kotimes \Hom{-,\mathrm{H} \cotimes \mathrm{K}} \right) \simeq \varinjlim_{\mathrm{p,q}} \Hom{-,\mathrm{p} \cotimes \mathrm{q}}
 \end{aligned}
 \]
 which is precisely the colimit of the product diagram
 \[\begin{tikzcd}[row sep=scriptsize]
	& {\mathrm{P_{0} \cotimes Q_{1}}} \\
	{\mathrm{P_{1} \cotimes Q_{1}}} && {\mathrm{P_{0} \cotimes Q_{0}}} \\
	& {\mathrm{P_{1} \cotimes Q_{0}}}
	\arrow["{\mathrm{p \otimes Q_{1}}}"'{pos=0.6}, from=2-1, to=1-2]
	\arrow["{\mathrm{P_{1}\otimes q}}"{pos=0.6}, from=2-1, to=3-2]
	\arrow["{\mathrm{P_{0}\otimes q}}"', from=1-2, to=2-3]
	\arrow["{\mathrm{p \otimes Q_{0}}}", from=3-2, to=2-3]
\end{tikzcd}\]
which is isomorphic to $\on{Coker}\left( \mathrm{P_{1} \cotimes Q_{0} \oplus P_{0} \cotimes Q_{1}} \xrightarrow{\left(\begin{smallmatrix} \mathrm{p \otimes Q_{0}} & \mathrm{P_{0} \otimes q} \end{smallmatrix} \right)} \mathrm{P_{0} \cotimes Q_{0}} \right) = \on{Coker}(\mathrm{P \boxtimes Q})$. Finally, note that this isomorphism is natural in $\mathrm{P,Q}$, being obtained from universal properties of colimits.
\end{proof}

Instead of the presheaf category $[\csym{C}^{\on{opp}}, \mathbf{vec}_{\Bbbk}]$, we could also use its Isbell dual, $[\csym{C},\mathbf{vec}_{\Bbbk}]^{\on{op}}$. The category $\csym{C}$ embeds in it as the injective objects. The corresponding diagrammatic construction, expressed in terms of kernels and injective coresolutions, is defined in \cite{MMMT} as the {\it injective abelianization} $\underline{\csym{C}}$ of $\csym{C}$.
\begin{corollary}\label{Abelianizations}
 We have monoidal equivalences $\overline{\csym{C}} \xrightarrow{\on{Coker}} [\csym{C}^{\on{opp}},\mathbf{vec}_{\Bbbk}]$ and $\underline{\csym{C}} \xrightarrow{\on{Ker}} [\csym{C},\mathbf{vec}_{\Bbbk}]^{\on{op}}$.
\end{corollary}

The fundamental consequence of \cite[Theorem~4.7]{MMMT} and \cite[Theorem~5.1]{MMMT} is the following:
\begin{theorem}{[\cite{MMMT}]}\label{MMMTMain}
 Assume $\csym{C}$ is rigid. Let $\mathbf{M,N}$ be finitary, transitive $\csym{C}$-module categories. For any object $X \in \mathbf{M}$, the functor $\Hom{X,-X}: \csym{C} \rightarrow \mathbf{vec}_{\Bbbk}$ gives rise to a coalgebra $A^{X}$ in the injective abelianization $\underline{\csym{C}}$. Given $Y \in \mathbf{M}$, the coalgebras $A^{X},A^{Y}$ are Morita equivalent if and only if $\mathbf{M,N}$ are equivalent.
\end{theorem}

In \cite{MMMT}, the object $A^{X}$ is obtained by observing that $\underline{\csym{C}}$ is {\it finite abelian} (i.e. equivalent to $A\!\on{--mod}$ for a finite-dimensional algebra $A$) and that $\Hom{X,-X}$ is left exact, thus representable, see e.g. \cite[Lemma~2.1]{FSS}. In view of the equivalence given by Corollary~\ref{Abelianizations}, we may identify $A^{X}$ with the functor $\Hom{X,-X}$ directly. Similarly to the case of Corollary~\ref{OstrikMonoidStructure}, one may verify that the coalgebra structure on $A^{X}$ defined in \cite{MMMT} coincides with the (oppositization of) monoid structure on $\Hom{X,-X}$ given in Definition~\ref{AbstractMonoid}.

\begin{corollary}\label{ReproveMMMT}
 Theorem~\ref{MMMTMain} follows from Theorem~\ref{MainConsequence} combined with Proposition~\ref{RigidMonoids}.
\end{corollary}

\begin{proof}
 The coalgebra structures $A^{X},A^{Y}$ correspond to the monoid structures on $\Hom{X,-X}, \Hom{Y,-Y}$; in particular, the coalgebras $A^{X},A^{Y}$ are Morita equivalent if and only if $\Hom{X,-X}, \Hom{Y,-Y}$ are Morita equivalent. Since $\csym{C}$ is rigid, by Proposition~\ref{RigidMonoids} we find that $\Hom{X,-X},\Hom{Y,-Y}$ are Morita equivalent if and only if the monoids $[X,X], [Y,Y]$ in $\tam$ are Morita equivalent, which, by Theorem~\ref{MainConsequence}, is equivalent to $\mathbf{M} \simeq \mathbf{N}$.
\end{proof}

In fact, we do not need the assumption of Theorem~\ref{MMMTMain} that $\mathbf{M,N}$ are transitive, and instead only need $\mathbf{M,N}$ to be cyclic. However, this is also true for the arguments used in \cite{MMMT}.

\subsection{Ostrik algebras and MMMTZ coalgebras are not complete Morita invariants}\label{s104}

We now give an example of a monoidal category $\csym{C}$ together with a family of mutually non-equivalent $\csym{C}$-module categories which yield the same Ostrik algebras and the same MMMTZ coalgebras, proving that in the non-rigid case, these invariants are not complete. We also show that the resulting monoids of the form $[X,X]$, as defined in Proposition~\ref{EndMonoid}, are not Morita equivalent, thus explicitly exemplifying Theorem~\ref{MainConsequence}.

We do so first in the setting of categories enriched in the monoidal poset $\mathbb{2}=(\setj{0,1},\land)$. We may interpret $0,1$ as truth values and $\land$ as the logical conjunction. We identify $\mathbb{2}$ with its corresponding posetal monoidal category. This category is in fact Cartesian monoidal, as $\land$ is the categorical product in $\mathbb{2}$. Further $\mathbb{2}$ is complete and cocomplete: coproducts are given by disjunctions, equalizers and coequalizers are trivially the domain and codomain of the morphisms, since the category is posetal.

Recall that a $\mathbb{2}$-enriched category $\mathcal{S}$ is the same as a preorder $S$. Given elements $s,t \in S$, we have
\[
 \on{Hom}_{\mathcal{S}}(s,t) =
 \begin{cases}
  1 \text{ if } s \leq_{s} t \\
  0 \text{ otherwise.}
 \end{cases}
\]
Similarly, a $\mathbb{2}$-enriched functor is the same as a preorder morphism, and a $\mathbb{2}$-transformation in $\mathbb{2}\!\on{-Cat}(\mathcal{S,T})(\euler{F,G})$ exists if only if $\euler{F}(s) \leq \euler{G}(s)$ for all $s \in \mathcal{S}$. In that case, it exists uniquely.

For $s,s' \in \mathcal{S}$, we write $s \sim s'$ if we have $s \leq s'$ and $s' \leq s$. This corresponds to isomorphism of objects of $\mathcal{S}$. A monoidal $\mathbb{2}$-enriched category is the same as a weakly monoidal preorder, i.e. a preorder $C$ together with a distinguished element $1_{C} \in C$ and a binary operation $- \cdot_{C} -: C \times C \rightarrow C$ which is monotone in both variables, and satisfies $a \cdot (b \cdot c) \sim (a \cdot b) \cdot c$ and $1_{C} \cdot b \sim b \sim b \cdot 1_{C}$.

Given a $\mathbb{2}$-enriched monoidal category $\csym{C}$, a $\mathbb{2}$-enriched $\csym{C}$-module category $\mathbf{S}$ consists of a preorder $S$ together with a function
$-\diamond -: C \times S \rightarrow S$ such that
 \begin{enumerate}
  \item if $c \leq c'$ and $s \leq s'$ then $c \diamond s \leq c' \diamond s'$;
  \item for all $s \in S$, we have $1_{C} \diamond s \sim s$;
  \item for all $c,c' \in C$ and $s \in S$, we have $c' \diamond (c \diamond s) \sim (c' \cdot c) \diamond s$.
 \end{enumerate}

 A morphism $\euler{\Gamma}: \mathbf{S} \rightarrow \mathbf{T}$ of $\csym{C}$-module $\mathbb{2}$-enriched categories is given by a function $\gamma: S \rightarrow T$ such that, for all $c \in C$ and all $s \in S$, we have $c \diamond \gamma(s) \sim \gamma(c \diamond s)$.

 Following \cite[Proposition~2.1]{Ros}, every $\mathbb{2}$-enriched category is $\mathbb{2}$-Cauchy complete, thus we say that a $\csym{C}$-module category $\mathbf{S}$ is cyclic if there is $s \in S$ such that for any $s' \in S$, there is $c \in C$ satisfying $c \diamond_{\mathbf{S}} s \sim s'$.

 A $\mathbb{2}$-enriched copresheaf $\phi: \mathcal{C} \rightarrow \mathbb{2}$ is a poset morphism from $\mathcal{C}$ to $\mathbb{2}$, which is uniquely determined by the downward closed subset $\phi^{-1}(\setj{0})$ of $C$; equivalently, by the upward closed subset $\phi^{-1}(\setj{1})$ of $C$. Thus, a $\mathbb{2}$-enriched profunctor $\mathcal{C} \xslashedrightarrow{} \mathcal{D}$ is a subset $P$ of $D \times C$ such that if $(d,c) \in P$ and $d' \leq_{D} d$ and $c \leq_{C} c'$, then $(d',c') \in P$.

 Given a monoidal $\mathbb{2}$-enriched category $\csym{C}$ and $\mathbb{2}$-enriched $\csym{C}$-module categories $\mathcal{S,T}$, a Tambara module $\mathcal{S} \xslashedrightarrow{} \mathcal{T}$ is a subset $P \subseteq T \times S$ which is a profunctor, and such that, for any $c \in C, s \in S$ and $t \in T$, if $(t,s) \in P$, then $(c\diamond_{\mathcal{T}} t, c\diamond_{\mathcal{S}} s) \in P$.

 The monoidal poset $(\mathbb{Z}_{\geq 0}, +)$ is a strict monoidal $\mathbb{2}$-enriched category, and, for every $k \in \mathbb{Z}_{\geq 0}$, the subposet $\setj{0,1,\ldots, k}$ of $\mathbb{Z}_{\geq 0}$ yields a strict $\mathbb{2}$-enriched $\mathbb{Z}_{\geq 0}$-module category $\mathbb{Z}_{0,k}$, by setting, for $c \in \mathbb{Z}_{\geq 0}$ and $s \in \setj{0,1,\ldots k}$:
 \[
  c \diamond s := \min\setj{c+s, k}.
 \]
 For every $k \in \mathbb{Z}_{\geq 0}$, the $\mathbb{Z}_{\geq 0}$-module category $\mathbb{Z}_{0,k}$ is cyclic, with the unique cyclic generator given by $0 \in \setj{0,1,\ldots,k}$.

 An equivalence of strict $\mathbb{2}$-enriched module categories is in particular a poset isomorphism (thus, it is automatically strict), hence, a bijection. Thus, for $k \neq m$ we have $\mathbb{Z}_{0,k} \not\simeq \mathbb{Z}_{0,m}$.

 However, for any $k \in \mathbb{Z}_{\geq 0}$, the $\mathbb{2}$-enriched presheaf $\Hom{-0,0}_{\mathbb{Z}_{0,k}}: \mathbb{Z}_{\geq 0}^{\on{op}} \rightarrow \mathbb{2}$ used to define the corresponding Ostrik algebra, is determined by $\Hom{-0,0}_{\mathbb{Z}_{0,k}}^{-1}(\setj{1}) = \setj{0}$. Thus, the Ostrik algebra does not depend on $k$. Similarly, the $\mathbb{2}$-enriched copresheaf $\Hom{0,-0}_{\mathbb{Z}_{0,k}}$ is determined by $\Hom{0,-0}_{\mathbb{Z}_{0,k}}^{-1}(\setj{0}) = \varnothing$. This shows that, in this setting, neither Ostrik algebras nor MMMTZ coalgebras give a complete invariant.

  \begin{proposition}
  The monoids $[0,0]_{0}$ and $[0,0]_{1}$ in $\mathbb{Z}_{\geq0}\!\on{-Tamb}(\mathbb{Z}_{\geq0},\mathbb{Z}_{\geq0})$ are not Morita equivalent.
 \end{proposition}

 \begin{proof}
  Given monoid objects $A,B$ in a tame monoidal category $\csym{D}$, a Morita equivalence $A \xrightarrow{{}_{B}M_{A}} B$ induces a monoidal equivalence $\on{Bimod}(\csym{D})(A,A) \simeq \on{Bimod}(\csym{D})(B,B)$ mapping ${}_{A}A_{A}$ to ${}_{B}(M \otimes_{A} M^{-1})_{B} \simeq {}_{B}B_{B}$. In particular, it induces an isomorphism of posets between the poset of subobjects of $A$ in $\on{Bimod}(\csym{D})(A,A)$ and the poset of subobjects of $B$ in $\on{Bimod}(\csym{D})(B,B)$. In particular, $A$ is simple if and only if $B$ is.

  Thus, it suffices to show that $[0,0]_{0}$ is a simple monoid, while $[0,0]_{1}$ is not.
  First, we show that $[0,0]_{0}$ is a simple monoid. For any $k,m \in \mathbb{Z}_{\geq 0}$, we have
  \[
  k\diamond 0 = \min\setj{k+ 0,0} = \min\setj{m+0,0} = m \diamond 0,
  \]
    thus $k \diamond 0 \leq m \diamond 0$, showing that $[0,0]_{0}(k,m) = \Hom{k\diamond 0, m \diamond 0}_{\mathbb{Z}_{0,0}} = 1$. Let $\Psi$ be an ideal in $[0,0]_{0}$. If $\Psi \neq [0,0]_{0}$, then there are $k,m \in \mathbb{Z}_{\geq 0}$ such that $\Psi(k,m) = 0$.
   The left action map
  \[
   \int^{l \in \mathbb{Z}_{\geq 0}} \Hom{k\diamond 0,l\diamond 0} \otimes \Psi(l,m) = \bigvee_{l} \Hom{k\diamond 0,l\diamond 0} \land \Psi(l,m) \rightarrow \Psi(k,m) = 0
  \]
  being well-defined shows that $\Hom{k\diamond 0,l\diamond 0} \land \Psi(l,m) = 0$ for all $l$. At the same time, $\Hom{k \diamond 0, l \diamond 0} = 1$. We conclude that $\Psi(l,m) = 0$ for all $l$. Similarly to the preceding argument, for any $j \in \mathbb{Z}_{\geq 0}$, the right action map
  \[
   \int^{l} \Psi(j,l) \otimes \Hom{l\diamond 0, n \diamond 0} = \bigvee_{l} \Psi(j,l) \land \Hom{l\diamond 0, n \diamond 0} \rightarrow \Psi(j,m) = 0
  \]
  being well-defined shows that $\Psi(j,l) = 0$ for all $l \in \mathbb{Z}_{\geq 0}$. Thus, $\Psi = 0$.

  It now suffices to show that $[0,0]_{1}$ admits a proper ideal. We claim that the Tambara module $\mathtt{\Sigma}_{\geq 1}$ given by
  \[
\mathtt{\Sigma}_{\geq 1}(k,m) = \Hom{(k+1)\diamond 0, m \diamond 0},
  \]
  gives such an ideal. By definition, $\mathtt{\Sigma}_{\geq 1}(k,m) = 1$ if and only if
  \[
   \min\setj{k+1, 1} \leq \min\setj{m,1},
  \]
  in which case we also have
  \[
   \min\setj{k,1} \leq \min\setj{m,1},
  \]
  showing that also $[0,0]_{1}(k,m) = 1$. This shows that $\mathtt{\Sigma}_{\geq 1}$ gives a Tambara submodule of $[0,0]_{1}$. Since the category $\mathbb{Z}_{\geq0}\!\on{-Tamb}(\mathbb{Z}_{\geq0},\mathbb{Z}_{\geq0})$ is posetal, it suffices to show that there are maps $[0,0]_{1} \circ \Sigma_{\geq 1} \rightarrow \Sigma_{\geq 1}$ and $\Sigma_{\geq 1} \circ [0,0]_{1} \rightarrow \Sigma_{\geq 1}$, in order to conclude that $\Sigma_{\geq 1}$ is an ideal in $[0,0]_{1}$. And indeed, for any $k,m \in \mathbb{Z}_{\geq 0}$, there is a map
  \[
   \int^{l} \Hom{k \diamond 0, l \diamond 0} \otimes \Hom{(l+1)\diamond 0, m \diamond 0} \rightarrow \Hom{(k+1)\diamond 0, m \diamond 0},
  \]
  since if, for some $l \in \mathbb{Z}_{\geq 0}$, we have $k \leq l$ and $l+1 \leq m$, we also have $k+1 \leq m$. Similarly, for any $k,m$, there is a map
    \[
   \int^{l} \Hom{(k+1) \diamond 0, l \diamond 0} \otimes \Hom{l \diamond 0, m \diamond 0} \rightarrow \Hom{(k+1)\diamond 0, m \diamond 0},
  \]
  since if, for some $l \in \mathbb{Z}_{\geq 0}$, we have $k+1 \leq l$ and $l \leq m$, we also have $k+1 \leq m$.
 \end{proof}

 One can show that $\Sigma_{\geq 1}$ is the only ideal in $[0,0]_{1}$, and for any $n$ one may consider analogously defined ideals $\setj{\Sigma_{\geq k} \subseteq [0,0]_{n} \; | \; 0 \leq  k \leq n}$.

 Further, $\mathbb{Z}_{0,n}$ may itself be viewed as a monoidal $\mathbb{2}$-enriched category, in fact it is the quotient of $\mathbb{Z}_{\geq 0}$ by the congruence whose equivalence classes are $\setj{0},\setj{1},\ldots,\setj{n-1},\setj{n,n+1,\ldots}$.
 Linearizing over $\Bbbk$ (and passing to Cauchy completions) the underlying structures for the above $\mathbb{2}$-enriched example yields a $\Bbbk$-linear category $\Bbbk\mathbb{Z}_{\geq 0}$ such that
 \[
 \on{Hom}_{\Bbbk\mathbb{Z}_{\geq 0}}(k,m) =
 \begin{cases}
  \Bbbk \text{ if } k \leq m \\
  0 \text{ otherwise.}
 \end{cases}
 \]
 We of course still have $k \otimes_{\Bbbk\mathbb{Z}_{\geq 0}} m = k+m$. Also in this case we obtain non-equivalent $\Bbbk\mathbb{Z}_{\geq 0}$-module categories $\Bbbk\mathbb{Z}_{0,0}, \Bbbk\mathbb{Z}_{0,1}$ which yield equal Ostrik algebras and MMMMTZ coalgebras, and similarly to above one can explicitly show that the associated monoids in $\Bbbk\mathbb{Z}_{\geq0}\!\on{-Tamb}(\Bbbk\mathbb{Z}_{\geq0},\Bbbk\mathbb{Z}_{\geq0})$ are not Morita equivalent. This still holds if we view $\Bbbk\mathbb{Z}_{0,0}, \Bbbk\mathbb{Z}_{0,1}$ as $\Bbbk\mathbb{Z}_{0,n}$-module categories, for any $n \geq 1$. In particular, the (monoidal) category $\Bbbk\mathbb{Z}_{0,n}$ is finitary, so we obtain an example of incompleteness of Ostrik algebras and MMMTZ coalgebras in the finitary setting.

\section{Module categories as categories of modules}\label{s11}

The results of \cite{Os} and \cite{MMMTZ1} do not only show that two $\csym{C}$-module categories are equivalent if and only if their associated (co)algebra objects in (the finite completion of) $\csym{C}$ are Morita equivalent. In fact, they show that a cyclic $\csym{C}$-module category $\mathbf{M}$ with a cyclic generator $X$ is equivalent as a $\csym{C}$-module category to the category $\on{mod-}\setj{X,X}$.

In order to reconstruct $\mathbf{M}$ as the category of {\it all} $\setj{X,X}$-modules, one needs the finiteness and cocompleteness assumptions of \cite{MMMTZ1}, \cite{Os}. In order to state a general result in our setting, we instead use the notion of {\it free} $\setj{X,X}$-modules. Let $\csym{C}$ be a monoidal category, $\mathrm{A}$ a monoid object in $\csym{C}$ with multiplication map $\mathrm{m}_{\mathrm{A}}$. Let $\on{mod}_{\ccf{C}}\!\on{--}\!\mathrm{A}$ denote the category of right $\mathrm{A}$-modules. For any object $\mathrm{G}$, the object $\mathrm{G} \cotimes \mathrm{A}$ together with the morphism $\mathrm{ra}_{\mathrm{G\cotimes A}}= \mathrm{G} \otimes \mathrm{m}_{\mathrm{A}}$, is a right $\mathrm{A}$-module.
We say that a module $\mathrm{M} \in \on{mod}_{\ccf{C}}\!\on{--}\!\mathrm{A}$ is {\it free} if there is an isomorphism of right $\mathrm{A}$-modules $\mathrm{M} \simeq \mathrm{G} \cotimes \mathrm{A}$ for some $\mathrm{G} \in \csym{C}$. We denote the full subcategory of $\on{mod}_{\ccf{C}}\!\on{--}\!\mathrm{A}$ given by free modules by $\on{frmod}_{\ccf{C}}\mathrm{A}$. It is canonically equivalent to the Kleisli category of the monad given by the endofunctor $-\cotimes A$ of $\csym{C}$.
Using this observation, we find the canonical isomorphism
\begin{equation}\label{KleisliHoms}
 \on{Hom}_{\ccf{C}}(\mathrm{G},-)\xiso \on{Hom}_{\on{mod}_{\ccf{C}}\!\on{--}\!\mathrm{A}}(\mathrm{G}\cotimes \mathrm{A},-),
\end{equation}
sending a morphism $\mathrm{f} \in \csym{C}(\mathrm{G,N})$ to the morphism $\mathrm{ra}_{\mathrm{N}} \circ (\mathrm{f \otimes A})$.

Recall from Section~\ref{s6} that a strong monoidal functor $\mathbb{F}: \csym{C} \rightarrow \csym{D}$ induces a pseudofunctor
\[
\on{Bimod}(\mathbb{F}): \on{Bimod}(\csym{C}) \rightarrow \on{Bimod}(\csym{D}).
\]

Similarly, for any monoid object $\mathrm{A} \in \csym{C}$, we find a functor $\mathbb{F}_{\ast,A}: \on{mod}_{\ccf{C}}\!\on{--}\!\!\mathrm{A} \rightarrow \on{mod}_{\ccf{D}}\!\on{--}\!\mathbb{F}(\mathrm{A})$, sending an $\mathrm{A}$-module object $\mathrm{M}$ to the $\mathbb{F}(\mathrm{A})$-module object $\mathbb{F}(\mathrm{M})$.
Moreover, the category $\on{mod}_{\ccf{C}}\!\on{--}\!\mathrm{A}$ is a $\csym{C}$-module subcategory of ${}_{\ccf{C}}\csym{C}$ - given an object $\mathrm{G} \in \csym{C}$ and a right $\mathrm{A}$-module $\mathrm{M}$ with structure map $\mathrm{M \cotimes A} \xrightarrow{\mathrm{ra}_{\mathrm{M}}} \mathrm{M}$, the object $\mathrm{G}\cotimes \mathrm{M}$ is a right $\mathrm{A}$-module with structure map $\mathrm{G} \otimes \mathrm{ra}_{\mathrm{M}}$.
Similarly, $\on{mod}_{\ccf{D}}\!\on{--}\!\mathbb{F}(\mathrm{A})$ is a $\csym{D}$-module subcategory of ${}_{\ccf{D}}\csym{D}$ and thus, after restricting the action along $\mathbb{F}$, also a $\csym{C}$-module subcategory of ${}_{\ccf{C}}\csym{D}$.
The functor $\mathbb{F}_{\ast,\mathrm{A}}$ is a $\csym{C}$-module functor with respect to the above $\csym{C}$-module category structures.
Clearly, if $\mathbb{F}$ is fully faithful, so is $\mathbb{F}_{\ast, \mathrm{A}}$ for any $\mathrm{A}$. Further, $\mathbb{F}_{\ast, \mathrm{A}}$ maps $\on{frmod}_{\ccf{C}}\mathrm{A}$ to $\on{frmod}_{\ccf{D}}\mathbb{F}(\mathrm{A})$.

In particular, the Yoneda embedding $\yo_{\!\ccf{C}}$ induces the pseudofunctor $\on{Bimod}(\csym{C}) \rightarrow \on{Bimod}([\csym{C}^{\on{opp}}, \mathbf{Vec}_{\Bbbk}])$. For a monoid $\mathrm{A} \in \csym{C}$, we obtain the monoid $\yo_{\ccf{C}}(\mathrm{A}) = \csym{C}(-,\mathrm{A})$ in $[\csym{C}^{\on{opp}},\mathbf{Vec}_{\Bbbk}]$ and a full and faithful $\csym{C}$-module functor $(\yo_{\ccf{C}})_{\ast,\mathrm{A}}:\on{mod}_{\ccf{C}}\!\on{--}\!\mathrm{A} \rightarrow \on{mod}_{[\ccf{C},\mathbf{Vec}_{\Bbbk}]}\!\on{--}\!\csym{C}(-,\mathrm{A})$. It restricts to a full and faithful $\csym{C}$-module functor from $\on{frmod}_{\ccf{C}}\mathrm{A}$ to $\on{frmod}_{[\ccf{C},\mathbf{Vec}_{\Bbbk}]}\csym{C}(-,\mathrm{A})$, whose essential image is given by presheaves isomorphic to those of the form
\[
\csym{C}(-,\mathrm{G}) \circledast \csym{C}(-,\mathrm{A}) \simeq \csym{C}(-,\mathrm{GA}).
\]
\begin{definition}
 Given a monoid object $\euler{P}$ in $[\csym{C}^{\on{opp}},\mathbf{Vec}_{\Bbbk}]$, let $\cplus{\euler{P}}$ be the full subcategory of $\on{frmod}_{[\ccf{C}^{\on{opp}},\mathbf{Vec}_{\Bbbk}]}\euler{P}$ given by right $\euler{P}$-module objects isomorphic to those of the form $\csym{C}(-,\mathrm{G}) \circledast \euler{P}$, for $\mathrm{G} \in \csym{C}$.
\end{definition}

Using this definition, we reformulate our earlier observations:
\begin{lemma}\label{Yo-YoDa}
 $(\yo_{\ccf{C}})_{\ast,\mathrm{A}}$ restricts to an equivalence of $\csym{C}$-module categories $\on{frmod}_{\ccf{C}}\mathrm{A} \xiso \cplus{\csym{C}(-,\mathrm{A})}$.
\end{lemma}

Further, observe that the category $\cplus{\euler{P}}$ is defined similarly to the category $[X,X]_{+}^{\ccf{C}}$ introduced in Definition~\ref{CTPlus}.
Recall that $[X,X]_{+}^{\ccf{C}}$ is given by objects of $[X,X]\!\on{--mod}_{\ccf{C}\!\on{-Tamb}(\ccf{C},\ccf{C})}$ isomorphic to those of the form $[X,X] \diamond \csym{C}(-,-\mathrm{G})$ for $\mathrm{G} \in \csym{C}$.

\begin{lemma}\label{AuxiliaryLemma}
 Let $\mathrm{A,B}$ be monoid objects in  $\csym{C}$. If there is a monoid isomorphism $\mathrm{A} \xrightarrow[\sim]{\mathrm{f}} \mathrm{B}$, there is an isomorphism of $\csym{C}$-module categories $\mathrm{f}^{\ast}: \on{mod}_{\ccf{C}}\!\on{--}\!\mathrm{B} \xiso \on{mod}_{\ccf{C}}\!\on{--}\!\mathrm{A}$.
\end{lemma}
\begin{proof}
 We let $\mathrm{f}^{\ast}(\mathrm{M}) = \mathrm{M}$, with $\mathrm{A}$-module structure $\mathrm{M \cotimes A} \xrightarrow{\mathrm{M} \otimes \mathrm{f}} \mathrm{M \cotimes B} \xrightarrow{\mathrm{ra}_{\mathrm{M}}} \mathrm{M}$. A $\mathrm{B}$-module morphism $\mathrm{t}:\mathrm{M} \rightarrow \mathrm{N}$ becomes an $\mathrm{A}$-module morphism from $\mathrm{f}^{\ast}(\mathrm{M})$ to $\mathrm{f}^{\ast}(\mathrm{N})$. Recall that given an object $\mathrm{F} \in \csym{C}$, the right $\mathrm{A}$-module structure on $\mathrm{F \cotimes M}$ is given by $\mathrm{F} \otimes \mathrm{ra}_{\mathrm{M}}$. To see that $\mathrm{f}^{\ast}$ is a $\csym{C}$-module functor, observe that the diagram
 \[\begin{tikzcd}
	{(\mathrm{F} \cotimes \mathrm{M}) \cotimes \mathrm{A}} && {(\mathrm{F} \cotimes \mathrm{M}) \cotimes \mathrm{B}} \\
	{\mathrm{F} \cotimes (\mathrm{M} \cotimes \mathrm{A})} && {\mathrm{F} \cotimes (\mathrm{M} \cotimes \mathrm{B})}
	\arrow["{\mathsf{a}_{\mathrm{F,M,A}}}", from=1-1, to=2-1]
	\arrow["{\mathsf{a}_{\mathrm{F,M,B}}}", from=1-3, to=2-3]
	\arrow["{(\mathrm{F \otimes M}) \otimes \mathrm{f}}", from=1-1, to=1-3]
	\arrow["{\mathrm{F} \otimes (\mathrm{M} \otimes \mathrm{f})}"', from=2-1, to=2-3]
\end{tikzcd}\]
commutes for all $\mathrm{F}$.
Clearly, $\mathrm{f}^{\ast} \circ (\mathrm{f}^{-1})^{\ast} = \on{id}_{\on{mod}_{\cccsym{C}}\!\on{--}\!\mathrm{A}}$ and $(\mathrm{f}^{-1})^{\ast} \circ \mathrm{f}^{\ast} = \on{id}_{\on{mod}_{\cccsym{C}}\!\on{--}\!\mathrm{B}}$.
\end{proof}

\begin{corollary}
 If $\euler{P}$ is a monoid object in $[\csym{C}^{\on{opp}},\mathbf{Vec}_{\Bbbk}]$ which is representable via $\euler{P} \simeq \csym{C}(-,\mathrm{A})$, then there is an equivalence of $\csym{C}$-module categories
 \[
 \on{frmod}_{\ccf{C}}\mathrm{A} \xiso \cplus{\euler{P}}.
 \]
\end{corollary}

\begin{proof}
 Lemma~\ref{Yo-YoDa} gives an equivalence of $\csym{C}$-module categories $\on{frmod}_{\ccf{C}}\mathrm{A} \xiso \cplus{\Hom{-,\mathrm{A}}}$ and Lemma~\ref{AuxiliaryLemma} yields an equivalence of $\csym{C}$-module categories $\cplus{\Hom{-,\mathrm{A}}} \simeq \cplus{\euler{P}}$.
\end{proof}

Finally, we observe that the $\csym{C}$-module functors of the form $\mathbb{F}_{\ast,\mathrm{A}}$ are compatible with the isomorphisms of Equation~\eqref{KleisliHoms}:
\begin{lemma}\label{ConjugateFreeForgetful}
 The diagram
 \[\begin{tikzcd}
	{\on{Hom}_{\ccf{C}}(\mathrm{G,N})} & {\on{Hom}_{\ccf{D}}(\mathbb{F}(\mathrm{G}),\mathbb{F}(\mathrm{N}))} & {\on{Hom}_{\on{mod}_{\ccf{D}}\!\on{--}\!\mathbb{F}(\mathrm{A})}(\mathbb{F}(\mathrm{G})\ktotimes{\cccsym{D}} \mathbb{F}(\mathrm{A}),\mathbb{F}(\mathrm{N}))} \\
	\\
	{\on{Hom}_{\on{mod}_{\ccf{C}}\!\on{--}\!\mathrm{A}}(\mathrm{G\cotimes A},\mathrm{N})} && {\on{Hom}_{\on{mod}_{\ccf{D}}\!\on{--}\!\mathbb{F}(\mathrm{A})}(\mathbb{F}(\mathrm{G\cotimes A)},\mathbb{F}(\mathrm{N}))}
	\arrow["\simeq", from=1-1, to=3-1]
	\arrow["\simeq", from=1-2, to=1-3]
	\arrow["{-\circ \mathbf{f}^{-1}_{\mathrm{G,A}}}", from=1-3, to=3-3]
	\arrow["{(\mathbb{F}_{\ast,\mathrm{A}})_{\mathrm{F\otimes A, N}}}", from=1-1, to=1-2]
	\arrow["{{(\mathbb{F})_{\mathrm{G,N}}}}"{description}, from=3-1, to=3-3]
\end{tikzcd}\]
  where $\mathbf{f}_{\mathrm{G,A}}$ is a coherence cell of $\mathbb{F}$ and the unlabelled isomorphisms are those of Equation~\eqref{KleisliHoms}, commutes.
\end{lemma}

\begin{proof}
 By definition we have $\mathrm{ra}_{\mathbb{F}\mathrm{N}} = \mathbb{F}(\mathrm{ra}_{\mathrm{N}})\circ \mathbf{f}_{\mathrm{N,A}}$.

 Chasing a morphism $\mathrm{g}: \mathrm{G} \rightarrow \mathrm{N}$ in the above diagram we obtain
 \[\begin{tikzcd}
	{\mathrm{g}} & {\mathbb{F}\mathrm{g}} & {\mathrm{ra}_{\mathbb{F}\mathrm{N}}\circ (\mathbb{F}\mathrm{g} \otimes \mathbb{F}\mathrm{A})} & {(\mathbb{F}(\mathrm{ra}_{\mathrm{N}}) \circ \mathbf{f}_{\mathrm{N,A}})\circ (\mathbb{F}\mathrm{g} \otimes \mathbb{F}\mathrm{A})} \\
	\\
	{\mathrm{ra}_{\mathrm{N}}\circ (\mathrm{g}\otimes \mathrm{A})} & {\mathbb{F}(\mathrm{ra}_{\mathrm{N}}\circ (\mathrm{g}\otimes \mathrm{A}))} & {\mathbb{F}(\mathrm{ra}_{\mathrm{N}})\circ \mathbb{F}(\mathrm{f\otimes A})} & {\mathbb{F}(\mathrm{ra}_{\mathrm{N}}) \circ \mathbf{f}_{\mathrm{N,A}}\circ (\mathbb{F}\mathrm{g} \otimes \mathbb{F}\mathrm{A})\circ \mathbf{f}_{\mathrm{G,A}}^{-1}}
	\arrow[maps to, from=1-1, to=1-2]
	\arrow[maps to, from=1-2, to=1-3]
	\arrow[maps to, from=1-3, to=1-4]
	\arrow[maps to, from=1-4, to=3-4]
	\arrow[maps to, from=1-1, to=3-1]
	\arrow[maps to, from=3-1, to=3-2]
	\arrow[from=3-2, to=3-3,equal]
	\arrow[from=3-3, to=3-4, equal]
\end{tikzcd}\]
where the equalities follow from functoriality of $\mathbb{F}$ and naturality of $\mathbf{f}$.
\end{proof}

Let $\mathbf{N}$ be a very cyclic $\csym{C}$-module category and let $Y$ be a very cyclic generator for $\mathbf{N}$. Assume that the functor $\Hom{-Y,Y}$ considered in Section~\ref{s10} is represented by an object $\setj{Y,Y}$. We see that Proposition~\ref{RepresentableMonoids} could be reformulated as showing a monoid isomorphism $\Hom{-Y,Y} \simeq \yo_{\ccf{C}}(\setj{Y,Y})$.

Recall from Section~\ref{s101} the strong monoidal functor $\mathbb{W}: [\csym{C}^{\on{opp}},\mathbf{Vec}_{\Bbbk}]^{\otimes\!\on{opp}} \rightarrow \csym{C}\!\on{-Tamb}(\csym{C},\csym{C})$. Since the domain of $\mathbb{W}$ is $[\csym{C}^{\on{opp}},\mathbf{Vec}_{\Bbbk}]^{\otimes\!\on{opp}}$ rather than $[\csym{C}^{\on{opp}},\mathbf{Vec}_{\Bbbk}]$, we obtain the functor
\[
\mathbb{W}_{\ast, \Hom{-Y,Y}}: \on{mod}_{[\ccf{C}^{\on{op}},\mathbf{Vec}_{\Bbbk}]}\!\on{--}\!\Hom{-Y,Y} \rightarrow \mathbb{W}(\Hom{-Y,Y})\!\on{--mod}_{\ccf{C}\!\on{-Tamb}(\ccf{C},\ccf{C})},
\]
mapping right modules to left modules. For the same reason, the action of $\csym{C}$, under which $\mathrm{F} \in \csym{C}$ acts by the functor sending a $\mathbb{W}(\Hom{-Y,Y})$-module $\mathrm{M}$ to the $\mathbb{W}(\Hom{-Y,Y})$-module $\mathrm{M} \diamond \csym{C}(-,-\mathrm{F})$, is a left action.
Following the above, $\mathbb{W}_{\ast, \Hom{-Y,Y}}$ is a $\csym{C}$-module functor, which maps $\on{frmod}_{[\ccf{C}^{\on{op}},\mathbf{Vec}_{\Bbbk}]}\Hom{-Y,Y}$ to $\mathbb{W}(\Hom{-Y,Y})\!\on{-frmod}_{\ccf{C}\!\on{-Tamb}(\ccf{C},\ccf{C})}$.

Analogously to $\cplus{\Hom{-Y,Y}}$ and $[Y,Y]_{+}^{\ccf{C}}$, we define the category $\mathbb{W}(\Hom{-Y,Y})_{+}^{\ccf{C}}$ consisting of left $\mathbb{W}(\Hom{-Y,Y})$-module objects isomorphic to those of the form $\mathbb{W}(\Hom{-Y,Y}) \diamond \csym{C}(-,-\mathrm{F})$, for $\mathrm{F} \in \csym{C}$.

Observe that we have isomorphisms
\begin{equation}\label{YonedaPersists}
 \mathbb{W}(\csym{C}(-,\mathrm{F}))(\mathrm{K,L}) = \int^{\mathrm{H}} \csym{C}(\mathrm{K,L \cotimes H}) \kotimes \csym{C}(\mathrm{H,F}) \xrightarrow[\sim]{(\mathtt{q}_{\mathrm{F}})_{\mathrm{K,L}}} \csym{C}(\mathrm{K,L\cotimes F}),
\end{equation}
yielding isomorphisms $\mathbb{W}(\csym{C}(-,\mathrm{F}))\xrightarrow[\sim]{\mathtt{q}_{\mathrm{F}}} \csym{C}(-,-\mathrm{F})$, for all $\mathrm{F}$.
We conclude that $\mathbb{W}_{\ast,\Hom{-Y,Y}}$ maps $\cplus{\Hom{-Y,Y}}$ to $\mathbb{W}(\Hom{-Y,Y})_{+}^{\ccf{C}}$.

\begin{definition}
 We denote the above described restriction of $\mathbb{W}_{\ast,\Hom{-Y,Y}}$ to a $\csym{C}$-module functor from $\cplus{\Hom{-Y,Y}}$ to $\mathbb{W}(\Hom{-Y,Y})_{+}^{\ccf{C}}$ by $\overline{\mathbb{W}}_{\ast,\Hom{-Y,Y}}$.
\end{definition}

\begin{proposition}\label{YonedaFullFaithful}
 For any $\mathrm{F} \in \csym{C}$ and $\euler{P}$ in $[\csym{C}^{\on{opp}},\mathbf{Vec}_{\Bbbk}]$, the map
 $\mathbb{W}_{\ccf{C}(-,\mathrm{F}),\euler{P}}$ is an isomorphism.
\end{proposition}

\begin{proof}
 Recall that by Yoneda lemma we have
 \[
 \begin{aligned}
  [\csym{C}^{\on{opp}},\mathbf{Vec}_{\Bbbk}](\csym{C}(-,\mathrm{F}),\euler{P}) &\xiso \euler{P}(\mathrm{F}); \\
  \euler{t} &\mapsto \euler{t}_{\mathrm{F}}(\on{id}_{\mathrm{F}}).
 \end{aligned}
 \]
 By Proposition~\ref{FreeTambara}, for any $\mathtt{\Psi} \in \csym{C}\!\on{-Tamb}(\csym{C},\csym{C})$ we have
 \[
  \begin{aligned}
   \csym{C}\!\on{-Tamb}(\csym{C},\csym{C})(\csym{C}(-,-\mathrm{F}), \mathtt{\Psi}) &\xiso \mathtt{\Psi}(\mathrm{F},\mathbb{1}) \\
   \mathtt{d}_{\mathrm{F},\mathbb{1}} &\mapsto \mathtt{d}_{\mathrm{F},\mathbb{1}}\left(\mathsf{l}^{-1}_{\mathrm{F}}\right).
  \end{aligned}
 \]
 We claim that the diagram
\begin{equation}\label{LastTrick}
\begin{tikzcd}
	{[\csym{C}^{\on{opp}}, \mathbf{Vec}_{\Bbbk}](\csym{C}(-,\mathrm{F}),\euler{P})} & {\csym{C}\!\on{-Tamb}(\csym{C},\csym{C})(\mathbb{W}(\csym{C}(-,\mathrm{F})),\mathbb{W}(\euler{P}))} & {\csym{C}\!\on{-Tamb}(\csym{C},\csym{C})(\csym{C}(-,-\mathrm{F}),\mathbb{W}(\euler{P}))} \\
	&& {\mathbb{W}(\euler{P})(\mathrm{F},\mathbb{1})} \\
	{\euler{P}(\mathrm{F})} & {\int^{\mathrm{H}}\euler{P}(\mathrm{H})\kotimes \csym{C}(\mathrm{F,H})} & {\int^{\mathrm{H}}\euler{P}(\mathrm{H}) \kotimes \csym{C}(\mathrm{F},\mathbb{1}\mathrm{H})}
	\arrow["\simeq", from=1-3, to=2-3]
	\arrow["\simeq", from=1-1, to=3-1]
	\arrow[from=2-3, to=3-3, equal]
	\arrow["{\mathbb{W}_{\cccsym{C}(-,\mathrm{F}),\euler{P}}}", from=1-1, to=1-2]
	\arrow["{-\circ \mathtt{q}_{\mathrm{F}}^{-1}}", from=1-2, to=1-3]
	\arrow["\yo"', from=3-2, to=3-1]
	\arrow["{\int^{\mathrm{H}}\euler{P}(\mathrm{H}) \otimes \ccf{C}(\mathrm{F},\mathsf{l})}"', from=3-3, to=3-2]
\end{tikzcd}
\end{equation}
 where the unlabelled isomorphisms are those specified above, $\yo$ indicates the isomorphism coming from Yoneda lemma, and $\mathtt{q}_{\mathrm{F}}$ is the isomorphism of Equation~\eqref{YonedaPersists}, commutes.

 Let $\euler{t}: \csym{C}(-,\mathrm{F}) \rightarrow \euler{P}$. Recall that $(\mathbb{W}_{\ccf{C}(-,\mathrm{F}),\euler{P}}(\euler{t}))_{\mathrm{K,L}} = \int^{\mathrm{H}} \euler{t}_{\mathrm{H}} \kotimes \csym{C}(\mathrm{K,LH})$. Precomposing with $\mathtt{q}_{\mathrm{F}}$, we obtain the map which, following Lemma~\ref{ProYoneda}, can be written as
 \[
  \begin{aligned}
   \csym{C}(\mathrm{K,LF}) &\rightarrow \coprod_{\mathrm{H}} \euler{P}(\mathrm{H}) \kotimes \csym{C}(\mathrm{K,LH}) \twoheadrightarrow \int^{\mathrm{H}} \euler{P}(\mathrm{H}) \kotimes \csym{C}(\mathrm{K,LH}) \\
   \mathrm{k} &\longmapsto \euler{t}_{\mathrm{F}}(\on{id}_{\mathrm{F}}) \otimes \mathrm{k} \longmapsto \euler{t}_{\mathrm{F}}(\on{id}_{\mathrm{F}}) \otimes \mathrm{k}.
  \end{aligned}
 \]
 In particular, setting $(\mathrm{K,L}) = (\mathrm{F},\mathbb{1})$ and $\mathrm{k} = \mathsf{l}_{\mathrm{F}}^{-1}$, we find the element $\euler{t}_{\mathrm{F}}(\on{id}_{\mathrm{F}}) \otimes \mathsf{l}^{-1}_{\mathrm{F}}$ in $\int^{\mathrm{H}} \euler{P}(\mathrm{H}) \kotimes \csym{C}(\mathrm{F},\mathbb{1}\mathrm{H})$. Finally, recall that $\yo(x \otimes \mathrm{f}) = \euler{P}(\mathrm{f})(x)$. Chasing $\euler{t}$ in Diagram~\eqref{LastTrick}, we obtain
 \[\begin{tikzcd}
	{\euler{t}} & {\mathbb{W}(\euler{t})} && {\mathbb{W}(\euler{t})\circ \mathtt{q}_{\mathrm{F}}} \\
	\\
	{\euler{t}(\on{id}_{\mathrm{F}})} & {\euler{P}(\on{id}_{\mathrm{F}})(\euler{t}_{\mathrm{F}}(\on{id}_{\mathrm{F}}))} & {\euler{t}(\on{id}_{\mathrm{F}}) \otimes \on{id}_{\mathrm{F}}} & {\euler{t}(\on{id}_{\mathrm{F}}) \otimes \mathsf{l}_{\mathrm{F}}^{-1}}
	\arrow[maps to, from=1-1, to=3-1]
	\arrow[maps to, from=1-1, to=1-2]
	\arrow[maps to, from=1-2, to=1-4]
	\arrow[maps to, from=1-4, to=3-4]
	\arrow[maps to, from=3-4, to=3-3]
	\arrow[maps to, from=3-3, to=3-2]
	\arrow[from=3-2, to=3-1, equal]
\end{tikzcd}\]
 showing that Diagram~\eqref{LastTrick} commutes. But we already know that all morphisms in Diagram~\eqref{LastTrick} which are not $\mathbb{W}_{\ccf{C}(-\mathrm{F}),\euler{P}}$ are isomorphisms. We conclude that also $\mathbb{W}_{\ccf{C}(-\mathrm{F}),\euler{P}}$ is an isomorphism.
\end{proof}

\begin{proposition}\label{WEquivalence}
 $\overline{\mathbb{W}}_{\ast,\Hom{-Y,Y}}$ is an equivalence of $\csym{C}$-module categories.
\end{proposition}

\begin{proof}
 Since $\mathbb{W}$ is monoidal, the isomorphisms given in Equation~\ref{YonedaPersists} yield isomorphisms
 \[
  \mathbb{W}(\csym{C}(-,\mathrm{F}) \circledast \Hom{-Y,Y}) \simeq \mathbb{W}(\Hom{-Y,Y}) \diamond \mathbb{W}(\csym{C}(-,\mathrm{F})) \simeq \mathbb{W}(\Hom{-Y,Y}) \diamond \csym{C}(-,-\mathrm{F}),
 \]
 showing that $\overline{\mathbb{W}}_{\ast,\Hom{-Y,Y}}$ is essentially surjective.

 To show that $\overline{\mathbb{W}}_{\ast,\Hom{-Y,Y}}$ is also full and faithful, we show that $(\mathbb{W}_{\ast, \Hom{-Y,Y}})_{\ccf{C}(-,\mathrm{F}) \circledast \Hom{-Y,Y},\ccf{C}(-,\mathrm{G}) \circledast \Hom{-Y,Y}}$ is an isomorphism for all $\mathrm{F,G}$. By Lemma~\ref{ConjugateFreeForgetful}, $(\mathbb{W}_{\ast, \Hom{-Y,Y}})_{\ccf{C}(-,\mathrm{F}) \circledast \Hom{-Y,Y},\ccf{C}(-,\mathrm{G}) \circledast \Hom{-Y,Y}}$ is conjugated via isomorphisms to $\mathbb{W}_{\ccf{C}(-,\mathrm{F}),\ccf{C}(-,\mathrm{G}) \circledast \Hom{-Y,Y}}$, so it suffices to show that the latter is an isomorphism. This follows from Proposition~\ref{YonedaFullFaithful}.
\end{proof}

\begin{lemma}\label{MXCPlus}
Let $\mathbf{M}$ be a Cauchy complete cyclic $\csym{C}$-module category with a cyclic generator $X$.
 There is an equivalence of $\csym{C}$-module categories $\mathbf{M} \simeq ([X,X]_{+}^{\ccf{C}})^{\mathsf{c}}$.
\end{lemma}

\begin{proof}
 By Definition~\ref{CyclicCauchyDefinition}, we have $\mathbf{M} \simeq (\mathbf{M} \star X)^{\mathsf{c}}$, where $\mathbf{M} \star X$ is very cyclic.
 Recall that, by Theorem~\ref{MainThm} the pseudofunctor $\na$ of Theorem~\ref{LaxPseudoFunctorialityNa} is a biequivalence.
 The proof of Corollary~\ref{MainEssSurj} shows that for a quasi-inverse $\na^{-1}$ of $\na$, we have $\na^{-1}([X,X]) \simeq [X,X]_{+}^{\ccf{C}}$ in $\csym{C}\!\on{-Tamb}$.
 At the same time, Theorem~\ref{LaxPseudoFunctorialityNa} gives $[X,X] \simeq \na(\mathbf{M}\star X)$. Thus
 \[
  \mathbf{M}\star X \simeq \na^{-1}(\na(\mathbf{M}\star X)) \simeq \na^{-1}([X,X]) \simeq [X,X]_{+}^{\ccf{C}} \text{ in } \csym{C}\!\on{-Tamb}.
 \]
 Passing to the Cauchy completions, we find
 \[
  \mathbf{M}\simeq (\mathbf{M}\star X)^{\mathsf{c}} \simeq ([X,X]_{+}^{\ccf{C}})^{\mathsf{c}} \text{ in } \csym{C}\!\on{-Mod}.
 \]
\end{proof}

\begin{corollary}\label{AuxiliaryCorollary}
 There is an equivalence of $\csym{C}$-module categories $(\cplus{\Hom{-X,X}})^{\mathsf{c}} \simeq \mathbf{M}$ if and only if there is an equivalence $\mathbb{W}(\Hom{-X,X})_{+}^{\ccf{C}} \simeq [X,X]_{+}^{\ccf{C}}$ in $\csym{C}\!\on{-Tamb}$, if and only if there is an equivalence $(\mathbb{W}(\Hom{-X,X})_{+}^{\ccf{C}})^{\mathsf{c}} \simeq ([X,X]_{+}^{\ccf{C}})^{\mathsf{c}}$ of $\csym{C}$-module categories.
\end{corollary}
\begin{proof}
 Lemma~\ref{MXCPlus} gives an equivalence $\mathbf{M} \simeq ([X,X]_{+}^{\ccf{C}})^{\mathsf{c}}$ in $\csym{C}\!\on{-Mod}$, thus also an equivalence $\mathbf{M} \simeq [X,X]_{+}^{\ccf{C}}$ in $\csym{C}\!\on{-Tamb}$.
 Proposition~\ref{WEquivalence} gives an equivalence $\mathbb{W}(\Hom{-X,X})_{+}^{\ccf{C}} \simeq \cplus{\Hom{-X,X}}$ in $\csym{C}\!\on{-Mod}$, thus also an equivalence $\mathbb{W}(\Hom{-X,X})_{+}^{\ccf{C}} \simeq \cplus{\Hom{-X,X}}$ in $\csym{C}\!\on{-Tamb}$.
 Thus, in $\csym{C}\!\on{-Tamb}$, there is an equivalence $\cplus{\Hom{-X,X}} \simeq \mathbf{M}$ if and only if there is an equivalence $\mathbb{W}(\Hom{-X,X})_{+}^{\ccf{C}} \simeq [X,X]_{+}^{\ccf{C}}$. Thus, in $\csym{C}\!\on{-Mod}$, there is an equivalence $(\cplus{\Hom{-X,X}})^{\mathsf{c}} \simeq \mathbf{M}^{\mathsf{c}}$ if and only if there is an equivalence $(\mathbb{W}(\Hom{-X,X})_{+}^{\ccf{C}})^{\mathsf{c}} \simeq ([X,X]_{+}^{\ccf{C}})^{\mathsf{c}}$.

 Since $\mathbf{M}$ is Cauchy complete, we have $\mathbf{M} \simeq \mathbf{M}^{\mathsf{c}}$ in $\csym{C}\!\on{-Mod}$. The result follows.

\end{proof}

\begin{theorem}\label{Formulae}
 If the monoid morphism $\omega_{X,X}$ of Corollary~\ref{RigidMonoids} is an isomorphism, then there is an equivalence of $\csym{C}$-module categories
 \[
  \mathbf{M} \simeq (\cplus{\Hom{-X,X}})^{\mathsf{c}}.
 \]
\end{theorem}
\begin{proof}
 Lemma~\ref{AuxiliaryLemma} applied to $\omega_{X,X}$ shows that the condition of Corollary~\ref{AuxiliaryCorollary} is satisfied.
\end{proof}

Assume that $\Hom{-X,X}$ is representable via $\Hom{-X,X} \simeq \Hom{-,\setj{X,X}}$, and let $\on{ev}_{X}$ be the counit of the adjunction $\Hom{-X,X} \simeq \csym{C}(-,\setj{X,X})$, as described in Section~\ref{s102}.
Stating the sufficient condition of Theorem~\ref{Formulae} explicitly in this case gives it a more natural, intuitive form:
\begin{corollary}\label{ExplicitEquations}
 If, for all $\mathrm{F,G}$, the map
 \begin{equation}\label{MoreFormulas}
 \begin{aligned}
  \csym{C}(\mathrm{F,G}\setj{X,X}) &\rightarrow \Hom{\mathbf{M}\mathrm{F}X,\mathbf{M}\mathrm{G}X} \\
  \mathrm{f} &\mapsto \mathbf{M}\mathrm{G}\on{ev}_{X} \circ \mathbf{m}_{\mathrm{G},\setj{X,X}} \circ \mathbf{M}\mathrm{f}_{X}
 \end{aligned}
 \end{equation}
 is an isomorphism, then there is an equivalence of $\csym{C}$-module categories
 \[
 (\on{frmod}_{\ccf{C}}\setj{X,X})^{\mathsf{c}} \simeq \mathbf{M}.
 \]
\end{corollary}

\begin{proof}
 The map given in Equation~\eqref{MoreFormulas} is the composite
 \[
 \begin{aligned}
  &\csym{C}(\mathrm{F,G}\setj{X,X}) \xiso \int^{\mathrm{H}} \csym{C}(\mathrm{F,GH}) \kotimes \csym{C}(\mathrm{H},\setj{X,X}) = \mathbb{W}(\csym{C}(-,\setj{X,X}))(\mathrm{F,G}) \\
  &\xiso \mathbb{W}(\Hom{-X,X})(\mathrm{F,G}) \xrightarrow{(\omega_{X,X})_{\mathrm{F,G}}} [X,X](\mathrm{F,G}) = \Hom{\mathbf{M}\mathrm{F}X, \mathbf{M}\mathrm{G}X},
 \end{aligned}
 \]
 which clearly is an isomorphim if and only if $(\omega_{X,X})_{\mathrm{F,G}}$ is an isomorphism. The result follows from Theorem~\ref{Formulae}.
\end{proof}

\begin{remark}
 If $\Hom{-X,X}$ is representable via $\Hom{-X,X} \simeq \Hom{-,\setj{X,X}}$ and $\omega_{X,X}$ is an isomorphism, then in particular, for any $\mathrm{G} \in \csym{C}$, we have a natural isomorphism
 \[
  \Hom{-X,\mathbf{M}\mathrm{G}X} \simeq \csym{C}(-,\mathrm{G}\cotimes \setj{X,X}),
 \]
 so the object $\mathrm{G}\cotimes \setj{X,X}$ also becomes an internal hom for the $\csym{C}$-module category $\mathbf{M}$ - we may write
 \[
 \mathrm{G} \otimes \setj{X,X} =: \setj{X,\mathrm{G}X}.
 \]
\end{remark}

\section{\texorpdfstring{$\tam$}{C-Tamb(C,C)}-enrichment}\label{s12}

The main focus in the presentation of the results of this document is their application to classification problems for $\csym{C}$-module categories, showing that these can be equivalently treated as Morita classification problems for monoids in a suitable monoidal category, $\tam$. This closely follows the approach taken in \cite{Os}, and even more so the approach taken in \cite{MMMT}, \cite{MMMTZ1}.

Closely related results, in a more category theoretic setting, can be found in \cite{GP}, and, formulated in more elementary terms, in \cite{JK}. To formulate its main results, let us first assume that $\csym{C}$ is closed and consider only $\csym{C}$-module categories $\mathbf{M}$ where the functors $\Hom{-X,Y}$ are representable for all $X,Y$ - in other words, the action is {\it right-closed}, admitting objects $\setj{X,Y}$ such that $\csym{C}(\mathrm{G},\setj{X,Y}) \simeq \Hom{\mathbf{M}\mathrm{G}X,Y}_{\mathbf{M}}$. In that case, we may define a {\it $\csym{C}$-enriched category} $\widecheck{\mathbf{M}}$, given by
\begin{itemize}
 \item $\on{Ob}\widecheck{\mathbf{M}} = \on{Ob}\mathbf{M}$;
 \item $\Hom{X,Y}_{\widecheck{\mathbf{M}}} = \setj{X,Y}$;
 \item The composition $\setj{Y,Z} \cotimes \setj{X,Y} \rightarrow \setj{X,Z}$ is given by the preimage under the Yoneda embedding of the transformations $\euler{c}_{X,Y,Z}$ described in Definition~\ref{AbstractMonoid};
 \item The unit $\mathbb{1} \rightarrow \setj{X,X}$ is given by the preimage under the Yoneda embedding of the transformation $\euler{u}_{X}$ described in Definition~\ref{AbstractUnit}.
 \item Associativity of the composition is given by Proposition~\ref{AbstractAssoc}, unitality by Proposition~\ref{AbstractUnitality}.
\end{itemize}

This construction extends to $\csym{C}$-module morphisms: a morphism $\Psi: \mathbf{M} \rightarrow \mathbf{N}$ yields a $\csym{C}$-transformation
\[
\widecheck{\Psi}_{X,Y}:\Hom{-X,Y}_{\mathbf{M}} \rightarrow \Hom{-\Psi X, \Psi Y}_{\mathbf{N}}
\]
for any $X,Y \in \mathbf{M}$, and these assemble to a $\csym{C}$-functor $\widecheck{\Psi}: \widecheck{\mathbf{M}} \rightarrow \widecheck{\mathbf{N}}$. This further extends to modifications, yielding a pseudofunctor $(\widecheck{-}): \csym{C}\!\on{-Mod} \rightarrow \csym{C}\!\on{-Cat}$. The following is a consequence of \cite[Theorem~3.7]{GP}:
\begin{theorem}[{\cite{GP}}]
 The pseudofunctor $(\widecheck{-})$ is a biequivalence.
\end{theorem}

Note that the monoids studied in \cite{Os} are just the endomorphism monoids in $\widecheck{\csym{C}}$. Observe also that the hom-objects and the composition maps are representations of objects and maps in $[\csym{C}^{\on{opp}},\mathbf{Vec}_{\Bbbk}]$. Thus, even when omitting the assumptions about $\csym{C}$ being closed and $\mathbf{M}$ being right-closed, we should be able to recover a $[\csym{C}^{\on{opp}}, \mathbf{Vec}_{\Bbbk}]$-enrichment on $\mathbf{M}$. And indeed, we have the following result:
\begin{theorem}[{\cite[Wood's theorem]{Ga}}]
 A $\csym{C}$-module category $\mathbf{M}$ gives rise to a $[\csym{C}^{\on{opp}}, \mathbf{Vec}_{\Bbbk}]$-enriched category $\widecheck{\mathbf{M}}$, given by
 \begin{itemize}
 \item $\on{Ob}\widecheck{\mathbf{M}} = \on{Ob}\mathbf{M}$;
 \item $\Hom{X,Y}_{\widecheck{\mathbf{M}}} = \Hom{-X,Y}$;
 \item The composition $\setj{Y,Z} \cotimes \setj{X,Y} \rightarrow \setj{X,Z}$ is given by the transformation $\euler{c}_{X,Y,Z}$ described in Definition~\ref{AbstractMonoid};
 \item The unit $\mathbb{1} \rightarrow \setj{X,X}$ is given by the transformation $\euler{u}_{X}$ described in Definition~\ref{AbstractUnit}.
 \item Associativity of the composition is given by Proposition~\ref{AbstractAssoc}, unitality by Proposition~\ref{AbstractUnitality}.
\end{itemize}
\end{theorem}

The constructions that follow concern categories and profunctors enriched in the non-symmetric monoidal category $\tam$. In the full generality of the non-strict case, we may follow \cite{GS}, which develops a theory of bicategories enriched in monoidal bicategories. In our case, both the enriching and the enriched bicategories have only identity $2$-cells. To facilitate the presentation, we assume that $\csym{C}$ is a $2$-category and that the $\csym{C}$-module categories are strict.

\begin{definition}\label{DefiningS}
 We define a $2$-functor $\mathbb{S}: \csym{C}\!\on{-Mod} \rightarrow \tam^{\otimes\!\on{opp}}\!\on{-Cat}$ as follows:
 \begin{enumerate}[label = (\alph*)]
  \item Given a $\csym{C}$-module category $\mathbf{M}$, we let $\on{Ob}\mathbb{S}(\mathbf{M}) = \on{Ob}\mathbf{M}$. Given $X,Y \in \on{Ob}\mathbf{M}$, we let $\Hom{X,Y}_{\mathbb{S}(\mathbf{M})} := [X,Y]$, the Tambara module defined following Equation~\eqref{BicatYo}.
  The composition maps $\mathsf{c}^{\mathbb{S}(\mathbf{M})}_{X,Y,Z}$ are given by
   \[
 \begin{aligned}
  ([X,Y] \circ [Y,Z])(\mathrm{F,G}) = \int^{\mathrm{H}} \Hom{\mathbf{M}\mathrm{F}X,\mathbf{M}\mathrm{H}Y} \kotimes \Hom{\mathbf{M}\mathrm{H}Y,\mathbf{M}\mathrm{G}Z} &\rightarrow \Hom{\mathbf{M}\mathrm{F}X,\mathbf{M}\mathrm{G}Z} = [X,Z](\mathrm{F,G}) \\
  f \otimes g &\mapsto g \circ f
 \end{aligned}
 \]
 and the unit morphism is given by
 \[
 \begin{aligned}
  \csym{C}(\mathrm{F,G}) &\rightarrow [X,X](\mathrm{F,G}) \\
  \mathrm{f} &\mapsto  (\mathbf{M}\mathrm{f})_{X}
 \end{aligned}
 \]
 \item Given a $\csym{C}$-module functor $\euler{\Phi}: \mathbf{M} \rightarrow \mathbf{N}$, we let $\on{Ob}\mathbb{S}(\euler{\Phi}) := \on{Ob}\euler{\Phi}$, and, for $X,Y \in \mathbf{M}$, we let
 \[
 \mathbb{S}(\euler{\Phi})_{X,Y}: [X,Y] \rightarrow [\euler{\Phi} X, \euler{\Phi} Y]
 \]
 be given by
 \[
 \begin{aligned}
(\mathbb{S}(\euler{\Phi})_{X,Y})_{\mathrm{F,G}}: \Hom{\mathbf{M}\mathrm{F}X,\mathbf{M}\mathrm{G}Y} &\rightarrow \Hom{\euler{\Phi} \mathbf{M}\mathrm{F}X, \euler{\Phi} \mathbf{M}\mathrm{G}Y} = \Hom{\mathbf{N}\mathrm{F}\euler{\Phi} X, \mathbf{N}\mathrm{G}\euler{\Phi} Y} \\
 f &\mapsto \euler{\Phi}(f).
 \end{aligned}
 \]
 \item
 Recall that, given $\tam^{\otimes\!\on{opp}}$-functors $\Phi,\Phi': \mathcal{M} \rightarrow \mathcal{N}$, a {\it $\tam^{\otimes\!\on{opp}}$-transformation} $\tau$ from $\Phi$ to $\Phi'$ consists of a collection
 \[
\setj{\tau_{X} \in \tam^{\otimes\!\on{opp}}(\csym{C}(-,-), \mathcal{N}(\Phi X, \Phi' X)) \simeq \mathcal{N}(\Phi X, \Phi' X)(\mathbb{1}_{\ccf{C}},\mathbb{1}_{\ccf{C}}) \; | \; X \in \mathcal{M}}
 \]
 such that, on the component indexed by $(\mathrm{F,G})$, we have the commutativity of
 \[
 \resizebox{.99\hsize}{!}{$
 \begin{tikzcd}[ampersand replacement=\&,row sep=scriptsize]
	{\mathcal{M}(X,Y)(\mathrm{F,G})} \& {\int^{\mathrm{H}}\mathcal{M}(X,Y)(\mathrm{F,H})\kotimes\csym{C}(\mathrm{H,G})} \& {\int^{\mathrm{H}} \mathcal{N}(\Phi X,\Phi Y)(\mathrm{F,H}) \kotimes \mathcal{N}(\Phi Y, \Phi' Y)(\mathrm{H,G})} \\
	{\int^{\mathrm{H}} \csym{C}(\mathrm{F,H}) \kotimes \mathcal{M}(X,Y)(\mathrm{H,G})} \& {\int^{\mathrm{H}} \mathcal{N}(\Phi X,\Phi'X)(\mathrm{F,H}) \kotimes \mathcal{N}(\Phi' X, \Phi' Y)(\mathrm{H,G})} \& {\mathcal{N}(\Phi X, \Phi'Y)(\mathrm{F,G})}
	\arrow["\simeq", from=1-1, to=1-2]
	\arrow["\simeq"', from=1-1, to=2-1]
	\arrow["{\tau_{X} \circ \Phi_{X,Y}'}"', shift right=1, from=2-1, to=2-2]
	\arrow["{\Phi_{X,Y}\circ \tau_{Y}}", from=1-2, to=1-3]
	\arrow["{(-\circ-)_{\mathcal{N}}}"', from=2-2, to=2-3]
	\arrow["{(-\circ-)_{\mathcal{N}}}", from=1-3, to=2-3]
\end{tikzcd}$}\]
 Chasing an element $a \in \mathcal{M}(X,Y)(\mathrm{F,G})$, we find the equation
 \begin{equation}\label{EnrichedNatEq}
  ((-\circ -)_{\mathcal{N}})_{\mathrm{F;F,G}}\big((\tau_{X})_{\mathrm{F,F}}(\on{id}_{\mathrm{F}}) \otimes (\Phi'_{X,Y})_{\mathrm{F,G}}(a)\big) = ((-\circ-)_{\mathcal{N}})_{\mathrm{G;F,G}}\big( (\Phi_{X,Y})_{\mathrm{F,G}}(a) \otimes (\tau_{Y})_{\mathrm{G,G}}(\on{id}_{\mathrm{G}}) \big)
 \end{equation}
 Let $\euler{t} \in \csym{C}\!\on{-Mod}(\mathbf{M},\mathbf{N})(\euler{\Phi},\euler{\Phi}')$ be a modification.
 We define a $\tam^{\otimes\!\on{opp}}$-transformation
 \[
\mathbb{S}(\euler{t}): \mathbb{S}(\Phi) \rightarrow \mathbb{S}(\Phi')
 \]
 by letting $\mathbb{S}(\euler{t})_{X}$ be the image of $\euler{t}_{X}$ under the isomorphism
 \begin{equation}\label{DefByUni}
\tam^{\otimes\!\on{opp}}\Big(\csym{C}(-,-), \Hom{\mathbb{S}(\euler{\Phi}) X, \mathbb{S}(\euler{\Phi'}) X}_{\mathbb{S}(\mathbf{N})}\Big) \simeq \Hom{\mathbb{S}(\euler{\Phi}) X, \mathbb{S}(\euler{\Phi'}) X}_{\mathbb{S}(\mathbf{N})}(\mathbb{1}_{\ccf{C}},\mathbb{1}_{\ccf{C}}) \simeq \Hom{\euler{\Phi} X, \euler{\Phi'} X}_{\mathbf{N}}.
 \end{equation}
 \end{enumerate}
\end{definition}

\begin{lemma}
 $\mathbb{S}$ is a well-defined $2$-functor.
\end{lemma}

\begin{proof}
\begin{enumerate}
\item Associativity and unitality axioms for $\mathbb{S}(\mathbf{M})$ are verified by the proofs of Proposition~\ref{EndMonoid} and Proposition~\ref{HomBimod}.
 \item
 The Tambara axiom for $\mathbb{S}(\euler{\Phi})$ follows from the commutativity of
 \[\begin{tikzcd}[row sep = scriptsize, column sep = huge]
	{\Hom{\mathbf{M}\mathrm{F}X,\mathbf{M}\mathrm{G}Y}} & {\Hom{\mathbf{N}\mathrm{F}\euler{\Phi} X,\mathbf{N}\mathrm{G}\euler{\Phi} Y}} \\
	{\Hom{\mathbf{M}\mathrm{KF}Y,\mathbf{M}\mathrm{KG}Y}} & {\Hom{\mathbf{N}\mathrm{KF}\euler{\Phi} X,\mathbf{N}\mathrm{KG}\euler{\Phi} Y}}
	\arrow["{\mathbf{M}\mathrm{K}_{\mathbf{M}\mathrm{F}X,\mathbf{M}\mathrm{G}Y}}"', from=1-1, to=2-1]
	\arrow["{\euler{\Phi}_{\mathbf{M}\mathrm{F}X,\mathbf{M}\mathrm{G}Y}}", from=1-1, to=1-2]
	\arrow["{\euler{\Phi}_{\mathbf{M}\mathrm{KF}Y,\mathbf{M}\mathrm{KG}Y}}"', from=2-1, to=2-2]
	\arrow["{\mathbf{N}\mathrm{K}_{\mathbf{N}\mathrm{F}\euler{\Phi} X,\mathbf{N}\mathrm{G}\euler{\Phi} Y}}", from=1-2, to=2-2]
\end{tikzcd}\]
which follows from $\euler{\Phi}$ being a $\csym{C}$-module functor.
\item
 The multiplicative axiom for $\tam^{\otimes\!\on{opp}}$-functoriality of $\mathbb{S}(\euler{\Phi})$ follows from the commutativity of
 \[\begin{tikzcd}[row sep = scriptsize]
	{[X,Y] \circ [Y,Z]} && {[Y,Z]} \\
	{[\euler{\Phi} X, \euler{\Phi} Y] \circ [\euler{\Phi} Y, \euler{\Phi} Z]} && {[\euler{\Phi} X,\euler{\Phi} Z]}
	\arrow["{\mathbb{S}(\euler{\Phi})_{X,Y}\circ \mathbb{S}(\euler{\Phi})_{Y,Z}}", from=1-1, to=2-1,swap]
	\arrow["{\mathsf{c}^{\mathbb{S}(\mathbf{M})}_{X,Y,Z}}", from=1-1, to=1-3]
	\arrow["{\mathbb{S}(\euler{\Phi})_{X,Z}}", from=1-3, to=2-3]
	\arrow["{\mathsf{c}^{\mathbb{S}(\mathbf{N})}_{\euler{\Phi} X, \euler{\Phi} Y, \euler{\Phi} Z}}"', from=2-1, to=2-3]
\end{tikzcd}\]
which we may verify component-wise, and, by identifying maps from coends with extranatural collections, reduce to the commutativity of
\[\begin{tikzcd}[row sep = scriptsize, column sep = huge]
	{\Hom{\mathbf{M}\mathrm{F}X,\mathbf{M}\mathrm{H}Y} \kotimes \Hom{\mathbf{M}\mathrm{H}Y,\mathbf{M}\mathrm{G}Z}} && {\Hom{\mathbf{M}\mathrm{F}X,\mathbf{M}\mathrm{G}Z}} \\
	{\Hom{\mathbf{N}\mathrm{F}\euler{\Phi} X,\mathbf{N}\mathrm{H}\euler{\Phi} Y} \kotimes \Hom{\mathbf{N}\mathrm{H}\euler{\Phi} Y,\mathbf{N}\mathrm{G}\euler{\Phi} Z}} && {\Hom{\mathbf{N}\mathrm{F}\euler{\Phi} X, \mathbf{N}\mathrm{G}\euler{\Phi} Z}}
	\arrow["{(-\circ-)_{\mathbf{M}\mathrm{F}X,\mathbf{M}\mathrm{H}Y,\mathbf{M}\mathrm{G}Z}}", from=1-1, to=1-3]
	\arrow["{\euler{\Phi}_{\mathbf{M}\mathrm{F}X,\mathbf{M}\mathrm{H}Y}\kotimes\euler{\Phi}_{\mathbf{M}\mathrm{H}Y,\mathbf{M}\mathrm{G}Z}}"', from=1-1, to=2-1]
	\arrow["{\euler{\Phi}_{\mathbf{M}\mathrm{F}X,\mathbf{M}\mathrm{G}Z}}", from=1-3, to=2-3]
	\arrow["{(-\circ-)_{\mathbf{N}\mathrm{F}\euler{\Phi} X,\mathbf{N}\mathrm{H}\euler{\Phi} Y,\mathbf{N}\mathrm{G}\euler{\Phi} Z}}"', from=2-1, to=2-3]
\end{tikzcd}\]
 which follows from functoriality of $\euler{\Phi}$. The unitality axiom follows similarly.
 \item Given a modification $\euler{t}$ as in Definition~\ref{DefiningS}, from Equation~\eqref{DefByUni} we find that
 \[
(\mathbb{S}(\euler{t})_{X})_{\mathbb{1},\mathbb{1}} = \euler{t}_{X}
 \text{ and }
 (\mathbb{S}(\euler{t})_{X})_{\mathrm{F,F}} = \mathbf{N}\mathrm{F}\euler{t}_{X}.
 \]
 Given $a \in \Hom{\mathbf{M}\mathrm{F}X, \mathbf{M}\mathrm{G}Y} = \mathbb{S}(\mathbf{M})(X,Y)(\mathrm{F,G})$, we have
 \[
\mathbf{N}\mathrm{G}\euler{t}_{X} \circ \euler{\Phi}(a) = \euler{t}_{\mathbf{M}\mathrm{G}Y} \circ \euler{\Phi}(a) = \euler{\Phi'}(a) \circ \euler{t}_{\mathbf{M}\mathrm{F}X} = \euler{\Phi'}(a) \circ \mathbf{N}\mathrm{F}\euler{t}_{X},
 \]
 verifying Equation~\eqref{EnrichedNatEq}. Observe that the first and the last equality use the fact that $\euler{t}$ is a modification.
 \item Given $\euler{\Psi} \in \csym{C}\!\on{-Mod}(\mathbf{K},\mathbf{M})$ and $\euler{\Sigma} \in \csym{C}\!\on{-Mod}(\mathbf{M},\mathbf{N})$, we clearly have
 \[
\on{Ob}\mathbb{S}(\euler{\Sigma \circ \Psi}) = \on{Ob}(\mathbb{S}(\euler{\Sigma}) \circ \mathbb{S}(\euler{\Psi})) = \on{Ob}(\euler{\Sigma \circ \Psi}).
 \]
 For $X,Y \in \mathbf{K}$, by definition we have
 \[
((\mathbb{S}(\euler{\Sigma}) \circ \mathbb{S}(\euler{\Psi}))_{X,Y})_{\mathrm{F,G}} = \euler{\Sigma}_{\mathbf{M}\mathrm{F}\euler{\Psi} X, \mathbf{M}\mathrm{G}\euler{\Psi} Y} \circ \euler{\Psi}_{\mathbf{M}\mathrm{F}X, \mathbf{M}\mathrm{G}Y} = \mathbb{S}(\euler{\Sigma \circ \Psi})_{\mathbf{M}\mathrm{F}X,\mathbf{M}\mathrm{G}Y},
 \]
 showing that we have $\mathbb{S}(\euler{\Sigma \circ \Psi}) = \mathbb{S}(\euler{\Sigma})\circ \mathbb{S}(\euler{\Psi})$.

 Similarly, given
 \[\begin{tikzcd}
	{\mathbf{K}} & {\mathbf{M}} & {\mathbf{N}}
	\arrow[""{name=0, anchor=center, inner sep=0}, "\euler{\Psi}", shift left=2, from=1-1, to=1-2]
	\arrow[""{name=1, anchor=center, inner sep=0}, "{\euler{\Psi}'}"', shift right=2, from=1-1, to=1-2]
	\arrow[""{name=2, anchor=center, inner sep=0}, "\euler{\Sigma}", shift left=2, from=1-2, to=1-3]
	\arrow[""{name=3, anchor=center, inner sep=0}, "{\euler{\Sigma}'}"', shift right=2, from=1-2, to=1-3]
	\arrow["{\euler{s}}", shorten <=1pt, shorten >=1pt, Rightarrow, from=0, to=1]
	\arrow["{\euler{t}}", shorten <=1pt, shorten >=1pt, Rightarrow, from=2, to=3]
\end{tikzcd}\]
in $\csym{C}\!\on{-Mod}$, for any $X \in \mathbf{K}$ we have
\[
((\mathbb{S}(\euler{t}) \hcomp \mathbb{S}(\euler{s}))_{X})_{\mathbb{1},\mathbb{1}} = (\mathbb{S}(\euler{t})_{\mathbb{S}(\euler{\Psi}')X})_{\mathbb{1},\mathbb{1}} \circ (\mathbb{S}(\euler{\Sigma})_{\mathbb{S}(\euler{\Psi}') X, \mathbb{S}(\euler{\Psi}') X}(\mathbb{S}(\euler{s})_{X})_{\mathbb{1},\mathbb{1}}) = \euler{t}_{\euler{\Psi}' X} \circ \euler{\Sigma}\euler{s}_{X} = (\euler{t} \hcomp \euler{s})_{X},
\]
which, in view of Equation~\eqref{DefByUni} defining $\mathbb{S}$ on modifications, shows that $\mathbb{S}(\euler{t}) \hcomp \mathbb{S}(\euler{s}) = \mathbb{S}(\euler{t}\hcomp \euler{s})$.
 \end{enumerate}
\end{proof}

\begin{definition}\label{DefiningR}
 We define a $2$-functor $\mathbb{R}: \tam^{\otimes\!\on{opp}}\!\on{-Cat} \rightarrow \csym{C}\!\on{-Mod}$ as follows:
 \begin{enumerate}[label = (\alph*)]
   \item Given a $\tam^{\otimes\!\on{opp}}$-category $\mathcal{M}$, we let $\on{Ob}\mathbb{R}(\mathcal{M}) := \on{Ob}\mathcal{M} \times \on{Ob}\csym{C}$, and we define the hom-space $\Hom{(X,\mathrm{F}), (Y,\mathrm{G})}_{\mathbb{R}(\mathcal{M})}$ as $\Hom{X,Y}_{\mathcal{M}}\!(\mathrm{F,G})$.
   The composition
   \[
\mathsf{c}^{\mathbb{R}(\mathcal{M})}_{(X,\mathrm{F}),(Y,\mathrm{H}),(Z,\mathrm{G})}: \Hom{X,Y}_{\mathcal{M}}(\mathrm{F,H}) \kotimes \Hom{Y,Z}_{\mathcal{M}}\!(\mathrm{H,G}) \rightarrow \Hom{X,Z}(\mathrm{F,G})
   \]
is defined as the component $(\mathsf{c}^{\mathcal{M}}_{X,Y,Z})_{\mathrm{H;F,G}}$ of the composition map
   \[
    \mathsf{c}_{X,Y,Z}^{\mathcal{M}}: \int^{\mathrm{H}} \Hom{X,Y}_{\mathcal{M}}\!(\mathrm{F,H}) \kotimes \Hom{Y,Z}_{\mathcal{M}}\!(\mathrm{H,G}) \rightarrow \Hom{X,Z}_{\mathcal{M}}(\mathrm{F,G})
   \]
  Similarly, the unit $\mathsf{e}_{(X,\mathrm{F})}^{\mathbb{R}(\mathcal{M})}: \Bbbk \rightarrow \Hom{X,X}_{\mathcal{M}}\!(\mathrm{F,F})$ is the restriction of $(\mathsf{e}_{X}^{\mathcal{M}})_{\mathrm{F,F}}\!:\! \csym{C}(\mathrm{F,F}) \rightarrow \Hom{X,X}_{\mathcal{M}}\!(\mathrm{F,F})$ to $\Bbbk\setj{\on{id}_{\mathrm{F}}}$. Associativity and unitality follows from the respective axioms for $\mathcal{M}$.

  We now specify the $\csym{C}$-module category structure. For $\mathrm{H} \in \csym{C}$, we define the functor $\mathbb{R}(\mathcal{M})\mathrm{H}$ as
  \[
\mathbb{R}(\mathcal{M})\mathrm{H}\big((X,\mathrm{F}) \xrightarrow{f} (Y,\mathrm{G})\big) = (X,\mathrm{HF}) \xrightarrow{\ta^{\Hom{X,Y}_{\mathcal{M}}}_{\mathrm{H;F,G}}(f)} (Y,\mathrm{HG}).
  \]
  Given a morphism $\mathrm{H} \xrightarrow{\mathrm{h}} \mathrm{H}'$, we define the natural transformation $\mathbb{R}(\mathcal{M})\mathrm{H} \xrightarrow{\mathbb{R}(\mathcal{M})\mathrm{h}} \mathbb{R}(\mathcal{M})\mathrm{H}'$ by setting $(\mathbb{R}(\mathcal{M})\mathrm{h})_{(X,\mathrm{F})} = (\mathsf{e}_{X}^{\mathcal{M}})_{\mathrm{HF,H'F}}(\mathrm{hF})$.
  \item Given a $\tam^{\otimes\!\on{opp}}$-functor $\Phi: \mathcal{M} \rightarrow \mathcal{N}$, we let $\mathbb{R}(\Phi)_{(X,\mathrm{F}),(Y,\mathrm{G})} := (\Phi_{X,Y})_{\mathrm{F,G}}$.
  \item Given a $\tam^{\otimes\!\on{opp}}$-transformation $\tau: \Phi \Rightarrow \Phi'$, we let $\mathbb{R}(\tau)_{(X,\mathrm{F})} := (\tau_{X})_{\mathrm{F,F}}(\on{id}_{\mathrm{F}})$.
 \end{enumerate}
\end{definition}

 \begin{lemma}
  $\mathbb{R}$ is a well-defined $2$-functor.
 \end{lemma}

 \begin{proof}
 Using the notation of Definition~\ref{DefiningR}, we find the following:
  \begin{enumerate}
   \item Functoriality of $\mathbb{R}(\mathcal{M})\mathrm{H}$ follows from the composition and unit maps in $\mathcal{M}$ being Tambara morphisms.
   \item Naturality of $\mathbb{R}(\mathcal{M})\mathrm{h}$ follows from the commutativity of
\[\resizebox{.99\hsize}{!}{$
\begin{tikzcd}[ampersand replacement=\&]
	\& {\int^{\mathrm{K}} \csym{C}(\mathrm{H'F,K}) \kotimes \Hom{X,Y}(\mathrm{K,H'G})} \& {\int^{\mathrm{K}} \csym{C}(\mathrm{HF,K}) \kotimes \Hom{X,Y}(\mathrm{K,H'G})} \\
	{\int^{\mathrm{K}} \csym{C}(\mathrm{F,K}) \kotimes \Hom{X,Y}(\mathrm{K,G})} \& {\int^{\mathrm{K}} \Hom{X,X}(\mathrm{H'F,K}) \kotimes \Hom{X,Y}(\mathrm{K,H'G})} \& {\int^{\mathrm{K}} \Hom{X,X}(\mathrm{HF,K}) \kotimes \Hom{X,Y}(\mathrm{K,H'G})} \\
	\& {\Hom{X,Y}(\mathrm{H'F,H'G})} \\
	{\Hom{X,Y}(\mathrm{F,G})} \&\& {\Hom{X,Y}(\mathrm{HF,H'G})} \\
	\& {\Hom{X,Y}(\mathrm{HF,HG})} \\
	{\int^{\mathrm{K}} \Hom{X,Y}(\mathrm{F,K}) \kotimes \csym{C}(\mathrm{K,G})} \& {\int^{\mathrm{K}}\Hom{X,Y}(\mathrm{HF,K}) \kotimes\Hom{Y,Y}(\mathrm{K,HG})} \& {\int^{\mathrm{K}}\Hom{X,Y}(\mathrm{HF,K}) \kotimes \Hom{Y,Y}(\mathrm{K,H'G})} \\
	\& {\int^{\mathrm{K}} \Hom{X,Y}(\mathrm{HF,K}) \kotimes \csym{C}(\mathrm{K,HG})} \& {\int^{\mathrm{K}} \Hom{X,Y}(\mathrm{HF,K}) \kotimes \csym{C}(\mathrm{K,H'G})}
	\arrow[""{name=0, anchor=center, inner sep=0}, "{\ccf{C}(\mathrm{hF,K}) \circ \Hom{X,Y}}", shift left=2, from=1-2, to=1-3]
	\arrow["{\mathsf{e}_{X} \circ\Hom{X,Y}}"', from=1-2, to=2-2]
	\arrow[""{name=1, anchor=center, inner sep=0}, "{\Hom{X,X}(\mathrm{hF,K})\circ \Hom{X,Y}}"', shift right=2, from=2-2, to=2-3]
	\arrow["{\mathsf{e}_{X} \circ\Hom{X,Y}}", from=1-3, to=2-3]
	\arrow["{\mathsf{c}_{\Hom{X,X},\Hom{X,Y}}}"', from=2-2, to=3-2]
	\arrow["{\mathsf{c}_{\Hom{X,X},\Hom{X,Y}}}"{description}, from=2-3, to=4-3]
	\arrow[""{name=2, anchor=center, inner sep=0}, "{\Hom{X,Y}(\mathrm{hF,H'G})}"{description}, from=3-2, to=4-3]
	\arrow["{(\euler{l}^{\Hom{X,Y}})_{\mathrm{F,G}}^{-1}}"{description}, from=4-1, to=2-1]
	\arrow[""{name=3, anchor=center, inner sep=0}, "{\ta^{\cccsym{C}\circ\Hom{X,Y}}}"{description}, from=2-1, to=1-2]
	\arrow[""{name=4, anchor=center, inner sep=0}, "{\ta^{\Hom{X,Y}\circ \cccsym{C}}}"{description}, from=6-1, to=7-2]
	\arrow[""{name=5, anchor=center, inner sep=0}, "{\Hom{X,Y} \circ \ccf{C}(\mathrm{K,hG})}"', shift right=1, from=7-2, to=7-3]
	\arrow["{\Hom{X,Y} \circ \mathsf{e}_{Y}}", from=7-2, to=6-2]
	\arrow["{\Hom{X,Y} \circ \mathsf{e}_{Y}}"', from=7-3, to=6-3]
	\arrow[""{name=6, anchor=center, inner sep=0}, "{\Hom{X,Y} \circ \Hom{Y,Y}(\mathrm{K,hG})}", shift left=2, from=6-2, to=6-3]
	\arrow["{\mathsf{c}_{\Hom{X,X},\Hom{X,Y}}}", from=6-2, to=5-2]
	\arrow["{\mathsf{c}_{\Hom{X,Y},\Hom{Y,Y}}}"{description}, from=6-3, to=4-3]
	\arrow[""{name=7, anchor=center, inner sep=0}, "{\ta^{\Hom{X,Y}}_{\mathrm{H';F,G}}}"{description}, from=4-1, to=3-2]
	\arrow[""{name=8, anchor=center, inner sep=0}, "{\ta^{\Hom{X,Y}}_{\mathrm{H;F,G}}}"{description}, from=4-1, to=5-2]
	\arrow[""{name=9, anchor=center, inner sep=0}, "{\Hom{X,Y}(\mathrm{HF,hG})}"{description}, from=5-2, to=4-3]
	\arrow["{(\euler{r}^{\Hom{X,Y}})_{\mathrm{F,G}}^{-1}}"{description}, from=4-1, to=6-1]
	\arrow["1"{description}, draw=none, from=7, to=3]
	\arrow["5"{description}, draw=none, from=4, to=8]
	\arrow["4"{description}, draw=none, from=9, to=7]
	\arrow["2"{description}, draw=none, from=1, to=0]
	\arrow["3"{description}, draw=none, from=2, to=1]
	\arrow["6"{description}, draw=none, from=6, to=9]
	\arrow["7"{description}, draw=none, from=5, to=6]
\end{tikzcd}$}\]
 where
 \begin{itemize}
  \item faces $1$ and $5$ commute due to unitality of composition;
  \item faces $2$ and $7$ commute due to units of composition being Tambara morphisms (and thus profunctor morphisms);
  \item faces $3$ and $6$ commute due to composition maps being Tambara morphisms;
  \item face $4$ commutes due to extranaturality of $\ta^{\Hom{X,Y}}$ in $\mathrm{H}$.
 \end{itemize}
 \item
 Chasing an element $f \in \Hom{X,Y}(\mathrm{F,G})$ along the exterior paths in the above diagram yields
 \[
 \begin{aligned}
 &\mathbb{R}(\mathcal{M})\mathrm{H'}(f)\circ (\mathbb{R}(\mathcal{M})\mathrm{h})_{X,\mathrm{F}} = \mathsf{c}_{\Hom{X,X},\Hom{X,Y}}\Big((\mathsf{e}_{X})_{\mathrm{HF,H'F}}(\mathrm{hF}) \otimes \ta^{\Hom{X,Y}}_{\mathrm{H';F,G}}(f)\Big) \\
  &= \mathsf{c}_{\Hom{X,Y},\Hom{Y,Y}}\Big(\ta^{\Hom{X,Y}}_{\mathrm{H;F,G}}(f) \otimes (\mathsf{e}_{Y})_{\mathrm{HG,H'G}}(\mathrm{hG}) \Big) = (\mathbb{R}(\mathcal{M})\mathrm{h})_{Y,\mathrm{G}}\circ \mathbb{R}(\mathcal{M})\mathrm{H}(f).
 \end{aligned}
 \]
 For any $f \in \Hom{X,Y}(\mathrm{F,G})$ we have
 \[
  \mathbb{R}(\mathcal{M})\mathbb{1}(f) = \ta_{\mathbb{1};\mathrm{F,G}}(f) = f, \quad (\mathbb{R}(\mathcal{M})\on{id}_{\mathrm{H}})_{(X,\mathrm{F})} = (\mathsf{e}_{X})_{\mathrm{HF,HF}}(\on{id}_{\mathrm{HF}}) = \on{id}_{(X,\mathrm{HF})},
 \]
 and
 \[
  \mathbb{R}(\mathcal{M})\mathrm{H}\mathbb{R}(\mathcal{M})\mathrm{K}(f) = \ta_{\mathrm{H;KF,KG}}\ta_{\mathrm{K;F,G}}(f) = \ta_{\mathrm{HK;F,G}}(f) = \mathbb{R}(\mathcal{M})\mathrm{HK}(f).
 \]
 Finally, for $\mathrm{h}$ as above and $\mathrm{k}: \mathrm{K} \rightarrow \mathrm{K}'$, we have
 \[
 \begin{aligned}
  &(\mathbb{R}(\mathcal{M})(\mathrm{h}) \hcomp \mathbb{R}(\mathcal{M})(\mathrm{k}))_{(X,\mathrm{F})} = (\mathbb{R}(\mathcal{M})\mathrm{h})_{(X,\mathrm{KF})} \circ (\mathbb{R}(\mathcal{M})\mathrm{H}\mathbb{R}(\mathcal{M})\mathrm{k})_{(X,\mathrm{F})} \\
  &=\mathsf{c}_{\Hom{X,X},\Hom{X,X}}\big(\ta_{\mathrm{H;KF,K'F}}((\mathsf{e}_{X})_{\mathrm{KF,K'F}}(\mathrm{kF})) \otimes (\mathsf{e}_{X})_{\mathrm{HKF,H'KF}}(\mathrm{hKF})\big) \\
  & = \mathsf{c}_{\Hom{X,X},\Hom{X,X}}\big((\mathsf{e}_{X})_{\mathrm{HKF,HK'F}}(\mathrm{HkF}) \otimes (\mathsf{e}_{X})_{\mathrm{HKF,H'KF}}(\mathrm{hKF})\big) \\
  &=(\mathsf{e}_{X})_{\mathrm{HKF,H'K'F}}(\mathrm{hkF}) = \mathbb{R}(\mathcal{M})(\mathrm{hk}).
 \end{aligned}
 \]
 The third equality holds since $\mathsf{e}_{X}$ is a Tambara morphism, and the fourth follows from unitality of composition in $\mathcal{M}$.
 This shows that the above assignments endow $\mathbb{R}(\mathcal{M})$ with the structure of a $\csym{C}$-module category.
  \item  The functoriality of $\mathbb{R}(\Phi): \mathbb{R}(\mathcal{M}) \rightarrow \mathbb{R}(\mathcal{N})$ follows directly from $\tam^{\otimes\!\on{opp}}$-functoriality of $\Phi$. Further, since $\Phi_{X,Y}: \Hom{X,Y}_{\mathcal{M}} \rightarrow \Hom{\Phi X, \Phi Y}_{\mathcal{N}}$ is a Tambara morphism, we have, for any $f \in \Hom{X,Y}(\mathrm{F,G})$, the following:
 \[
  \begin{aligned}
   \mathbb{R}(\Phi)_{(X,\mathrm{HF}),(Y,\mathrm{HG})}\mathbb{R}(\mathcal{M})\mathrm{H}(f) = (\Phi_{X,Y})_{\mathrm{HF,HG}}\ta^{\Hom{X,Y}}_{\mathrm{H;F,G}}(f) = \ta^{\Hom{\Phi X, \Phi Y}}_{\mathrm{H;F,G}}(\Phi_{X,Y})_{\mathrm{F,G}}(f) = \mathbb{R}(\mathcal{N})\mathrm{H}\mathbb{R}(\Phi)_{(X,\mathrm{F}),(Y,\mathrm{G})}(f),
  \end{aligned}
 \]
 showing that $\mathbb{R}(\Phi)\mathbb{R}(\mathcal{M}) \rightarrow \mathbb{R}(\mathcal{N})$ gives a $\csym{C}$-module functor.
 \item The fact that $\mathbb{R}(\tau)$ defines a transformation from $\mathbb{R}(\Phi)$ to $\mathbb{R}(\Phi')$ is an immediate consequence of Equation~\eqref{EnrichedNatEq}. Further, $\mathbb{R}(\tau)$ is a modification. This follows from
 \[
  (\mathbb{R}(\mathcal{N})\mathrm{H}\,\mathbb{R}(\tau))_{(X,\mathrm{F})} = \ta^{\Hom{\Phi X, \Phi' X}}_{\mathrm{H;F,F}}(\tau_{X})_{\mathrm{F,F}}(\on{id}_{\mathrm{F}}) = (\tau_{X})_{\mathrm{HF,HF}}(\on{id}_{\mathrm{HF}}) = (\mathbb{R}(\tau)\mathbb{R}(\mathcal{M})\mathrm{H})_{(X,\mathrm{F})},
 \]
 where the second equality follows from $\tau_{X}$ being a Tambara morphism. We also have that
 \[
  \mathbb{R}(\on{id}_{\Phi})_{(X,\mathrm{F})} = ((\on{id}_{\Phi})_{X})_{\mathrm{F,F}}(\on{id}_{\mathrm{F}}) = (\mathsf{e}_{\Phi X})_{\mathrm{F,F}}(\on{id}_{\mathrm{F}}) = \on{id}_{(X,\mathrm{F})}
 \]
 and, given $\sigma: \Phi' \Rightarrow \Phi''$, we have
 \[
  \mathbb{R}(\sigma \circ \tau)_{(X,\mathrm{F})} = (\sigma \circ \tau)_{X}(\on{id}_{\mathrm{F}}) = (\mathsf{c}_{\Hom{\Phi X, \Phi' X},\Hom{\Phi' X, \Phi'' X}})_{\mathrm{F;F,F}}((\tau_{X})_{\mathrm{F,F}}(\on{id}_{\mathrm{F}}) \otimes (\sigma_{X})_{\mathrm{F,F}}(\on{id}_{\mathrm{F}})) = \mathbb{R}(\sigma)_{(X,\mathrm{F})}\circ \mathbb{R}(\tau)_{(X,\mathrm{F})}.
 \]
 \item Given $\tam^{\otimes\!\on{opp}}$-functors $\Phi: \mathcal{K} \rightarrow \mathcal{M}$ and $\Psi: \mathcal{M} \rightarrow \mathcal{N}$, we have
 \[
\on{Ob}\mathbb{R}(\Psi \circ \Phi) = \on{Ob}\Psi \circ \Phi \times \on{Ob}\mathbb{1}_{\ccf{C}} = \on{Ob}(\mathbb{R}(\Psi) \circ \mathbb{R}(\Phi)).
 \]
 Further, we have
 \[
\mathbb{R}(\Psi \circ \Phi)_{(X,\mathrm{F}),(Y,\mathrm{G})} = ((\Psi \circ \Phi)_{X,Y})_{\mathrm{F,G}} = (\Psi_{\Phi X, \Phi Y})_{\mathrm{F,G}} \circ (\Phi_{X,Y})_{\mathrm{F,G}} = \mathbb{R}(\Psi)_{(\Phi X, \mathrm{F}), (\Phi Y, \mathrm{G})} \circ \mathbb{R}(\Phi)_{(X,\mathrm{F}),(Y,\mathrm{G})}
 \]
 showing that $\mathbb{R}(\Psi \circ \Phi) = \mathbb{R}(\Psi) \circ \mathbb{R}(\Phi)$. Finally, for
 \[\begin{tikzcd}
	{\mathcal{K}} & {\mathcal{M}} & {\mathcal{N}}
	\arrow[""{name=0, anchor=center, inner sep=0}, "\Phi", shift left=3, from=1-1, to=1-2]
	\arrow[""{name=1, anchor=center, inner sep=0}, "{\Phi'}"', shift right=3, from=1-1, to=1-2]
	\arrow["\Psi", shift left=3, from=1-2, to=1-3]
	\arrow["{\Psi'}"', shift right=3, from=1-2, to=1-3]
	\arrow["\tau", shorten <=2pt, shorten >=2pt, Rightarrow, from=0, to=1]
\end{tikzcd}\]
in $\tam^{\otimes\!\on{opp}}\!\on{-Cat}$, we have
 \[
  \begin{aligned}
   &\mathbb{R}(\sigma \hcomp \tau)_{(X,\mathrm{F})} = ((\sigma \hcomp \tau)_{X})_{\mathrm{F,F}}(\on{id}_{\mathrm{F}})
   =(\mathsf{c}_{\Psi\Phi X, \Psi\Phi' X, \Psi' \Phi' X})_{\mathrm{F;F,F}}\Big((\Psi_{\Phi X, \Phi' X}\circ \tau_{X})_{\mathrm{F,F}}(\on{id}_{\mathrm{F}}) \otimes (\sigma_{\Phi' X})_{\mathrm{F,F}}(\on{id}_{\mathrm{F}})\Big) \\
   &= \mathbb{R}(\sigma)_{(\Phi' X, \mathrm{F})} \circ \mathbb{R}(\Psi)_{(\Phi X,\mathrm{F}),(\Phi' X, \mathrm{F})}(\mathbb{R}(\tau_{X,\mathrm{F}})) = (\mathbb{R}(\sigma)\hcomp \mathbb{R}(\tau))_{(X,\mathrm{F})},
  \end{aligned}
 \]
 showing the $2$-functoriality.
  \end{enumerate}
 \end{proof}

\begin{proposition}\label{OtherUnit}
 There is a $2$-natural equivalence $\mathbb{u}: \mathbb{1}_{\ccf{C}\!\on{-Mod}} \rightarrow \mathbb{RS}$.
\end{proposition}

\begin{proof}
 By definition we have $\Hom{(X,\mathbb{1}),(Y,\mathbb{1})}_{\mathbb{RS}(\mathbf{M})} = \Hom{X,Y}_{\mathbf{M}}$, which yields a full, faithful functor $\mathbb{u}_{\mathbf{M}}: \mathbf{M} \rightarrow \mathbb{RS}$, sending $X$ to $(X,\mathbb{1})$. Further, for any $\mathrm{H} \in \csym{C}$, we have $\mathbb{RS}(\mathbf{M})\mathrm{H} = \ta^{\mathbb{S}(\mathbf{M})}_{H;-,-} = \mathbf{M}\mathrm{H}$, which shows that the above embedding is a $\csym{C}$-module functor.

 Further, we have
 \[
\Hom{(X,\mathrm{F}),(\mathbf{M}\mathrm{F}X,\mathbb{1})}_{\mathbb{RS}(\mathbf{M})} = [X,\mathbf{M}\mathrm{F}X](\mathrm{F},\mathbb{1}) = \Hom{\mathbf{M}\mathrm{F}X,\mathbf{M}\mathrm{F}X}_{\mathbf{M}}
 \]
 and similarly, $\Hom{(\mathbf{M}\mathrm{F}X,\mathbb{1}),(X,\mathrm{F})}_{\mathbb{RS}(\mathbf{M})} = \Hom{\mathbf{M}\mathrm{F}X,\mathbf{M}\mathrm{F}X}_{\mathbf{M}}$. It is easy to show that composing the images of $\on{id}_{\mathbf{M}\mathrm{F}X}$ in the respective Hom-spaces yields the respective identity morphisms, thus $(X,\mathrm{F}) \simeq (\mathbf{M}\mathrm{F}X,\mathbb{1})$, showing that the embedding $\mathbb{u}_{\mathbf{M}}$ is essentially surjective and thus an equivalence.

  Let $\euler{\Phi} \in \csym{C}\!\on{-Mod}(\mathbf{M},\mathbf{N})$. To establish $2$-naturality of $\mathbb{u}$, we show that $\mathbb{RS}(\euler{\Phi}) \circ \mathbb{u}_{\mathbf{M}} = \mathbb{u}_{\mathbf{N}} \circ \euler{\Phi}$. On objects, we have
  \[
   (\mathbb{RS}(\euler{\Phi}) \circ \mathbb{u}_{\mathbf{M}})(X) = \mathbb{RS}(\euler{\Phi})(X,\mathbb{1}) = (\euler{\Phi} X, \mathbb{1}) = (\mathbb{u}_{\mathbf{N}}\circ \euler{\Phi})(X).
  \]
  Given any $X,Y \in \mathbf{M}$, we have
  \[
   (\mathbb{RS}\euler{\Phi})_{(X,\mathbb{1}), (Y, \mathbb{1})} = ((\mathbb{S}\euler{\Phi})_{X,Y})_{\mathbb{1},\mathbb{1}} = \euler{\Phi}_{\mathbf{M}\mathbb{1}X, \mathbf{M}\mathbb{1}Y} = \euler{\Phi}_{X,Y},
  \]
  concluding the proof.
\end{proof}

\begin{proposition}\label{UnitProposition}
 There is a $2$-transformation $\mathbb{n}: \mathbb{1}_{\ccf{C}\!\on{-Tamb}(\ccf{C},\ccf{C})^{\otimes\!\on{opp}}} \rightarrow \mathbb{SR}$
 such that $\mathbb{n}_{\mathcal{M}}$ is full and faithful for all $\mathcal{M} \in {\csym{C}\!\on{-Tamb}(\csym{C},\csym{C})^{\otimes\!\on{opp}}}\!\on{-Cat}$.
\end{proposition}

\begin{proof}
 By definition we have
 \begin{equation}\label{UnitActionEnrichment}
  \Hom{(X,\mathrm{F}), (Y,\mathrm{G})}_{\mathbb{SR}(\mathcal{M})}\!(\mathrm{K,L})\!=\!\Hom{\mathbb{R}(\mathcal{M})\mathrm{K} (X,\mathrm{F}), \mathbb{R}(\mathcal{M})\mathrm{L}(Y,\mathrm{G})}_{\mathbb{R}(\mathcal{M})}\!=\!\Hom{(X,\mathrm{KF}), (Y,\mathrm{LG})}_{\mathbb{R}(\mathcal{M})}\!=\!\Hom{X,Y}_{\mathcal{M}}\!(\mathrm{KF,LG})
 \end{equation}
 and
 \[
  (\mathsf{c}^{\mathbb{SR}(\mathcal{M})}_{(X,\mathbb{1}),(Y,\mathbb{1}),(Z,\mathbb{1})})_{\mathrm{F,G}} = \setj{\mathsf{c}^{\mathbb{R}(\mathcal{M})}_{(X,\mathrm{F}),(Y,\mathrm{H}), (Z,\mathrm{G})} \; | \; \mathrm{H} \in \csym{C}} = \setj{(\mathsf{c}^{\mathcal{M}}_{X,Y,Z})_{\mathrm{H;F,G}} \; | \; \mathrm{H} \in \csym{C}} = (\mathsf{c}_{X,Y,Z}^{\mathcal{M}})_{\mathrm{F,G}},
 \]
 where the middle two terms give the extranatural collections which specify maps from the coend that is indicated by giving the variable in which the collection is extranatural.
 Similarly one shows that $\mathsf{e}^{\mathbb{SR}(\mathcal{M})}_{(X,\mathbb{1})} = \mathsf{e}^{\mathcal{M}}_{X}$, which proves that the assignment $X \mapsto (X, \mathbb{1})$ on the level of objects, together with the identification of Equation~\eqref{UnitActionEnrichment} on the level of morphisms, define a full and faithful $\tam^{\otimes\!\on{opp}}$-functor $\mathbb{n}_{\mathcal{M}}$.

 To show $2$-naturality, we verify that for any $\tam^{\otimes\!\on{opp}}$-functor $\Phi: \mathcal{M} \rightarrow \mathcal{N}$, we have $\mathbb{SR}(\Phi) \circ \mathbb{n}_{\mathcal{M}} = \mathbb{n}_{\mathcal{N}} \circ \Phi$. On the level of objects, for any $X \in \mathcal{M}$, we have
 \[
  (\mathbb{SR}(\Phi) \circ \mathbb{n}_{\mathcal{M}})(X) = (\Phi X, \mathbb{1}) = (\mathbb{n}_{\mathcal{N}} \circ \Phi)(X)
 \]
 On the level of morphisms, we have
 \begin{equation}\label{CompositionIdentification}
  ((\mathbb{SR}\Phi \circ \mathbb{n}_{\mathcal{M}})_{X,Y})_{\mathrm{F,G}} = ((\mathbb{SR}\Phi )_{X,Y})_{\mathrm{F,G}} = (\mathbb{R}\Phi)_{(X,\mathrm{F}),(Y,\mathrm{G})} = (\Phi_{X,Y})_{\mathrm{F,G}} = (\mathbb{n}_{\mathcal{N}} \circ \Phi)_{X,Y},
 \end{equation}
 where the first and the last equality are the identification made in Equation~\eqref{UnitActionEnrichment}.
\end{proof}

\begin{proposition}\label{LastStep}
 For any $\mathbf{M} \in \csym{C}\!\on{-Mod}$, the $\tam^{\otimes\!\on{opp}}$-functor $\mathbb{n}_{\mathbb{S}(\mathbf{M})}$ is an equivalence.
\end{proposition}

\begin{proof}
 In view of Proposition~\ref{UnitProposition}, it suffices to show that $\mathbb{n}_{\mathbb{S}(\mathbf{M})}$ is essentially surjective. To this end, we show that, for any $X \in \mathbf{M}$ and $\mathrm{F} \in \csym{C}$, there is an isomorphism $(X,\mathrm{F}) \simeq (\mathbf{M}\mathrm{F}X,\mathbb{1})$.

 Following Equation~\eqref{UnitActionEnrichment}, we have the equality
 \[
  \Hom{(X,\mathrm{F}),(Y,\mathrm{G})}_{\mathbb{SRS}(\mathbf{M})}(\mathrm{K,L}) = \Hom{X,Y}_{\mathbb{S}(\mathbf{M})}(\mathrm{KF,LG}) = \Hom{\mathbf{M}\mathrm{KF}X, \mathbf{M}\mathrm{LG}Y} = [\mathbf{M}\mathrm{F}X, \mathbf{M}\mathrm{G}Y](\mathrm{K,L}),
 \]
 and thus the equality
  $\Hom{(X,\mathrm{F}),(Y,\mathrm{G})}_{\mathbb{SRS}(\mathbf{M})} = [\mathbf{M}\mathrm{F}X, \mathbf{M}\mathrm{G}Y]$ in $\tam^{\otimes\!\on{opp}}$. Similarly to Equation~\eqref{CompositionIdentification}, we show that the diagrams
  \begin{equation}\label{RealIdentification}
  \begin{tikzcd}[ampersand replacement=\&]
	{[\mathbf{M}\mathrm{F}X,\mathbf{M}\mathrm{F}X] \circ [\mathbf{M}\mathrm{F}X,\mathbf{M}\mathrm{F}X]} \& {\Hom{(X,\mathrm{F}),(\mathbf{M}\mathrm{F}X,\mathbb{1})}_{\mathbb{SRS}(\mathbf{M})}\circ \Hom{(\mathbf{M}\mathrm{F}X,\mathbb{1}),(X,\mathrm{F})}_{\mathbb{SRS}(\mathbf{M})}} \\
	{[\mathbf{M}\mathrm{F}X,\mathbf{M}\mathrm{F}X]} \& {\Hom{(X,\mathrm{F}),(X,\mathrm{F})}_{\mathbb{SRS}(\mathbf{M})}}
	\arrow["{\mathtt{m}^{[\mathbf{M}\mathrm{F}X,\mathbf{M}\mathrm{F}X]}}", from=1-1, to=2-1, swap]
	\arrow[from=1-1, to=1-2, equal]
	\arrow[from=2-1, to=2-2, equal]
	\arrow["{\mathsf{c}^{\mathbb{SRS}(\mathbf{M})}_{(X,\mathrm{F}), (\mathbf{M}\mathrm{F}X,\mathbb{1}),(X,\mathrm{F})}}"', from=1-2, to=2-2, swap]
\end{tikzcd}
\end{equation}
and
\[\begin{tikzcd}[ampersand replacement=\&]
	{[\mathbf{M}\mathrm{F}X,\mathbf{M}\mathrm{F}X] \circ [\mathbf{M}\mathrm{F}X,\mathbf{M}\mathrm{F}X]} \& {\Hom{(\mathbf{M}\mathrm{F}X,\mathbb{1}),(X,\mathrm{F})}_{\mathbb{SRS}(\mathbf{M})}\circ \Hom{(X,\mathrm{F}),(\mathbf{M}\mathrm{F}X,\mathbb{1})}_{\mathbb{SRS}(\mathbf{M})}} \\
	{[\mathbf{M}\mathrm{F}X,\mathbf{M}\mathrm{F}X]} \& {\Hom{(\mathbf{M}\mathrm{F}X,\mathbb{1}),(\mathbf{M}\mathrm{F}X,\mathbb{1})}_{\mathbb{SRS}(\mathbf{M})}}
	\arrow["{\mathtt{m}^{[\mathbf{M}\mathrm{F}X,\mathbf{M}\mathrm{F}X]}}", from=1-1, to=2-1,swap]
	\arrow[from=1-1, to=1-2]
	\arrow[from=2-1, to=2-2]
	\arrow["{\mathsf{c}^{\mathbb{SRS}(\mathbf{M})}_{(\mathbf{M}\mathrm{F}X,\mathbb{1}),(X,\mathrm{F}),(\mathbf{M}\mathrm{F}X,\mathbb{1})}}"', from=1-2, to=2-2,swap]
\end{tikzcd}\]
commute.

Indeed, we have
 \[
 \begin{aligned}
  &\left(\mathsf{c}^{\mathbb{SRS}(\mathbf{M})}_{(X,\mathrm{F}),(\mathbf{M}\mathrm{F}X,\mathbb{1}), (X,\mathrm{F})}\right)_{\mathrm{K,L}} = \setj{\mathsf{c}^{\mathbb{RS}(\mathbf{M})}_{(X,\mathrm{KF}),(\mathbf{M}\mathrm{F}X,\mathrm{H}),(X,\mathrm{LF})} \; | \; \mathrm{H}} = \setj{\left(\mathsf{c}^{\mathbb{S}(\mathbf{M})}_{X,\mathbf{M}\mathrm{F},X}\right)_{\mathrm{H;KF,LF}} \; | \; \mathrm{H}} \\
  &= \setj{\mathsf{c}^{\mathbf{M}}_{\mathbf{M}\mathrm{KF}X,\mathbf{M}\mathrm{H}\mathbf{M}\mathrm{F}X,\mathbf{M}\mathrm{LF}X} \; | \; \mathrm{H}}
  = \setj{\mathsf{c}^{\mathbf{M}}_{\mathbf{M}\mathrm{K}\mathbf{M}\mathrm{F}X,\mathbf{M}\mathrm{H}\mathbf{M}\mathrm{F}X,\mathbf{M}\mathrm{L}\mathbf{M}\mathrm{F}X} \; | \; \mathrm{H}} = \setj{\mathtt{m}^{[\mathbf{M}\mathrm{F}X,\mathbf{M}\mathrm{F}X]}_{\mathrm{H;K,L}} \; | \; \mathrm{H}} = \left(\mathtt{m}^{[\mathbf{M}\mathrm{F}X,\mathbf{M}\mathrm{F}X]}\right)_{\mathrm{K,L}}
 \end{aligned}
 \]
 and
 \[
 \begin{aligned}
  &\left(\mathsf{c}^{\mathbb{SRS}(\mathbf{M})}_{(\mathbf{M}\mathrm{F}X,\mathbb{1}),(X,\mathrm{F}),(\mathbf{M}\mathrm{F}X,\mathbb{1})}\right)_{\mathrm{K,L}} = \setj{\mathsf{c}^{\mathbb{RS}(\mathbf{M})}_{(\mathbf{M}\mathrm{F}X,\mathrm{K}),(X,\mathrm{HF}),(\mathbf{M}\mathrm{F}X,\mathrm{L})} \; | \; \mathrm{H}} = \setj{\left(\mathsf{c}_{\mathbf{M}\mathrm{F}X,X,\mathbf{M}\mathrm{F}X}\right)_{\mathrm{HF;K,L}}\; | \; \mathrm{H}} \\
  & = \setj{\mathsf{c}^{\mathbf{M}}_{\mathbf{M}\mathrm{KF}X,\mathbf{M}\mathrm{HF}X,\mathbf{M}\mathrm{LF}X}}
   = \left(\mathtt{m}^{[\mathbf{M}\mathrm{F}X,\mathbf{M}\mathrm{F}X]}\right)_{\mathrm{K,L}}.
 \end{aligned}
 \]
 Next, we show that the diagram
 \begin{equation}\label{UnitIdentification}
 \begin{tikzcd}[ampersand replacement=\&,row sep=scriptsize]
	{\csym{C}} \& {\Hom{(X,\mathrm{F}),(X,\mathrm{F})}_{\mathbb{SRS}(\mathbf{M})}} \\
	\& {[\mathbf{M}\mathrm{F}X,\mathbf{M}\mathrm{F}X]}
	\arrow["{\mathsf{e}^{\mathbb{SRS}(\mathbf{M})}_{(X,\mathrm{F})}}", from=1-1, to=1-2]
	\arrow[from=1-2, to=2-2]
	\arrow["{\mathtt{e}^{[\mathbf{M}\mathrm{F}X,\mathbf{M}\mathrm{F}X]}}"', from=1-1, to=2-2]
\end{tikzcd}
 \end{equation}
commutes. Indeed, for any $\mathrm{k} \in \csym{C}(\mathrm{K,L})$, we have
\[
 \begin{aligned}
  &(\mathsf{e}^{\mathbb{SRS}(\mathbf{M})}_{(X,\mathrm{F})})_{\mathrm{K,L}}(\mathrm{k}) = (\mathbb{RS}(\mathbf{M})\mathrm{k})_{(X,\mathrm{F})} = (\mathsf{e}^{\mathbb{S}(\mathbf{M})}_{X})_{\mathrm{KF,LF}}(\mathrm{kF}) \\
  &= (\mathtt{e}^{[X,X]})_{\mathrm{KF,LF}}(\mathrm{kF}) = (\mathbf{M}\mathrm{kF})_{X} = (\mathbf{M}\mathrm{k})_{\mathbf{M}\mathrm{F}X} = (\mathtt{e}^{[\mathbf{M}\mathrm{F}X,\mathbf{M}\mathrm{F}X]})_{\mathrm{K,L}}(\mathrm{k}).
 \end{aligned}
\]
Consider the following diagram:
\[
 \resizebox{.99\hsize}{!}{$
\begin{tikzcd}[ampersand replacement=\&,row sep=scriptsize]
	{[\mathbf{M}\mathrm{F}X,\mathbf{M}\mathrm{F}X] \circ [\mathbf{M}\mathrm{F}X,\mathbf{M}\mathrm{F}X]} \&\& {\Hom{(X,\mathrm{F}),(\mathbf{M}\mathrm{F}X,\mathbb{1})}_{\mathbb{SRS}(\mathbf{M})} \circ \Hom{(\mathbf{M}\mathrm{F}X,\mathbb{1}),(X,\mathrm{F})}_{\mathbb{SRS}(\mathbf{M})}} \\
	\& {\csym{C}(-,-)\circ \csym{C}(-,-)} \\
	\& {\csym{C}(-,-)} \\
	{[\mathbf{M}\mathrm{F}X,\mathbf{M}\mathrm{F}X]} \&\& {\Hom{(X,\mathrm{F}),(X,\mathrm{F})}_{_{\mathbb{SRS}(\mathbf{M})}}}
	\arrow[""{name=0, anchor=center, inner sep=0}, "{\mathsf{c}^{\mathbb{SRS}(\mathbf{M})}_{(X,\mathrm{F}),(\mathbf{M}\mathrm{F}X,\mathbb{1}), (X,\mathrm{F})}}"', swap, from=1-3, to=4-3]
	\arrow[from=1-1, to=1-3, equal]
	\arrow["{\mathsf{c}^{\cccsym{C}}}"', from=2-2, to=3-2]
	\arrow[""{name=1, anchor=center, inner sep=0}, "{\mathtt{e}^{[\mathbf{M}\mathrm{F}X,\mathbf{M}\mathrm{F}X]}}"{pos=0.6}, from=2-2, to=1-1]
	\arrow[""{name=2, anchor=center, inner sep=0}, "{\mathtt{m}^{[\mathbf{M}\mathrm{F}X,\mathbf{M}\mathrm{F}X]}}"', from=1-1, to=4-1]
	\arrow["{\mathtt{e}^{[\mathbf{M}\mathrm{F}X,\mathbf{M}\mathrm{F}X]}}"', from=3-2, to=4-1]
	\arrow[""{name=3, anchor=center, inner sep=0}, from=4-1, to=4-3,equal]
	\arrow["{\mathsf{e}^{\mathbb{SRS}(\mathbf{M})}_{(X,\mathrm{F})}}", from=3-2, to=4-3]
	\arrow["1", shift left=3, draw=none, from=1, to=2]
	\arrow["2", draw=none, from=2-2, to=0]
	\arrow["3", draw=none, from=3-2, to=3]
\end{tikzcd}$}
\]
Its outer face commutes due to the commutativity of Diagram~\eqref{RealIdentification}. Face 1 commutes due to unitality of multiplication in $[\mathbf{M}\mathrm{F}X,\mathbf{M}\mathrm{F}X]$, and face 3 commutes due to the commutativity of Diagram~\eqref{UnitIdentification}. Chasing the element $\on{id}_{\mathbb{1}} \otimes \on{id}_{\mathbb{1}}$ in $(\csym{C}(-,-) \circ \csym{C}(-,-))(\mathbb{1,1})$, we find that ${\mathsf{c}^{\mathbb{SRS}(\mathbf{M})}_{(X,\mathrm{F}),(\mathbf{M}\mathrm{F}X,\mathbb{1}), (X,\mathrm{F})}}(\on{id}_{\mathbf{M}\mathrm{F}X} \otimes \on{id}_{\mathbf{M}\mathrm{F}X}) = \mathsf{e}^{\mathbb{SRS}(\mathbf{M})}_{(X,\mathrm{F})}(\on{id}_{\mathbb{1}}) = \on{id}_{(X,\mathrm{F})}$.

Similarly one shows that ${\mathsf{c}^{\mathbb{SRS}(\mathbf{M})}_{(\mathbf{M}\mathrm{F}X,\mathbb{1}), (X,\mathrm{F}),(\mathbf{M}\mathrm{F}X,\mathbb{1})}}(\on{id}_{\mathbf{M}\mathrm{F}X} \otimes \on{id}_{\mathbf{M}\mathrm{F}X}) = \mathsf{e}^{\mathbb{SRS}(\mathbf{M})}_{(X,\mathrm{F})}(\on{id}_{\mathbb{1}}) = \on{id}_{(\mathbf{M}\mathrm{F}X,\mathbb{1})}$. Thus, $\on{id}_{\mathbf{M}\mathrm{F}X}$ yields mutually inverse morphisms between $(X,\mathrm{F})$ and $(\mathbf{M}\mathrm{F}X,\mathbb{1})$ in $\mathbb{SRS}(\mathbf{M})$.
\end{proof}

Combining Proposition~\ref{OtherUnit}, Proposition~\ref{UnitProposition} and Proposition~\ref{LastStep}, we obtain the following:
\begin{theorem}\label{ActionViaEnrichment}
The $2$-functor $\mathbb{S}: \csym{C}\!\on{-Mod} \rightarrow \tam^{\otimes\!\on{opp}}\!\on{-Cat}$ is a $2$-equivalence onto its essential image.
\end{theorem}

\vspace{5mm}

\noindent
Department of Mathematics, Uppsala University, Box. 480, SE-75106, Uppsala, SWEDEN, \\
email: {\tt mateusz.stroinski\symbol{64}math.uu.se}

\begin{thebibliography}{99999999999}
 \bibitem[ASS]{ASS} I.~Assem, D.~Simson, A.~Skowro{\' n}ski, Elements of the representation theory of associative algebras. Vol. 1. Techniques of representation theory. London Mathematical Society Student Texts, {\bf 65}. Cambridge University Press, Cambridge, 2006.
\bibitem[BGG]{BGG} J. Bernstein, I. Gelfand, S. Gelfand. A certain category of $\mathbb{g}$-modules. Funkcional. Anal. i Prilozen. 10 (1976), no. 2, 1–8
\bibitem[Bo]{Bo} F.~Borceux. Handbook of Categorical Algebra: Volume 1, Basic Category Theory.
Encyclopedia of Mathematics Vol. 50, Cambridge Univ. Press, 1994.
\bibitem[BD]{BD} F.~Borceux, D.~Dejean. Cauchy completion in category theory. Cahiers de Topologie et G{\' e}om{\' e}trie Diff{\' e}rentielle Cat{\' e}goriques, Tome {\bf 27}, 133-146 (1986).
\bibitem[Cr]{Cr} L.~Crane. Clock and category: is quantum gravity
algebraic?  J. Math. Phys. {\bf 36}  (1995),  no. 11, 6180--6193.
\bibitem[CF]{CF} L.~Crane, I.~Frenkel. Four-dimensional topological quantum field theory, Hopf categories, and the canonical bases. Topology and physics. J. Math. Phys. {\bf 35} (1994), no. 10, 5136--5154.
\bibitem[CEGLMPR]{CEGLMPR} B. Clarke, D. Elkins, J. Gibbons, F. Loregian, B. Milewski, E. Pillmore, and M. Rom{\' a}n, Profunctor optics: a categorical update. Preprint, arXiv: 2001.07488
\bibitem[Da]{Day} B. Day, Construction of Biclosed Categories, PhD thesis. School of Mathematics of the University of New South Wales, September 1970.
\bibitem[DS{-}PS]{DSPS} C. L. Douglas, C. Schommer-Pries, and N. Snyder. The balanced tensor product of module categories. Kyoto J. Math., {\bf 59}(1):167–179, 2019.
\bibitem[EGNO]{EGNO} P.~Etingof, S.~Gelaki, D.~Nikshych, V.~Ostrik. Tensor categories.
Mathematical Surveys and Monographs {\bf 205}. American Mathematical Society, Providence, RI, 2015.
\bibitem[FSS]{FSS} J.~Fuchs, G.~Schaumann,  C.~Schweigert. Eilenberg-Watts calculus for finite categories and a bimodule Radford $S^{4}$ theorem. Trans. Amer. Math. Soc., 373(1):1–40, 2020
\bibitem[GJ]{GJ} N.~Gambino, A.~Joyal. On operads, bimodules and analytic functors. Mem. Amer. Math. Soc. {\bf 289} (2017), no.1184, v+110pp.
\bibitem[Ga]{Ga} R.~Garner. Bicategories, course given at Research School on Bicategories, Categorification and Quantum Theory, University of Leeds, 2022. Notes available at \url{https://conferences.leeds.ac.uk/bcqt2022/}.
\bibitem[GS]{GS} R.~Garner, M.~Shulman. Enriched categories as a free cocompletion. Adv. Math. {\bf 289} (2016), 1-94.
\bibitem[GP]{GP} R.~Gordon, A.~J.~Power. Enrichment through variation. J. Pure Appl. Algebra, {\bf 120}(2):167–185, 1997.
\bibitem[IK]{IK} G .B. Im and G. M. Kelly, A universal property of the convolution monoidal structure, J. Pure Appl. Algebra {\bf 43} (1986), no. 1, 75-88
\bibitem[JK]{JK} G.~Janelidze and G.M.~Kelly. A note on actions of a monoidal
category. Theory Appl. Categ., {\bf 9}:No 2, 61-91, 2001.
\bibitem[JY]{JY} N.~Johnson, D.~Yau. 2-Dimensional Categories. New York: Oxford University Press, 2021.
\bibitem[KL]{KL} T.~Kerler, V.~V.~Lyubashenko. Non-semisimple topological quantum field theories for
3-manifolds with corners. Lecture Notes in Math. {\bf 1765} (2001).
 \bibitem[Kh]{Kh} M.~Khovanov. A categorification of the Jones polynomial. Duke Math. J.
 {\bf 101} (2000), no. 3, 359--426.
\bibitem[La]{La} F.~W.~Lawvere. Metric spaces, generalized logic, and closed categories. Seminario Mat. e. Fis. di Milano {\bf 43}, 135–166 (1973).
\bibitem[Lo]{Lo} F.~Loregian. Coend Calculus. London Mathematical Society Lecture Note Series {\bf 468} (2021).
\bibitem[MM1]{MM1} V. Mazorchuk, V. Miemietz. Cell 2-representations of finitary 2-categories.
Compositio Math \textbf{147}  (2011), 1519--1545.
\bibitem[MM2]{MM5} V.~Mazorchuk, V.~Miemietz. Transitive $2$-representations of finitary $2$-categories.
Trans. Amer. Math. Soc. {\bf 368} (2016), no. 11, 7623--7644.
\bibitem[MMMT]{MMMT} M.~Mackaay, V.~Mazorchuk, V.~Miemietz, D.~Tubbenhauer. Simple transitive 2-representations
via (co)algebra $1$-morphisms. Indiana Univ. Math. J. {\bf 68} (2019), no. 1, 1--33.
\bibitem[MMMTZ]{MMMTZ1} M. Mackaay, V. Mazorchuk, V. Miemietz, D. Tubbenhauer, X. Zhang. Finitary birepresentations of finitary bicategories. Forum Math. {\bf 33} (5) (2021), 1261--1320
\bibitem[MacL]{MacL} S. MacLane. Categories for the Working Mathematician. Springer-Verlag, 1998.
\bibitem[Os]{Os} V.~Ostrik. Module categories, weak Hopf algebras and modular invariants. Transform. Groups {\bf 8} (2003), no. 2, 177-206.
\bibitem[PS]{PS} C.~Pastro, R.~Street. Doubles for monoidal categories, Theory Appl. Categ., {\bf21}:No.
4, 61–75, 2008.
\bibitem[Ros]{Ros} G. Rosolini. A note on Cauchy completeness for preorders. Riv. Mat. Univ.
Parma (6), 2*:89–99 (2000), 1999.
\bibitem[Rou]{Rou} R. Rouquier. 2-Kac-Moody algebras. Preprint arXiv:0812.5023.
\bibitem[S{-}P]{SP} C.~Schommer-Pries. The classification of two-dimensional extended topological field theories. Ph.D. Thesis, University of California, Berkeley, 2009.
\bibitem[So]{So} W. Soergel. Kategorie $\mathcal{O}$, perverse Garben und Moduln {\" u}ber den Koinvarianten zur Weylgruppe. J. Amer. Math. Soc. 3 (1990), no. 2, 421–445.
\bibitem[Str]{Str} M.~Stroi{\' n}ski. Weighted colimits of 2-representations and star algebras. Preprint arXiv: 2106.14452
\bibitem[Ta]{Ta} D. Tambara. Distributors on a tensor category, Hokkaido Mathematics
Journal {\bf 35} (2006) pp. 379–425.
\bibitem[Wo]{Wo} R.J. Wood. Abstract pro arrows I. Cahiers de Topologie et G{\' e}om{\' e}trie Diff{\' e}rentielle Cat{\' e}goriques, Tome {\bf 23} (1982) no. 3, pp. 279-290.
\end{thebibliography}
\end{document}